\documentclass{article}
\usepackage{amssymb,amsmath,amsthm,enumerate,graphicx,mathtools,subcaption, url}
\usepackage[margin=1in]{geometry}

\usepackage{pgfbaseplot,pgfplots,tikz,ifthen,tikz-3dplot}
\usetikzlibrary{decorations.text,hobby,arrows.meta}

\usepackage[makeroom]{cancel}
\usepackage[margin=1in]{geometry}

\newtheorem*{theorem*}{Theorem}
\newtheorem{theorem}{Theorem}
\newtheorem{lemma}[theorem]{Lemma}
\newtheorem*{lemma*}{Lemma}
\newtheorem{proposition}[theorem]{Proposition}
\newtheorem{corollary}[theorem]{Corollary}
\newtheorem{conjecture}[theorem]{Conjecture}
\newtheorem{remark}[theorem]{Remark}
\newtheorem{definition}[theorem]{Definition}
\newtheorem{notation}[theorem]{Notation}

\def\E{\mathcal{E}}
\def\cD{\mathcal{D}}
\def\R{\mathcal{R}}
\def\L{\mathcal{L}}
\def\F{\mathcal{F}}
\def\Id{\mathrm{Id}}
\def\rr{\rho}
\def\kk{\kappa}
\def\de{\delta}

\def\wt{\widetilde}

\def\Diff{\mathrm{Diff}}

\def\L{\mathcal{L}}

\def\R{\mathcal{R}}
\def\F{\mathcal{F}}

\def\Id{\mathrm{Id}}
\def\S{\mathcal{S}}
\def\T{\mathbb{T}}
\def\ii{\mathrm{i}}
\def\nn{\mathrm{n}}

\mathtoolsset{showonlyrefs}

 \author{
   Andrew Clarke
   \and
   Jacques Fejoz
   \and
   Marcel Guardia
 }

\newcommand{\Addresses}{{% additional braces for segregating \footnotesize
  \bigskip
  \footnotesize

  \textsc{Andrew Clarke, \newline Departament de Matem\`atiques i Inform\`atica, Universitat de Barcelona,   Gran Via, 585, 08007 Barcelona, Spain}\par\nopagebreak
  \textit{E-mail address}: \texttt{andrew.clarke@ub.edu}

  \medskip

  \textsc{Jacques Fejoz, \newline Universit\'e Paris Dauphine PSL, CEREMADE,  Place du Mar\'echal de Lattre de Tassigny,  75016 Paris, France \newline  Observatoire de Paris PSL, IMCCE,  77 avenue Denfert Rochereau, 75014 Paris, France}\par\nopagebreak
  \textit{E-mail address}: \texttt{jacques.fejoz@dauphine.psl.eu}

  \medskip

  \textsc{Marcel Guardia, \newline Departament de Matem\`atiques i Inform\`atica, Universitat de Barcelona,
  Gran Via, 585, 08007 Barcelona, Spain \newline 
  Centre de Recerca Matem\`atica, 
  Edifici C, Campus Bellaterra,  08193 Bellaterra, Spain}\par\nopagebreak
  \textit{E-mail address}: \texttt{guardia@ub.edu}

}}

\date{\today}

\title{Why are inner planets not inclined?}

\begin{document}
\maketitle

\begin{abstract}
  \label{sec:abstract}
  Poincar\'e's work more than one century ago, or Laskar's numerical simulations from the 1990's on, have irrevocably impaired the long-held belief that the Solar System should be stable. But mathematical mechanisms explaining this instability have remained mysterious. In 1968, Arnold conjectured the existence of ``Arnold diffusion'' in celestial mechanics. We prove Arnold's conjecture in the planetary spatial $4$-body problem as well as in the corresponding hierarchical problem (where the bodies are increasingly separated), and show that this diffusion leads, on a long time interval, to some large-scale instability. Along the diffusive orbits, the mutual inclination of the two inner planets is close to $\pi/2$, which hints at why even marginal stability in planetary systems may exist only when inner planets are not inclined. 

More precisely, consider the normalised angular momentum of the second planet, obtained by rescaling the angular momentum by the square root of its semimajor axis and by an adequate mass factor (its direction and norm give the plane of revolution and the eccentricity of the second planet). It is a vector of the unit $3$-ball. We show that any finite sequence in this ball may be realised, up to an arbitrary precision, as a sequence of values of the normalised angular momentum in the $4$-body problem. For example, the second planet may flip from prograde nearly horizontal revolutions to retrograde ones. As a consequence of the proof, the non-recurrent set of any finite-order secular normal form accumulates on circular motions -- a weak form of a celebrated conjecture of Herman.

% Our proof relies on the existence of a hyperbolic invariant cylinder inherited from a hyperbolic singularity of the quadrupolar secular system (the Keplerian approximation being degenerate); on the control of the distorsion of the inner dynamics along the cylinder; on the control of the splitting of its stable and unstable manifolds in relevant directions; on showing that other directions may be discarded to build an outer dynamics of the cylinder (scattering map) obtained by a limit process of excursions along the unstable and stable foliations (here we use an ad hoc theorem that we have defered to another article); and on shadowing pseudo-orbits built by randomly concatenating the inner and outer dynamics. 

%%% Local Variables: 
%%% mode: latex
%%% TeX-master: "4bp_secular_diffusion_v6.tex"
%%% End: 

\end{abstract}

\tableofcontents

\section{Introduction}
\label{sec:intro}
% To do:
% - Regularize Deprit coordinates at coplanar motions
% - Merge the two main thms in Deprit coord?
% - Clean up references to papers of Chierchia-Pinzari (JMD useless)

\subsection{A case for instability in the solar system}
\label{sec:solar-system}

Hook's and Newton's discovery of universal attraction in the \textsc{xvii} century masterly reconciles two seemingly contradictory physical principles: the principle of inertia, put forward by Galileo and Descartes in terrestrial mechanics, and the laws of Kepler, governing the elliptical motion of planets around the Sun \cite{Albouy:2013, Arnold:1990, Newton:1687}.  The unforeseen mathematical consequence of Hook’s and Newton's discovery was to question the belief that the solar system be stable: it was no longer obvious that planets kept moving immutably, without collisions or ejections, because of their mutual (``universal'') attraction. Newton himself, in an additional and staggering tour de force, estimated the first order effect on Mars of the attraction of other planets. But infinitesimal calculus was in its infancy and the necessary mathematical apparatus to understand the long-term influence of mutual attractions did not exist.

In their study of Jupiter's and Saturn's motions, Lagrange, Laplace, Poisson and others laid the foundations of the Hamiltonian theory of the variation of constants. They  managed to compute the secular dynamics, i.e. the slow deformations of Keplerian ellipses, at the first order with respect to the masses, eccentricities and inclinations of the planets. This level of approximation is still integrable, and Keplerian ellipses have slow but non vanishing precession and rotation frequencies, thus departing from the dynamical degeneracy decribed by Bertrand's theorem. Besides, the analysis of the spectrum of the linearised vector field entailed a resounding stability theorem for the solar system: the observed variations of adiabatic invariants in the motion of Jupiter and Saturn come from resonant terms of large amplitude and long period, but with zero average~(\cite[p.~164]{Laplace:1785}, \cite{laskar2006Lagrange}). Yet it is a mistake, which Laplace made, to infer the topological stability of the non-truncated planetary system.

In the \textsc{xviii} and \textsc{xix} centuries, mathematicians spent an inordinate amount of energy trying to prove the stability of the Solar system... until Poincar\'e discovered a remarkable set of arguments strongly speaking \emph{against} stability: generic divergence of perturbation series, non-integrability of the three-body problem, and entanglement of the stable and unstable manifolds of the Lagrange relative equilibrium in the restricted 3-body problem~\cite{poincare1892methodes}. 

In the mid \textsc{xx} century, Siegel and Kolmogorov still proved that, respectively for the linearisation problem of a one-dimensional complex map and for the perturbation of an invariant torus of fixed frequency in a Hamiltonian system, perturbation series do converge, albeit non uniformly, under some arithmetic assumption of Diophantine type, ensuring that the frequencies of the motion are far from low order resonances, in a quantitative way. The obtained solutions are quasiperiodic and densely fill Lagrangian invariant tori. They form a large set in the measure theoretic sense, but a small set from the topological viewpoint. Besides, starting from dimension $6$, invariant tori do not separate energy levels and thus do not confine neighboring motions, so, outside invariant tori, nothing prevents adiabatic invariants to drift. Kolmogorov's theorem was successfully adapted to the planetary system, despite the numerous degeneracies of the latter, and assuming that the masses of the planets are very small~\cite{Arnold:1963,Chierchia:2011,fejoz2004arnold,Robutel:1995}. The obtained solutions are small perturbations of (Diophantine) Laplace-Lagrange motions.

\medskip Soon afterward, Arnold imagined an example of dynamical instability in a near-integrable Hamiltonian system with many degrees of freedom, where action variables may drift (for some well chosen orbits), by an amount uniform with respect to the smallness of the perturbation~\cite{arnold1964instability}. Of course, the drifting time tends to infinity as the size of the perturbation tends to $0$, consistently with the continuity of the time-$t$ map of the flow  with respect to parameters. Drifting orbits shadow the stable and unstable manifolds of a chain of hyperbolic invariant tori (``transition chain''). This phenomenon has been called \emph{Arnold diffusion}, since Chirikov coined the phrase, referring to the (in part conjectural) stochastic properties of such a dynamics~\cite{Chirikov:1959}. In fact, in this seminal paper, Arnold conjectured the following.

\begin{conjecture}[Arnold~\cite{arnold1964instability}]
  \label{conj:arnold}
  The mechanism of transition chains [...] is also applicable to the case of general Hamiltonian systems (for example, to the problem of three bodies).
\end{conjecture}

Arnold's example has proved difficult to generalise because of the so-called \emph{large gap problem}: usually the transition chain is a (totally disconnected) Cantor set of hyperbolic tori and it is not obvious whether there exist orbits shadowing these tori. A better strategy has emerged, consisting in shadowing normally hyperbolic cylinders (whether they contain invariant tori or not). Nearly integrable Hamiltonian systems are usually classified as \emph{a priori unstable} and \emph{a priori stable} \cite{CherchiaG94}. A priori unstable models are those whose integrable approximation presents some hyperbolicity (the paradigmatic example being a pendulum weakly coupled with several rotators). In this case, the unperturbed model has a normally hyperbolic invariant manifold with attached invariant manifolds that one can use, for the perturbed model, as a ``highway'' for diffusing orbits. The existence of Arnold diffusion generically in these models is nowadays rather well understood, at least for two and a half degrees of freedom (see \cite{moeckel2002drift,ChengY04,Treschev04,delshams2006biggaps,gidea2006topological,Bernard08, DelshamsH05}, or \cite{Treschev:2012,MR3479576,MR4033892} for results in higher dimension).

A priori stable systems are those whose integrable approximation is foliated by quasiperiodic invariant Lagrangian tori. Since the unperturbed Hamiltonian does not possess hyperbolic invariant objects, in order to construct the diffusing ``highway'' one has to rely on a first perturbation and face involved singular perturbation problems. One of the difficulties is that one cannot avoid double resonances, where the system is intrinsically non-integrable. Arnold conjecture refers to these models. The work of Mather on minimizing measures has been deeply influential. In the finite smoothness category, the papers \cite{Kaloshin:2016,Cheng:2017,Kaloshin:2020} show the typicality (in the cusped residual sense as defined by J. Mather) of Arnold diffusion in a priori stable Hamiltonian systems of 3 degrees of freedom. Yet, many questions remain unsolved. In particular, the original Arnold conjecture on the typicality of Arnold diffusion for analytic non degenerate nearly integrable Hamiltonian systems of 3 or more degrees of freedom remains open (see however \cite{MR1784083,MR1806373}).\footnote{One can also consider the so called \emph{a priori chaotic case}, where the unperturbed Hamiltonian presents ``local non-integrability''. In particular, it has a first integral and a periodic orbit with transverse homoclinics at each energy level. Examples of such settings are certain geodesic flows with a time dependent potential, see \cite{Bolotin:1999,clarke2022arnold,DelshamsLS00,DelshamsLS06b,Gelfreich:2008,GelfreichTuraev2017}.}

\medskip In the 1990s, with extensive numerical computations Laskar showed that over the physical life span of the Sun, or even over a few hundred million years, collisions and ejections of inner planets occur with some probability~\cite{laskar1989chaos,laskar2010}.\footnote{Such long term computations are checked to pass various consistency tests (e.g. the preservation of first integrals). But due to the exponential divergence of solutions, they are statistical in nature: an uncertainty of a few centimeters on the initial position of the Earth leads to an uncertainty of the size of the Solar System after a few hundred millions years. But one likes to believe that such Hamiltonian systems have good shadowing properties, i.e. that any finite-time pseudo-orbit (as computed numerically) is shadowed by orbits.}  Our solar system is now believed only marginally stable. This has been corroborated by abundant numerical evidence, as overviewed in Morbidelli's book~\cite{morbidelli2002}. In particular, the effect of mean motion (Keplerian) resonances in the asteroid belt has been described by \cite{lecar2001chaos}. Numerical evidence has also been suggesting that secular resonances are a major source of chaos in the Solar system~\cite{Laskar:1993:chaotic, Laskar:2008:chaotic, Froeschle:1989}. For example, astronomers have established that Mercury's eccentricity is chaotic and can increase so much that collisions with Venus or the Sun become possible, as a result from an intricate network of secular resonances~\cite{Boue:2012:simple}. On the other hand, that Uranus's obliquity ($97^o$) is essentially stable, is explained, to a large extent, by the absence of any low-order secular resonance~\cite{Boue:2010:collisionless, Laskar:1993:chaotic}. The effects of secular resonances of the inner planets have later been studied systematically using both computer algebra and numerics and the main ``sources of chaos'' in the inner solar system have been identified~\cite{batygin2015chaotic,Mogavero:2022}.

\medskip The mathematical theory of instability remains in its infancy and the Astronomers' instability mechanisms are still mysterious. A matter of discontent with Arnold diffusion is that the time needed for actions to drift looks larger than in other, far from integrable, instability mechanisms that astronomers observe. Resonance overlapping, a phenomenon described by Chirikov~\cite{Chirikov:1959}, would be a fantastic competing mechanism. But to our knowledge it lacks mathematical explanation (see however a simple example in~\cite{Fejoz:2018:overlap}).

Regarding ``the oldest problem in dynamical systems'', in his ICM lecture~\cite{herman1998icm} Herman formulated the following two precise conjectures.  Consider the $N$-body problem in space, with $N\geq 3$. Assume that the center of mass is fixed at the origin and that on the energy surface of level $e$ we $C^\infty$-reparametrise the flow by a $C^\infty$ function $\varphi_e$ such that %the flow is complete : we replace $H$ by $H_e = \varphi_e.(H- e)$ so that
the collisions now occur only in infinite time ($\varphi_e > 0$ is a $C^\omega$ function outside collisions).\footnote{Here indeed it is natural to count the non-wandering set without collision orbits, since positions do not necessarily go to infinity at collisions. Herman furthermore claims that this reparametrised flow is complete. This is not clear due to the potential presence of non-collision singularities. But conjecture~\ref{conj:global} remains relevant with a possibly incomplete flow.}

%Then, M. Herman asked the following question
%(see also A.N. Kolmogorov \cite{Kolmogorov54,abraham1978foundations}).
% 7.3 - Following G.D. Birkho [B3] (who only considers the case n = 3 and the
% angular momentum 6= 0) (see also A.N. Kolmogorov [ICM 54]), we ask :
\begin{conjecture}[Global instability]
  \label{conj:global}
  Is for every $e$ the non-wandering set of the Hamiltonian flow of $H_e$ on $H_e^{-1} (0)$ nowhere dense in $H_e^{-1} (0)$?
\end{conjecture}

This would imply that bounded orbits are nowhere dense and no topological stability occurs.\footnote{This incidentally contradicts a conjecture of Poincar\'e, that periodic orbits are dense \cite[End of section 36]{poincare1892methodes}.} The conjecture is wide open, and we are still at the stage of looking for non-recurrent orbits having negative energy (due to the Lagrange-Jacobi identity, non-wandering orbits have negative energy) \cite{Fejoz:2015:nbp}. A growing number of such orbits are known to exist, most often associated with some kind of symbolic dynamics: close to parabolic motions~\cite{Alekseev68,delshams2019instability, Guardia:2013:oscillatory, guardia2016oscillatory, GuardiaPSV20, LlibreS80, moser1973stableunstable, SimoL80, Sitnikov60}, close to triple collisions~\cite{moeckel1989chaotic, moeckel2007symbolicdynamics}, close to Poincar\'e's periodic orbits of second species \cite{bolotin2006symbolic}, close to the Euler-Lagrange points \cite{Capinski12}, on the submanifold of zero angular momentum of the $3$-body problem~\cite{Montgomery:2007, Montgomery:2019}, and so on. It is part of the richness of the $N$-body problem to include such diverse kinds of behavior.

Yet, scarce mathematical mechanisms have been described regarding more astronomical regimes, which would be plausible for subsystems of solar or extra-solar systems. Within the domain of negative energy, of special astronomical relevance is the planetary problem, where planets with small masses revolve around the Sun.\footnote{Jupiter, the largest planet of the solar system, weighs roughly $1/1000$ the mass of the Sun.} Another well-known problem is the hierarchical problem, an extension to $N$-bodies of the so-called Hill problem or lunar problem, where one body (the Sun) revolves far away around the other two (the Earth and the Moon).\footnote{The lunar distance is roughly $1/400$ the distance from Earth to Sun.} These problems are rendered difficult by the proximity to a degenerate (``super-integrable'') integrable system of two uncoupled Kepler problems. As mentioned above, the recurrent set has positive Lebesgue measure due to Arnold-Herman's invariant tori theorem, while the existence of unstable orbits relies on the so-called Arnold diffusion. Herman further argues that \emph{What seems not an unreasonable question to ask (and possibly prove in a finite time with a lot of technical details) is that:}

\begin{conjecture}[Planetary instability]
  \label{conj:local}
  If one of the masses is fixed ($m_0 = 1$) and the other masses $m_j =\rr \widetilde m_j$, $1 \leq j \leq n - 1$, $\widetilde m_j > 0$, $\rr > 0$, then in any neighbourhod of fixed different circular orbits around $m_0$ moving in the same direction in a plane, when $\rho$ is small, there are wandering domains. 
\end{conjecture}

Even if Arnold in his seminal paper conjectured that his Arnold diffusion mechanism via transition chains should be present in the 3 body problem, even nowadays the results in this direction are rather scarce. Indeed, as far as the authors know, the \emph{only complete analytical} proof of Arnold diffusion in Celestial Mechanics, prior to the present work, is contained in an article by  Delshams, Kaloshin, de la Rosa and Seara~\cite{delshams2019instability}. In this paper, the authors consider the Restricted Planar Elliptic 3 Body Problem and construct orbits with large drift in angular momentum, assuming the mass ratio and the eccentricities of the primaries are sufficiently small. A key point of this paper is that the body of zero mass is close to the so called parabolic motions. That is, it relies on the invariant manifolds of infinity which are already present in the two body problem. They then perform a delicate perturbative analysis of the model. 

Some works have uncovered instability mechanisms in celestial mechanics related to Arnold diffusion, relying on computer assisted computations (see \cite{capinski2017diffusion,fejoz2016kirkwood} and also \cite{CapinskiGidea} which relies on computer assisted proofs) or conditionally to a plausible transversality hypothesis~\cite{Xue:2014:4bp}. The article~\cite{fejoz2016secular} proves the existence of some hyperbolic invariant set with symbolic dynamics in the hierarchical truncated secular spatial $3$-body problem, with no hope of finding Arnold diffusion in the full $3$-body problem because of the lack of additional slow degrees of freedom.

Regarding Conjecture \ref{conj:local},  the construction of wandering domains using Arnold diffusion-like mechanisms is a difficult problem. As far as we know, the only positive result in this direction is in \cite{Lazzarini19}, where wandering domains for Gevrey nearly integrable symplectic maps are constructed. The methods used in that paper to construct such domains do not admit an immediate extension to the analytic category.

The present work 
\begin{itemize}
\item proves Arnold's Conjecture~\ref{conj:arnold} in the planetary spatial $4$-body problem (see Theorem \ref{thm:main} below)
\item and proves a weak local version of Herman's Conjecture~\ref{conj:local}, dealing with non-recurrent orbits instead of wandering orbits
(Theorem~\ref{thm:WLHC}), for the corresponding secular dynamics.
\end{itemize}

\subsection{Main results}
\label{sec:main-short}

Consider the $4$-body problem, that is $4$  bodies numbered from $0$ to $3$ moving in $3$-dimensional space according to the Newtonian gravitational law, 
\[
 \ddot x_j=\sum_{\substack {0 \leq i \leq 3  \\ i\neq j}} m_i\frac{x_i-x_j}{\|x_i-x_j\|^3},
\]
where $x_j \in \mathbb{R}^3$ is the position and $m_j>0$ is the mass of body $j$ for $j=0,1,2,3$.

For the sake of simplicity, let us first focus on the ``hierarchical regime'' where body $2$ revolves around and far away from bodies $0$ and $1$, while body $3$ revolves around and even farther away from bodies $0$, $1$ and $2$. Each body thus primarily undergoes the attraction of one other body: bodies $0$ and $1$ are close to being isolated, body $2$ primarily undergoes the attraction of a fictitious body located at the center of mass of $0$ and $1$, and body $3$ primarily undergoes the attraction of a fictitious body located at the center of mass of $0$, $1$ and $2$. We think of body $0$ as the Sun and of the three other bodies as planets. The Jacobi coordinates are well suited for this regime (Figure~\ref{fig:Jacobi}), but we defer their definition to a later stage. Assuming that the center of mass is fixed, the small displacements of the Sun may be recovered from the positions of the planets.

The fast dynamics consists in those planets moving along Keplerian ellipses according to the above approximation, with elliptical elements as first integrals, in addition to the total energy and the total angular momentum $C = C_1 + C_2 + C_3$, where $C_i$ is the angular momentum of planet $i$. 
Let  $a_1$, $a_2$ and $a_3$ be the semimajor axes. In our regime, 
\[a_1\ll a_2\ll a_3\]
(we will later make a more specific hypothesis on how the semimajor axis ratios $a_i/a_{i+1}$ compare to each other).  Let also $e_1$, $e_2$ and $e_3$ be the eccentricities. The angular momentum $C_i$, seen as a vector in space, is normal to the plane of ellipse $i$ and its length is $\sqrt{a_i(1-e_i^2)}$, up to a mass factor (see \eqref{def:C2tilde} below). Let $\theta_{ij}$ be the mutual inclinations, i.e. the oriented angle $\alpha_{C_i \times C_j}(C_i,C_j)$ between the angular momenta of planets $i$ and $j$, the orientation being defined by the normal vector $C_i \times C_j$ (here assumed non-zero).

The so-called secular dynamics describes the slow evolution of the three Keplerian ellipses. At the first order of approximation, it is governed by the vector field obtained by averaging out the mean anomalies, thus defining a dynamical system on the ``secular space'' of triples of Keplerian ellipses with fixed semimajor axes. In the hierarchical regime, the dominating term is what is usually called the ``quadrupolar'' Hamiltonian $F^{12}_{\mathrm{quad}}$ of the two inner planets. It was introduced in various particular cases by  Lidov and then Kozai~\cite{Kozai:1962, Lidov:1962, Ziglin:1975}. It may come as a surprise that $F^{12}_{\mathrm{quad}}$ is integrable (defined in \eqref{eq_quad12unscaled} in Section  \ref{section_secularexpansion}), as noticed by Harrington~\cite{Harrington:1968}, due to the fact that it does not depend on the argument of the outer pericenter $g_2$. This dynamics was later studied more globally in the secular space by Lidov, Kozai and others (see a review in \cite{Naoz:2016}).

Our analysis follows from a higher order, non-integrable approximation of the system (it will rely also on the quadrupolar Hamiltonian of planets $2$ and $3$, and the octupolar term of planets $1$ and $2$, as introduced later). Precisely understanding the various time scales within the secular dynamics will be key to our analysis. At this stage, let us only loosely describe the different roles played by the three planets:
\begin{itemize}
\item Since the semimajor axis $a_3$ is so large, it is planet $3$ which most contributes to the total angular momentum $C$, so planet $3$ cannot change substantially in eccentricity or inclination. Yet, planet $3$ is a source of angular momentum for the two inner planets, and a minor change of $C_3$ results in major changes of the elliptical elements of planets $1$ or $2$. 

\item The elliptical elements of the first (inner) planet  vary faster than the elements of the second one. So the approximate conservation of $F^{12}_{\mathrm{quad}}$ introduces a coupling between the eccentricity and the inclination of the first planet. It can drive an initially near-circular orbit to arbitrarily high eccentricity, and flip an initially moderately inclined orbit between a prograde and a retrograde motion. This is an integrable, quasiperiodic dynamics, reportedly discovered by Lidov and thus called the \emph{Kozai mechanism}.\footnote{We will call it the Lidov-Kozai effect, as suggested by A. Neishtadt and~\cite{Naoz:2016}.} After reduction by the symmetry of rotations, the existence of an elliptic secular fixed point, which is the continuation (for non-zero eccentricities and inclinations) of the singularity known to Lagrange and Laplace, explains the oscillation of the argument of the inner pericenter.\footnote{The Lidov-Kozai mechanism has had useful many applications to a variety of systems from planetary and stellar scales to supermassive black holes. The orbits' eccentricity can reach extreme values, leading to a nearly radial motion, which can further evolve into short orbit periods and merging binaries. Furthermore, the orbits' mutual inclinations may change dramatically from pure prograde to pure retrograde, leading to misalignment and a wide range of inclinations. These dynamics are accessible from a large part of the triple-body parameter space and can be applied to a diverse range of astrophysical settings and used to gain insights into many puzzles~\cite{Naoz:2016}. But since it is an integrable behavior, it cannot bring light to Herman's conjectures.}

For our part, we will instead localise in a region where the two inner planets are mutually highly inclined, i.e. $\theta_{12}$ is close to $\pi/2$. In this region, there is a  hyperbolic secular singularity, which was used for instance by Jefferys-Moser to prove the existence of normally hyperbolic invariant tori~\cite{jefferys1966}. This singularity, when lifted to the full phase space, yields a normally hyperbolic cylinder which will be crucial to our construction. In particular we will show that its codimension-$1$ stable and unstable invariant manifolds split, and we will need to control the splitting of the underlying foliations more or less carefully depending on directions. The motions of interest to us will occasionally shadow the stable and unstable foliations of the cylinder, thus moving away from the cynlinder itself, to a homoclinic channel and then back close to the cylinder (along a so-called homoclinic excursion). So $\theta_{12}$ will vary substantially but with no drift. On the other hand, it is planet $1$ which most contributes to the secular Hamiltonian, so, because of the near conservation of the latter, the eccentricity $e_1$ will be bounded within a small interval.

\item In contrast, no first integral prevents the eccentricity $e_2$ or the mutual inclination $\theta_{23}$ to vary, even dramatically. It is the goal of this article to prove that these two quantities do vary arbitrarily, since any finite sequence of points in the $(e_2,\theta_{12})$-cylinder is shadowed by the projection of some solutions of the $4$-body problem.
\end{itemize}

The following theorem is a more precise statement of some of these assertions, in terms of the normalised angular momentum vector
\begin{equation}\label{def:C2tilde}
\tilde C_2 = \frac{\sqrt{m_0+m_1+m_2}}{m_2 (m_0+m_1)\sqrt{ a_2}} C_2
\end{equation}
of planet $2$, which lies in the unit Euclidean ball $B^3$ since $|\tilde C_2| = \sqrt{1-e_2^2}$ and $e_2\in (0,1)$). An even more precise statement will be given in Section~\ref{sec:DepritResults}. Fix $0 < \eta \ll 1$ and choose masses in the set
\begin{equation}
  \label{eq:M}
  \mathcal{M}_\eta = \left\{ (m_0,m_1,m_2,m_3) \in \, (0,+\infty)^4, \; m_0+m_1+m_2+m_3=1, \; |m_0 - m_1| \geq \eta, \; |m_0+m_1 \neq m_2| \geq \eta \right\};
\end{equation}
the sum of masses is fixed in order to avoid non-compactness issues, which are artificial due to the invariance of the dynamics with respect to a change of mass units; and the two inequalities are meant to avoid some degeneracies in the secular dynamics. (Note that in the planetary regime these degeneracies do not exist.)

\begin{theorem}[Main result]\label{thm:main}
Fix $\eta>0$.  Consider $(m_0,m_1,m_2,m_3) \in \mathcal{M}_\eta$ and  any finite sequence of points in $B^3$. For every $0<\delta\ll 1$, there exists an orbit whose normalised angular momentum $\tilde C_2$ passes  successively $\delta$-close to each point of the prescribed itinerary. 
\end{theorem}

The statement calls for a few comments:
\begin{enumerate}
\item The orbits of interest will be found first in the \emph{hierarchical regime} $a_1 \ll a_2 \ll a_3$, with fixed masses such that
  \[m_0 \neq m_1 \quad \mbox{and} \quad m_0+m_1 \neq m_2.\]

\item Now consider the \emph{planetary regime}, where $m_0=1$ and $m_i = \rho \tilde m_i$ ($i=1,2,3$), with $\rho \to 0$. A more precise statement of the Theorem~\ref{thm:main} in Section~\ref{sec:DepritResults} will show that the semimajor axes may be chosen independently from $0 \leq \rho \leq 1$, i.e. the conclusion holds in the planetary problem. 

\item Let us be more specific on the orbit of the conclusion. Let $\tilde C_2^0,..., \tilde C_2^N \in B^3$ be the prescribed itinerary. These points determine values 
  \begin{itemize}
  \item $e_2^k$ of the eccentricity
  \item $\theta_{23}^k$ of the inclination between planets $2$ and $3$ (the inclination of planet $3$ being nearly fixed)
  \item and $h_2^k$ of the longitude of the node of planet $2$.
  \end{itemize}
  The proof will show that there exist times $t_0< t_1 < \cdots < t_N$ such that the osculating orbital elements satisfy, as stated,
\begin{equation}\label{def:drifte2theta23}
 \begin{cases}
   |e_2(t_k)-e_2^k|&\leq \delta\\
   |\theta_{23}(t_k)-\theta_{23}^k|&\leq \delta\\
   |h_2(t_k)-h_2^k|&\leq \delta \quad (k=0,...,N).
 \end{cases}
\end{equation}
Moreover, the orbit can be chosen so that
\begin{equation}\label{def:drifte1theta12}
 \begin{cases}
   |e_1(t_k)|&\leq \delta\\
   |\theta_{12}(t_k)-\theta_{12}^k|&\leq \delta \quad (k=0,...,N),
\end{cases}
\end{equation}
where $\theta_{12}^k\in (0,\pi)$ is determined by the relation 
\begin{equation}\label{def:inclincation12relation}
  (1-(e_2^0)^2)^{3/2} \cos\theta_{12}^k=(1-(e_2^k)^2)^{3/2}\cos\theta_{12}^0 
\end{equation}
whereas 
\begin{equation}\label{def:drifte3aj}
\begin{cases}
%    |\theta_{12}(-t)- ??|\\
   |e_3(t)-e_3^0|&\leq \delta\\
   |a_j(t)-a_j^0|&\leq \delta \quad \text{for}\quad (j=1,2,3,  \; t\in [0,t_N]).  
\end{cases}
\end{equation}
Note that we can choose any initial condition $e_3^0\in (0,1)$ and $e_3$ remains almost constant along the trajectories we consider.

\item We also obtain estimates for the drifting time $T=t_N$. In the hierarchical regime where the fixed masses belong to $\mathcal{M}_\eta$, we have 
  \begin{equation}\label{def:timefinal}
  T = C(m_0,m_1,m_2,m_3)  \, \frac{N}{\delta^{\kappa} }
  \end{equation}
  where $C$ is a constant depending only on the masses, for some exponent $\kappa$ which does not depend on $N$ or the itinerary. To be more precise, call $\alpha_i = a_i/a_{i+1}$, $i=1,2$, the semimajor axis ratios. As $\delta$ tends to zero, the $\alpha_i$'s will be chosen polynomially smaller, and the drifting time itself depends polynomially on the $\alpha_i$'s. 
  
  In the planetary regime where $m_j = \rho \, \tilde m_j$ with $\rho$ small, the drifting time satisfies
    \begin{equation}\label{def:timefinal2}
  T = C(m_0,\tilde m_1,\tilde m_2,\tilde m_3)  \, \frac{N}{\delta^{\kappa} \rho^2}
  \end{equation}
  for some exponent $\kappa$ which does not depend on $N$ or the itinerary.

\item One may consider the setting in the present paper as a mix between a priori stable and a priori unstable.  Indeed, this is a model with multiple time scales and degeneracies. This implies that some directions can be treated as a priori unstable, when one has to face a regular perturbation problem. This are the so-called secular variables. Other directions, which encapsulate the Keplerian mean motions, are much faster. In the present paper we manage to obtain drifting orbits along the a priori unstable directions, that is in the secular actions (angular momenta and inclinations), with a sufficiently robust mechanism, so that that fast directions do not interfere with the slow ones. Drift in the actions conjugated to the mean anomalies, that is the semimajor axis, would require a deeper analysis since would fall into an a priori stable regime. In particular, the drift in actions obtained in Theorem \ref{thm:main} requires time scales which are polynomial in the perturbative parameter. On the contrary, drift of the semimajor axis must take an exponentially long time since Nekhoroshev Theory applies along these directions (see \cite{MR1419468}).
  
\end{enumerate}

Here are some noteworthy consequences.
\begin{enumerate}
  
\item Why are inner planets not inclined? If the two inner planets have their mutual inclination close to $\pi/2$, Theorem~\ref{thm:main} proves that the next planet might be subject to large instabilities. In reality, where semimajor axes are fixed, there is a competition between the semimajor axes ratios and $e_2$, which we do not control here. We conjecture that this mechanism leads $e_2$ to become so large that planets $1$ and $2$ may collide. 

  In the Solar System, indeed most planets have relatively small mutual inclinations. On the other hand, the two largest  dwarf planets Pluto and Eris are trans-Neptunian (i.e. they play the role of our third planet) and have large inclinations to the ecliptic (17$^{\circ}$ and 44$^{\circ}$ respectively).
%   *** compare ***

\item Orbits described by Theorem~\ref{thm:main} may flip from prograde to retrograde.
  % https://en.wikipedia.org/wiki/Retrograde_and_prograde_motion

  In the Solar System, the orbits around the Sun of all planets and most other objects, except many comets, are prograde.\footnote{They orbit around the Sun in the same direction as the sun rotates about its axis, consistently with the most probable scenarios of formation of the Solar System.}
  However, protoplanetary disks can collide with or steal material from molecular clouds and this can lead to disks and their resulting planets having retrograde orbits around their stars. Retrograde motion may also result from the Lidov-Kozai mechanism, as already mentioned. Here we thus provide another mechanism possibly explaining the existence of some retrograde planets. 

\item As far as the authors know, up to the present paper there has been no result (analytical or otherwise) on Arnold diffusion in a (non-restricted) $N$-body problem. As we have already explained, it relies on the existence of normally hyperbolic structure along secular resonances. We consider a purely elliptic regime, that is, at short time scales the bodies perform aproximate ellipses, and therefore no hyperbolic invariant objects exist at first order. We rely on the analysis of the perturbed secular dynamics to detect such objects and use its invariant manifolds to obtain drifting orbits. This allows us to obtain Arnold diffusion in the classical planetary regime.

Note that our diffusion mechanism is robust in the following sense; if we consider the $N$-body problem and assume that the initial conditions of the first four bodies are as in Theorem \ref{thm:main} and the remaining $N-4$ bodies  revolve sufficiently far away, the conclusion of Theorem \ref{thm:main} still holds.
\end{enumerate}

We now consider implications of our main theorem regarding Herman's local conjecture~\ref{conj:local}, in the case of the spatial $4$-body problem. We need the following complement to the main theorem, which follows from our construction and \cite[Remark~2.13]{clarke2022topological}, because in the secular directions aligned windows are bounded. 

\begin{proposition}\label{prop:SecularChains}
  Transition chains obtained for the secular system (thus diregarding the fast, Keplerian dynamics) have infinite length and the subsequent shadowing orbits are defined for all future times.
%   in the past and future. 
\end{proposition}

We can thus conclude the following.

\begin{theorem}[Weak local Herman conjecture]\label{thm:WLHC} % weak local Herman conjecture
  Under the hypotheses of Theorem~\ref{thm:main}, circular orbits are in the closure of the non-recurrent set (as well as of the recurrent set) of the secular dynamics at any finite order.
\end{theorem}

\medskip The main theorem above (Theorem \ref{thm:main}) will be reworded more precisely as two theorems (Theorems \ref{thm:MainHierarch:Deprit} and \ref{thm:PlanetHierarch:Deprit} below) which imply Theorem \ref{thm:main} above. These new theorems are more detailed and analyse the dynamics in a different set of coordinates called Deprit variables. The key point of these coordinates is that they are symplectic and moreover adapted to the symplectic reduction the 4 body problem with respect to all its symmetries (translation and rotation).

\subsection{Main ideas of the proof of Theorem \ref{thm:main}}
 
We first prove the theorem for the hierarchical regime (and fixed values of the masses), and then give a continuation argument which allows to vary the masses within $\mathcal{M}_\eta$ (see \eqref{eq:M}). 

The starting point of our construction is a now classical scheme: 
\begin{enumerate}
\item Establish the existence of a normally hyperbolic invariant manifold (see Appendix~\ref{sec:heuristics} for the definition).
\item Using Poincar\'e-Melnikov theory, prove that its stable and unstable manifolds have a transverse homoclinic intersection.
\item Compute the first order of the scattering maps in certain variables.
\item Using a shadowing theorem, find true orbits that shadow those of the scattering maps.
\end{enumerate}
However there are many caveats. Since we assume that the semimajor axes of the Keplerian ellipses are of different orders, there are numerous time scales. In particular the mean anomalies revolve much faster than the other angles (see below for a description of the angles), and so the splitting of separatrices in the directions of their symplectic conjugate variables is exponentially small. This makes the computation of the scattering maps in those directions a significant challenge. The shadowing results of \cite{clarke2022topological} are well equipped to dealing with problems of this nature; actually, this paper was precisely written for the purpose of its application to the present problem. Therefore we proceed by checking that the assumptions of that paper (summarised in Appendix \ref{appendix_shadowing}) are satisfied by the secular system (defined in Section \ref{section_secularexpansion}), and thus the full four-body problem.
 
Beginning with the usual Hamiltonian (see \eqref{eq_4bpham} in Section \ref{sec:DepritResults}) of the four-body problem, we have a conservative system with 12 degrees of freedom. It is well known that this system possesses many integrals of motion. We first pass to Jacobi coordinates to perform the symplectic reduction by translational invariance, which removes 3 degrees of freedom. For a three-body problem, the next step would usually be to pass to orbital elements (i.e. Delaunay variables) in order to perform Jacobi's classical reduction of the nodes. It is known however that this approach does not extend to the $N$-body problem when $N \geq 4$. Instead, we use the so-called \emph{Deprit variables}. These were introduced by Deprit in the paper \cite{deprit1983}, but he stated ``Whether the new phase variables... are practical in the General Theory of Perturbations is an open question. At least, for planetary theories, the answer is likely to be in the negative''. The variables were subsequently forgotten for a number of years, until they were rediscovered by Chierchia and Pinzari \cite{chierchia2011deprit,Chierchia:2011}, who pointed out that they are in fact \emph{very} useful, by defying Deprit's advice and nonetheless implementing the coordinates in the planetary problem. The use of these coordinates reduces a further 2 degrees of freedom from our system, resulting in a system with 7 degrees of freedom. To our knowledge, this paper represents the first use of the Deprit variables since the papers of Chierchia and Pinzari. 

Denote by $C_j$ the angular momentum of the $j^{\mathrm{th}}$ fictitious Keplerian body. The (Deprit) coordinates we are left with after the symplectic reduction are as follows: the mean anomalies $\ell_1, \ell_2, \ell_3$ and their symplectic conjugates $L_1, L_2, L_3$ which are proportional to the square root of the semimajor axes; the arguments of the perihelia $\gamma_1, \gamma_2, \gamma_3$ (\emph{not} with respect to the ascending node, but rather some different nodes; see Section \ref{sec:DepritResults} for a precise definition) and their symplectic conjugates the absolute angular momenta $\Gamma_1=|C_1|$, $\Gamma_2=|C_2|$, $\Gamma_3=|C_3|$; an abstractly defined angle $\psi_1$ and its symplectic conjugate $\Psi_1 = | C_1 + C_2|$. 

In these coordinates the Hamiltonian can be written $H = F_{\mathrm{Kep}} + F_{\mathrm{per}}$ where the \emph{Keplerian function} $F_{\mathrm{Kep}}$ describes the motion of three uncoupled two-body problems, and the \emph{perturbing function} $F_{\mathrm{per}}$ describes the gravitational forces between the planets. The assumptions we make on the semimajor axes of the Keplerian ellipses imply that the mean anomalies evolve much faster than the other angles, and so we can use the averaging theory to make a near-to-the-identity symplectic transformation that averages the angles $\ell_j$ out of $F_{\mathrm{per}}$ up to arbitrarily high order. The Hamiltonian resulting from averaging $F_{\mathrm{per}}$ over the angles $\ell_j$ is called the \emph{secular} Hamiltonian and is denoted by $F_{\mathrm{sec}}$. Since we assume that the semimajor axes are of different orders we have $L_1 \ll L_2 \ll L_3$ and so, using the Legendre polynomials, we can expand $F_{\mathrm{sec}}$ in powers of $\frac{L_1}{L_2}$ and $\frac{L_2}{L_3}$. The first order terms in the expansion of the secular Hamiltonian are referred to as the \emph{quadrupolar} and \emph{octupolar} Hamiltonians of the interactions between bodies 1 and 2 ($F_{\mathrm{quad}}^{12}$ and $F_{\mathrm{oct}}^{12}$), and the interactions between bodies 2 and 3 ($F_{\mathrm{quad}}^{23}$ and $F_{\mathrm{oct}}^{23}$); all other terms are of higher order, and are not required for our analysis. All parts of the proof described so far are contained in Section \ref{section_secularexpansion}. 

In order to continue, we observe that the actions are of different orders; indeed, $\Gamma_1 = O (L_1)=O(1)$, $\Gamma_2 = O(L_2)$, $\Gamma_3 = O(L_3)$, $\Psi_1=O(L_2)$. As it is more convenient to deal with actions of order 1, we perform a linear symplectic coordinate transformation and we ``chop'' the new action space in rectangles with sizes of order 1. We perform our analysis in each of these rectangles. Since the assumptions in \cite{clarke2022topological} are local, one can verify them in each rectangle separately.
% so that the new actions are of this form. 
This coordinate transformation leaves the variables $\gamma_1$, $\Gamma_1$ unchanged, and we denote by $\tilde{\gamma}_j$, $\tilde{\Gamma}_j$, $\tilde{\psi}_1$, $\tilde{\Psi}_1$ the new variables. Substituting these variables into the Hamiltonians $F_{\mathrm{quad}}^{12}$, $F_{\mathrm{oct}}^{12}$, $F_{\mathrm{quad}}^{23}$, and $F_{\mathrm{oct}}^{23}$ illuminates the time scales of our problem, and allows us to perform further Taylor expansions of each of these Hamiltonians, keeping a careful account of the term in which each variable appears for the first time (see Proposition \ref{proposition_secularexpansion}). This linear change, localization and expansion of the secular Hamiltonian are also performed in Section \ref{section_secularexpansion}. 

\begin{figure}[h]
    \centering
        \includegraphics[width=0.7\textwidth]{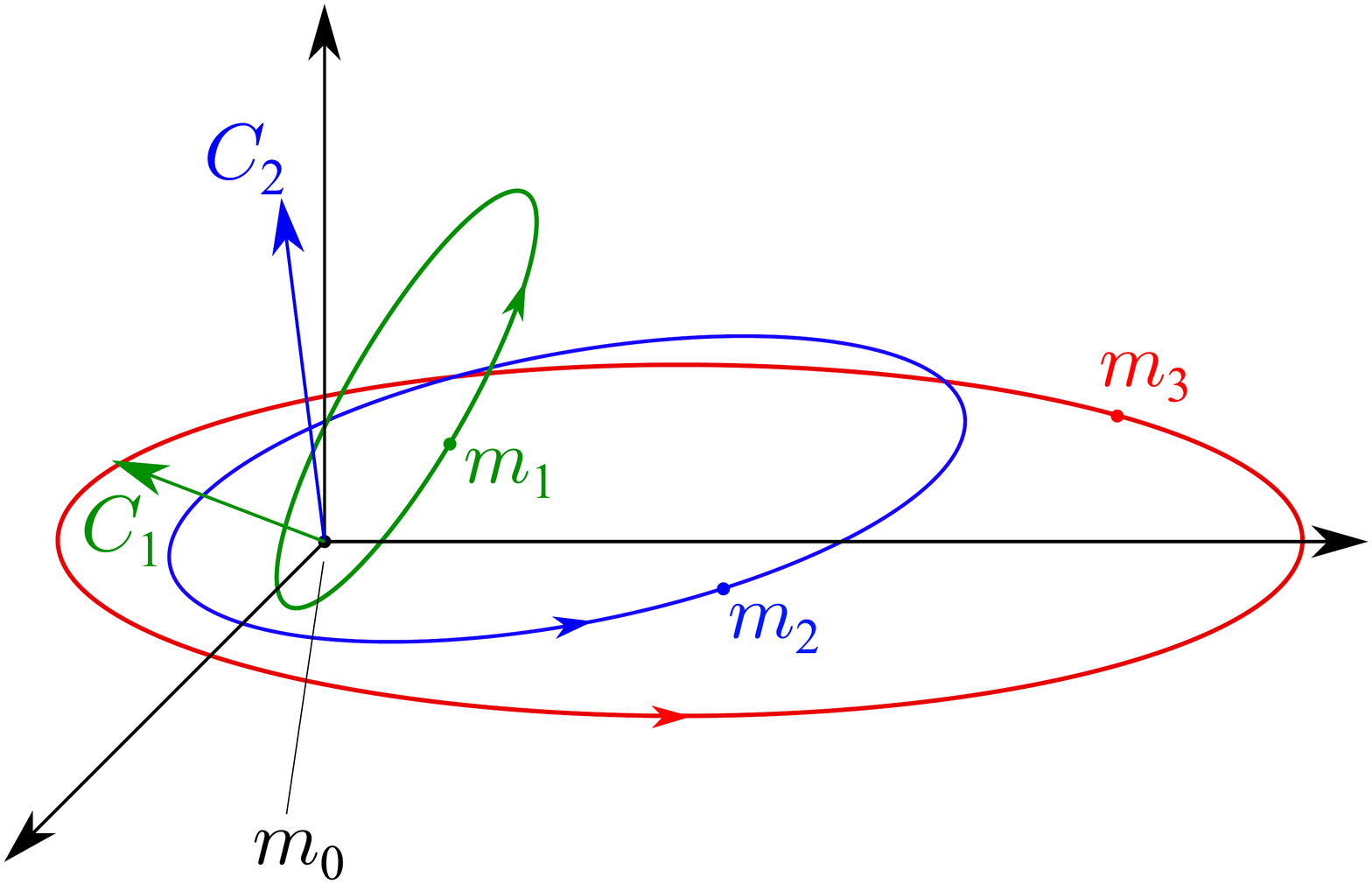}   
    \caption{We assume that the semimajor axes of the Keplerian ellipses are of different orders. Moreover, assuming that the mutual inclination of bodies 1 and 2 (i.e. the angle between the angular momentum vectors $C_1$ and $C_2$ in the figure) is sufficiently large, the first order term in the expansion of the secular Hamiltonian has a saddle periodic orbit. Our main theorems say that, given any predetermined itinerary of mutual inclinations between bodies 2 and 3, and eccentricities of the $2^{\mathrm{nd}}$ ellipse, there exist trajectories of the four-body problem shadowing that itinerary.}\label{figure_4bpellipses}
\end{figure}

Part of the beauty of the Deprit variables is that the Hamiltonians $F_{\mathrm{quad}}^{12}$ and $F_{\mathrm{oct}}^{12}$ are identical to the quadrupolar and octupolar Hamiltonians (respectively) in the three-body problem expressed in Delaunay coordinates (see \cite{fejoz2016secular}, for example). Moreover, we choose the  coordinate transformation in Section \ref{section_secularexpansion} carefully to ensure that the first-order term in the expansion of $F_{\mathrm{quad}}^{12}$ is identical to the first-order term in the expansion of the quadrupolar Hamiltonian in \cite{fejoz2016secular} (modulo some errata; see Appendix \ref{appendix_errata}). This Hamiltonian, which we denote by $H_0^{12}$, possesses a well-known saddle with a separatrix (see for example \cite{jefferys1966}) whenever the mutual inclination between bodies 1 and 2 is sufficiently large (see Figure \ref{figure_4bpellipses}); moreover, the saddle is present for an interval of values of the energy $H_0^{12}$. Collecting these saddles on this interval of energy levels, we obtain a normally hyperbolic invariant manifold $\Lambda$ (see Appendix~\ref{sec:heuristics} for definitions), whose stable and unstable manifolds coincide. Furthermore, Fenichel theory implies that the secular Hamiltonian $F_{\mathrm{sec}}$ inherits the normally hyperbolic cylinder $\Lambda$ \cite{fenichel1971persistence,fenichel1974asymptotic,fenichel1977asymptotic}. The results of \cite{fejoz2016secular} imply that the octupolar Hamiltonian splits the separatrix in the three-body problem; as the equations are identical, $F_{\mathrm{oct}}^{12}$ splits the stable and unstable manifolds of $\Lambda$ in our case too. The manifold $\Lambda$ for the secular Hamiltonian is written as a graph over the variables $\tilde{\gamma}_2$, $\tilde{\Gamma}_2$, $\tilde{\psi}_1$, $\tilde{\Psi}_1$, $\tilde{\gamma}_3$, $\tilde{\Gamma}_3$, and so has the structure of a cylinder $\T^3 \times [0,1]^3$. This is described in detail in Section \ref{section_analysisofh0}. 

In Section \ref{sec:innerdyn} we analyse the inner dynamics: the restriction of the secular flow to the normally hyperbolic cylinder $\Lambda$. The primary goal of this section is to find a new system of coordinates on the cylinder $\Lambda$, denoted by $(\hat \gamma_2, \hat \Gamma_2, \hat \psi_1, \hat \Psi_1, \hat \Gamma_3, \hat \Gamma_3)$, in which we can continue our analysis. This system is constructed in two steps. The first step produces a coordinate transformation that brings the restriction to $\Lambda$ of the symplectic structure into its canonical form, using Moser's trick from his proof of Darboux's theorem. The second step uses standard methods from
averaging theory to push the dependence of the inner Hamiltonian $\left. F_{\mathrm{sec}} \right|_{\Lambda}$ on the angles $\hat \gamma_2, \hat \psi_1, \hat \gamma_3$ to higher order terms. 

In Section \ref{sec:scattering} we prove that there are two homoclinic channels relative to $\Lambda$ that give rise to two globally defined scattering maps (see Appendix~\ref{sec:heuristics} for definitions). Moreover we compute the first order of the jumps in the scattering maps in the variables $\hat \Psi_1$, $\hat \Gamma_3$. This section relies on the computation of three Poincar\'e-Melnikov integrals: the first of these is identical to the computation in \cite{fejoz2016secular} (see also Appendix \ref{appendix_errata}), whereas the computation of the remaining two integrals is performed in Section \ref{section_melnikovcomp} by considering complex values of time and integrating over certain contours using the residue theorem. Since these computations are conducted in `tilde' variables, we then have to do further computations to determine the jumps in the scattering maps in the `hat' variables (Lemma \ref{lemma_scatteringhatcoords}). 

In Section \ref{sec:mapreduction} we consider the return map to the Poincar\'e section $\{ \hat \gamma_2 = 0 \}$ in an energy level of the secular Hamiltonian. The restriction to an energy level eliminates $\hat \Gamma_2$, and so we are left with a map of a four-dimensional cylinder. We prove that the map satisfies a twist condition, and we show that the jumps in the scattering maps in the $\hat \Psi_1$, $\hat \Gamma_3$ directions are the same as those for the scattering maps corresponding to the flow. We deduce from the formulas for the scattering maps that there are pseudo-orbits (i.e. orbits of the iterated function system consisting of the Poincar\'e map and the two scattering maps) that follow any itinerary in the actions $\hat \Psi_1$, $\hat \Gamma_3$, and that the scattering maps map tori corresponding to constant values of $\hat \Psi_1$, $\hat \Gamma_3$ transversely across other such tori . These are the assumptions of the first main theorem of \cite{clarke2022topological} (these results have been summarised in Appendix \ref{appendix_shadowing} -- see Theorem \ref{theorem_main1}), and so we obtain orbits of the secular Hamiltonian that follow any predetermined itinerary of the variables $\hat \Psi_1, \hat \Gamma_3$ (thus proving Proposition \ref{prop:SecularChains}). As a consequence, we can then show that the full four-body problem therefore satisfies the assumptions of the second main theorem in \cite{clarke2022topological} (see Theorem \ref{theorem_main2} in Appendix \ref{appendix_shadowing}), and so we obtain an analogous result for the full four-body problem.
% In Section \ref{subsection_rescalingandproof} it is shown that this in turn implies Theorem \ref{thm:Main:Hiearch:orbelements}. 
Furthermore, the shadowing methods provide us with (non-optimal) time estimates, which are contained in Section \ref{section_applicshadowingresults}. 

In Section \ref{sec:planetary} we explain how we can pass from the hierarchical  to the planetary case, where the masses of the planets are arbitrarily small of order $0<\rho\ll 1$. We show that the separation of the semimajor axes of the planets can be taken independently of the smallness of the masses of the planets. This effectively amounts to a demonstration that we can rescale the planetary Hamiltonian by $\rho^{-1}$ and time by a factor of $\rho^{-2}$ in order to obtain the hierarchical Hamiltonian, where $\rho$ is the small mass parameter.
% Therefore the proof of Theorem \ref{thm:Main:Hiearch:orbelements} applies equally to Theorem \ref{thm:Main:Planet:orbelements}. 

\subsection*{Acknowledgments}
A. Clarke and M. Guardia are supported by the European Research Council (ERC) under the European Union's Horizon 2020 research and innovation programme (grant agreement No. 757802). M. Guardia is also supported by the Catalan Institution for Research and Advanced Studies via an ICREA Academia Prize 2019. This work is also supported by the Spanish State Research Agency, through the Severo Ochoa and María de Maeztu Program for Centers and Units of Excellence in R\&D (CEX2020-001084-M).

This work is also partially supported by the project of the French Agence Nationale pour la Recherche CoSyDy (ANR-CE40-0014).

%%% Local Variables: 
%%% mode: latex
%%% TeX-master: "4bp_secular_diffusion_v6.tex"
%%% End: 

\section{Main results in Deprit coordinates}\label{sec:DepritResults}
In Section \ref{sec_setup} we introduce the Hamiltonian of the four-body problem, perform the reduction by translational symmetry, state our assumptions precisely, and expand the perturbing function using the Legendre polynomials. In Section \ref{sec:deprit} we recall the definition of the Deprit coordinates. Then, in Section \ref{sec:thm:deprit} we give  more detailed versions of Theorem
\ref{thm:main}
% s \ref{thm:Main:Planet:orbelements} and \ref{thm:Main:Hiearch:orbelements} 
in terms of Deprit coordinates.

\subsection{Setting up the problem}\label{sec_setup}

Consider four point masses in space, interacting via gravitational attraction in the sense of Newton. Denote by $m_j$ the mass, by $x_j \in \mathbb{R}^3$ the position, and by $y_j \in \mathbb{R}^3$ the linear momentum of body $j$ for each $j=0,1,2,3$. This system is described by the flow of the Hamiltonian function
\begin{equation}\label{eq_4bpham}
H= \sum_{0 \leq j \leq 3} \frac{y_j^2}{2m_j} - \sum_{0 \leq i < j \leq 3} \frac{m_i m_j}{\| x_j - x_i \|}.
\end{equation}
Hamilton's equations of motion give a system of 24 differential equations. It is well-known that this system has (at least) 10 integrals of motion: the integral $H$ corresponds to conservation of energy; 6 integrals correspond to the motion of the barycenter (translational symmetry); and 3 integrals correspond to conservation of angular momentum (rotational symmetry). By making suitable changes of coordinates we can make these integrals visible, thus reducing the equations of motion. First, we pass to Jacobi coordinates to perform the symplectic reduction by translational symmetry, and then use Deprit coordinates to reduce by rotational symmetry.

\paragraph{Reduction by Translational Symmetry}
Let
\begin{equation}
M_j = \sum_{i=0}^j m_i, \quad \sigma_{ij} = \frac{m_i}{M_j},
\end{equation}
and define the Jacobi coordinates $(q_j,p_j) \in \mathbb{R}^3 \times \mathbb{R}^3$ for $j=0,1,2,3$ by
\begin{equation}
\begin{cases}
q_0 = x_0 \\
q_1 = x_1 - x_0 \\
q_2 = x_2 - \sigma_{01} \, x_0 - \sigma_{11} \, x_1 \\
q_3 = x_3 - \sigma_{02} \, x_0 - \sigma_{12} \, x_1 - \sigma_{22} \, x_2
\end{cases}
\quad
\begin{cases}
p_0 = y_0 + y_1 + y_2 + y_3 \\
p_1 = y_1 + \sigma_{11} \, y_2 +\sigma_{11} \, y_3 \\
p_2 = y_2 + \sigma_{22} \, y_3 \\
p_3 = y_3.
\end{cases}
\end{equation}

\begin{figure}[h]
  \centering
  \begin{tikzpicture}
    \draw [->] (0,0) node[anchor=north east] {$O$}
    -- (1,1) node[anchor=north west] {$q_0$}
    -- (2,2) node[anchor=west] {$x_0$};
    \draw [->] (2,2)
    -- (1.5,2.5) node[anchor=north east] {$q_1$}
    -- (1,3) node[anchor=south east] {$x_1$};
    \draw [->] (1.5,2.5)
    -- (2,4) node[anchor=south east] {$q_2$}
    -- (2.5,5.5) node[anchor=south] {$x_2$};
    \draw [->] (2,4)
    -- (3.5,4.5) node[anchor=south east] {$q_3$}
    -- (6.5,5.5) node[anchor=west] {$x_3$};
  \end{tikzpicture}
  \caption{Jacobi coordinates}
  \label{fig:Jacobi}
\end{figure}
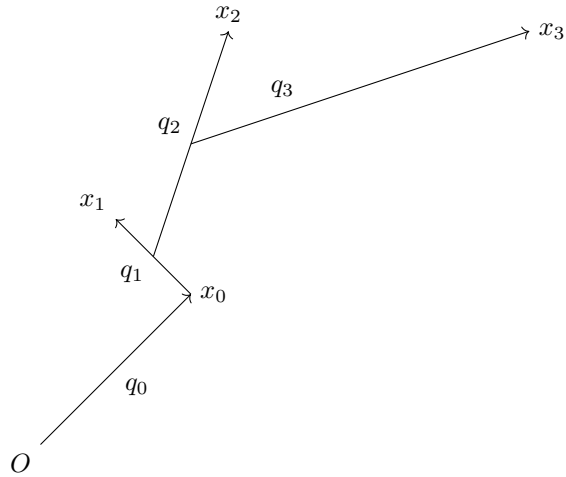

This is a symplectic change of coordinates, and the variable $q_0$ does not appear in the Hamiltonian~\eqref{eq_4bpham}. The reduced phase space has coordinates $(q_j,p_j)_{j=1,2,3}$, and, without loss of generality, we may restrict to $p_0=0$. Thus write
\begin{equation}\label{def:HKepplusper}
H = F_{\mathrm{Kep}} + F_{\mathrm{per}}
\end{equation}
where
\begin{equation}
F_{\mathrm{Kep}} = \sum_{j=1}^3 \left( \frac{p_j^2}{2 \mu_j} - \frac{\mu_j M_j}{\| q_j \|} \right)
\end{equation}
and
\begin{align}
F_{\mathrm{per}} ={}& \sum_{j=2}^3 \frac{\mu_j M_j}{\| q_j \|} - \frac{m_0 \, m_2}{\| q_2 + \sigma_{11} \, q_1 \|} - \frac{m_0 \, m_3}{\| q_3 + \sigma_{22} \, q_2 + \sigma_{11} \, q_1 \|} - \frac{m_1 \, m_2}{\| q_2 - \sigma_{01} \, q_1 \|}  \\
& - \frac{m_1 \, m_3}{\| q_3 + \sigma_{22} \, q_2 + (\sigma_{11} - 1) \, q_1 \|} - \frac{m_2 \, m_3}{\| q_3 + (\sigma_{22} - 1) \, q_2 
% + (\sigma_{12} + \sigma_{11} \, \sigma_{22} - \sigma_{11}) \, q_1
\|}
\end{align}
with the reduced masses $\mu_j$ defined by
\begin{equation}
\mu_j^{-1} = M_{j-1}^{-1} + m_j^{-1} \quad (j=1,2,3).
\end{equation}
The function $F_{\mathrm{Kep}}$ is an integrable Hamiltonian describing the motion of 3 uncoupled \emph{Kepler problems}, and $F_{\mathrm{per}}$ is the so-called \emph{perturbing function}.

\paragraph{Assumptions and Expansion in Legendre Polynomials}
In this paper, we assume that each term of $F_{\mathrm{Kep}}$ is negative, so each of the 3 uncoupled Keplerian trajectories is elliptical. We assume that the orbit of body 2 is far away from the orbits of the first two bodies. We assume moreover that the orbit of body 3 lies far from the orbits of the first 2 bodies, in a range dictated by the position of body 2. More precisely, if we denote by $a_j$ the semimajor axis of the $j^{\mathrm{th}}$ Keplerian ellipse, our assumptions are:
\begin{equation} \label{eq_assumption1}
% \begin{dcases}
O(1) = a_1 \ll a_2  \to \infty \qquad\text{and}\qquad
a_2^{\frac{11}{6}}  \ll a_3 \ll a_2^2.
% \end{dcases}
\end{equation} 
The purpose of these assumptions is twofold. First, the fact that $a_1 \ll a_2 \ll a_3 \ll a_2^2$ implies that the three mean anomalies are the fastest three angles, and so we can average out these three angles (see below in this section). Second, the fact that $a_2^{11/6} \ll a_3$ implies both that the first two terms in the expansion of the secular Hamiltonian come from the quadrupolar Hamiltonian of bodies 1 and 2, and that the first nontrivial term in the Melnikov potential comes from the octupolar Hamiltonian of bodies 1 and 2 (see later).

We are in a near-integrable setting, as $F_{\mathrm{per}}$ is smaller than $F_{\mathrm{Kep}}$. Indeed, denote by $\zeta_{j}$ the angle between $q_j$ and $q_{j+1}$ for $j=1,2$, and denote by $P_n$ the Legendre polynomial of degree $n$. Our assumptions imply that $\| q_1 \| \ll \| q_2 \| \ll \| q_3 \|$. Therefore we can write the perturbing function as 
\begin{equation}\label{eq_fperdef}
F_{\mathrm{per}} = F_{\mathrm{per}}^{12} + F_{\mathrm{per}}^{23} + O \left( \frac{\| q_1 \|^2}{\| q_3 \|^3} \right)
\end{equation}
where
\begin{equation}
F_{\mathrm{per}}^{12} =\frac{\mu_2 M_2}{\| q_2 \|} - \frac{m_0 \, m_2}{\| q_2 + \sigma_{11} \, q_1 \|} - \frac{m_1 \, m_2}{\| q_2 - \sigma_{01} \, q_1 \|} 
% \\
% ={}& 
=- \frac{\mu_1 m_2}{\| q_2 \|} \sum_{n=2}^{\infty} \tilde{\sigma}_{1,n} P_n (\cos \zeta_{1}) \left( \frac{\| q_1 \|}{\| q_2 \|} \right)^n \label{eq_perfn12}
\end{equation}
is the perturbing function from the 3-body problem (see \cite{fejoz2016secular}, for example), and where 
\begin{equation} \label{eq_perfn23}
F_{\mathrm{per}}^{23} = - \frac{\mu_2 m_3}{\| q_3 \|} \sum_{n=2}^{\infty} \tilde{\sigma}_{2,n} P_n (\cos \zeta_{2}) \left( \frac{\| q_2 \|}{\| q_3 \|} \right)^n
\end{equation}
with
\begin{equation}
\tilde{\sigma}_{1,n} = \sigma_{01}^{n-1} + (-1)^n \sigma_{11}^{n-1}, \quad \tilde{\sigma}_{2,n} = (\sigma_{02} + \sigma_{12})^{n-1} + (-1)^n \sigma_{22}^{n-1}.
\end{equation}

\subsection{The Deprit coordinates and reduction by rotational symmetry}\label{sec:deprit}

To prove Theorem \ref{thm:main} and to take advantage of the rotational symmetry, we rely on Deprit coordinates. These coordinates were discovered originally by Deprit \cite{deprit1983} and rediscovered recently by Chierchia and Pinzari \cite{chierchia2011deprit}.

Denote by 
\begin{equation}
C_j = q_j \times p_j
\end{equation} 
the angular momentum of the $j^{\mathrm{th}}$ fictitious Keplerian body (Keplerian refering to $F_{\mathrm{Kep}}$), and let $k_j$ be the $j^{\mathrm{th}}$ element of the standard orthonormal basis of $\mathbb{R}^3$. Define the nodes $\nu_j$ by
\begin{equation}
\nu_1=\nu_2=C_1 \times C_2, \quad \nu_3 = (C_1 + C_2) \times C_3, \quad \nu_4=k_3 \times C
\end{equation}
where
\begin{equation}
C = C_1 + C_2 + C_3
\end{equation}
is the total angular momentum vector. For a non-zero vector $z \in \mathbb{R}^3$ and two non-zero vectors $u,v$ lying in the plane orthogonal to $z$, denote by $\alpha_z (u,v)$ the oriented angle between $u,v$, with orientation defined by the right hand rule with respect to $z$. Denote by $\Pi_j$ the pericenter of $q_j$ on its Keplerian ellipse.
% \footnote{Despite the common notation, there is no risk of confusion with Legendre polynomials.} 
The Deprit variables $(\ell_j,L_j,\gamma_j,\Gamma_j,\psi_j,\Psi_j)_{j=1,2,3}$ are defined as follows:
\begin{itemize}
\item $\ell_j$ is the mean anomaly of $q_j$ on its Keplerian ellipse;
\item
$L_j = \mu_j \sqrt{M_j a_j}$;
\item
$\gamma_j= \alpha_{C_j}(\nu_j,\Pi_j)$;
\item
$\Gamma_j = \left\| C_j \right\|$;
\item
$\psi_1 = \alpha_{(C_1+C_2)}(\nu_3,\nu_2)$, $\psi_2=\alpha_C(\nu_4,\nu_3)$, $\psi_3 = \alpha_{k_3} (k_1, \nu_4)$;
\item
$\Psi_1 = \| C_1 + C_2 \|$, $\Psi_2 = \| C_1 + C_2 +C_3 \| = \| C \|$, $\Psi_3 = C \cdot k_3$.
\end{itemize}
The Deprit variables are analytic over the open subset $\cD$ over which the $3$ terms of $F_{\mathrm{Kep}}$ are negative, the eccentricities of the Keplerian ellipses lie strictly between 0 and 1, and the nodes $\nu_j$ are nonzero.
 
\begin{figure}[h]
\centering
\tdplotsetmaincoords{60}{50}
  \begin{tikzpicture}[tdplot_main_coords,scale=4] %,x=3cm/360
     \draw[thick,->] (0,0,0) -- (1.8,0,0) node[anchor=north east]{$x$};
     \draw[thick,->] (0,0,0) -- (0,2,0) node[anchor=north west]{$y$};
     \draw[thick,->] (0,0,0) -- (0,0,1.6) node[anchor=south]{$z$};
 
  %   First plane
     \draw (1,0,0) -- (1,-1,.5) -- (-1.5,-1,.5) -- (-1.5,0,0);
    \draw (1,0,0) -- (1,1,-.5) -- (.2,1,-.5);
     \draw[dashed] (1,0,0) -- (1,1,-.5) -- (-1.5,1,-.5) -- (-1.5,0,0);
 
 %    Angular momentum
     \draw[->] (-1,0.5,0.25) -- (-1,0,1) node[anchor=south west]{$C_j$};
  \draw[->] (0.8,0.8,-0.4) -- (0.8,1,0) node[anchor=west]{$
       \begin{cases}
         C_2 &(j=1)\\
         C_1 &(j=2)\\
         C_1+C_2 &(j=3)
       \end{cases}
       $};
 
%      Node
     \draw (1,0,0) -- (-1.5,0,0);
     \draw[very thick,->] (0,0,0) -- (0.7,0,0) node[anchor=south west]{$\nu_j$};
 
     % gamma_j
     \draw (0.15,0.06,0) node [anchor=west] {$\gamma_j$};
 
     % Second plane
     \draw (1,0,0) -- (1,1,.5) -- (-1.5,1,.5) -- (-1.5,0,0);
     \draw (1,0,0) -- (1,-1,-.5) -- (-1.5,-1,-.5) -- (-1.5,0,0);
 
     \input{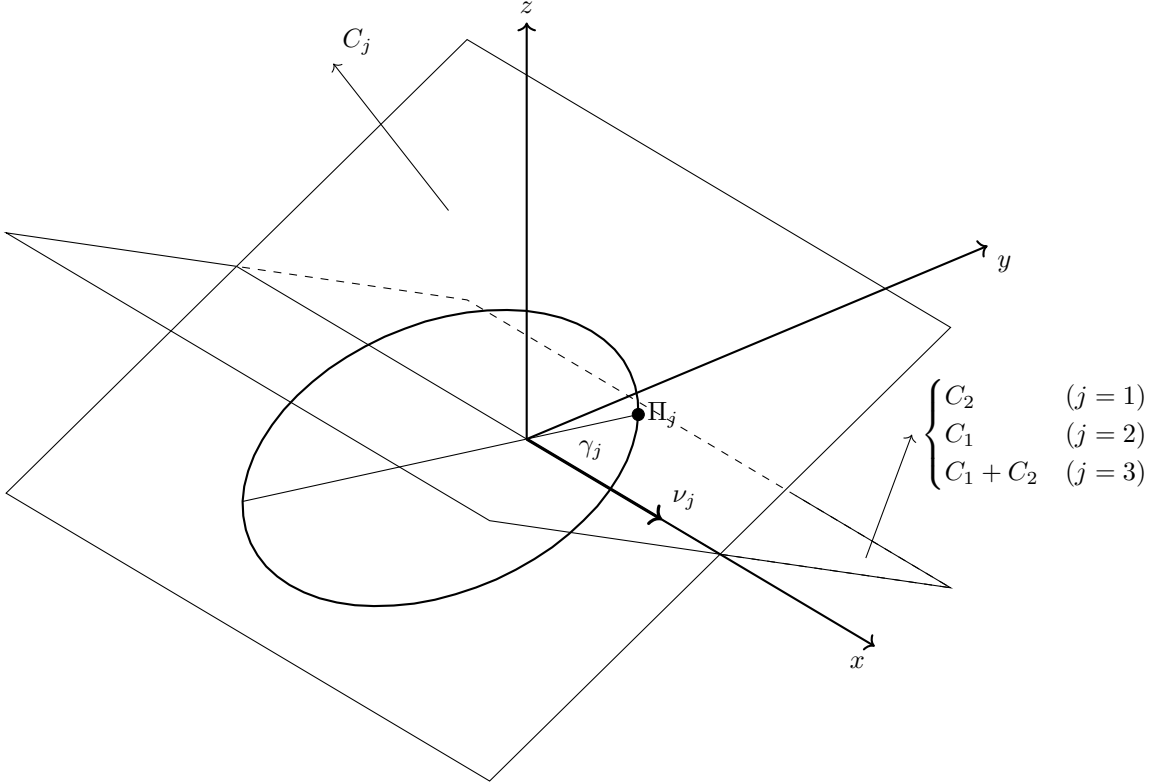}
 
     % \useasboundingbox (-2,2); % make bounding box higher
   \end{tikzpicture}
   \caption{The node $\nu_j$ and the argument $\gamma_j$ of pericenter $\Pi_j$}
   \label{fig:gammaj}
 \end{figure}

% \begin{lemma}
The Deprit variables form a symplectic coordinate system over $\cD$. In Appendix \ref{sec:app:Deprit} we give an alternative proof of its symplecticity from that in \cite{deprit1983,chierchia2011deprit}. 
% \end{lemma}

In these coordinates the Keplerian Hamiltonian becomes
\begin{equation}
F_{\mathrm{Kep}} = - \sum_{j=1}^3 \frac{\mu_j^3 M_j^2}{2 L_j^2}
\end{equation}
and so the mean anomalies $\ell_j$ undergo a rigid rotation with frequency $\sqrt{\frac{M_j}{a_j^3}}$, while all other variables remain fixed. Moreover  the variables $\psi_3, \Psi_2, \Psi_3$ are integrals of motion for the full 4-body problem (due to conservation of angular momentum), and so the Hamiltonian does not depend on $\Psi_3, \psi_2, \psi_3$. %See \cite{pinzari2015canonical} for a complete description of the reduction by rotational symmetry in these coordinates.

Note that one can express the orbital elements in terms of Deprit as follows. The osculating eccentricities are defined by
\begin{equation}\label{def:eccentricity}
 e_i=\sqrt{1-\frac{\Gamma_j^2}{L_j^2}},\qquad j=1,2,3.
\end{equation}
The mutual inclination between the bodies $1$ and $2$ is given by
 $i_{12}$, which is defined by
\begin{equation}\label{def:inclinationi12}
\cos i_{12} = \frac{\Gamma_1^2 + \Gamma_2^2 - \Psi_1^2}{2 \, \Gamma_1 \, \Gamma_2}.
\end{equation}
For the third planet one can measure its inclination with respect to the plane orthogonal to the vector $S_1=C_1+C_2$, that is with respect to the plane orthogonal to the angular momentum given by planet 1 and 2. If one denotes this angle by $i_{23}$, one has
\begin{equation}\label{def:inclinationi23}
 \cos i_{23}=\frac{\Gamma_3^2+\Psi_1^2-\Psi_2^2}{2\Gamma_3\Psi_1}.
\end{equation}
Note that we are in a regime where $L_2\gg L_1$ and therefore the angular momentum of the first two planets is essentially carried by planet 2.

\subsection{Arnold diffusion in Deprit coordinates}\label{sec:thm:deprit}
In this paper we consider both the hierarchical and planetary regimes. We first state the results in the hierarchical regime, which imply Theorem \ref{thm:main} for fixed values of masses. Later, we consider the planetary regime, that is we take masses of the planets arbitrarily small.
% \ref{thm:Main:Hiearch:orbelements}.

To this end, let us specify what is the hierarchical regime in Deprit coordinates. For now we assume that the masses $m_0,m_1,m_2,m_3>0$ are fixed and belong to $\mathcal{M}_\eta$ (see \eqref{eq:M}).

Then, we consider that semimajor axes of the bodies are well separated
\[
 a_1\ll a_2\ll a_3
\]
and equivalently
\[
 L_1\ll L_2\ll L_3.
\]
Note that we assume that the eccentricities of the bodies are uniformily bounded away from 0 and 1 and therefore 
% \[
%  \Gamma_i\sim L_i\qquad \text{for}\qquad i=1,2,3.
% \]
% and 
\begin{equation}\label{def:psigamma}
 \Gamma_i\sim L_i\quad \text{for}\quad i=1,2,3\qquad\text{and}\qquad \Psi_i\sim \Gamma_{i+1}\qquad \text{for}\qquad i=1,2
\end{equation}
(recall that $\Psi_2$ is the total angular momentum which is a conserved quantity).

In this regime the semimajor axes are very stable. The same happens for $\Gamma_3$ due to the conservation of angular momentum vector. However a ``small'' transfer of angular momentum from the third planet to the second can have a big effect on the eccentricty and inclination of the second planet. 

In particular,  $\Gamma_3$
% note that since $\Gamma_1\ll \Gamma_2\ll \Gamma_3$, $\Gamma_3$ 
may transition from 
\begin{equation}\label{def:Gamma3size}
 \Gamma_3\sim \Psi_2-\Psi_1 \qquad\text{to}\qquad  \Gamma_3\sim \Psi_2+\Psi_1.
\end{equation}
This corresponds to having $C_3$ and $C_1+C_2$ close to parallel and either with the same sense or opposite sense. Such transition is equivalent to make a transition in the inclination $\theta_{23}$ (see Section \ref{sec:FromDepritToElements} below).

On the other hand, note that $0<\Gamma_2<L_2$. So the maximal transition $\Gamma_2$ can make is
\begin{equation}\label{def:Gamma2size}
 \Gamma_2\sim L_2  \qquad\text{to}\qquad \Gamma_2\sim  0.
\end{equation}
By \eqref{def:eccentricity}, this corresponds to the second planet transitioning from a close to circular orbit ($e_2\sim 0$) to a highly eccentric one ($e_2\sim 1$). 

The next theorem shows that such transitions are possible and that one  can freely  vary $\Gamma_3$ and $\Gamma_2$ within there admitted ranges. In the constructed orbits the changes in the other actions is determined by those two.

\begin{theorem}[Hierarchical regime]\label{thm:MainHierarch:Deprit}
Fix masses $m_0, m_1, m_2, m_3>0$ such that 
\begin{equation}\label{def:massconditions:2}
m_0\neq m_1\qquad \text{ and }\qquad m_0+m_1\neq m_2. 
\end{equation}
There exists $0<\kappa\ll 1$, $\alpha>0$, $\beta>0$, such that the following is satisfied.

Fix $N\geq 1$ any $\{\nu^k\}_{k=0}^N\subset (0,1)$, $\{\eta^k\}_{k=0}^N\subset (-1,1)$ and constants $\Gamma_3^0\in [\kappa L_3,(1-\kappa)L_3]$, $\Psi_1^0,\Gamma_2^0\in  [\kappa L_2,(1-\kappa)L_2]$ such that 
%  \[
 $|\Psi_1^0-\Gamma_2^0|\leq \kappa$ , $\Psi_2^0$ such that $\Psi_2^0\in[\Gamma_3^0-(1-\kk)\Psi_1^0,\Gamma_3^0+(1-\kk)\Psi_1^0] $
%  \]
and 
$L^0=(L_1^0,L_2^0,L_3^0)$ satisfying $L_1^0\in [1/2,2]$ and
 \begin{equation}\label{eq_semimajorassumptionsinlvariable}
  L_1^0\ll L_2^0 \quad\text{and}\quad (L_2^0)^\frac{11}{6}\ll L_3^0\ll (L_2^0)^2.
 \end{equation}
 Then, 
%  for $L_2^0/L_1^0$ large enough
 there exists an orbit 
%  $(\ell(t), L(t), g(t),\Gamma(t), \psi(t), \Psi(t))$ 
 of the Hamiltonian $H$ in \eqref{def:HKepplusper} expressed in Deprit coordinates  and  times $\{t_k\}_{k=0}^N$ satisfying
 \[
t_0=0 \qquad\text{and}\qquad  |t_k|\leq \left(L_2^0\right)^\alpha,\, k\geq 1
 \]
 such that
%  its osculating orbital elements satisfy
 \[
 %  \begin{aligned}
% %   |e_2(0)-e_2^0|&\leq \delta\qquad
   |\Gamma_2(t_k)-\nu_k L_2^0|\leq  (L_2^0)^{-\beta}
%    \left(\frac{L_1^0}{L_2^0}\right)^\beta
   ,\qquad %  |\theta_{23}(0)-\theta_{23}^0|&\leq \delta\qquad  
 |\Psi_2^0-\Gamma_3(t_k)-\eta_k\Psi_1(t_k)|\leq  (L_2^0)^{-\beta}
%  \left(\frac{L_1^0}{L_2^0}\right)^\beta
% % \end{aligned}
 \]
 and
 \[
% %  \begin{aligned}
% %   |e_2(0)-e_2^0|&\leq \delta\qquad
   |\Gamma_1(t_k)-L_1^0|\leq  (L_2^0)^{-\beta}
%    \left(\frac{L_1^0}{L_2^0}\right)^\beta
   ,\qquad\qquad
% %  |\theta_{23}(0)-\theta_{23}^0|&\leq \delta\qquad  
 |\Psi_1(t_k)-\Gamma_2(t_k)-M_k|\leq  (L_2^0)^{-\beta}
%  \left(\frac{L_1^0}{L_2^0}\right)^\beta
% % \end{aligned}
 \]
where $M_k\in (0,\kappa)$ is determined by
\[
\frac{M_k^2}{(L_2^0-\Gamma_2(t_k))^{3/2}}=\frac{M_0^2}{(L_2^0-\Gamma_2^0)^{3/2}}  \qquad \text{and}\qquad M_0=\Psi_1^0-\Gamma_2^0
\]
% \[
% \cos\theta_{12}^k=\frac{(1-(e_2^k)^2)^{3/2}}{(1-(e_2^0)^2)^{3/2}}\cos\theta_{12}^0
% \]
whereas for all $t\in [0,t_N]$,
\[
% \begin{split}
%    |\theta_{12}(t)- ??|\\
   |\Gamma_3(t)-\Gamma_3^0|\leq 2L_2^0
%    \left(\frac{L_1^0}{L_2^0}\right)^\beta
   , \qquad  \text{and}\qquad 
   |L_j(t)-L_j^0|\leq (L_2^0)^{-\beta}
%    \left(\frac{L_1^0}{L_2^0}\right)^\beta 
   \quad \text{for}\quad j=1,2,3.
% \end{split}
   \]

\end{theorem}

One can obtain an analogous statement in the planetary regime, that is when one considers arbitrarily small masses for the bodies 1,2,3. That is, 
$ m_0\sim 1$ whereas $m_i=\rr \tilde m_i$, $i=1,2,3$, with $\tilde m_i\sim 1$ and $0<\rr\ll 1$. 
Note that when $\rr$ tends to $0$, the actions $(L,\Gamma, \Psi)$ all satisfy
\[
 L,\Gamma,\Psi\to 0.
\]
Therefore, to be able to capture the drift in actions it is convenient to perform a symplectic scaling
\[
 L=\rr \wt L, \quad \Gamma=\rr \wt \Gamma, \quad  \Psi=\rr \wt \Psi, \quad  
\]
Note that all the orbital elements are homogeneous functions of degree 0 of the Deprit coordinates (see \eqref{def:eccentricity}, \eqref{def:inclinationi12}, \eqref{def:inclinationi23}). Therefore, the drift in the scaled coordinates will determine the behavior described in Theorem \ref{thm:main} for the planetary regime.

\begin{theorem}[Planetary regime]\label{thm:PlanetHierarch:Deprit}
Fix $m_0,\wt m_1,\wt m_2,\wt m_3>0$  and consider the Hamiltonian $H$ in \eqref{def:HKepplusper} expressed in Deprit coordinates with  masses $m_0, m_j=\rr \wt m_j$ with $j=1,2,3$.  There exists $0<\kk_0\ll 1$, $\alpha>0$, $\beta>0$, such that the following is satisfied for any $\kk\in (0,\kk_0)$.

Fix $N\geq 1$ any $\{\nu^k\}_{k=0}^N\subset (0,1)$, $\{\eta^k\}_{k=0}^N\subset (-1,1)$ and constants
  $\wt \Psi_1^0,\wt \Psi_2^0,\wt \Gamma_2^0,\wt \Gamma_3^0,\wt L_1^0,\wt L_2^0,\wt L_3^0$ such that 
 \[
\wt\Psi_1^0,\wt\Gamma_2^0\in  [\kappa L_2,(1-\kappa)L_2],\, |\wt \Psi_1^0-\wt \Gamma_2^0|\leq \kappa ,\,\wt \Gamma_3^0\in (\kappa \wt L_3,(1-\kappa)\wt L_3) \quad\text{and}\quad  \wt\Psi_2^0\in[\wt\Gamma_3^0-(1-\kk)\wt\Psi_1^0,\wt\Gamma_3^0+(1-\kk)\wt\Psi_1^0] 
 \]
and 
% $\wt L^0=(\wt L_1^0,\wt L_2^0,\wt L_3^0)$ satisfying
 \[
\wt L_1^0\in [1/2,2],\quad \wt L_1^0\ll\wt L_2^0 \quad\text{and}\quad (\wt L_2^0)^\frac{11}{6}\ll\wt L_3^0\ll (\wt L_2^0)^2.
\]
%  with $\wt L_1^0/\wt L_2^0>0$ small enough.

%  with $C>0$ independent of $\rr$
 Then, there exists $\rr_0$ such that for any $\rr\in (0,\rr_0)$,
 there exists an orbit  of the Hamiltonian $H$ in \eqref{def:HKepplusper} expressed in scaled Deprit coordinates
%  $(\ell(t), \wt L(t), g(t),\wt \Gamma(t), \psi(t),\wt  \Psi(t))$ 
 and  times $\{t_k\}_{k=0}^N$ satisfying $t_0=0$ and 
 \[
  |t_k|\leq C\frac{(L_2^0)^\alpha}{\rr^2}
%   \left(\frac{L_2^0}{L_1^0}\right)^\alpha \qquad\text{for}\qquad k=1\ldots N
 \]
 with $C>0$ independent of $\rr$ and $L_2^0$, such that 
%  its osculating orbital elements satisfy
 \[
 %  \begin{aligned}
% %   |e_2(0)-e_2^0|&\leq \delta\qquad
   |\wt\Gamma_2(t_k)-\nu_k \wt L_2^0|\leq (\wt L_2^0)^{-\beta}
%    \left(\frac{\wt L_1^0}{\wt L_2^0}\right)^\beta
%    \kk \wt L_2^0
   ,\qquad %  |\theta_{23}(0)-\theta_{23}^0|&\leq \delta\qquad  
 |\wt\Psi_2^0-\wt\Gamma_3(t_k)-\eta_k\wt\Psi_1(t_k)|\leq  (\wt L_2^0)^{-\beta}
%  \left(\frac{\wt L_1^0}{\wt L_2^0}\right)^\beta
%  \kk \wt L_2^0
% % \end{aligned}
 \]
 and
 \[
% %  \begin{aligned}
% %   |e_2(0)-e_2^0|&\leq \delta\qquad
   |\wt\Gamma_1(t_k)-\wt L_1^0|\leq  (\wt L_2^0)^{-\beta}
%    \left(\frac{\wt L_1^0}{\wt L_2^0}\right)^\beta
%    \kk \wt L_1^0
   ,\qquad\qquad
% %  |\theta_{23}(0)-\theta_{23}^0|&\leq \delta\qquad  
 |\wt\Psi_1(t_k)-\wt\Gamma_2(t_k)-M_k|\leq  (\wt L_2^0)^{-\beta}
%  \left(\frac{\wt L_1^0}{\wt L_2^0}\right)^\beta
%  \kk \wt L_1^0
% % \end{aligned}
 \]
where $M_k\in (0,\kappa)$, $k=1\ldots N$, is determined by
\[
M_k^2=\frac{(\wt\Psi_1^0-\wt\Gamma_2^0)^2}{(\wt L_2^0-\wt \Gamma_2^0)^{3/2}}(\wt L_2^0-\wt \Gamma_2(t_k)^2)^{3/2}
\]
% \[
% \cos\theta_{12}^k=\frac{(1-(e_2^k)^2)^{3/2}}{(1-(e_2^0)^2)^{3/2}}\cos\theta_{12}^0
% \]
whereas for all $t\in [0,t_N]$,
\[
% \begin{split}
%    |\theta_{12}(t)- ??|\\
   |\wt\Gamma_3(t)-\wt\Gamma_3^0|\leq 2\wt L_2^0
%    \left(\frac{\wt L_1^0}{\wt L_2^0}\right)^\beta
%    \kk \wt L_2^0, 
   \qquad  \text{and}\qquad 
   |\wt L_j(t)-\wt L_j^0|\leq  (\wt L_2^0)^{-\beta}
%    \left(\frac{\wt L_1^0}{\wt L_2^0}\right)^\beta
%    \kk\wt L_j^0
   \quad \text{for}\quad j=1,2,3.
% \end{split}
   \]
\end{theorem}

Sections \ref{section_secularexpansion} - \ref{sec:mapreduction} are devoted to prove Theorem \ref{thm:MainHierarch:Deprit}. Then, in Section \ref{sec:planetary} we explain how to extend the proof in the planetary regime.

\section{The averaging procedure and the secular Hamiltonian}\label{section_secularexpansion}
The purpose of this section is to define the secular Hamiltonian, and compute all  terms relevant for our later analysis. This involves averaging the perturbing function \eqref{eq_fperdef} in Deprit coordinates, and performing a Taylor expansion of the resulting expressions in powers of the small parameters $\frac{1}{L_2}$ and $\frac{L_2}{L_3}$. In Section \ref{sec_averagingandsecular} we define the secular Hamiltonian, and compute the first two terms in its expansion. In Section \ref{subsection_secularexpansion}, we introduce affine coordinate transformations to split the action space in  substrips such that the new action variables in each strip are of order 1. This makes the multiple time scales of the problem visible, and allows us to perform further Taylor expansions. 

\subsection{The averaging procedure and the secular Hamiltonian}\label{sec_averagingandsecular}

It can be seen in equations \eqref{eq_perfn12}, \eqref{eq_perfn23} that $F_{\mathrm{per}}^{12}$, $F_{\mathrm{per}}^{23}$ are of order $O \left(\frac{1}{a_2^{3}} \right)$, $O \left( \frac{a_2^2}{a_3^{3}} \right)$ respectively. Therefore our assumption \eqref{eq_assumption1} implies that the mean anomalies $\ell_j$ have faster frequencies than the other angles, and moreover that the frequencies of the mean anomalies are of different orders. As a result we can apply the normal form theory (see \cite{fejoz2002quasiperiodic,zhao2014quasi}): we may construct an arbitrary number $k$ of non-resonant normal forms and find a change of coordinates close to the identity in order to obtain the Hamiltonian 
\begin{equation}\label{eq_4bphamafterave}
F = F_{\mathrm{Kep}} + F_{\mathrm{sec},k} +O \left(\left(\frac{1}{a_2} \right)^{k+1}, \, \frac{a_2^k}{a_3^{k+1}} \right)
\end{equation}
where $F_{\mathrm{sec},k}$ is the \emph{secular Hamiltonian} of order $k$, defined by 
\begin{equation} \label{eq_secularhamiltoniandef}
F_{\mathrm{sec},k} = F_{\mathrm{sec},k}^{12} + F_{\mathrm{sec},k}^{23}  + O \left( \frac{1}{a_3^3} \right)
\end{equation}
with
\begin{equation} \label{eq_secularhamiltonian12}
F_{\mathrm{sec},k}^{12} = \frac{1}{(2 \pi)^2} \int_{\mathbb{T}^2} F_{\mathrm{per}}^{12} \, d \ell_1 \, d \ell_2 + O \left( \frac{1}{a_2^{9/2}}  \right), 
\end{equation}
and
\begin{equation} \label{eq_secularhamiltonian23}
 F_{\mathrm{sec},k}^{23} = \frac{1}{(2 \pi)^2} \int_{\mathbb{T}^2} F_{\mathrm{per}}^{23} \, d \ell_2 \, d \ell_3 + O \left( \frac{a_2^4}{a_3^{9/2}} \right).
\end{equation}
In what follows we drop the $k$ subscript and simply write $F_{\mathrm{sec}}^{12}$ and $F_{\mathrm{sec}}^{23}$ to simplify notation. 

Now, the first term in the expansions \eqref{eq_perfn12}, \eqref{eq_perfn23} after averaging is called the \emph{quadrupolar} Hamiltonian (with respect to bodies 1 and 2, and with respect to bodies 2 and 3, respectively), and the second is called the \emph{octupolar} Hamiltonian (following the terminology of electrostatic multipoles). Since $\tilde{\sigma}_{j,2}=1$ for $j=1,2$, we write
\begin{equation}\label{eq_sechamexpansion}
F_{\mathrm{sec}}^{j,j+1}= - \frac{\mu_j \, m_{j+1}}{(2 \pi)^2} \left( F_{\mathrm{quad}}^{j,j+1} + \tilde{\sigma}_{j,3} \, F_{\mathrm{oct}}^{j,j+1} + O \left( \frac{ a_j^4}{a_{j+1}^5} \right) \right)\qquad \text{for $j=1,2$. }
\end{equation}
The eccentricity of the $j^{\mathrm{th}}$ Keplerian ellipse is given in terms of Deprit coordinates by
\begin{equation} \label{eq_eccentricity}
e_j = \sqrt{1 - \frac{\Gamma_j^2}{L_j^2}}.
\end{equation}

\begin{lemma}\label{lemma_quadoct12comp}
The quadrupolar and octupolar Hamiltonians of bodies 1 and 2 are given by
\begin{equation} \label{eq_quad12unscaled}
F_{\mathrm{quad}}^{12}= \frac{a_1^2}{8 \, a_{2}^3 \, (1-e_{2}^2)^{\frac{3}{2}}} \left( \left( 15 \, e_1^2 \cos^2 \gamma_1 - 12 \, e_1^2 - 3 \right) \sin^2 i_{12} + 3 e_1^2 + 2 \right)
\end{equation}
and 
\begin{equation} \label{eq_oct12unscaled}
\begin{split} 
F_{\mathrm{oct}}^{12} ={}& - \frac{15}{64} \frac{a_1^3}{a_2^4} \frac{e_1 \, e_2}{\left(1 -e_2^2 \right)^{\frac{5}{2}}} \\
& \times \left\{
\begin{split}
\cos \gamma_1 \, \cos \gamma_2 \left[
\begin{split}
\frac{\Gamma_1^2}{L_1^2} \left(5 \, \sin^2 i_{12} \left(6-7 \cos^2 \gamma_1 \right) - 3 \right) \\
-35 \, \sin^2 \gamma_1 \, \sin^2 i_{12} + 7
\end{split}
\right] \\
+ \sin \gamma_1 \, \sin \gamma_2 \, \cos i_{12} \left[
\begin{split}
\frac{\Gamma_1^2}{L_1^2} \left(5 \, \sin^2 i_{12} \left(4 - 7 \, \cos^2 \gamma_1 \right) - 3 \right) \\
- 35 \sin^2 \gamma_1 \, \sin^2 i_{12} + 7
\end{split}
\right]
\end{split}
\right\} 
\end{split}
\end{equation}
respectively, where $e_j$ is the eccentricity of the $j^{\mathrm{th}}$ Keplerian ellipse, and $i_{12}$ is the mutual inclination of Keplerian bodies $1$ and $2$, defined by
\begin{equation}\label{eq_cosi12def}
\cos i_{12} = \frac{\Gamma_1^2 + \Gamma_2^2 - \Psi_1^2}{2 \, \Gamma_1 \, \Gamma_2}.
\end{equation}
\end{lemma}
\begin{proof}
The quadrupolar and octupolar Hamiltonians are defined by
\begin{equation}
F_{\mathrm{quad}}^{12} = \int_{\mathbb{T}^2} P_2 \left( \cos \zeta_{1} \right) \frac{\| q_1 \|^2}{\| q_2 \|^3} \, d \ell_1 \, d \ell_{2}, \quad F_{\mathrm{oct}}^{12} = \int_{\mathbb{T}^2} P_3 \left( \cos \zeta_{1} \right) \frac{\| q_1 \|^3}{\| q_2 \|^4} \, d \ell_1 \, d \ell_{2}.
\end{equation}
In order to measure $\cos \zeta_1$, it is sufficient to derive expressions for $q_1,q_2$ in Deprit coordinates with respect to any orthonormal basis of $\mathbb{R}^3$. Recall the definition of the nodes $\nu_i$. As in Claim 1 of \cite{chierchia2011deprit} we consider the \emph{orbital basis} $B_i = (k_{i,1},k_{i,2},k_{i,3})$ for $i=1,2$, where
\begin{equation}
k_{i,1} = \frac{\nu_i}{\| \nu_i \|}, \quad k_{i,3} = \frac{C_i}{\| C_i \|},
\end{equation}
and where $k_{i,2}$ is chosen to make the basis orthonormal. We assume that $q_i \in \mathbb{R}^3 \setminus \{ 0 \}$. Let $\bar{q}_i = \| q_i \|^{-1} q_i$. Then, by the definition of the angle $\gamma_i$, the point $\bar{q}_i$ is written with respect to $B_i$ as $\bar{Q}_i=(\cos (\gamma_i+v_i), \sin (\gamma_i + v_i), 0)$, where $v_i$ is the true anomaly corresponding to the mean anomaly $\ell_i$. Define the standard rotation matrix by an angle $\theta \in \mathbb{T}$ around the $x$-axis by
\begin{equation}
\R_1 (\theta) = \left(
\begin{matrix}
1 & 0 & 0 \\
0 & \cos \theta & - \sin \theta \\
0 & \sin \theta & \cos \theta 
\end{matrix}
\right).
\end{equation}
Since $\nu_1=\nu_2$, the change of basis matrix from $B_1$ to $B_2$ is $\R_1 (i_{12})$, and so
\begin{equation}
\cos \zeta_1 = \R_1 (i_{12}) \bar{Q}_1 \cdot \bar{Q}_2.
\end{equation}
A standard computation (see Appendix C of \cite{fejoz2002quasiperiodic}; compare also equations (6), (7) of \cite{fejoz2016secular}) completes the proof.
\end{proof}

\begin{remark}
Analogous expressions exist for $F_{\mathrm{quad}}^{23}$ and $F_{\mathrm{oct}}^{23}$. However these expressions are long, and difficult to interpret. In order to transform these expressions into a more easily understandable form, we first perform a coordinate transformation that allows us to perform a further Taylor expansion. This coordinate transformation and expansion is carried out in Section \ref{subsection_secularexpansion}. 
\end{remark}

\subsection{Expansion of the secular Hamiltonian}\label{subsection_secularexpansion}

In what follows, we perform a perturbative analysis with respect to the small parameters $\frac{1}{L_2}$, $\frac{L_2}{L_3}$, $\frac{1}{L_3}$. In order to do this, first notice that the actions $\Gamma_1, \Gamma_2, \Psi_1, \Gamma_3$ are of different order; we therefore make an affine symplectic change of variables that results in actions of order 1. Indeed, our assumptions regarding the semi-major axes imply that $\Gamma_2, \Psi_1$ are of order $L_2$, while $\Gamma_3, \Psi_2$ are of order $L_3$. Recall that $\Psi_2$ is the norm of the total angular momentum, and therefore it is a first integral. We fix it as 
\[
 \Psi_2 = \delta_2 \, L_3
\]
for some fixed $\delta_2 \in (0,1)$ independent of $L_2$ and $L_3$. Note that different values of $\delta_2$ give different (approximate) values for $\Gamma_3$ (of order $L_3$) and therefore different (approximate) values for the osculating eccentricity of the ellipse of the third planet.

We  make the symplectic change of variables (with respect to the symplectic form $\Omega=\sum_{j=1}^3 d\Gamma_i\wedge d\gamma_i+d\Psi_1\wedge d\psi_1$):
\begin{equation} \label{eq_changeofcoordstilde}
\begin{dcases}
\tilde{\Psi}_1 = \Psi_1 - \delta_1 \, L_2, \quad & \tilde{\psi}_1 = \psi_1 + \gamma_2 \\
% \Psi_2 = \delta_2 \, L_3, \quad & \tilde{\psi}_2 = \psi_2 + \gamma_3 \\
\tilde{\Gamma}_2=\Psi_1-\Gamma_2, \quad & \tilde{\gamma}_2=- \gamma_2 \\
\tilde{\Gamma}_3 = \Psi_2 - \Gamma_3 - \delta_3 \, L_2, \quad & \tilde{\gamma}_3 =-  \gamma_3 
\end{dcases}
\end{equation}
where $\delta_1 \in (0,1)$ and $\delta_3 \in (-1,1)$ are constant with respect to the secular Hamiltonian. Note that this symplectic transformation does not modify the variables $\gamma_1,\Gamma_1$.

We assume that
\begin{equation} \label{eq_gamma2positive}
\tilde{\Gamma}_2 > 0
\end{equation}
as the case where $\tilde{\Gamma}_2$ is negative can be treated analogously.
Moreover, we assume that the new actions $\Gamma_1, \tilde{\Gamma}_2, \tilde{\Gamma}_3, \tilde{\Psi}_1$ live in a compact set away from the origin which is independent of $L_2$ and $L_3$. Indeed, we can choose $(\tilde{\Gamma}_3, \tilde{\Psi}_1)\in [-1,1]^2$. Then, choosing a discrete set of pairs $(\delta_1, \delta_3)$ appropriately, we can cover the whole domain that we want to analyse (see \eqref{def:psigamma}, \eqref{def:Gamma3size},\eqref{def:Gamma2size}). Since all the analysis we have to perform can be done locally, it is then enough to do it in each of these rectangles to achieve Arnold diffusion in the whole range of actions.

The rest of this section is dedicated to a Taylor expansion of the secular Hamiltonian in powers of $\frac{L_2}{L_3}$ and $\frac{1}{L_2}$; observe that these quantities are both small as a result of our assumption \eqref{eq_assumption1}. The following proposition summarises the results of the rest of this section, and its proof is effectively contained in the subsequent Lemmas \ref{lemma_quad12expansion}, \ref{lemma_oct12expansion}, \ref{lemma_quad23exp}, and \ref{lemma_oct23exp} in which we expand $F_{\mathrm{quad}}^{12}$, $F_{\mathrm{oct}}^{12}$, $F_{\mathrm{quad}}^{23}$, and $F_{\mathrm{oct}}^{23}$ respectively. The proposition is important as it tells us the order of the speed of each variable $\gamma_1, \tilde{\gamma}_2, \tilde{\psi}_1, \tilde{\gamma}_3, \Gamma_1, \tilde{\Gamma}_2, \tilde{\Psi}_1, \tilde{\Gamma}_3$. It also tells us the first order terms containing products of trigonometric functions of $\tilde{\psi}_1$ (resp. $\tilde{\gamma}_3$) with functions of $\gamma_1, \Gamma_1, \tilde{\gamma}_2$; this will be of great significance later on, as these will be the lowest order terms that give a nontrivial Poincar\'e-Melnikov potential in $\tilde{\psi}_1$ (resp. $\tilde{\gamma}_3$; see Proposition \ref{proposition_melnikovtildeexp} in Section \ref{sec:scattering}, as well as Section \ref{section_melnikovcomp}).

\begin{proposition}\label{proposition_secularexpansion}
The secular Hamiltonian \eqref{eq_secularhamiltoniandef} has the form
\begin{equation}\label{eq_secularhamexpansions}
F_{\mathrm{sec}} = c +  \sum_{i,j=0}^{\infty} \epsilon^i \mu^j F_{ij}
\end{equation}
where $\epsilon = \frac{1}{L_2}$ and $\mu = \frac{L_2}{L_3}$. 
Moreover the terms in this expansion satisfy the following properties.
\begin{enumerate}

\item \label{item_propquad12}
The first two nontrivial terms in the expansion are $F_{6,0} = \alpha_0^{12} \, H_0^{12}$, $F_{7,0} = \alpha_1^{12} \, H_1^{12}$ where $\alpha_i^{12}$ are nontrivial constants, and where the Hamiltonians $H_0^{12}$, $H_1^{12}$ are defined by \eqref{eq_H012def} and \eqref{eq_H112def} respectively, are integrable, and do not depend on the masses. The Hamiltonians $H_0^{12}$, $H_1^{12}$ are the first order terms from $F_{\mathrm{quad}}^{12}$ (see \eqref{eq_quad12unscaled}). The variables $\gamma_1, \tilde{\gamma}_2, \tilde{\Gamma}_2$ appear in $H_0^{12}$, and the action $\tilde{\Psi}_1$ first appears in $H_1^{12}$. 
\item \label{item_propoct12}
The angle $\tilde{\gamma}_2$ first appears in $H_2^{12}$, which is contained in $F_{8,0}$. The Hamiltonian $H_2^{12}$ is defined by \eqref{eq_foct12firstorder}, and is the first order term in the expansion of $F_{\mathrm{oct}}^{12}$ (see \eqref{eq_oct12unscaled}). 

\item \label{item_propquad23}
The angle $\tilde{\psi}_1$ first appears in $F_{2,6} = H_0^{23}$ where the Hamiltonian $H_0^{23}$ is defined by \eqref{eq_H023H123def}, and is the first order term coming from $F_{\mathrm{quad}}^{23}$ (see \eqref{eq_sechamexpansion}). 

\item \label{item_propquad23p2}
The term $H_2^{23}$, contained in $F_{3,6}$, is the first order term containing products of trigonometric functions of $\tilde{\psi}_1$ with functions of $\gamma_1, \Gamma_1, \tilde{\gamma}_2$. The Hamiltonian $H_2^{23}$ is defined by \eqref{eq_H223def}, and comes from $F_{\mathrm{quad}}^{23}$. 

\item \label{item_propactiongamma3}
The action $\tilde{\Gamma}_3$ first appears in $\tilde{H}_3$, which is contained in $F_{3,6}$. The Hamiltonian $\tilde{H}_3$ is defined by \eqref{eq_lowestorderGamma3}, and comes from $F_{\mathrm{quad}}^{23}$. 

\item \label{item_propoct23}
The angle $\tilde{\gamma}_3$ first appears in $H_3^{23}$, which is contained in $F_{2,8}$. The Hamiltonian $H_3^{23}$ is defined by \eqref{eq_H323def}, and is the first order term from $F_{\mathrm{oct}}^{23}$ (see \eqref{eq_sechamexpansion}). 

\item \label{item_propoct23p2}
The term $H_5^{23}$, contained in $F_{3,8}$, is the first order term containing products of trigonometric functions of $\tilde{\gamma}_3$ with functions of $\gamma_1, \Gamma_1, \tilde{\gamma}_2$. The Hamiltonian $H_5^{23}$ is defined by \eqref{eq_H523def} and comes from $F_{\mathrm{oct}}^{23}$. 

\end{enumerate}
\end{proposition}

\begin{proof}
Notice that, as a result of our assumption \eqref{eq_assumption1} on the semimajor axes, we have the inequalities
\begin{equation} \label{eq_ordersofmagnitude}
\frac{1}{L_2^6} \gg \frac{1}{L_2^7} \gg \frac{L_2^4}{L_3^6} \gg \frac{L_2^5}{L_3^7} \gg \frac{L_2^3}{L_3^6}, \qquad \frac{1}{L_2^8} \gg \frac{L_2^3}{L_3^6} \gg \frac{L_2^6}{L_3^8} 
\end{equation}
Combining these inequalities with the contents of Lemmas \ref{lemma_quad12expansion}, \ref{lemma_oct12expansion}, \ref{lemma_quad23exp}, and \ref{lemma_oct23exp} completes the proof of the proposition. 
\end{proof}

\begin{notation}
Throughout this paper, in order to simplify notation, we use ellipsis to mean the following. Fix some sufficiently large integer $r \in \mathbb{N}$. The notation $F = \epsilon^i \, \mu^j \, G + \cdots $ means that there are $\eta_1, \, \eta_2 \in \mathbb{N}_0$, not both 0, and a positive constant $C$ such that
\[
\left\| F - \epsilon^i \, \mu^j \, G \right\|_{C^r} \leq C \, \epsilon^{i + \eta_1} \, \mu^{j + \eta_2}. 
\]
Moreover, we use the expression \emph{nontrivial constant} to mean a constant depending only on the masses and the parameters $\delta_j$ that is nonzero for all $m_0, \, m_1, \, m_2, \, m_3 >0$, all $\delta_1, \, \delta_2 \in (0,1)$, and all $\delta_3 \in (-1,1)$. 
\end{notation}

\begin{lemma}\label{lemma_quad12expansion}
The Hamiltonian $F_{\mathrm{quad}}^{12}$ can be written in the  variables \eqref{eq_changeofcoordstilde} as 
\begin{equation}
F_{\mathrm{quad}}^{12} = \tilde{c}_0^{12} + \frac{1}{L_2^6} \,  \alpha_0^{12} \,  H_0^{12} + \frac{1}{L_2^7} \, \alpha_1^{12} \, H_1^{12} + \frac{1}{L_2^8} \, \tilde{\alpha}_2 \, \tilde{H}_2 + \cdots
\end{equation}
where
\begin{align}
H_0^{12} ={}& \left( 1 -\frac{\Gamma_1^2}{L_1^2} \right) \left[ 2 - 5 \left(1 - \frac{\tilde{\Gamma}_2^2}{\Gamma_1^2} \right) \sin^2 \gamma_1 \right] + \frac{\tilde{\Gamma}_2^2}{L_1^2} \label{eq_H012def} \\
H_1^{12} ={}& \left( 3 H_0^{12} \left( \gamma_1, \Gamma_1, \tilde{\Gamma}_2 \right) - 1 \right) \tilde{\Psi}_1 - 4 \tilde{\Gamma}_2 H_0^{12} \left( \gamma_1, \Gamma_1, \tilde{\Gamma}_2 \right) + 3 \tilde{\Gamma}_2 - \frac{\Gamma_1^2 \tilde{\Gamma}_2}{L_1^2} \label{eq_H112def} \\
\tilde{H}_2 ={}& \left( 3 H_0^{12} \left( \gamma_1, \Gamma_1, \tilde{\Gamma}_2 \right) - 1 \right) \tilde{\Psi}_1^2 + \left( 6 - 8 H_0^{12} \left( \gamma_1, \Gamma_1, \tilde{\Gamma}_2 \right) - 2 \frac{\Gamma_1^2}{L_1^2} \right) \tilde{\Gamma}_2 \tilde{\Psi}_1 \\
& \quad + \frac{1}{8}  \Bigg[ \sin^2 \gamma_1 \, \left( 5\,\Gamma_{1}^2 - {{5\,\Gamma_{1}^4}\over{\,L_{1}^2}}+{{210\,
\Gamma_{1}^2 \,  \tilde{\Gamma}_{2}^2}\over{\,L_{1}^2}}-{{205\,
 \tilde{\Gamma}_{2}^4}\over{\,L_{1}^2}}+{{205\,\tilde{\Gamma}_{2}^4}\over{
 \,\Gamma_{1}^2}} \right) \label{eq_quad12termoforderepsilonsquared}\\
 & \quad + {{\Gamma_{1}^4}\over{\,L_{1}^2}}-{{66\,\Gamma_{1}^2 \,\tilde{\Gamma}_{2}^2}\over{
 \,L_{1}^2}}+{{41\,\tilde{\Gamma}_{2}^4}\over{\,L_{1}^2}}+40\,\tilde{\Gamma}_{2}^2  \Bigg]
\end{align}
and
\begin{equation}
\alpha_0^{12} = {{3\,L_{1}^4\,M_{2}^3\,\mu_{2}^6}\over{8\,M_{1}^2\,\delta_{1}^3\,
 \mu_{1}^4}}, \quad \alpha_1^{12} = -{{3\,L_{1}^4\,M_{2}^3\,\mu_{2}^6}\over{8\,M_{1}^2\,\delta_{1}^4\,
 \mu_{1}^4}}, \quad \tilde{\alpha}_2 = {{3\,L_{1}^4\,M_{2}^3\,\mu_{2}^6}\over{4\,M_{1}^2\,\delta_{1}^5\,
 \mu_{1}^4}}. 
\end{equation}
Moreover $F_{\mathrm{quad}}^{12}$ is integrable. 
\end{lemma}

\begin{proof}
Recall from the definition of the Deprit variables in Section \ref{sec:DepritResults} that $a_j = \frac{L_j^2}{\mu_j^2 \, M_j}$, the eccentricity $e_j$ is defined by \eqref{eq_eccentricity}, and the inclination $i_{12}$ is defined via its cosine in \eqref{eq_cosi12def}. It follows from \eqref{eq_changeofcoordstilde} that 
\begin{equation}
\Gamma_2 = \delta_1 \, L_2 + \tilde{\Psi}_1 - \tilde{\Gamma}_2, \quad \Psi_1 = \delta_1 \, L_2 + \tilde{\Psi}_1. 
\end{equation}
Therefore
\begin{equation}\label{eq_oneminuse2sqexp}
1 - e_2^2 = \frac{\Gamma_2^2}{L_2^2} = \delta_1^2 + \frac{1}{L_2} \, 2 \delta_1 \, \left( \tilde{\Psi}_1 - \tilde{\Gamma}_2 \right) + \frac{1}{L_2^2} \left( \tilde{\Psi}_1 - \tilde{\Gamma}_2 \right)^2
\end{equation}
and
\begin{equation}\label{eq_cosi12exp}
\cos i_{12} = \frac{\Gamma_1^2 + \Gamma_2^2 - \Psi_1^2}{2 \, \Gamma_1 \, \Gamma_2} = - \frac{\tilde{\Gamma}_2}{\Gamma_1} + \frac{1}{L_2} \, \frac{\Gamma_1^2 - \tilde{\Gamma}_2^2}{2 \, \delta_1 \, \Gamma_1} - \frac{1}{L_2^2} \, \frac{\left( \Gamma_1^2 - \tilde{\Gamma}_2^2 \right) \tilde{\Psi}_1 - \Gamma_1^2 \, \tilde{\Gamma}_2 + \tilde{\Gamma}_2^3}{2 \, \delta_1^2 \, \Gamma_1} + O \left( \frac{1}{L_2^3} \right). 
\end{equation}
Combining these formulas with \eqref{eq_quad12unscaled} and expanding in powers of $\frac{1}{L_2}$ yields the formulas \eqref{eq_H012def}, \eqref{eq_H112def}, and \eqref{eq_quad12termoforderepsilonsquared}. Finally, the integrability of $F_{\mathrm{quad}}^{12}$ is due to the fact that it does not depend on $\gamma_2$, and therefore $\Gamma_2$ is a constant of motion. Indeed, it is thus a Hamiltonian system with two degrees of freedom and two integrals of motion. 
\end{proof}

\begin{lemma}\label{lemma_oct12expansion}
The Hamiltonian $F_{\mathrm{oct}}^{12}$ can be written in the rescaled variables \eqref{eq_changeofcoordstilde} as
\begin{equation}
F_{\mathrm{oct}}^{12} = \frac{1}{L_2^8} \, \alpha_2^{12} \, H_2^{12} + \cdots 
\end{equation}
where
\begin{equation}
\begin{split} 
H_{2}^{12} ={}& \sqrt{1 - \frac{\Gamma_1^2}{L_1^2}}  \left\{
\begin{split}
\cos \gamma_1 \, \cos \tilde{\gamma}_2 \left[
\begin{split}
\frac{\Gamma_1^2}{L_1^2} \left(5 \, \left( 1 - \frac{\tilde{\Gamma}_2^2}{\Gamma_1^2}\right) \left(6-7 \cos^2 \gamma_1 \right) - 3 \right) \\
-35 \, \sin^2 \gamma_1 \, \left( 1 - \frac{\tilde{\Gamma}_2^2}{\Gamma_1^2}\right) + 7
\end{split}
\right] \\
+ \frac{\tilde{\Gamma}_2}{\Gamma_1} \, \sin \gamma_1 \, \sin \tilde{\gamma}_2 \,  \left[
\begin{split}
\frac{\Gamma_1^2}{L_1^2} \left(5 \, \left( 1 - \frac{\tilde{\Gamma}_2^2}{\Gamma_1^2}\right) \left(4 - 7 \, \cos^2 \gamma_1 \right) - 3 \right) \\
- 35 \sin^2 \gamma_1 \, \left( 1 - \frac{\tilde{\Gamma}_2^2}{\Gamma_1^2}\right) + 7
\end{split}
\right]
\end{split}
\right\} \label{eq_foct12firstorder}
\end{split}
\end{equation}
and
\begin{equation}
\alpha_2^{12} = - \frac{15}{64} \frac{L_1^6 \mu_2^8 M_2^4}{\mu_1^6 M_1^3} \frac{\sqrt{1 - \delta_1^2}}{\delta_1^5}.
\end{equation}
\end{lemma}

\begin{proof}
Similarly to the proof of Lemma \ref{lemma_quad12expansion}, we combine the formulas for $a_j$, the eccentricity $e_j$, as well as the expansions \eqref{eq_oneminuse2sqexp} and \eqref{eq_cosi12exp} with the formula \eqref{eq_oct12unscaled} for $F_{\mathrm{oct}}^{12}$ and expand in powers of $\frac{1}{L_2}$ to obtain \eqref{eq_foct12firstorder}. 
\end{proof}

\begin{lemma}\label{lemma_quad23exp}
The Hamiltonian $F_{\mathrm{quad}}^{23}$ can be written in the rescaled variables \eqref{eq_changeofcoordstilde} as
\begin{equation}
F_{\mathrm{quad}}^{23} = \tilde{c}_0^{23} + \frac{L_2^4}{L_3^6} \, \alpha_0^{23} \, H_0^{23} + \frac{L_2^5}{L_3^7} \, \alpha_1^{23} \, H_1^{23} + \frac{L_2^3}{L_3^6} \, \alpha_2^{23} \, H_2^{23} + \cdots
\end{equation}
with
\begin{equation} \label{eq_H023H123def}
H_0^{23} = \cos^2 \tilde{\psi}_1 + O \left( \frac{L_2}{L_3}, \frac{1}{L_2} \right), \quad H_1^{23} = \sin^2 \tilde{\psi}_1 + O \left( \frac{L_2}{L_3},  \frac{1}{L_2} \right)
\end{equation}
where the higher order terms depend only on $\tilde{\psi}_1, \tilde{\Psi}_1, \tilde{\Gamma}_2, \tilde{\Gamma}_3$, and
\begin{equation}\label{eq_H223def}
H_2^{23} = \sqrt{\Gamma_1^2 - \tilde{\Gamma}_2^2} \, \left( c_1 \, \cos \tilde{\psi}_1 \, \cos \tilde{\gamma}_2 + c_2 \, \sin \tilde{\psi}_1 \, \sin \tilde{\gamma}_2 \right) + O \left( \frac{L_2}{L_3},  \frac{1}{L_2} \right)
\end{equation}
where
\begin{equation} \label{eq_c1c2def}
c_1 = \delta_1^2, \quad c_2=-(5-4 \, \delta_1^2).
\end{equation}
Moreover, the term of order $L_2^{-1}$ in the expansion of $H_0^{23}$ is $\frac{1}{L_2} \, \tilde{c}_3 \,  \tilde{H}_3$ where $\tilde{c}_3$ is a nontrivial constant, and 
\begin{equation}
\begin{split} \label{eq_lowestorderGamma3}
\tilde{H}_3 ={}& \tilde{\Psi}_1 \left[ 5  (\delta_3^2 - \delta_1^4)\, \cos^2 \tilde{\psi}_1  - 5 \delta_3^2 + 3 \delta_1^4 \right]  + \tilde{\Gamma}_2 \left[ 5 ( \delta_1^4 - \delta_1^2 \delta_3^2) \, \cos^2 \tilde{\psi}_1 + 4 \delta_1^2 \delta_3^2 - 3 \delta_1^4 \right] \\
& \qquad + \tilde{\Gamma}_3 \left[ 5 (\delta_1^3 \delta_3 - \delta_1 \delta_3) \, \cos^2 \tilde{\psi}_1 + 5 \delta_1 \delta_3 - 4 \delta_1^3 \delta_3 \right].
\end{split}
\end{equation}
\end{lemma}

\begin{proof}
By definition, we have
\begin{equation}\label{eq_quad23beforeaver}
F_{\mathrm{quad}}^{23} = \int_{\mathbb{T}^2} P_2 \left( \cos \zeta_{2} \right) \frac{\| q_2 \|^2}{\| q_3 \|^3} \, d \ell_2 \, d \ell_{3}.
\end{equation}
Denote by $\R_1(\theta)$, $\R_3(\theta)$ the rotation matrix by an angle $\theta$ around the $x, z$-axis respectively, and let $I_3 = \R_3 (\pi)$. Write $\bar{q}_j= \| q_j \|^{-1} \, q_j$, and $\bar{Q}_j = (\cos (\gamma_j + v_j), \sin(\gamma_j + v_j),0)$ where $v_j$ is the true anomaly corresponding to the mean anomaly $\ell_j$. By Proposition 4.1 of \cite{pinzari2009kolmogorov}, we have
\begin{equation}
\bar{q}_2 = \R_3 (\psi_3) \, \R_1 (i) \, \R_3 (\psi_2) \, \R_1 (\tilde{i}_2) \, \R_3 (\psi_1) \, I_3 \, \R_1 (i_2) \, \bar{Q}_2
\end{equation}
and
\begin{equation}
\bar{q}_3 = \R_3 (\psi_3) \, \R_1 (i) \, \R_3 (\psi_2) \, I_3 \, \R_1 (i_3) \, \bar{Q}_3
\end{equation}
where
\begin{equation}\label{eq_inclinationsdef1}
\cos i = \frac{\Psi_3}{\Psi_2}, \quad \cos \tilde{i}_2 = \frac{\Psi_2^2 + \Psi_1^2 - \Gamma_3^2}{2 \, \Psi_1 \, \Psi_2}, 
\end{equation}
\begin{equation}\label{eq_inclinationsdef2}
\cos i_2 = \frac{\Gamma_2^2 + \Psi_1^2 - \Gamma_1^2}{2 \, \Psi_1 \, \Gamma_2}, \quad \cos i_3 = \frac{\Gamma_3^2 + \Psi_2^2 - \Psi_1^2}{2 \, \Psi_2 \, \Gamma_3}.
\end{equation}
Since the last 3 rotations performed in each expression $\bar{q}_2, \bar{q}_3$ are the same, they can be ignored in the computation of $\cos \zeta_2 = \bar{q}_2 \cdot \bar{q}_3$. 

First, we focus on the rotations by the angles $i_2, i_3$. Observe that, in our rescaled variables,
\begin{equation}
\cos i_2 = 1 - \frac{1}{L_2^{2}} \frac{\Gamma_1^2 - \tilde{\Gamma}_2^2}{2 \, \delta_1^2} + O \left(\frac{1}{ L_2^{3}} \right), \quad \sin i_2 =\frac{1}{ L_2} \frac{ \sqrt{\Gamma_1^2 - \tilde{\Gamma}_2^2}}{\delta_1} + O \left(\frac{1}{L_2^2} \right),
\end{equation}
\begin{equation}
\cos i_3 = 1 - \left( \frac{L_2}{L_3} \right)^2 \frac{\delta_1^2 - \delta_3^2}{2 \, \delta_2^2} + O \left( \left( \frac{L_2}{L_3} \right)^3, \frac{L_2}{L_3^2}  \right), 
\end{equation}
\begin{equation}
\sin i_3 = \frac{L_2}{L_3} \frac{ \sqrt{ \delta_1^2 - \delta_3^2}}{\delta_2} + O \left( \left( \frac{L_2}{L_3} \right)^2, \frac{L_2}{L_3^2} \right).
\end{equation}
The square roots in the above terms are real-valued: indeed, the triangle inequality implies that
\begin{equation}
\Psi_1 \leq \Gamma_1 + \Gamma_2, \quad \Psi_2 \leq \Psi_1 + \Gamma_3.
\end{equation}
Inserting the rescaled variables into these inequalities and using \eqref{eq_gamma2positive} yields
\begin{equation}
\Gamma_1 \geq \tilde{\Gamma}_2 > 0, \quad \delta_1 > \delta_3 > 0.
\end{equation}
Therefore we can write
\begin{equation}
\R_1 (i_2) = \Id + L_2^{-1} M_2 + O \left( L_2^{-2} \right), \quad \R_1 (i_3) = \Id + \frac{L_2}{L_3} M_3 + O \left( \left( \frac{L_2}{L_3} \right)^2, \frac{L_2}{L_3^2} \right)
\end{equation}
where
\begin{equation}
M_j = \left(
\begin{matrix}
0 & 0 & 0 \\
0 & 0 & -b_j \\
0 & b_j & 0
\end{matrix}
\right) \quad \mathrm{with} \quad 
\begin{dcases}
b_2 = \frac{\sqrt{\Gamma_1^2 - \tilde{\Gamma}_2^2}}{\delta_1} \\
b_3 = \frac{\sqrt{\delta_1^2 - \delta_3^2}}{\delta_2}.
\end{dcases}
\end{equation}
We thus obtain the expression 
\begin{equation}\label{eq_coszeta2exp}
\cos \zeta_2 = \bar{q}_2 \cdot \bar{q}_3 = W_0 + \frac{L_2}{L_3} W_1 + \frac{1}{L_2} W_2 + O \left( \left( \frac{L_2}{L_3} \right)^2, \frac{1}{L_3}, \frac{1}{L_2^2} \right)
\end{equation}
where
\begin{align}
W_0& = \R_1 (\tilde{i}_2) \, \R_3 (\psi_1) \, I_3 \, \bar{Q}_2 \cdot I_3 \, \bar{Q}_3,\label{eq_sec23expw0}\\
% \end{equation}
% \begin{equation}
W_1 &= \R_1 (\tilde{i}_2) \, \R_3 (\psi_1) \, I_3 \, \bar{Q}_2 \cdot I_3 \, M_3 \, \bar{Q}_3,\label{eq_sec23expw1}\\
% \end{align}
% \begin{equation}\label{eq_sec23expw2}
W_2 &= \R_1 (\tilde{i}_2) \, \R_3 (\psi_1) \, I_3 \, M_2 \, \bar{Q}_2 \cdot I_3 \, \bar{Q}_3.\label{eq_sec23expw2}
\end{align}
Recall the Legendre polynomial of degree 2 is $P_2 (x)= \frac{1}{2} \left( 3 \, x^2 - 1 \right)$. Thus
\begin{equation}
P_2 (\cos \zeta_2) = P_2 (W_0)+ 3 \, \frac{L_2}{L_3} \, W_0 \, W_1 + 3 \, \frac{1}{L_2} \, W_0 \, W_2 + O \left( \left( \frac{L_2}{L_3} \right)^2 \right).
\end{equation}
Integrating these 3 terms separately (using again the technique introduced in Appendix C of \cite{fejoz2002quasiperiodic}) and applying the change of coordinates \eqref{eq_changeofcoordstilde}, we obtain
\begin{equation}
F_{\mathrm{quad}}^{23} = \frac{a_2^2}{a_3^3 \, (1 - e_3^2)^{\frac{3}{2}}} \left( K_0 + \frac{L_2}{L_3} K_1 + \frac{1}{L_2} K_2 + O \left( \left( \frac{L_2}{L_3} \right)^2 \right) \right)
\end{equation}
where
\begin{equation}
\begin{split}
K_0& = \frac{1}{8} \left(\left(15\,e_{2}^2\,\cos ^2 \tilde{\psi}_1-12\,e_{2}^2-3 \right)\,\sin ^2\tilde{i}_2+3\,e_{2}^2+2 \right),\\
% \end{equation}
% \begin{equation}
K_1& = - \frac{3 \, b_3}{4} \,\cos \tilde{i}_2\,\sin \tilde{i}_2\, \left( (1-e_2^2) +5 \, e_2^2 \,  \sin^2 \tilde{\psi}_1  \right), \\
K_2 &  = 
\frac{3 \, b_2}{4} \,\cos \tilde{i}_2\,\sin \tilde{i}_2\,\left(\left(1-e_{2}^2 \right) \,\cos \tilde{\psi}_{1} \cos \tilde{\gamma}_2 - (1+4 \,e_{2}^2) \,\sin  \tilde{\psi}_{1} \,\sin \tilde{\gamma}_{2}
 \right).
 \end{split}
\end{equation}
Finally, combining the above equations with \eqref{eq_sechamexpansion} and observing that, due to \eqref{eq_eccentricity},
\begin{equation}
e_2^2 = (1 - \delta_1^2) + \frac{1}{L_2} \, 2 \delta_1 \left( \tilde{\Gamma}_2 - \tilde{\Psi}_1 \right) + O \left( L_2^{-2} \right), \quad e_3^2 = 1 - \delta_2^2 + O \left( \frac{L_2}{L_3} \right),
\end{equation}
\begin{equation}
\cos \tilde{i}_2 = \frac{\delta_3}{\delta_1} + \frac{L_2}{L_3} \frac{\delta_1^2 - \delta_3^2}{2 \delta_1 \delta_2} + \frac{1}{L_2} \frac{\delta_1 \tilde{\Gamma}_3 - \delta_3 \tilde{\Psi}_1}{\delta_1^2} + O \left( \frac{1}{L_3} \right), 
\end{equation}
and
\begin{equation}
\sin^2 \tilde{i}_2 = \left( 1 - \frac{\delta_3^2}{\delta_1^2} \right) - \frac{L_2}{L_3} \frac{2 \delta_3}{\delta_1} \frac{\delta_1^2 - \delta_3^2}{2 \delta_1 \delta_2} - \frac{1}{L_2} \frac{ 2 \delta_3}{\delta_1} \frac{ \delta_1 \tilde{\Gamma}_3 - \delta_3 \tilde{\Psi}_1}{\delta_1^2}  + O \left( \left(\frac{L_2}{L_3} \right)^2, \frac{1}{L_3} \right)
\end{equation}
where the higher-order terms depend only on $\tilde{\Psi}_1, \tilde{\Gamma}_2, \tilde{\Gamma}_3$, we obtain the expressions \eqref{eq_H023H123def} and \eqref{eq_H223def}. Moreover, expanding $K_0$ up to terms of order $L_2^{-1}$ and using the above formulas, we see that the lowest-order term containing $\tilde{\Gamma}_3$ is the expression $\tilde{H}_3$ defined by \eqref{eq_lowestorderGamma3}. 
\end{proof}

\begin{lemma}\label{lemma_oct23exp}
The Hamiltonian $F_{\mathrm{oct}}^{23}$ can be written in the rescaled variables \eqref{eq_changeofcoordstilde} as
\begin{equation}
F_{\mathrm{oct}}^{23} = \tilde{c}_1^{23} + \frac{L_2^6}{L_3^8} \, \alpha_3^{23} \, H_3^{23}+  \frac{L_2^7}{L_3^9} \, \alpha_4^{23} \, H_4^{23} + \frac{L_2^5}{L_3^8} \, \alpha_5^{23} \, H_5^{23} + \cdots
\end{equation}
where:
\begin{itemize}
\item
The Hamiltonian $H_3^{23}$ is defined by 
\begin{equation}\label{eq_H323def}
H_3^{23} = \nu_0 \, \cos \left( \tilde{\gamma}_3 + 3 \tilde{\psi}_1 \right) + \nu_1 \, \cos \left( \tilde{\gamma}_3 + \tilde{\psi}_1 \right) + \nu_2 \, \cos \left( \tilde{\gamma}_3 -  \tilde{\psi}_1 \right) + \nu_3 \, \cos \left( \tilde{\gamma}_3 - 3 \tilde{\psi}_1 \right)
\end{equation}
where
\begin{equation}\label{eq_constantsmuline1}
\nu_0 = \frac{35}{8 \, \delta_1^3} \, \left(\delta_1^2 - 1 \right) \left(\delta_3-\delta_1 \right) \left(\delta_3+\delta_1 \right)^2, \quad \nu_1 = - \frac{1}{8 \, \delta_1^3} \, \left(3 \delta_1^2-7 \right) \left(\delta_3+\delta_1 \right) \left(15 \delta_3^2-10 \delta_1 \delta_3 - \delta_1^2 \right),
\end{equation}
\begin{equation} \label{eq_constantsmuline2}
\nu_2 = \frac{1}{8 \, \delta_1^3} \, \left(3 \delta_1^2-7 \right) \left(\delta_3-\delta_1 \right) \left(15 \delta_3^2 +10 \delta_1 \delta_3 - \delta_1^2 \right), \quad \nu_3 = - \frac{35}{8 \, \delta_1^3} \, \left(\delta_1^2 - 1 \right) \left(\delta_3+\delta_1 \right) \left(\delta_3-\delta_1 \right)^2.
\end{equation}
Moreover $H_3^{23}$ is the lowest-order term containing $\tilde{\gamma}_3$.
\item
The Hamiltonian $H_4^{23}$ does not depend on $\gamma_1, \Gamma_1, \tilde \gamma_2$.
\item
The Hamiltonian $H_5^{23}$ is given by 
\begin{equation}\label{eq_H523def}
H_5^{23} = \sqrt{ \Gamma_1^2 - \tilde{\Gamma}_2^2 } \, \left( J_1 \left( \tilde{\psi}_1, \tilde{\gamma}_3 \right) \, \cos \tilde{\gamma}_2 + J_2 \left( \tilde{\psi}_1, \tilde{\gamma}_3 \right) \, \sin \tilde{\gamma}_2 \right)
\end{equation}
where
\[
\begin{aligned}
J_1 \left( \tilde{\psi}_1, \tilde{\gamma}_3 \right) ={}& 
30\,\delta_{3}^2\,\sin \tilde{\gamma}_{3}\,\cos \tilde{\psi}_{1}\,\sin 
 \tilde{\psi}_{1}-10\,\delta_{1}^2\,\sin \tilde{\gamma}_{3}\,\cos \tilde{\psi}_{1}\,
 \sin \tilde{\psi}_{1}-20\,\delta_{1}\,\delta_{3}\,\cos \tilde{\gamma}_{3}\,\cos ^2
 \tilde{\psi}_{1}+10\,\delta_{1}\,\delta_{3}\,\cos \tilde{\gamma}_{3} \\
J_2 \left( \tilde{\psi}_1, \tilde{\gamma}_3 \right) ={}& 
-50\,\delta_{1}\,\delta_{3}\,\cos \tilde{\gamma}_{3}\,\cos \tilde{\psi}_{1}\,\sin 
 \tilde{\psi}_{1}+ \frac{70\,\delta_{3}}{\delta_1}\,\cos \tilde{\gamma}_{3}\,\cos \tilde{\psi}_{1}\,
 \sin \tilde{\psi}_{1}+\frac{105\,\delta_{3}^2}{\delta_1^2}\,\sin 
 \tilde{\gamma}_{3}\,\cos ^2\tilde{\psi}_{1} \\
 & -75\,\delta_{3}^2\,
 \sin \tilde{\gamma}_{3}\,\cos ^2\tilde{\psi}_{1}+25\,\delta_{1}^2\,\sin \tilde{\gamma}_{3}\,\cos 
 ^2\tilde{\psi}_{1}-35\,\sin \tilde{\gamma}_{3}\,\cos ^2\tilde{\psi}_{1}- \frac{105\,
 \delta_{3}^2}{\delta_1^2} \,\sin \tilde{\gamma}_{3} \\
 & +60\,\delta_{3}^2\,
 \sin \tilde{\gamma}_{3}-17\,\delta_{1}^2\,\sin \tilde{\gamma}_{3}+28\,\sin \tilde{\gamma}_{3}.
\end{aligned}
\]
\end{itemize}
\end{lemma}

\begin{proof}
Similarly to the proof of Lemma \ref{lemma_quad23exp}, we use formula \eqref{eq_coszeta2exp} for $\cos \zeta_2$ and the Legendre polynomial $P_3 (x) = \frac{1}{2} (5 x^3 - 3x)$ to obtain
\begin{equation}
P_3 \left( \cos \zeta_2 \right) = P_3 \left( W_0 \right) + \frac{L_2}{L_3} \, \frac{1}{2} \, \left( 15 \, W_0^2 - 3 \right) \, W_1 + \frac{1}{L_2} \frac{1}{2} \, \left( 15 \, W_0^2 - 3 \right) \, W_2 + O \left( \left( \frac{L_2}{L_3} \right)^2 \right).
\end{equation}
Computing this using formulas \eqref{eq_sec23expw0}, \eqref{eq_sec23expw1}, and \eqref{eq_sec23expw2} for $W_0, W_1, W_2$, substituting the result into 
\begin{equation}
F_{\mathrm{oct}}^{23} = \int_{\mathbb{T}^2} P_3 \left( \cos \zeta_{2} \right) \frac{\| q_2 \|^3}{\| q_3 \|^4} \, d \ell_2 \, d \ell_{3},
\end{equation}
integrating these 3 terms separately using the technique introduced in Appendix C of \cite{fejoz2002quasiperiodic}, and applying the change of coordinates \eqref{eq_changeofcoordstilde} completes the proof of the lemma. 
\end{proof}

\section{Analysis of the first-order Hamiltonian}
\label{section_analysisofh0}
% The Hamiltonian $F_{\mathrm{sec}}^{12}$ does not  depend on the angles $\psi_1$ 
% and $\gamma_3$ nor on the action $\Gamma_3$. Therefore, one can reduce the 
% Hamiltonian to a two degree Hamiltonian with respect to the variables 
% $(\gamma_1,\Gamma_1,\gamma_2,\Gamma_2)$ which depends on the parameter $\Psi_1$ 
% (and certainly on $L_1$ and $L_2$). 
% 
% The dynamics of $F_{\mathrm{sec}}^{12}$ was analyzed  in the paper  
% \cite{fejoz2016secular}. We devote this section to provide all the results of 
% \cite{fejoz2016secular} which are needed for this paper. We also analyze the 
% dynamics of $F_{\mathrm{sec}}^{12}$ in the extended phase space, that is 
% considering its dynamics also in the variables $(\psi_1,\Psi_1)$
% 
% This is done in two steps. First, in Section  \ref{sec:DynamicsQuad12}, we 
% analyze the dynamics of the first order of $F_{\mathrm{sec}}^{12}$ which is 
% given by $H_0^{12}$ (see Proposition \ref{proposition_secularexpansion}). Then, 
% in Section \ref{sec:DynamicsSec12}, we analyze the dynamics of 
% $F_{\mathrm{sec}}^{12}$  in the extended phase space.
% 
% \subsection{Analysis of $H_0^{12}$}\label{sec:DynamicsQuad12}

The purpose of this section is to analyse the first order term $H_0^{12}$ in the expansion of the secular Hamiltonian. We establish the existence of a saddle periodic orbit (in an interval of energy levels), the stable and unstable manifolds of which coincide. Collecting the saddle periodic orbits in this interval of energy levels yields a normally hyperbolic invariant manifold with a separatrix. In addition, we obtain an explicit time parametrisation of the separatrix. 

It follows from Proposition \ref{proposition_secularexpansion} that the first term in the expansion of the secular Hamiltonian $F_{\mathrm{sec}}$ is $H_0^{12}$. A convenient property of the Deprit coordinates is that the quadrupolar Hamiltonian of the interaction between planets 1 and 2 in the 4-body problem (and indeed the $N$-body problem) coincides with the quadrupolar Hamiltonian from the 3-body problem, expressed in Delaunay coordinates. Furthermore, $H_0^{12}$ corresponds precisely to its counterpart from the 3-body problem \cite{fejoz2016secular} (see also Appendix \ref{appendix_errata}). In this section we recall results from \cite{fejoz2016secular} regarding the existence of 2 hyperbolic periodic orbits of $H_0^{12}$, connected by a separatrix. 

Since $H_0^{12}$ does not depend on $\tilde{\gamma}_2$, we may consider $\tilde{\Gamma}_2$ as a parameter. Differentiating \eqref{eq_H012def}, we see that Hamilton's equations of motion are
\begin{equation}
\begin{dcases}
\dot{\gamma}_1 = \frac{\partial H_0^{12}}{\partial \Gamma_1} = \frac{2 \, \Gamma_1}{L_1^2} \left[ 5 \, \left(1 - \frac{\tilde{\Gamma}_2^2}{\Gamma_1^2} \right) \, \sin^2 \gamma_1 - 2 \right] - 10 \, \left( 1 - \frac{\Gamma_1^2}{L_1^2} \right) \, \frac{\tilde{\Gamma}_2^2}{\Gamma_1^3} \, \sin^2 \gamma_1 \\
\dot{\Gamma}_1 = - \frac{\partial H_0^{12}}{\partial \gamma_1} = 5 \, \left(1 - \frac{\Gamma_1^2}{L_1^2} \right) \, \left( 1 - \frac{\tilde{\Gamma}_2^2}{\Gamma_1^2} \right) \, \sin 2 \gamma_1.
\end{dcases}
\end{equation}
We seek equilibria of the Hamiltonian vector field. Although we have assumed that the eccentricities $e_j$ satisfy $0 < e_j < 1$, which implies that $\Gamma_1 \in (0,L_1)$, the Hamiltonian $H_0^{12}$ is analytic on a neighbourhood of the cylinder $(\gamma_1, \Gamma_1) \in \mathbb{T} \times (0,L_1)$ in $\mathbb{T} \times \mathbb{R}$. Observe that $\dot{\Gamma}_1=0$ if $\Gamma_1=L_1$. In this case we have $\dot{\gamma}_1=0$ if and only if
\begin{equation} \label{eq_equilibriumcondition1}
\sin^2 \gamma_1 = \frac{2}{5 \, \left(1 - \frac{\tilde{\Gamma}_2^2}{L_1^2} \right)}.
\end{equation}
Assuming that
\begin{equation} \label{eq_equilibriumcondition2}
\left| \tilde{\Gamma}_2 \right| < L_1 \sqrt{\frac{3}{5}},
\end{equation}
equation \eqref{eq_equilibriumcondition1} has two solutions in the interval $\gamma_1 \in (0,\pi)$ (and two more for $\gamma_1\in(\pi,2\pi)$). One of these solutions, which we denote by $\gamma_1^{\mathrm{min}}$ lies in the interval $\left(0, \frac{\pi}{2} \right)$, and the other solution is $\gamma_1^{\mathrm{max}} = \pi - \gamma_1^{\mathrm{min}}$. 

We have thus found two equilibria $(\gamma_1, \Gamma_1) = (\gamma_1^{\mathrm{min,max}},L_1)$. These equilibria correspond to circular ellipses, and it can be shown by making a suitable change of coordinates that they are hyperbolic \cite{fejoz2016secular}. Lifting the equilibria to the full phase space $(\gamma_1, \Gamma_1, \tilde{\gamma}_2, \tilde{\Gamma}_2)$ of $H_0^{12}$, we obtain the two hyperbolic periodic orbits
\begin{equation}
Z^0_{\mathrm{min,max}} \left(t, \tilde{\gamma}_2^0 \right) = \left( \gamma_1^{\mathrm{min,max}}, L_1, \tilde{\gamma}_2^0 + \tilde{\gamma}_2^1 (t), \tilde{\Gamma}_2 \right)
\end{equation}
where $\tilde{\gamma}_2^0\in\mathbb{T}$ is the initial condition, and
\begin{equation} \label{eq_gammafrequency}
\tilde{\gamma}_2^1 (t) = \frac{2 \, \tilde{\Gamma}_2}{L_1^2}t
\end{equation}
is determined by differentiating \eqref{eq_H012def} with respect to $\tilde{\Gamma}_2$ and setting $\Gamma_1=L_1$. 

Suppose \eqref{eq_equilibriumcondition2} holds, and recall moreover we have assumed in  \eqref{eq_gamma2positive} that $\tilde{\Gamma}_2>0$. Define the positive constants
\begin{equation} \label{eq_chia2def}
\chi = \sqrt{\frac{2}{3}} \frac{\tilde{\Gamma}_2}{L_1} \frac{1}{\sqrt{1 - \frac{5}{3} \frac{\tilde{\Gamma}_2^2}{L_1^2}}}, \quad A_2 = \frac{6}{L_1} \sqrt{ \frac{2}{3}} \sqrt{1 - \frac{5}{3} \frac{\tilde{\Gamma}_2^2}{L_1^2}}.
\end{equation}
The proof of the following result is identical to the proof of Lemma 3.1 in \cite{fejoz2016secular}. 

\begin{lemma}\label{lemma_separatrixformulas}
There is a heteroclinic orbit of $H_0^{12}$ joining $Z^0_{\mathrm{max}}$ and $Z^0_{\mathrm{min}}$ backward and forward in time respectively. It is defined by the equation
\begin{equation}
\left(1 - \frac{\tilde{\Gamma}_2^2}{\Gamma_1^2} \right) \, \sin^2 \gamma_1 = \frac{2}{5}
\end{equation}
where $\gamma_1 \in (\gamma_1^{\mathrm{min}}, \gamma_1^{\mathrm{max}}) \subset (0, \pi)$, and its time parametrisation is given by
\begin{equation}
Z^0 (t, \tilde{\gamma}_2^0) = \left( \gamma_1 (t), \Gamma_1 (t), \tilde{\gamma}_2 (t), \tilde{\Gamma}_2 \right)
\end{equation}
where
\begin{equation} \label{eq_cosgamma1separatrix}
\cos \gamma_1 (t) = \sqrt{ \frac{3}{5}} \, \frac{ \sinh (A_2 \, t)}{\sqrt{\chi^2 + (1+\chi^2) \, \sinh^2 (A_2 \, t)}},
\end{equation}
\begin{equation} \label{eq_Gamma1separatrix}
\Gamma_1 (t) = \tilde{\Gamma}_2 \sqrt{\frac{5}{3}} \, \frac{\sqrt{1 + \frac{3}{5} \frac{L_1^2}{\tilde{\Gamma}_2^2} \, \sinh^2 (A_2 \, t)}}{\cosh (A_2 \, t)},
\end{equation}
and
\begin{equation}\label{eq_gamma22separatrix}
\tilde{\gamma}_2 (t) = \tilde{\gamma}_2^0 + \tilde{\gamma}_2^1 (t) + \tilde{\gamma}_2^2 (t),
% \end{equation}
\qquad\text{with}\qquad
% \begin{equation} 
\tilde{\gamma}_2^2 (t) = \arctan \left( \chi^{-1} \tanh (A_2 \, t) \right).
\end{equation}
\end{lemma}

Even though the Hamiltonian function $H_0^{12}$ is analytic near $\{ \Gamma_1=L_1 \}$, the Deprit coordinates, as is the case with Delaunay coordinates, are singular on this hypersurface. Indeed, on the circular ellipse $\Gamma_1=L_1$, the argument $\gamma_1$ of the perihelion is without meaning. We therefore introduce the Poincar\'e variables
\begin{equation} \label{eq_poincarevariables}
\xi = \sqrt{2 \, (L_1 - \Gamma_1)} \, \cos \gamma_1, \quad \eta = - \sqrt{2 \, (L_1 - \Gamma_1)} \, \sin \gamma_1.
\end{equation}
This is a symplectic change of variables, in the sense that
\begin{equation} \label{eq_poincarevariablesaresymplectic}
d \xi \wedge d \eta = d \Gamma_1 \wedge d \gamma_1. 
\end{equation}
In these variables, the Hamiltonian $H_0^{12}$ becomes
\begin{equation}\label{eq_H012inpoincarevariables}
\tilde{H}_0^{12} = \frac{1}{L_1} \left[ 2 \, \xi^2 - \left( 3 - 5 \frac{\tilde{\Gamma}_2^2}{L_1^2} \right) \eta^2 \right] + \frac{\tilde{\Gamma}_2^2}{L_1^2} + O_2 \left( \xi^2 + \eta^2 \right)
\end{equation}
and the entire hypersurface $\{ \Gamma_1 = L_1 \}$ becomes a single hyperbolic periodic orbit
\begin{equation}
\left(\xi, \eta, \tilde{\gamma}_2, \tilde{\Gamma}_2 \right) = \left(0,0, \tilde{\gamma}_2^0 + \tilde{\gamma}_2^1 (t), \tilde{\Gamma}_2 \right).
\end{equation}
Moreover, the heteroclinic connection established in Lemma \ref{lemma_separatrixformulas} becomes a homoclinic connection to this hyperbolic periodic orbit. 

On the hyperbolic periodic orbit and the separatrix, the energy is given by $\frac{\tilde{\Gamma}_2^2}{L_1^2}$. It follows that we have a hyperbolic periodic orbit and a homoclinic connection for each positive value of $\tilde{\Gamma}_2$ satisfying \eqref{eq_equilibriumcondition2}. In other words, the Hamiltonian $\tilde{H}_0^{12}$ has a normally hyperbolic invariant manifold given by
\begin{equation} \label{eq_nhim0}
\Lambda_0 = \left\{ \left(\xi, \eta, \tilde{\gamma}_2, \tilde{\Gamma}_2 \right) :(\xi,\eta)=(0,0),  \tilde{\gamma}_2\in\mathbb{T}, \tilde{\Gamma}_2  \in [\zeta_1, \zeta_2] \right\}
\end{equation}
where $\zeta_1, \zeta_2$ satisfy 
\begin{equation}\label{eq_nhimparametersdef}
0<\zeta_1<\zeta_2<L_1 \sqrt{\frac{3}{5}}. 
\end{equation}
Moreover the stable and unstable manifolds of $\Lambda_0$ coincide.

\section{The inner dynamics}
\label{sec:innerdyn}
The goal of this section is to establish the existence of a normally hyperbolic invariant manifold $\Lambda$ for the secular Hamiltonian, and to determine a set of coordinates in which the inner dynamics (i.e. the dynamics carried by $\Lambda$) can be analysed. 

In Section \ref{sec:cylinder}, we consider a graph parameterization of the normally hyperbolic invariant cylinder and analyse the pull back of the Hamiltonian $F_{\mathrm{sec}}$ defined in \eqref{eq_secularhamexpansions} into the cylinder. We also average this Hamiltonian and show that it is very close to integrable. Finally, in Section~\ref{sec:Hessian}, we compute the first and second derivatives of this Hamiltonian with respect to the actions. Such analysis will be fundamental in Section~\ref{sec:mapreduction} to prove that this Hamiltonian has torsion.

\subsection{The parametrisation of the cylinder and the inner Hamiltonian}\label{sec:cylinder}
The normally hyperbolic invariant manifold $\Lambda_0$ defined by \eqref{eq_nhim0} can be lifted to the full secular phase space simply by including the remaining variables, to obtain 
\begin{equation} \label{eq_nhim0tilde}
\tilde{\Lambda}_0 = \left\{ \left(\xi, \eta, \tilde{\gamma}_2, \tilde{\Gamma}_2, \tilde{\psi}_1, \tilde{\Psi}_1, \tilde{\gamma}_3, \tilde{\Gamma}_3 \right): \xi=\eta=0,\, \tilde{\gamma}_2, \tilde{\psi}_1,  \tilde{\gamma}_3\in\mathbb{T}, \, \tilde{\Gamma}_2\in [\zeta_1,\zeta_2], \, \tilde{\Psi}_1 \in [-1,1], \, \tilde{\Gamma}_3\in [-1,1]\right\}.
\end{equation}
This set obviously remains a normally hyperbolic invariant manifold for $H_0^{12}$. It is diffeomorphic to $\mathbb{T}^3 \times \left[ 0,1 \right]^3$ and its stable and unstable manifolds, each of dimension 7, coincide. In a neighbourhood of $\tilde{\Lambda}_0$, the symplectic form is
\begin{equation}
\Omega = d \xi \wedge d \eta + d \tilde{\Gamma}_2 \wedge d \tilde{\gamma}_2 + d \tilde{\Psi}_1 \wedge d \tilde{\psi}_1 + d \tilde{\Gamma}_3 \wedge d \tilde{\gamma}_3
\end{equation}
due to \eqref{eq_poincarevariablesaresymplectic}.
The restriction of $\Omega$ to $\tilde{\Lambda}_0$ is
\begin{equation} \label{eq_omega0}
\Omega_0 = \left. \Omega \right|_{\tilde{\Lambda}_0} = d \tilde{\Gamma}_2 \wedge d \tilde{\gamma}_2 + d \tilde{\Psi}_1 \wedge d \tilde{\psi}_1 + d \tilde{\Gamma}_3 \wedge d \tilde{\gamma}_3.
\end{equation}

Note that, as a result of our assumption \eqref{eq_assumption1} on the semimajor axes of the Keplerian ellipses, we have
\begin{equation}
\frac{L_2^{11}}{L_3^6} \ll 1.
\end{equation}
 The following is the main result of this section. 

\begin{theorem} \label{theorem_innerdynamics}
For any $r \geq 2$ there is $L_2^* > 0$ such that for all $L_2 \geq L_2^*$ and all $L_3 \in \left[ L_3^-, L_3^+ \right]$ (where $L_3^+ >L_3^- >0$ depend on $L_2$) the following holds.
\begin{enumerate}
\item
There is a $C^r$ smooth normally hyperbolic invariant manifold $\Lambda$ of $F_{\mathrm{sec}}$ that is $O \left( L_2^{-1} \right)$ close to $\tilde{\Lambda}_0$ in the $C^r$ topology. Moreover the variables $\left(\tilde{\gamma}_2, \tilde{\Gamma}_2, \tilde{\psi}_1, \tilde{\Psi}_1, \tilde{\gamma}_3, \tilde{\Gamma}_3 \right)$ define coordinates on $\Lambda$ with respect to which the restriction of the symplectic form $\Omega$ to $\Lambda$ is closed and non-degenerate (but not necessarily in Darboux form). 
\item
Fix any $k_1, k_2 \in \mathbb{N}$. Then there is a coordinate transformation 
\begin{equation}\label{eq_tildehatcoordtransf}
\left(\tilde{\gamma}_2, \tilde{\Gamma}_2, \tilde{\psi}_1, \tilde{\Psi}_1, \tilde{\gamma}_3, \tilde{\Gamma}_3 \right) \mapsto \left(\hat{\gamma}_2, \hat{\Gamma}_2, \hat{\psi}_1, \hat{\Psi}_1, \hat{\gamma}_3, \hat{\Gamma}_3 \right)
\end{equation}
on $\Lambda$ that is $O \left( \frac{L_2^{11}}{L_3^6}\right)$ close to the identity in the $C^r$ topology such that $\left. \Omega \right|_{\Lambda}$ becomes the standard symplectic form, and the secular Hamiltonian $F_{\mathrm{sec}}$, when restricted to $\Lambda$, becomes
\begin{equation}\label{eq_innerhamiltonianave}
\hat{F} = \hat{F}_0 \left(\hat{\Gamma}_2, \hat{\Psi}_1, \hat{\Gamma}_3; \epsilon, \mu \right) + \epsilon^{k_1} \mu^{k_2} \hat{F}_1 \left( \hat{\gamma}_2, \hat{\Gamma}_2, \hat{\psi}_1, \hat{\Psi}_1, \hat{\gamma}_3, \hat{\Gamma}_3; \epsilon, \mu \right)
\end{equation}
where $\hat{F}_0 = \epsilon^6  c_0 \hat{\Gamma}_2^2 + \epsilon^7 \hat{h}_0 \left(\hat{\Gamma}_2, \hat{\Psi}_1, \hat{\Gamma}_3; \epsilon, \mu \right)$, where $\epsilon = \frac{1}{L_2}$, $\mu = \frac{L_2}{L_3}$, and where the $C^r$ norms of $\hat{h}_0$  and $\hat{F}_j$ are uniformly bounded in $\epsilon, \mu$ for $j=0,1$. 
\end{enumerate}
\end{theorem}

\begin{remark}\label{remark_fenicheltheory}
The existence of the normally hyperbolic invariant manifold $\Lambda$ for $F_{\mathrm{sec}}$ follows from Fenichel theory. Indeed, $F_{\mathrm{sec}}$ is $O( L_2^{-7})$ close to its first-order term $\alpha_0^{12} L_2^{-6} H_0^{12}$ (see Proposition \ref{proposition_secularexpansion}). Since $L_2^{-1}$ is the small parameter in this instance, and since Fenichel theory applies to vector fields of order $O(1)$, we must scale $F_{\mathrm{sec}}$ by $L_2^6$ (by scaling time). We thus obtain a Hamiltonian that is $O(L_2^{-1})$ close to its first-order term $\alpha_0^{12} H_0^{12}$. Since $H_0^{12}$ has a normally hyperbolic invariant manifold $\tilde{\Lambda}_0$ defined by \eqref{eq_nhim0tilde}, Fenichel theory implies that $L_2^6 F_{\mathrm{sec}}$ has a normally hyperbolic invariant manifold $\Lambda$ that is $O(L_2^{-1})$ close to $\tilde{\Lambda}_0$ in the $C^r$ topology \cite{fenichel1971persistence,fenichel1974asymptotic,fenichel1977asymptotic}. Clearly $\Lambda$ is a normally hyperbolic invariant manifold for the secular Hamiltonian $F_{\mathrm{sec}}$ itself. The smoothness $r$ of $\Lambda$ in Theorem \ref{theorem_innerdynamics} can be made as large as required by increasing $L_2^*$ if necessary; indeed, $\tilde{\Lambda}_0$ is $C^{\infty}$, and the Lyapunov exponents of the flow of $H_0^{12}$ on $\tilde{\Lambda}_0$ are 0 in the directions transverse to the flow. We fix some sufficiently large value of $r$, and choose $L_2^*$ large enough so that $\Lambda$ is $C^r$ smooth. 
\end{remark}

\begin{remark}\label{remark_shrinkingnhim}
To be precise, whenever we make a coordinate transformation (see \eqref{eq_tildehatcoordtransf} for example; the same applies to Lemmas \ref{lemma_straightsympform} and \ref{lemma_averaging} below), we must shrink the range of the actions (see \eqref{eq_nhim0tilde}) on the cylinder $\Lambda$ in order to continue to use the results of Section \ref{section_analysisofh0}: there is a small $\varrho>0$ such that for all sufficiently large $L_2, \, L_3$, the new cylinder is defined for $\hat \Psi_1 \in [-1 + \varrho, 1 - \varrho], \, \hat \Gamma_2 \in [\zeta_1 + \varrho, \zeta_2 - \varrho], \hat \Gamma_3 \in [-1 + \varrho, 1 - \varrho]$. To maintain simplicity of notation, we continue to refer to this cylinder as $\Lambda$. 
\end{remark}

The rest of the section is dedicated to the proof of Theorem \ref{theorem_innerdynamics}. Then, in Section \ref{sec:Hessian}, we analyse the second derivatives $\hat{h}_0$. Such analysis will be used later on in Section \ref{sec:mapreduction} to show that certain Poincar\'e map associated to the flow $\hat{F_0}$ has non-degenerate torsion.

By Remark \ref{remark_fenicheltheory}, Fenichel theory guarantees the existence of a function $\rho : \tilde{\Lambda}_0 \to \mathbb{R}^2$ such that 
\begin{equation}\label{eq_nhimdef}
\Lambda = \mathrm{graph} (\rho) = \left\{ ( \rho(x), x) : x = 
\left(\tilde{\gamma}_2, \tilde{\Gamma}_2, \tilde{\psi}_1, \tilde{\Psi}_1, 
\tilde{\gamma}_3, \tilde{\Gamma}_3 \right) \, \mathrm{ with } \, \tilde \gamma_2, \tilde \psi_1, \tilde \gamma_3 \in \mathbb{T}; \, \tilde \Gamma_2 \in [\zeta_1, \zeta_2]; \, \tilde \Psi_1, \, \tilde \Gamma_3 \in [-1,1] \right\}
\end{equation}
where $\zeta_1, \zeta_2$ satisfy \eqref{eq_nhimparametersdef} and such that $\rho$ is $O \left( L_2^{-1} \right)$ close to 0 in the $C^r$ topology. Note that $\rho$ represents the values of the Poincar\'e variables $\xi, \eta$ on the cylinder $\Lambda$, expressed as a function of the variables $\left(\tilde{\gamma}_2, \tilde{\Gamma}_2, \tilde{\psi}_1, \tilde{\Psi}_1, \tilde{\gamma}_3, \tilde{\Gamma}_3 \right)$, and so we can consider these variables as coordinates on $\Lambda$. The following lemma provides information regarding the orders at which each of these variables first appears in the Taylor expansion of $\rho$. 
\begin{lemma}\label{lemma_fenichelgraphexp}
The function $\rho$ can be expanded in the form
\begin{equation}\label{eq_fenichelgraphexp}
\rho = \frac{1}{L_2} \rho_0 + \frac{1}{L_2^2} \rho_1 + \frac{L_2^{10}}{L_3^6} \rho_2+ \hat{\epsilon} \, \rho_3
\end{equation}
where $\hat{\epsilon} \lesssim \frac{L_2^9}{L_3^6}$ and
\begin{equation}\label{eq_fenichelgraphdependence}
\begin{dcases}
\rho_0 = \rho_0 \left(\tilde{\Gamma}_2 \right), \\
\rho_1 = \rho_1 \left( \tilde{\gamma}_2, \tilde{\Gamma}_2, \tilde{\Psi}_1 \right), \\
\rho_2 = \rho_2 \left( \tilde{\gamma}_2, \tilde{\Gamma}_2, \tilde{\psi}_1, \tilde{\Psi}_1\right),\\
\rho_3 = \rho_3 \left( \tilde{\gamma}_2, \tilde{\Gamma}_2, \tilde{\psi}_1, \tilde{\Psi}_1, \tilde{\gamma}_3, \tilde{\Gamma}_3 \right),
\end{dcases}
\end{equation}
where the $C^r$ norms of $\rho_i$, $i=0,1,2,3$, are uniformly bounded with respect to $L_2$ and $L_3$.
\end{lemma}
\begin{proof}
As already explained in Remark \ref{remark_fenicheltheory}, the existence of the $C^r$ normally hyperbolic invariant manifold and its graph parameterization is a direct consequence of Fenichel Theory. It only remains to compute its expansion. We follow the approach considered in \cite{delshams2006biggaps} by Delshams, de la Llave and Seara (see also \cite{hirsch1970invariant}).

The actions $\tilde{\Psi}_1$ and $\tilde{\Gamma}_3$ first appear in the terms of order $\frac{1}{L_2^7}$ and $\frac{L_2^3}{L_3^6}$, respectively, in the expansion of the  secular Hamiltonian (see Proposition \ref{proposition_secularexpansion}). It is therefore clear that $\rho_0, \rho_1$ do not depend on $\tilde{\gamma}_3, \tilde{\Gamma}_3$. However it is not clear a priori that $\rho_0$ does not depend on $\tilde{\psi}_1, \tilde{\Psi}_1$, and that $\rho_1$ does not depend on $\tilde{\gamma}_3, \tilde{\Gamma}_3$. 

% This proof is split into two parts, in which we establish these facts separately. 

We make the proof in  steps to analyse each of the terms $\rr_i$, $i=0,1,2,3$.

For the first term, it is enough to point out that $H_0^{12}+L_2^{-1}H_1^{12}$ is an integrable Hamiltonian and its vector field only depends on $\tilde\gamma_1$, $\tilde \Gamma_1$ and $\tilde \Gamma_2$. Then,  it can be seen as a 1 degree of freedom Hamiltonian depending on the parameter $\tilde \Gamma_2$. If one writes this vector field in Poincar\'e coordinates $(\eta,\xi)$, it  has a saddle close to $(\eta,\xi)=(0,0)$ which is of the form 
\[
 (\eta,\xi)=\frac{1}{L_2}\rr_0(\tilde\Gamma_2).
\]
In the full phase space, this saddle defines the normally hyperbolic invariant cylinder of $H_0^{12}+L_2^{-1}H_1^{12}$.

The second step is to analyse $\rr_1$ by means of perturbative arguments. Let us consider the vector field $X$ associated to $L_2^6F_{\mathrm{sec}}^{12}$ expressed in $(\eta,\xi)$ coordinates, which can be written as 
$ X=X_0+\delta X_1$ where $X_0$ is the vector field associated to Hamiltonian $H_0^{12}+L_2^{-1}H_1^{12}$ analysed in the first step, 
$X_1$ is the vector field associated to the rest of the terms in $L_2^6F_{\mathrm{sec}}^{12}$ and $\delta=L_2^{-2}$.

We look for an equation for $\rr_1$. To this end, let us denote $x=( \tilde{\gamma}_2, \tilde{\Gamma}_2, \tilde{\psi}_1, \tilde{\Psi}_1, \tilde{\gamma}_3, \tilde{\Gamma}_3)$ and write $X^\ii$ for the $x$-components and $X^\nn$ for the $(\eta,\xi)$-components. The superscripts $\ii$ and $\nn$ stand for inner and normal. One can define analgously the components of $X_0$ and $X_1$.

Fenichel Theory ensures that the vector field $X=X_0+\delta X_1$ has a normally hyperbolic invariant manifold. Since $\delta$ is small, we look for a parameterization of it of the form
\[
 (\eta,\xi)=\frac{1}{L_2}\rr_0(\tilde\Gamma_2)+\de \rr_1(x).
\]
We consider the invariance equation to show that $\rr_1$ only depends on $(\tilde{\gamma}_2, \tilde{\Gamma}_2, \tilde{\Psi}_1)$. Indeed, it reads
\[
 X^\nn\left(L_2^{-1}\rr_0(\tilde\Gamma_2)+\de \rr_1(x)\right)=\left(L_2^{-1}\rr_0(\tilde{\Gamma}_2)+D\rr_1(x)\right)X^\ii\left(L_2^{-1}\rr_0(\tilde\Gamma_2)+\de \rr_1(x)\right).
\]
Note that $DX^\nn_0$ is a hyperbolic matrix independent of $L_2$. One can obtain the solution of this equation by means of a fixed point argument inverting the operator
\[
\mathcal{L}h=\left[ DX^\nn(0)-D\right]h.
\]
Then, since $X^\nn$ and $X^\ii$ only depend on $(\tilde{\gamma}_2, \tilde{\Gamma}_2, \tilde{\Psi}_1)$, it is straightforward to show that $\rr_1$ only depends on these variables. 

One can proceed analgously to analyse $\rr_2$ and $\rr_3$.  It is enough to use the particular form and size of all the terms in the expansion of the Hamiltonian given in Proposition \ref{proposition_secularexpansion}.
\end{proof}

Denote by $F_{\mathrm{in}} = \left. F_{\mathrm{sec}} \right|_{\Lambda}$ the restriction of the secular Hamiltonian to the normally hyperbolic invariant manifold $\Lambda$. 
\begin{corollary}\label{lemma_hampullbackerrors}
The expansion of the inner Hamiltonian $F_{\mathrm{in}}$ in powers of $\epsilon = \frac{1}{L_2}$ and $\mu = \frac{L_2}{L_3}$ satisfies the following properties:
\begin{enumerate}
\item
The lowest order term containing $\tilde{\Gamma}_2$ is
\begin{equation}\label{eq_innerhamleadingterm}
\epsilon^6 \bar{H}_0^{12} = \epsilon^6 \tilde{\Gamma}_2^2,
\end{equation}
up to a nonzero multiplicative constant independent of $\epsilon$ and $\mu$, which is the leading term in the expansion of $F_{\mathrm{in}}$. 
\item
The lowest order terms containing the actions $\tilde{\Psi}_1$ and $\tilde{\Gamma}_3$ respectively, up to nonzero multiplicative constants independent of $\epsilon$ and $\mu$, are $\epsilon^7 \bar{H}_1^{12}$ and $\epsilon^3 \mu^6 \tilde{H}_3$, where
\begin{equation} \label{eq_innerh012h112}
\bar{H}_1^{12} =  \left(3 \, \frac{\tilde{\Gamma}_2^2}{L_1^2} - 1 \right) \tilde{\Psi}_1 + H_1' \left( \tilde{\Gamma}_2 \right),
\end{equation}
for some function $H_1'$, and where $\tilde{H}_3$ is defined by \eqref{eq_lowestorderGamma3}. 
\item
The lowest order term containing the angle $\tilde{\gamma}_2$ is at least of order $\epsilon^8$. 
\item
The lowest order terms containing the angles $\tilde{\psi}_1$ and $\tilde{\gamma}_3$ respectively, up to nonzero multiplicative constants, are $\epsilon^2 \mu^6 H_0^{23}$ and $\epsilon^2 \mu^8 H_3^{23}$, where $H_0^{23}$ is defined by \eqref{eq_H023H123def} and $H_3^{23}$ is defined by \eqref{eq_H323def}. 
\end{enumerate}
\end{corollary}

\begin{proof}
The Poincar\'e variables $(\xi, \eta)$ first appear in $H_0^{12}$, where they always appear at least quadratically (see \eqref{eq_H012inpoincarevariables}). Squaring the formula \eqref{eq_fenichelgraphexp} for the function $\rho$, using \eqref{eq_fenichelgraphdependence}, and recalling that the order of $H_0^{12}$ is $\epsilon^6 = L_2^{-6}$, we see that the function $\rho$ satisfies the following properties:
\begin{itemize}
\item
The first-order correction from $\rho$ to the inner Hamiltonian that contains $\tilde{\gamma}_2, \tilde{\Gamma}_2$ is of order $O\left( \epsilon^8 \right)$;
\item
The first-order correction from $\rho$ to the inner Hamiltonian that contains $\tilde{\Psi}_1$ is of order $O\left( \epsilon^9 \right)$;
\item
The first-order correction from $\rho$ to the inner Hamiltonian that contains $\tilde{\psi}_1$ is of order $O\left( \epsilon^3\mu^6\right)$;
\item
The first-order correction from $\rho$ to the inner Hamiltonian that contains $\tilde{\gamma}_3, \tilde{\Gamma}_3$ is of order $O\left( \hat{\epsilon} \, \epsilon^7  \right)$ where $\hat{\epsilon} \lesssim \frac{L_2^9}{L_3^6}$. 
\end{itemize}
Comparing these orders with the first order of appearance of each variable $\tilde{\gamma}_2, \tilde{\Gamma}_2, \tilde{\psi}_1, \tilde{\Psi}_1, \tilde{\gamma}_3, \tilde{\Gamma}_3$ in the expansion of the Hamiltonian (see Proposition \ref{proposition_secularexpansion}), we see that parts 1-4 of the corollary follow. The formulas \eqref{eq_innerhamleadingterm} for $\bar{H}_0^{12}$, and \eqref{eq_innerh012h112} for $\bar{H}_1^{12}$ come from restricting formula \eqref{eq_H012def} for $H_0^{12}$ and formula \eqref{eq_H112def} for $H_1^{12}$ to $\tilde{\Lambda}_0 = \{ \xi = \eta = 0 \}$. 
\end{proof}

Due to Corollary \ref{lemma_hampullbackerrors} and Proposition \ref{proposition_secularexpansion}, we have the expansion
\[F_{\mathrm{in}} = F_{\mathrm{in}}^{12} + F_{\mathrm{in}}^{23} \]
with
\begin{equation} \label{eq_innersecularhamiltonians}
\begin{dcases}
F_{\mathrm{in}}^{12}= c_0^{12} + \frac{1}{L_2^6} \, \alpha_0^{12} \, \bar{H}_0^{12} + \frac{1}{L_2^7} \, \alpha_1^{12} \, \bar{H}_1^{12} + \cdots \\
F_{\mathrm{in}}^{23}= c_0^{23} + \frac{L_2^4}{L_3^6} \, \alpha_0^{23} \, \bar{H}_0^{23} + \frac{L_2^5}{L_3^7} \, \alpha_1^{23} \, \bar{H}_1^{23} + \cdots
\end{dcases}
\end{equation}
for some nontrivial constants $\alpha_n^{j,j+1}$, where $\bar{H}_0^{12}$ is defined by \eqref{eq_innerhamleadingterm}, and $\bar{H}_1^{12}$ is defined by \eqref{eq_innerh012h112}. 
The term $F_{\mathrm{in}}^{12}$ contains the $L_3$-independent terms and $F_{\mathrm{in}}^{23}$ contains the rest.

The following lemma allows us to find a coordinate transformation that straightens the symplectic form.  

\begin{lemma} \label{lemma_straightsympform}
There is a change of coordinates
\begin{equation}
\left(\tilde{\gamma}_2, \tilde{\Gamma}_2, \tilde{\psi}_1, \tilde{\Psi}_1, \tilde{\gamma}_3, \tilde{\Gamma}_3 \right) \mapsto \left(\gamma_2', \Gamma_2', \psi_1', \Psi_1', \gamma_3', \Gamma_3' \right)
\end{equation}
on $\Lambda$ that is $O \left( L_2^{-2} \right)$ close to the identity satisfying
\begin{equation}
 \begin{dcases}
  \Gamma_2' &=\tilde{\Gamma}_2+ 
%   \frac{1}{L_2^2} Q'_2 \left( \tilde{\gamma}_2, \tilde{\Gamma}_2 \right)
%   + 
  \frac{1}{L_2^3} P_2' \left( \tilde{\gamma}_2, \tilde{\Gamma}_2, \tilde{\psi}_1, \tilde{\Psi}_1  \right) + \frac{\hat{\epsilon}}{L_2} N_2'\left( \tilde{\gamma}_2, \tilde{\Gamma}_2, \tilde{\psi}_1, \tilde{\Psi}_1,  \tilde{\gamma}_3, \tilde{\Gamma}_3 \right),\\
  \Psi_1'&=\tilde{\Psi}_1+ \frac{1}{L_2^3} P'_1 \left( \tilde{\gamma}_2, \tilde{\Gamma}_2, \tilde{\psi}_1, \tilde{\Psi}_1 \right) + \frac{\hat{\epsilon}}{L_2} N_1' \left( \tilde{\gamma}_2, \tilde{\Gamma}_2, \tilde{\psi}_1, \tilde{\Psi}_1,  \tilde{\gamma}_3, \tilde{\Gamma}_3 \right),\\
  \Gamma_3' &=\tilde{\Gamma}_3 + \frac{\hat{\epsilon}}{L_2} N_3' \left( \tilde{\gamma}_2, \tilde{\Gamma}_2, \tilde{\psi}_1, \tilde{\Psi}_1,  \tilde{\gamma}_3, \tilde{\Gamma}_3 \right),
 \end{dcases}
\end{equation}
where $\hat{\epsilon} \lesssim \frac{L_2^9}{L_3^6}$ comes from Lemma \ref{lemma_fenichelgraphexp}, such that
\begin{equation} \label{eq_straightsympform}
\left. \Omega \right|_{\Lambda} = d \Gamma_2' \wedge d \gamma_2' + d \Psi_1' \wedge d \psi_1' + d \Gamma_3' \wedge d \gamma_3'.
\end{equation}
\end{lemma}

\begin{proof}
Recall in Lemma \ref{lemma_fenichelgraphexp} we obtained the expansion \eqref{eq_fenichelgraphexp} for the function $\rho$. Denote by $U$ the neighbourhood of the origin in $\mathbb{R}^2$ in which the Poincar\'e variables $(\xi, \eta)$ are valid, and by $V$ the subset of $\mathbb{R}^3$ to which the actions $(\tilde{\Gamma}_2, \tilde{\Psi}_1, \tilde{\Gamma}_3)$ belong. Define the inclusion
\begin{equation}
P :  \mathbb{T}^3 \times V  \longrightarrow U \times \left( \mathbb{T}^3 \times V \right) 
\end{equation}
by $(\xi, \eta, x) =P(x)= (\rho (x), x)$. Then
$\Omega_1 = P^* \Omega,
$ is the pullback of the canonical form in the tilde coordinates

We claim that $\Omega_1$ is exact. Indeed, we have $\Omega = d \lambda$ where
\begin{equation}
\lambda = \xi \, d \eta +  \tilde{\Gamma}_2 \, d \tilde{\gamma}_2 + \tilde{\Psi}_1 \, d \tilde{\psi}_1 + \tilde{\Gamma}_3 \, d \tilde{\gamma}_3
\end{equation}
is the Liouville 1-form. Let $\lambda_1 = P^* \lambda$. Then $\Omega_1 = d \lambda_1$ since the pullback commutes with the differential, and so $\Omega_1$ is exact. 

Now, denote by $\rho_{\xi}, \rho_{\eta}$ the $\xi, \eta$ components of $\rho$, and by $\rho_{j,\xi}, \rho_{j,\eta}$ the $\xi, \eta$ components of $\rho_j$ for each $j=0,1,2,3$. Recall $\Omega_0$ is defined by \eqref{eq_omega0}. We have
\begin{equation}
\Omega_1 = P^* \Omega = d \rho_{\xi} \wedge d \rho_{\eta} +\Omega_0 = \Omega_0 + R_1 + R_2
\end{equation}
where, since $\rr_0$ only depends on  $\tilde{\Gamma}_2$,
% \begin{equation}
% \Omega_* = \Omega_0 + \frac{1}{L_2^2} d \rho_{0,\xi} \wedge d \rho_{0,\eta}
% \end{equation}
% depends on $\tilde{\psi}_1, \tilde{\Psi}_1, \tilde{\gamma}_3, \tilde{\Gamma}_3$ only in $\Omega_0$, and where
% where
\begin{align}
R_1 ={}& \frac{1}{L_2^3} \left[ d \rho_{0,\xi} \wedge \left(d \rho_{1,\eta}+\frac{L_2^8}{L_3^6}d \rho_{2,\eta}\right) +\left( d \rho_{1,\xi} +\frac{L_2^8}{L_3^6}d \rho_{2,\xi}\right)\wedge d \rho_{0,\eta} \right]\\
{}&+ \frac{1}{L_2^4}\left( d \rho_{1,\xi} +\frac{L_2^8}{L_3^6}d \rho_{2,\xi}\right) \wedge\left(\frac{L_2^8}{L_3^6}d \rho_{2,\eta}\right) \\
R_2 ={}& d \rho_{\xi} \wedge d \rho_{\eta} - R_1
\end{align}
Then $R_1$ is of order $\frac{1}{L_2^3}$ and does not depend on $\tilde{\gamma}_3, \tilde{\Gamma}_3$, and $R_2$ is of order $\frac{\hat{\epsilon}}{L_2}$. We first eliminate $R_1$ by a change of coordinates which does not depend on $\tilde{\gamma}_3, \tilde{\Gamma}_3$ and then we eliminate $R_2$.
% the higher-order terms from $\Omega_*$ by constructing a near-identity coordinate transformation that does not alter $\tilde{\psi}_1, \tilde{\Psi}_1, \tilde{\gamma}_3, \tilde{\Gamma}_3$. 

Suppose there exists a coordinate transformation
\begin{equation}
h_0 : \left( \tilde{\gamma}_2, \tilde{\Gamma}_2, \tilde{\psi}_1, \tilde{\Psi}_1, \tilde{\gamma}_3, \tilde{\Gamma}_3 \right) \mapsto \left( \gamma_2^*, \Gamma_2^*, \psi_1^*, \Psi_1^*, \tilde{\gamma}_3, \tilde{\Gamma}_3 \right)
\end{equation}
so that
\begin{equation} \label{eq_straighttransf}
h_0^* \Omega' = \Omega_0\qquad \text{where}\qquad \Omega'=\Omega_0+R_1
\end{equation}
Suppose moreover there is a (nonautonomous) vector field $X_t$ such that $h_0=\phi_{\epsilon}$ where $\phi_t$ is the time-$t$ map of $X_t$, and where $\epsilon = L_2^{-3}$. Differentiating \eqref{eq_straighttransf} with respect to $\epsilon$ and using Cartan's magic formula, we get
% \begin{align}
\[0 ={} \frac{d}{d \epsilon} \left[ \phi_{\epsilon}^* \Omega' \right] = \phi_{\epsilon}^* \left[ \frac{d}{d \epsilon} \Omega' + \mathcal{L}_{X_t} \Omega' \right] \\
={} \phi_{\epsilon}^* \left[ \frac{d}{d \epsilon} \Omega' + i_{X_t} d \Omega' + d i_{X_t} \Omega' \right]\]
% \end{align}
where $\mathcal{L}_{X_t}$ is the Lie derivative with respect to $X_t$ and $i_{X_t}$ is the contraction operator of $X_t$. Now, $\Omega'=\Omega_0+R_1$ is also exact, with $\Omega' = d \lambda'$ for some 1-form $\lambda'$, by an argument analogous to the above proof of exactness of $\Omega_1$. Therefore
\begin{equation}
di_{X_t} \Omega' =- \frac{d}{d \epsilon} \Omega'= - \frac{d}{d \epsilon} d \lambda'= - d \left( \frac{d}{d \epsilon} \lambda'\right).
\end{equation}
Then, to obtain a vector field $X_t$ with this property, it is enough to look for $X_t$ satisfying
\begin{equation}
i_{X_t} \Omega'= - \frac{d}{d \epsilon} \lambda'.
\end{equation}
It is easy to check that a solution $X_t$ to this equation exists; its flow exists at least for time $\epsilon$; and its time-$\epsilon$ map, by construction, is the required coordinate transformation $h_0$. 

In the new coordinates, the symplectic  $\Omega_1$ becomes of the form $\hat \Omega_1 = h_0^*\Omega_1= \Omega_0 + \hat{R}_2$, where $\hat{R}_2 = h_0^* R_2$. 
% Since $\tilde{R}_1$ is of order $\frac{1}{L_2^3}$ and does not depend on $(\tilde{\gamma}_3, \tilde{\Gamma}_3)$, we can repeat the above procedure with $\epsilon =  \frac{1}{L_2^3}$ to obtain a coordinate transformation 
% \begin{equation}
% h_1 : \left( \gamma_2^*, \Gamma_2^*, \tilde{\psi}_1, \tilde{\Psi}_1, \tilde{\gamma}_3, \tilde{\Gamma}_3 \right) \mapsto \left( \check{\gamma}_2, \check{\Gamma}_2, \check{\psi}_1, \check{\Psi}_1, \tilde{\gamma}_3, \tilde{\Gamma}_3 \right) 
% \end{equation}
% that is $O \left( \frac{1}{L_2^3} \right)$ close to the identity and eliminates $\tilde{R}_1$. Now the symplectic form becomes $\Omega_2 = \Omega_0 + \hat{R}_2$ where $\hat{R}_2 = h_1^* \tilde{R}_2$. 
Then, since $\hat{R}_2$ is of order $\frac{\hat{\epsilon}}{L_2}$, we may repeat the procedure again to complete the proof. 
\end{proof}

Since the coordinate transformation provided by Lemma \ref{lemma_straightsympform} is $O\left( L_2^{-2} \right)$ close to the identity, the expansions \eqref{eq_innersecularhamiltonians} are unchanged (at least) up to the terms $\bar{H}_2^{j,j+1}$. The restriction of the secular Hamiltonian to $\Lambda$ can be written as
\begin{equation} \label{eq_finexpansion}
F_{\mathrm{in}} = F_{\mathrm{in}}^{12} + F_{\mathrm{in}}^{23} = c_0 + \sum_{n=0}^{\infty} \epsilon_n \alpha_n h_n
\end{equation}
where
\begin{equation}
\epsilon_0 = \frac{1}{L_2^6} \gg \epsilon_1 = \frac{1}{L_2^7} \gg \epsilon_2 = \frac{L_2^4}{L_3^6} \gg \cdots
\end{equation}
and
\begin{equation} \label{eq_innersecexpansion}
h_0 = \bar{H}_0^{12}, \quad h_1 = \bar{H}_1^{12}, \quad h_2 = \bar{H}_0^{23}, \ldots \quad \alpha_0 = \alpha_0^{12}, \quad \alpha_1 = \alpha_1^{12}, \quad \alpha_2 = \alpha_0^{23}, \ldots
\end{equation}
(see \eqref{eq_ordersofmagnitude}). In the next lemma we perform an arbitrary number $k$ of steps of symplectic averaging. 
\begin{lemma} \label{lemma_averaging}
Let $(k_1,k_2) \in \mathbb{N}^2$. There is a symplectic coordinate transformation 
\begin{equation}
\Phi : \left(\gamma_2', \Gamma_2', \psi_1', \Psi_1', \gamma_3', \Gamma_3' \right) \mapsto
\left(\hat{\gamma}_2, \hat{\Gamma}_2, \hat{\psi}_1, \hat{\Psi}_1, \hat{\gamma}_3, \hat{\Gamma}_3 \right)
\end{equation}
on $\Lambda$ that is $O \left(\frac{L_2^{11}}{L_3^6} \right)$ close to the identity such that the restriction $\hat{F}$ of the secular Hamiltonian to $\Lambda$, defined in \eqref{eq_finexpansion}, becomes
\begin{equation}
\hat{F} = \hat{F}_0 \left(\hat{\Gamma}_2, \hat{\Psi}_1, \hat{\Gamma}_3; \epsilon, \mu \right) + \epsilon^{k_1} \mu^{k_2} \hat{F}_1 \left( \hat{\gamma}_2, \hat{\Gamma}_2, \hat{\psi}_1, \hat{\Psi}_1, \hat{\gamma}_3, \hat{\Gamma}_3; \epsilon, \mu \right)
\end{equation}
where $\hat{F}_0 = \epsilon^6  c_0 \hat{\Gamma}_2^2 + \epsilon^7 \hat{h}_0 \left(\hat{\Gamma}_2, \hat{\Psi}_1, \hat{\Gamma}_3; \epsilon, \mu \right)$, where $\epsilon = \frac{1}{L_2}$, $\mu = \frac{L_2}{L_3}$, and where $\hat{F}_j$ is uniformly bounded in $\epsilon, \mu$ for $j=0,1$. 
Moreover, the transformations of the actions satisfy
\begin{align}
\hat{\Gamma}_2& = \Gamma_2' + O \left( \frac{L_2^9}{L_3^6} \right),  \label{eq_gamma2gamma3symptransf}\\
% \end{equation}
% \begin{equation}
\hat{\Psi}_1 &= \Psi_1' - \frac{L_2^{11}}{L_3^6} \, \frac{\alpha_2}{\alpha_1} \, \frac{ \cos \left( 2 \psi_1' \right)}{ 2 \left(3 \frac{\left(\Gamma_2' \right)^2}{L_1^2} - 1 \right) } + \cdots \label{eq_psi1symptransf}\\
% \end{align}
% and
% \begin{equation}
% \begin{split}\label{eq_gamma3symptransf}
\hat{\Gamma}_3 &= \Gamma_3' + \frac{L_2^{13}}{L_3^8} \, \tilde{\alpha} \, \left( 3 \frac{\left(\Gamma_2' \right)^2}{L_1^2} - 1 \right)^{-1} \bigg[ \frac{1}{3} \, \nu_0 \, \cos \left( \gamma_3' + 3 \psi_1' \right) + \nu_1 \, \cos \left( \gamma_3' + \psi_1' \right) - \nu_2 \, \cos \left( \gamma_3' - \psi_1' \right) \\
& \qquad \qquad \qquad \qquad \qquad \qquad \qquad \qquad \qquad \qquad \qquad \qquad   - \frac{1}{3} \, \nu_3 \, \cos \left( \gamma_3' - 3 \psi_1' \right) \bigg] + \cdots\label{eq_gamma3symptransf}
% \end{split}
\end{align}
where $\tilde{\alpha}$ is a nontrivial constant, and the constants $\nu_j$ are given by \eqref{eq_constantsmuline1}, \eqref{eq_constantsmuline2}. 
\end{lemma}

\begin{proof}
The fact that such a coordinate transformation exists is a standard result from averaging theory and the fact that the frequencies $\left(\tilde{\gamma}_2, \tilde{\psi}_1,  \tilde{\gamma}_3\right)$ are at different time scales. We compute the first order. There is a Hamiltonian function $K$ such that the corresponding flow $\phi^t_K$ with respect to the standard symplectic form $\Omega_1$ (given by \eqref{eq_straightsympform}) satisfies $\phi^1_K = \Phi$. We can write
\begin{equation}
K = \sum_{n=2}^k \hat{\epsilon}_n K_n
\end{equation}
where the sum starts at $n=2$ because $h_0, h_1$ do not depend on the angles $\gamma_2', \psi_1', \gamma_3'$. In order to obtain a first order approximation of $\Phi$ we determine $\hat{\epsilon}_2$, $K_2$. The purpose of $K_2$ is to eliminate the angle $\psi_1'$ from $h_2$ by adjusting the action $\Psi_1'$, which first appears in $h_1$. Set
\begin{equation}
\hat{\epsilon}_2 = \frac{\epsilon_2}{\epsilon_1} = \frac{L_2^{11}}{L_3^6}.
\end{equation}
Suppose $K_2$ does not depend on $\gamma_2'$. Then, from \eqref{eq_innerh012h112} and \eqref{eq_innersecexpansion},
\begin{align}
F_{\mathrm{in}} \circ \phi^{- \hat{\epsilon}_2}_{K_2} ={}& c_0 + \epsilon_0 \alpha_0 h_0 \circ \phi^{- \hat{\epsilon}_2}_{K_2} + \epsilon_1 \alpha_1 h_1 \circ \phi^{- \hat{\epsilon}_2}_{K_2} + \epsilon_2 \alpha_2 h_2 \circ \phi^{- \hat{\epsilon}_2}_{K_2} + \cdots \\
={}& c_0 + \epsilon_0 \alpha_0 \hat{\Gamma}_2^2 + \epsilon_1 \alpha_1 h_1 + \epsilon_2 \left( \alpha_2 h_2 - \alpha_1 \left\{ h_1, K_2 \right\} \right) + \cdots \label{eq_rescaledfinexp}
\end{align}
where $\{ \cdot, \cdot \}$ is the Poisson bracket. We therefore require that $K_2$ solves the equation
\begin{equation} \label{eq_averagingrequirement}
\{ h_1, K_2 \} = \frac{\alpha_2}{\alpha_1} \left( h_2 - \left\langle h_2 \right\rangle \right) = \frac{\alpha_2}{\alpha_1} \left( \cos^2 \psi_1' - \frac{1}{2} \right).
\end{equation}
Let
\begin{equation}
K_2 = \frac{\alpha_2}{\alpha_1} \frac{\sin 2 \psi_1'}{4 \left( 3 \frac{\left( \Gamma_2' \right)^2}{L_1^2} - 1 \right)}.
\end{equation}
Since $K_2$ does not depend on $\gamma_2'$, equation \eqref{eq_rescaledfinexp} holds; moreover $K_2$ solves \eqref{eq_averagingrequirement}. From Hamilton's equations of motion we have
\begin{equation}
\hat{\Psi}_1 = \Psi_1' - \hat{\epsilon}_2 \, \frac{\partial K_2}{\partial \psi_1'} + \cdots
\end{equation}
which implies \eqref{eq_psi1symptransf}. 

Since $F_{\mathrm{oct}}^{12}$ has a factor of $e_1$ (see \eqref{eq_oct12unscaled}), it is zero on $\tilde{\Lambda}_0$, and thus $O\left( L_2^{-1} \right)$ on $\Lambda$. It follows that the angle $\gamma_2'$ does not appear in the inner secular Hamiltonian until the term of order $O \left( L_2^3 L_3^{-6} \right)$. Since the action $\Gamma_2'$ appears in the term of order $O \left( L_2^{-6} \right)$, we get the approximation \eqref{eq_gamma2gamma3symptransf}. 

Let us now prove formula \eqref{eq_gamma3symptransf}. By part \ref{item_propoct23} of Proposition \ref{proposition_secularexpansion}, the first-order term where we see $\tilde{\gamma}_3$ is $\frac{L_2^6}{L_3^8} \, \alpha_3^{23} \, H_3^{23}$ where $H_3^{23}$ is defined by \eqref{eq_H323def}. Notice that
\begin{equation}
\left\langle H_3^{23} \right\rangle_{\tilde{\psi}_1} = \frac{1}{2 \pi} \int_{\mathbb{T}} H_3^{23} \, d \tilde{\psi}_1 = 0.
\end{equation}
So this term is completely eliminated in the averaging process if one considers as integrable first order $L_2^{-6}\bar{H}_0^{12} +L_2^{-7}\bar{H}_1^{12}$ since then $\tilde\psi_1$ has  a non-trivial frequency (see \eqref{eq_innerh012h112}).
% .
%  by the application of the coordinate transformation $\Phi$ to the first-order term containing $\tilde{\Psi}_1$. Recall that the first-order term containing $\tilde{\Psi}_1$ is $\frac{1}{L_2^7} \alpha_1^{12} H_1^{12}$, where $H_1^{12}$ is defined by \eqref{eq_H112def}. 

By similar logic to that used in the definition of $\hat{\epsilon}_2$ and $K_2$, we define $\hat{\epsilon} = \frac{L_2^{13}}{L_3^8}$, and we search for a Hamiltonian function $\hat{K} = \tilde{\alpha} \, \tilde{K}$, with $\tilde{\alpha} = \frac{\alpha_3^{23}}{\alpha_1^{12}}$, such that $H_3^{23} - \left\{ H_1^{12}, \tilde{K} \right\} = 0$. Notice that this equation is satisfied by the Hamiltonian function
\begin{equation}
\tilde{K} = - \left( 3 \frac{\left(\Gamma_2' \right)^2}{L_1^2} - 1 \right)^{-1} \bigg[ \frac{1}{3} \, \nu_0 \, \sin \left( \gamma_3' + 3 \psi_1' \right) + \nu_1 \, \sin \left( \gamma_3' + \psi_1' \right) - \nu_2 \, \sin \left( \gamma_3' - \psi_1' \right) - \frac{1}{3} \, \nu_3 \, \sin \left( \gamma_3' - 3 \psi_1' \right) \bigg].
\end{equation}
Differentiating this expression with respect to $\gamma_3'$ and using the formula
\begin{equation}
\hat{\Gamma}_3 = \Gamma_3' - \hat{\epsilon} \, \frac{\partial \hat{K}}{\partial \gamma_3'} + \cdots = \Gamma_3' - \hat{\epsilon} \, \tilde{\alpha} \, \frac{\partial \bar{K}}{\partial \gamma_3'} + \cdots
\end{equation}
yields \eqref{eq_gamma3symptransf}.
\end{proof}

This completes the proof of Theorem \ref{theorem_innerdynamics}. Now it only remains to compute the second derivatives of the averaged Hamiltonian in the cylinder. This will be used to show that certain Poincar\'e map associated to this Hamiltonian has nonzero torsion (see Section \ref{sec:mapreduction}).

\subsection{The Hessian of the averaged inner Hamiltonian}\label{sec:Hessian}

Now, consider the restriction of the secular Hamiltonian to $\Lambda$ given by Lemma \ref{lemma_averaging}. We have denoted this Hamiltonian by  $\hat{F}$ and it is expressed as a function of the ``hat'' coordinates $\left(\hat{\gamma}_2, \hat{\Gamma}_2, \hat{\psi}_1, \hat{\Psi}_1, \hat{\gamma}_3, \hat{\Gamma}_3 \right)$. 

The following lemma gives first-order expressions for the partial derivatives of $\hat{F}$ of orders 1 and 2 with respect to the actions $\hat{\Gamma}_2, \hat{\Psi}_1, \hat{\Gamma}_3$. Note that the Hamiltonian  $\hat{F}$ given in Lemma \ref{lemma_averaging} is just given by $\hat{F}_0$ plus very small terms. Therefore, we only need to analyse the derivatives of $\hat{F}_0$.

The Hessian matrix of $\hat{F}$ will be useful to determine the torsion of a certain Poincar\'e map of the inner dynamics. Indeed, the highest order term of the torsion will depend only on those derivatives.
 
\begin{lemma}\label{lemma_innerhamderivatives}
The first and second-order partial derivatives of $\hat{F}$ with respect to the actions $\hat{\Gamma}_2, \hat{\Psi}_1, \hat{\Gamma}_3$ are 
\begin{align*}
\frac{\partial \hat{F}}{\partial \hat{\Gamma}_2} &= \epsilon^6 C_{12} 
{{6\,\hat{\Gamma}_{2}}\over{L_{1}^2\,\delta_{1}^3}} + \cdots,& \quad 
\frac{\partial \hat{F}}{\partial \hat{\Psi}_1} &= 3\epsilon^7 C_{12} 
{{L_{1}^2-3\, \hat{\Gamma}_{2}^2}\over{L_{1}^2\,\delta_{1}^4}} + \cdots, &
\quad \frac{\partial \hat{F}}{\partial \hat{\Gamma}_3}&=\epsilon^3 \mu^6 
C_{23} {{\left(20 - 
12\,\delta_{1}^2\right)\,\delta_{3}}\over{\delta_{1}^2\,\delta_{2}^3}} + 
\cdots,\\
\frac{\partial^2 \hat{F}}{\partial \hat{\Gamma}_2^2} &= \epsilon^6 C_{12} 
{{6}\over{L_{1}^2\,\delta_{1}^3}} + \cdots, &\quad \frac{\partial^2 
\hat{F}}{\partial \hat{\Psi}_1^2}&= 12\epsilon^8 C_{12} \left({{3\, 
\hat{\Gamma}_{2}^2 - L_{1}^2}\over{L_{1}^2\,\delta_{1}^5}} \right) + \cdots, 
& \quad \frac{\partial^2 \hat{F}}{\partial \hat{\Gamma}_3^2}&=\epsilon^4 \mu^6 
C_{23} {{\left(20 - 12\,\delta_{1}^2\right)}\over{\delta_{1}^2\,\delta_{2}^3}}  
+ \cdots,\\
\frac{\partial^2 \hat{F}}{\partial \hat{\Gamma}_2 \partial \hat{\Psi}_1} &= - 
\epsilon^7 C_{12} {{18\, \hat{\Gamma}_{2}}\over{L_{1}^2\,\delta_{1}^4}} + 
\cdots, & \frac{\partial^2 \hat{F}}{\partial \hat{\Gamma}_2 \partial 
\hat{\Gamma}_3} &= \epsilon^4 \mu^6 C_{23} 
{{24\,\delta_{3}}\over{\delta_{1}\,\delta_{2}^3}} + \cdots,& 
\frac{\partial^2 \hat{F}}{ \partial \hat{\Psi}_1 \partial \hat{\Gamma}_3} &= - 
\epsilon^4 \mu^6 C_{23} {{40\,\delta_{3}}\over{\delta_{1}^3\,\delta_{2}^3}} + 
\cdots
\end{align*}
where $\epsilon = \frac{1}{L_2}$ and $\mu = \frac{L_2}{L_3}$, and where $C_{12}, 
C_{23}$ are nonzero constants independent of $L_2$ and $L_3$ coming from 
$F_{\mathrm{quad}}^{12}$, $F_{\mathrm{quad}}^{23}$ respectively. 
\end{lemma}

To prove this lemma we proceed in several steps. 
First we analyse the averaged Hamiltonian 
\[
\langle F_{\mathrm{in}}\rangle_{(\tilde{\gamma}_2,\tilde{\psi}_1,\tilde{\gamma}_3)}
\]
where $F_{\mathrm{in}}$ is the Hamiltonian given by Corollary \ref{lemma_hampullbackerrors}. In particular, we analyse the   linear and quadratic terms  with respect to the actions of the averaged Hamiltonian.

Then, we prove that these terms remain unchanged (at first order) when one applies the change of coordinates given by Lemmas \ref{lemma_straightsympform} and \ref{lemma_averaging}. This implies that these terms give the first order of the first and second derivatives given in Lemma \ref{lemma_innerhamderivatives}.

%The proof requires the following lemma, in the old coordinates.
\begin{proof}[Proof of Lemma~\ref{lemma_innerhamderivatives}]
The formulas for the first derivatives are straightforward using the expansion of the Hamiltonian $F_{\mathrm{sec}}$ given by Proposition \ref{proposition_secularexpansion} and the expansion of the normally hyperbolic invariant cylinder  given by Lemma \ref{lemma_fenichelgraphexp}. Indeed, the expansion of  $F_{\mathrm{sec}}$ given in Proposition \ref{proposition_secularexpansion} gives explicitly the first order of the linear terms in the actions $\tilde{\Gamma}_2, \tilde{\Psi}_1, \tilde{\Gamma}_3$ of the Hamiltonian. One can easily compute the pullback of these terms in the cylinder and then show that the changes of coordinates  given by Lemmas \ref{lemma_straightsympform} and \ref{lemma_averaging} do not alter these first orders.

Now we analyse the second derivatives. Note that Proposition \ref{proposition_secularexpansion} does not give terms which are quadratic in the actions and therefore it does not provide first orders of the second derivatives. To this end we have to do a deeper analysis of the Hamiltonian $\langle F_{\mathrm{in}}\rangle_{(\tilde{\gamma}_2,\tilde{\psi}_1,\tilde{\gamma}_3)}
$.

To analyse  it, we first analyse the Hamiltonian $\langle F_{\mathrm{sec}}\rangle_{(\tilde{\gamma}_2,\tilde{\psi}_1,\tilde{\gamma}_3)}
$, where 
 $F_{\mathrm{sec}}$ is the Hamiltonian analysed in Proposition \ref{proposition_secularexpansion}, restricted to  the normally hyperbolic invariant cylinder  given by Lemma \ref{lemma_fenichelgraphexp}. We use the expansions of the Hamiltonian given by the Proposition \ref{proposition_secularexpansion} and the  expansion of the parameterization of the cylinder  given in Lemma \ref{lemma_fenichelgraphexp}. 
%  .. Then, we analyse the pull back of this Hamiltonian into the normally hyperbolic invariant cylinder  given by Lemma \ref{lemma_fenichelgraphexp}. We use the  expansion of its parameterization given in Lemma \ref{lemma_fenichelgraphexp} to show that the pullback does not alter the hierarchy in the expansion Hamiltonian $\langle F_{\mathrm{sec}}\rangle_{(\tilde{\gamma}_2,\tilde{\psi}_1,\tilde{\gamma}_3)}
% $.

The first orders in the actions in  the Hamiltonian $\langle F_{\mathrm{in}}\rangle_{(\tilde{\gamma}_2,\tilde{\psi}_1,\tilde{\gamma}_3)}
$ come from $F_{\mathrm{quad}}^{12}$ and $F_{\mathrm{quad}}^{23}$. In order to compute the derivatives, we go back to the original formulas \eqref{eq_quad12unscaled} for $F_{\mathrm{quad}}^{12}$ and \eqref{eq_quad23beforeaver} for $F_{\mathrm{quad}}^{23}$. We average out the angles $\tilde{\gamma}_2, \tilde{\psi}_1$ from these Hamiltonians (recall $F_{\mathrm{quad}}^{12}$ and $F_{\mathrm{quad}}^{23}$ do not depend on $\gamma_3$), and compute the derivatives with respect to the actions. We compute the expansions of these derivatives in powers of $\epsilon=\frac{1}{L_2}$ and $\mu=\frac{L_2}{L_3}$. 
% This gives the first order of the derivatives of the restriction of $F_{\mathrm{quad}}^{12}$ and $F_{\mathrm{quad}}^{23}$ to $\tilde{\Lambda}_0$; we then use Corollary \ref{lemma_hampullbackerrors} to infer that the same formulas hold for the first order of the derivatives on $\Lambda$. 

First we compute the derivatives of $F_{\mathrm{quad}}^{12}$ restricted to $\tilde{\Lambda}_0$. Recall that $\Gamma_1=L_1$ on $\tilde{\Lambda}_0$; therefore $\left. e_1 \right|_{\tilde{\Lambda}_0}=0$  and so
\begin{equation}
K_1 = \left. F_{\mathrm{quad}}^{12} \right|_{\tilde{\Lambda}_0} = \frac{a_1^2}{8 \, a_{2}^3 \, (1-e_{2}^2)^{\frac{3}{2}}} \left(  3 \cos^2 i_{12} - 1 \right)
\end{equation}
where $e_2$ is defined by \eqref{eq_eccentricity}, and $\cos i_{12}$ is defined by \eqref{eq_cosi12def}. Therefore, using again the fact that $\left. \Gamma_1 \right|_{\tilde{\Lambda}_0} = L_1$, we see that the first and second partial derivatives of $K_1$ are
\begin{align}\label{eq_K1diffsoriginalactions}
\begin{dcases}
\frac{\partial^2 K_1}{\partial \Gamma_2^2} &=
\frac{a_1^2}{8 \, a_{2}^3} \left(
{{45\,L_{2}^3\,\Psi_{1}^4}\over{2\,\Gamma_{2}^7\,L_{1}^2}}-{{18\,
 L_{2}^3\,\Psi_{1}^2}\over{\Gamma_{2}^5\,L_{1}^2}}-{{45\,L_{2}^3\,
 \Psi_{1}^2}\over{\Gamma_{2}^7}}+{{45\,L_{1}^2\,L_{2}^3}\over{2\,
 \Gamma_{2}^7}}+{{3\,L_{2}^3}\over{2\,\Gamma_{2}^3\,L_{1}^2}}+{{6\,
 L_{2}^3}\over{\Gamma_{2}^5}}
 \right) \\
 \frac{\partial^2 K_1}{\partial \Psi_1^2} &= \frac{a_1^2}{8 \, a_{2}^3} \left(
 {{9\,L_{2}^3\,\Psi_{1}^2}\over{\Gamma_{2}^5\,L_{1}^2}}-{{3\,L_{2}^3
 }\over{\Gamma_{2}^3\,L_{1}^2}}-{{3\,L_{2}^3}\over{\Gamma_{2}^5}}
 \right) \\
  \frac{\partial^2 K_1}{\partial \Gamma_2 \partial \Psi_1} &= \frac{a_1^2}{8 \, a_{2}^3} \left(
 -{{15\,L_{2}^3\,\Psi_{1}^3}\over{\Gamma_{2}^6\,L_{1}^2}}+{{9\,L_{2}
 ^3\,\Psi_{1}}\over{\Gamma_{2}^4\,L_{1}^2}}+{{15\,L_{2}^3\,\Psi_{1}
 }\over{\Gamma_{2}^6}}
  \right) \\
  \frac{\partial K_1}{\partial \Gamma_2} &= \frac{a_1^2}{8 \, a_{2}^3} \left(
 -{{15\,L_{2}^3\,\Psi_{1}^4}\over{4\,\Gamma_{2}^6\,L_{1}^2}}+{{9\,
 L_{2}^3\,\Psi_{1}^2}\over{2\,\Gamma_{2}^4\,L_{1}^2}}+{{15\,L_{2}^3\,
 \Psi_{1}^2}\over{2\,\Gamma_{2}^6}}-{{15\,L_{1}^2\,L_{2}^3}\over{4\,
 \Gamma_{2}^6}}-{{3\,L_{2}^3}\over{4\,\Gamma_{2}^2\,L_{1}^2}}-{{3\,
 L_{2}^3}\over{2\,\Gamma_{2}^4}}
  \right) \\
  \frac{\partial K_1}{\partial \Psi_1} &= \frac{a_1^2}{8 \, a_{2}^3} \left(
 {{3\,L_{2}^3\,\Psi_{1}^3}\over{\Gamma_{2}^5\,L_{1}^2}}-{{3\,L_{2}^3
 \,\Psi_{1}}\over{\Gamma_{2}^3\,L_{1}^2}}-{{3\,L_{2}^3\,\Psi_{1}
 }\over{\Gamma_{2}^5}}
  \right). 
  \end{dcases}
 \end{align}
Due to \eqref{eq_changeofcoordstilde} we have
\begin{equation}\label{eq_K1chainrule1}
\frac{\partial K_1}{\partial \tilde{\Gamma}_2} = - \frac{\partial K_1}{\partial \Gamma_2}, \quad \frac{\partial K_1}{\partial \tilde{\Psi}_1} = \frac{\partial K_1}{\partial \Psi_1} + \frac{\partial K_1}{\partial \Gamma_2}, \quad \frac{\partial^2 K_1}{\partial \tilde{\Gamma}_2^2} = \frac{\partial^2 K_1}{\partial \Gamma_2^2},
\end{equation}
and
\begin{equation}\label{eq_K1chainrule2}
\frac{\partial^2 K_1}{\partial \tilde{\Psi}_1^2} = \frac{\partial^2 K_1}{\partial \Psi_1^2} + 2 \frac{\partial^2 K_1}{\partial \Gamma_2 \partial \Psi_1} + \frac{\partial^2 K_1}{\partial \Gamma_2^2}, \quad \frac{\partial^2 K_1}{\partial \tilde{\Gamma}_2 \partial \tilde{\Psi}_1} = - \frac{\partial^2 K_1}{\partial \Gamma_2 \partial \Psi_1} - \frac{\partial^2 K_1}{\partial \Gamma_2^2}.
\end{equation}
Combining \eqref{eq_K1diffsoriginalactions}, \eqref{eq_K1chainrule1}, and \eqref{eq_K1chainrule2}, effecting the coordinate transformation \eqref{eq_changeofcoordstilde}, and expanding in powers of $\epsilon = \frac{1}{L_2}$, we see that
\begin{align}\label{eq_quad12avederivatives}
\begin{dcases}
\frac{\partial^2 K_1}{\partial \tilde{\Gamma}_2^2} &= \epsilon^6 C_{12} {{6}\over{L_{1}^2\,\delta_{1}^3}} + O \left( \epsilon^7 \right) \\
\frac{\partial^2 K_1}{\partial \tilde{\Psi}_1^2} &= 12 \epsilon^8 C_{12} 
\left({{3\, \tilde{\Gamma}_{2}^2 -L_{1}^2}\over{L_{1}^2\,\delta_{1}^5}} 
\right) +O \left( \epsilon^9 \right) \\
\frac{\partial^2 K_1}{\partial \tilde{\Gamma}_2 \partial \tilde{\Psi}_1} &= - \epsilon^7 C_{12} {{18\, \tilde{\Gamma}_{2}}\over{L_{1}^2\,\delta_{1}^4}} + O \left( \epsilon^8 \right) \\
\frac{\partial K_1}{\partial \tilde{\Gamma}_2} &= \epsilon^6 C_{12} {{6\,\tilde{\Gamma}_{2}}\over{L_{1}^2\,\delta_{1}^3}} + O \left( \epsilon^7 \right) \\
\frac{\partial K_1}{\partial \tilde{\Psi}_1} &= 3\epsilon^7 C_{12} 
{{L_{1}^2-3\, \tilde{\Gamma}_{2}^2}\over{L_{1}^2\,\delta_{1}^4}} + O \left( 
\epsilon^8 \right)
\end{dcases}
\end{align}
where the coefficient $C_{12}$ is a nonzero constant depending on $L_1$ and $\mu_j, M_j$ for $j=1,2$ (but independent of $L_2$ and $L_3$). 

Now, the restriction to $\tilde{\Lambda}_0$ of $F_{\mathrm{quad}}^{23}$ depends on $\gamma_2, \psi_1$. Since we have shown in Lemma \ref{lemma_averaging} that a $C^r$ near-to-the-identity coordinate transformation transforms the Hamiltonian to its average over $\tilde{\gamma}_2, \tilde{\psi}_2$, we instead consider
\begin{equation}
K_2 = \left. \left[ \int_{\mathbb{T}^4} P_2 \left( \cos \zeta_{2} \right) \frac{\| q_2 \|^2}{\| q_3 \|^3} \, d \ell_2 \, d \ell_{3} \, d \tilde{\gamma}_2 \, d \tilde{\psi}_1 \right] \right|_{\Gamma_1=L_1} = \left. \left[ \int_{\mathbb{T}^2} \left\langle P_2 \left( \cos \zeta_{2} \right) \right\rangle_{\tilde{\gamma}_2, \tilde{\psi}_1} \frac{\| q_2 \|^2}{\| q_3 \|^3} \, d \ell_2 \, d \ell_{3} \right] \right|_{\Gamma_1 = L_1}
\end{equation}
where we have taken formula \eqref{eq_quad23beforeaver} for $F_{\mathrm{quad}}^{23}$, restricted attention to $\tilde{\Lambda}_0$ by setting $\Gamma_1=L_1$ (and noticing that $F_{\mathrm{quad}}^{23}$ does not depend on $\gamma_1$), averaged over the variables $\tilde{\gamma}_2, \tilde{\psi}_1$, and used the fact that $\| q_2 \|, \| q_3\|$ do not depend on $\tilde{\gamma}_2, \tilde{\psi}_1$. Recall from the proof of Lemma \ref{lemma_quad23exp} and the coordinate transformation \eqref{eq_changeofcoordstilde} that
\begin{equation}
\cos \zeta_2  = \left( \R_1 (\tilde{i}_2) \, \R_3 (\tilde{\psi}_1 + \tilde{\gamma}_2) \, I_3 \, \R_1 (i_2) \, \bar{Q}_2 \right) \cdot  \left( I_3 \, \R_1 (i_3) \, \bar{Q}_3 \right)
\end{equation}
where $\R_1, \R_3$ is the rotation around the $x,z$-axis respectively, $I_3 = \R_3 ( \pi)$, the angles $\tilde{i}_2, i_2, i_3$ are defined by \eqref{eq_inclinationsdef1} and \eqref{eq_inclinationsdef2}, and $\bar{Q}_j= \left(\cos \left( v_j - \tilde{\gamma}_j \right), \sin \left( v_j - \tilde{\gamma}_j \right),0 \right)$ with $v_j$ denoting the true anomaly of the $j^{\mathrm{th}}$ fictitious body. We can thus compute $\left\langle P_2 \left( \cos \zeta_{2} \right) \right\rangle_{\tilde{\gamma}_2, \tilde{\psi}_1} = \int_{\mathbb{T}_2} P_2 \left( \cos \zeta_{2} \right) \, d \tilde{\gamma}_2 \, d \tilde{\psi}_1$, and use the resulting expression to obtain
\begin{align}
K_2 ={}& \frac{-4 \, \pi^2 \, a_2^2}{32 \, a_3^3 \, (1 - e_3^2)^{\frac{3}{2}}} \left(3\,e_{2}^2+2\right)\,\left(3\,\cos ^2i_{2}-1\right)\, \Big[ 6
 \,\cos \tilde{i}_2 \,\sqrt{1-\cos ^2 \tilde{i}_2}\,\cos i_{3}\,\sqrt{1-\cos ^2
 i_{3}}\\
& \quad -6\,\cos ^2\tilde{i}_2 \,\cos ^2i_{3}+3\,\cos ^2i_{3}+3\,\cos ^2 \tilde{i}_2
 -2 \Big]. 
\end{align}
To compute this integral we have used again the technique described in Appendix C of \cite{fejoz2002quasiperiodic} (used above in Lemmas \ref{lemma_quadoct12comp}, \ref{lemma_quad23exp}, and \ref{lemma_oct23exp}). We now effect the coordinate transformation \eqref{eq_changeofcoordstilde}, and expand $K_2$ in powers of $\mu = \frac{L_2}{L_3}$ and $\epsilon=\frac{1}{L_2}$ to obtain
\begin{align}\label{eq_fquadaveragedexp}
K_2 ={}& \epsilon^2 \mu^6  C_{23} \Big[ \left( \bar{K}_{0,0} + \epsilon \bar{K}_{0,1} + \epsilon^2 \bar{K}_{0,2} + \cdots \right) + \mu \left( \bar{K}_{1,0} + \epsilon \bar{K}_{1,1} + \epsilon^2 \bar{K}_{1,2} + \cdots \right) \\
\quad &  + \mu^2 \left( \bar{K}_{2,0} + \epsilon \bar{K}_{2,1} + \epsilon^2 \bar{K}_{2,2} + \cdots \right) \Big]
\end{align}
where the formulas for $\bar{K}_{i,j}$ can be found in Appendix \ref{sec_appendixexpformulas}. Therefore we can compute
\begin{align}\label{eq_quad23avederivatives}
\begin{dcases}
\frac{\partial^2 K_2}{\partial \tilde{\Gamma}_3^2} &=\epsilon^4 \mu^6 C_{23} {{\left(20 - 12\,\delta_{1}^2\right)}\over{\delta_{1}^2\,\delta_{2}^3}}  + O \left( \epsilon^5 \mu^6, \epsilon^4 \mu^7 \right) \\
\frac{\partial^2 K_2}{\partial \tilde{\Gamma}_3 \partial \tilde{\Psi}_1} &= - \epsilon^4 \mu^6 C_{23} {{40\,\delta_{3}}\over{\delta_{1}^3\,\delta_{2}^3}} + O \left( \epsilon^5 \mu^6, \epsilon^4 \mu^7 \right) \\
\frac{\partial^2 K_2}{\partial \tilde{\Gamma}_3 \partial \tilde{\Gamma}_2} &= \epsilon^4 \mu^6 C_{23} {{24\,\delta_{3}}\over{\delta_{1}\,\delta_{2}^3}} + O \left( \epsilon^5 \mu^6, \epsilon^4 \mu^7 \right) \\
\frac{\partial K_2}{\partial \tilde{\Gamma}_3} &= \epsilon^3 \mu^6 C_{23} {{\left(20 - 12\,\delta_{1}^2\right)\,\delta_{3}}\over{\delta_{1}^2\,\delta_{2}^3}} + O \left( \epsilon^4 \mu^6, \epsilon^3 \mu^7 \right). 
\end{dcases}
\end{align}

%  , we first show that  the  lowest order terms containing products of the form $\tilde{\Gamma}_2 \tilde{\Psi}_1$, $\tilde{\Gamma}_2 \tilde{\Gamma}_3$, $\tilde{\Psi}_1 \tilde{\Gamma}_3$, $\tilde{\Gamma}_2^2$, $\tilde{\Psi}_1^2$, $\tilde{\Gamma}_3^2$ are equal to the lowest order terms containing such products in the expansion of  the averaged secular Hamiltonians  $\langle F_{\mathrm{sec}}\rangle_{(\tilde{\gamma}_2,\tilde{\psi}_1,\tilde{\gamma}_3)}$ in Proposition \ref{proposition_secularexpansion}

% 
%   We must check the order at which products of the form $\tilde{\Gamma}_2 \tilde{\Psi}_1$, $\tilde{\Gamma}_2 \tilde{\Gamma}_3$, $\tilde{\Psi}_1 \tilde{\Gamma}_3$, $\tilde{\Gamma}_2^2$, $\tilde{\Psi}_1^2$, $\tilde{\Gamma}_3^2$ first appear. This is done in the proof of Lemma \ref{lemma_innerhamderivatives} - see equations \eqref{eq_quad12avederivatives} and \eqref{eq_quad23avederivatives}. 
%   
Note that these second derivatives are  the derivatives of the pullback of the Hamiltonian to the unperturbed cylinder $\Lambda_0$. To ensure that the first order of the  derivatives remain unchanged if one considers the pullback to the perturbed normally hyperbolic invariant cylinder given by Lemma \ref{lemma_fenichelgraphexp} it is enough to use the expansion of its parameterization given by this lemma. 

The sizes of each of the terms in the Hamiltonian and the cylinder parameterization are gathered in Table \ref{table_orderappproducts}. In each case the dominant term is the term coming from the averaged secular Hamiltonian; this can be checked using the assumption \eqref{eq_assumption1} and the inequality $\hat{\epsilon} \ll \frac{L_2^9}{L_3^6}$.
% , and it completes the proof of the lemma. 
% \end{proof}

\begin{table}[!ht]
  \begin{center}
    \begin{tabular}{|c|c|c|c|c|c|c|c|c|c|}
      \hline 
      & \vphantom{$\Big[$} $\tilde{\Gamma}_2 \tilde{\Psi}_1$ & $\tilde{\Gamma}_2 \tilde{\Gamma}_3$ & $\tilde{\Psi}_1 \tilde{\Gamma}_3$ & $\tilde{\Gamma}_2^2$ & $\tilde{\Psi}_1^2$ & $\tilde{\Gamma}_3^2$ \\
      \hline
      From $\langle F_{\mathrm{sec}}\rangle_{(\tilde{\gamma}_2,\tilde{\psi}_1,\tilde{\gamma}_3)}
$ \vphantom{$\Big[$} & $\epsilon^7$ & $\epsilon^4 \mu^6$ & $\epsilon^4 \mu^6$ & $\epsilon^6$ & $\epsilon^8$ & $\epsilon^4 \mu^6$ \\
      From $\rho$ \vphantom{$\Big[$} & $ \epsilon^9$ & $\hat{\epsilon} \, \epsilon^7$ & $\hat{\epsilon} \, \epsilon^7$ & $\epsilon^8$ & $\epsilon^9$ & $\hat{\epsilon} \, \epsilon^7$  \\
      \hline 
    \end{tabular}
    \caption{\label{table_orderappproducts}The order at which the relevant products of the actions appear in the inner Hamiltonian, as a result of appearing explicitly in the averaged secular Hamiltonian $\langle F_{\mathrm{sec}}\rangle_{(\tilde{\gamma}_2,\tilde{\psi}_1,\tilde{\gamma}_3)}
$, and as a result of restricting the Hamiltonian $\langle F_{\mathrm{sec}}\rangle_{(\tilde{\gamma}_2,\tilde{\psi}_1,\tilde{\gamma}_3)}
$ to the manifold $\Lambda$ using the function $\rho$.
    }
  \end{center}
\end{table}

% \begin{proof}[Proof of lemma~\ref{lemma_innerhamderivatives}]
% We compute the first order of the first and second terivatives with respect to $\left(\tilde{\gamma}_2, \tilde{\Gamma}_2, \tilde{\psi}_1, \tilde{\Psi}_1, \tilde{\gamma}_3, \tilde{\Gamma}_3 \right)$ coordinates, and use the fact that the coordinate transformation \eqref{eq_tildehatcoordtransf} is close to the identity in $C^r$ to prove the lemma. Note that we need to know the first order terms in the expansion of $\hat{F}$ containing pairwise products of $\tilde{\Gamma}_2, \tilde{\Psi}_1, \tilde{\Gamma}_3$ to compute the Hessian;

To complete the proof of Lemma \ref{lemma_innerhamderivatives} it is enough to ensure that the change of coordinates given in  \eqref{eq_tildehatcoordtransf} does not alter the first order of the second derivatives given in  \eqref{eq_quad12avederivatives} and \eqref{eq_quad23avederivatives}. This change of coordinates is the composition of the changes given by Lemmas \ref{lemma_straightsympform} and \ref{lemma_averaging}. Then, it is enough to use the estimates given in these lemmas on how close to the identity are these changes.
\end{proof}

%%% Local Variables: 
%%% mode: latex
%%% TeX-master: "4bp_secular_diffusion_v6.tex"
%%% End: 

\section{The outer dynamics}
\label{sec:scattering}

Recall from Section \ref{sec:innerdyn} that Theorem \ref{theorem_innerdynamics} ensures the existence of the normally 
hyperbolic (weakly) invariant cylinder $\Lambda$ (see \eqref{eq_nhimdef} and \eqref{eq_nhimparametersdef}) with stable and unstable invariant 
manifolds.  The purpose of this section is to show that these invariant manifolds 
intersect along two homoclinic channels. Such channels allow us to  establish 
the existence of two scattering maps. This is done in Section \ref{sec:scatteringsub} were we provide asymptotic formulas for those maps. These formulas are given by an associated Poincar\'e-Melnikov potential. Its computation is contained in Section \ref{section_melnikovcomp}.

\subsection{The scattering map}\label{sec:scatteringsub}
Recall from Section \ref{sec:innerdyn} that Theorem \ref{theorem_innerdynamics} ensures the existence of the normally 
hyperbolic (weakly) invariant cylinder $\Lambda$ (see \eqref{eq_nhimdef} and \eqref{eq_nhimparametersdef}) with stable and unstable invariant 
manifolds.  The purpose of this section is to show that these invariant manifolds 
intersect along two homoclinic channels. Such channels allow us to  establish 
the existence of two scattering maps $S_{\pm} : \Lambda \to \Lambda'$, for some 
$\Lambda'$ such that $\Lambda \subset \Lambda' \subset \mathbb{T}^3 \times 
\mathbb{R}^3$ (see Appendix \ref{sec:heuristics} for the definition). Then, we derive a 
first-order approximation of the changes in the actions $\hat{\Psi}_1$, 
$\hat{\Gamma}_3$ in the `hat' coordinates defined by Theorem 
\ref{theorem_innerdynamics}. The following theorem is the main result of this 
section. 

\begin{theorem}\label{theorem_outerdynamics}
The stable and unstable invariant manifolds of the normally hyperbolic 
invariant cylinder $\Lambda$ introduced in Theorem \ref{theorem_innerdynamics} 
intersect transversally along (at least) two homoclinic channels. Associated to these channels, there exist two 
scattering maps $S_{\pm} : \Lambda \to \Lambda'$ such that
\begin{equation}
S_{\pm} : \left(\hat{\gamma}_2, \hat{\Gamma}_2, \hat{\psi}_1, \hat{\Psi}_1, \hat{\gamma}_3, \hat{\Gamma}_3 \right) \longmapsto \left(\hat{\gamma}_2^*, \hat{\Gamma}_2^*, \hat{\psi}_1^*, \hat{\Psi}_1^*, \hat{\gamma}_3^*, \hat{\Gamma}_3^* \right)
\end{equation}
with
\begin{equation}
\hat{\Psi}_1^* = \hat{\Psi}_1 + \frac{L_2^9}{L_3^6} \S_1^{\pm} \left( \hat{\psi}_1, \hat{\Gamma}_2 \right) + \cdots, \quad \hat{\Gamma}_3^* = \hat{\Gamma}_3 + \frac{L_2^{11}}{L_3^8} \S^{\pm}_3 \left( \hat{\psi}_1, \hat{\gamma}_3, \hat{\Gamma}_2 \right) + \cdots
\end{equation}
where
\begin{align}
\S_1^{\pm}\left( \hat{\psi}_1, \hat{\Gamma}_2 \right) ={}& \mp \alpha_2^{23} \, \kappa \left( \frac{\pi \, \hat{\Gamma}_2}{A_2 \, L_1^2} \right) \, c_2 \, \cos \hat{\psi}_1 + \frac{\alpha_0^{23}}{\alpha_1^{12}} \, \beta_2 \frac{L_1}{6} \, \sqrt{ \frac{3}{2} } \hat{\Gamma}_2 \sqrt{1 - \frac{5}{3} \frac{\hat{\Gamma}_2^2}{L_1^2}} \frac{ \sin 2 \hat{\psi}_1}{3 \frac{\hat{\Gamma}_2^2}{L_1^2} - 1}, \label{eq_scatteringpsi1} \\
\S_3^{\pm}  \left( \hat{\psi}_1, \hat{\gamma}_3, \hat{\Gamma}_2 \right) ={}& \mp \alpha_5^{23} \, \kappa \left( \frac{\pi \, \hat{\Gamma}_2}{A_2 \, L_1^2} \right) \, \bigg[ 50\,\delta_{1}\,\delta_{3}\,\sin \hat{\gamma}_{3}\,\cos \hat{\psi}_{1}\,\sin 
 \hat{\psi}_{1}- \frac{70\,\delta_{3}}{\delta_1}\,\sin \hat{\gamma}_{3}\,\cos \hat{\psi}_{1}\,
 \sin \hat{\psi}_{1} \\
 & +\frac{105\,\delta_{3}^2}{\delta_1^2}\,\cos \hat{\gamma}_{3}\,\cos ^2\hat{\psi}_{1} -75\,\delta_{3}^2\,\cos \hat{\gamma}_{3}\,\cos ^2\hat{\psi}_{1} +25\,\delta_{1}^2\,\cos \hat{\gamma}_{3}\,\cos^2\hat{\psi}_{1} -35\,\cos \hat{\gamma}_{3}\,\cos ^2\hat{\psi}_{1} \\
  & - \frac{105\, \delta_{3}^2}{\delta_1^2} \,\cos \hat{\gamma}_{3} +60\,\delta_{3}^2\, \cos \hat{\gamma}_{3}-17\,\delta_{1}^2\,\cos \hat{\gamma}_{3}+28\,\cos \hat{\gamma}_{3} \bigg] \label{eq_scatteringgamma3} \\
  & +  \frac{L_1}{6} \sqrt{ \frac{3}{2}} \, \beta_2 \, \hat{\Gamma}_2 \sqrt{1 - \frac{5}{3} \frac{\hat{\Gamma}_2^2}{L_1^2}} \, \tilde{\alpha} \, \left( 3 \frac{\hat{\Gamma}_2^2}{L_1^2} - 1 \right)^{-1} \bigg[  \, \nu_0 \, \cos \left( \hat{\gamma}_3 + 3 \hat{\psi}_1 \right) \\
  &+ \nu_1 \, \cos \left( \hat{\gamma}_3 + \hat{\psi}_1 \right) + \nu_2 \, \cos \left( \hat{\gamma}_3 - \hat{\psi}_1 \right)  + \nu_3 \, \cos \left( \hat{\gamma}_3 - 3 \hat{\psi}_1 \right) \bigg]. \\
\end{align}
where the function $\kappa$ is defined by 
\begin{equation}\label{eq_melnikovcoefficientfn}
\kappa \left( x \right) = \sqrt{\frac{2}{3}} \, \frac{L_1^2}{\chi} \left[1 - 
\frac{x}{\sinh x} \right],
\end{equation}
the constant $c_2$ is defined in \eqref{eq_c1c2def}, and $\chi, \, 
A_2$ are the constants defined in \eqref{eq_chia2def}, the constants $\alpha_{ij}^ k$ come from the expansion of the secular Hamiltonian, and are defined in Section \ref{section_secularexpansion}, the constant $\beta_2$ is defined in Appendix \ref{appendix_psi1phaseshift}, the nontrivial constant $\tilde \alpha$ comes from Lemma \ref{lemma_averaging}, and the constants $\nu_j$ are defined by \eqref{eq_constantsmuline1} and \eqref{eq_constantsmuline2}. 
\end{theorem}

The proof of  Theorem \ref{theorem_outerdynamics} is done in several steps. 
The first step is to prove that the invariant manifolds of the cylinder 
intersect. This analysis is performed in the 
`tilde' coordinates introduced in  \eqref{eq_changeofcoordstilde}. The analysis of the transversality of these invariant manifolds allows us to 
derive first-order expressions for the associated scattering maps in 
`tilde' coordinates. Then, the second step is to  express the scattering 
maps in the  `hat' variables provided by Theorem \ref{theorem_innerdynamics}.

The first step is achieved by means of a suitable Poincar\'e-Melnikov Theory. 
We follow the approach by Delshams, de la Llave and Seara in 
\cite{delshams2006biggaps} which deals with the transverse intersection of the 
invariant manifolds of normally hyperbolic invariant manifolds in nearly 
integrable regimes.

Recall that the Hamiltonian $F_{\mathrm{quad}}^{12}$ analysed in 
Lemma \ref{lemma_quad12expansion}  is integrable. We show that the higher order 
terms in the Hamiltonian make the invariant manifolds of the cylinder split. 

Recall that in Section 
\ref{section_analysisofh0} we established the existence of two hyperbolic 
periodic orbits $Z^0_{\min}, Z^0_{\max}$ for the Hamiltonian $H_0^{12}$; 
moreover we found that the stable and unstable manifolds of these saddles 
coincide along a heteroclinic trajectory $Z^0$. Furthermore, in Poincar\'e 
variables \eqref{eq_poincarevariables} the two saddles are reduced to a single 
hyperbolic periodic orbit, which we denote by $Z^0_*$, and this hyperbolic 
periodic orbit possesses a homoclinic connection, which for convenience we 
continue to denote by $Z^0$. Write $\tilde{F}_{\mathrm{quad}}^{12} = L_2^6 
F_{\mathrm{quad}}^{12}$. Then $\tilde{F}_{\mathrm{quad}}^{12}$ possesses a 
hyperbolic periodic orbit $Z^{\mathrm{quad}}_*$ that is $O(L_2^{-1})$ close to 
$Z^0_*$; moreover there is a homoclinic orbit $Z^{\mathrm{quad}}$ to 
$Z^{\mathrm{quad}}_*$ that is $O(L_2^{-1})$ close to $Z^{0}$. Since 
$\tilde{F}_{\mathrm{quad}}^{12}$ is integrable, the homoclinic trajectory 
$Z^{\mathrm{quad}}$ corresponds to a non-transverse homoclinic manifold 
intersection of the stable and unstable manifolds of $Z^{\mathrm{quad}}_*$. 
Note that, as happened in Section \ref{section_analysisofh0}, such objects are 
referred to the Hamiltonian $\tilde{F}_{\mathrm{quad}}^{12}$ seen as a 1 degree 
of freedom Hamiltonian in the $(\tilde \gamma_1,\tilde\Gamma_1)$ variables (or 
equivalently in $(\eta,\xi)$ variables). One can consider the same objects in 
the full phase space. Then, the hyperbolic periodic orbit becomes the normally 
hyperbolic invariant cylinder and the homoclinic becomes a 7 dimensional 
homoclinic manifold. Note that, by the particular form of 
$F_{\mathrm{quad}}^{12}$ provided by
Lemma \ref{lemma_quad12expansion}, the invariant cylinder of 
$F_{\mathrm{quad}}^{12}$ is just foliated by invariant 
quasiperiodic tori. Abusing notation, we use the same notation for the objects 
in the reduced and the full phase space.

Now, write $\bar{H} = L_2^6 F_{\mathrm{sec}} - \tilde{F}_{\mathrm{quad}}^{12}$, and define the Poincar\'e-Melnikov potential by
\begin{align}
\begin{split}\label{eq_melnikovdef}
\L \left( \tilde{\gamma}_2, \tilde{\psi}_1, \tilde{\gamma}_3, \tilde{\Gamma}_2, \tilde{\Psi}_1, \tilde{\Gamma}_3 \right) ={}& \int_{- \infty}^{\infty} \Big( \bar{H} \left( Z^{\mathrm{quad}} \left( t, \tilde{\gamma}_2, \tilde{\psi}_1, \tilde{\gamma}_3, \tilde{\Gamma}_2, \tilde{\Psi}_1, \tilde{\Gamma}_3 \right) \right)  \\
& \quad- \bar{H} \left( Z_*^{\mathrm{quad}} \left( t, \tilde{\gamma}_2, 
\tilde{\psi}_1, \tilde{\gamma}_3, \tilde{\Gamma}_2, \tilde{\Psi}_1, 
\tilde{\Gamma}_3 \right) \right) \Big) \, dt. 
\end{split}
\end{align}
As with $\bar{H}$ itself, the Poincar\'e-Melnikov potential $\L$ can be 
expanded in ratios of powers of $L_2$ and $L_3$. The following result gives an 
expression for the first-order term at which each angle $\tilde{\gamma}_2, 
\tilde{\psi}_1, \tilde{\gamma}_3$ appears in the expansion of $\L$. The 
proposition is proved in Section \ref{section_melnikovcomp}. 

\begin{proposition}\label{proposition_melnikovtildeexp}
The expansion of the Poincar\'e-Melnikov potential $\L$ satisfies the following properties, where the notation $\alpha^{ij}_k$, $H^{ij}_k$ is as in Proposition \ref{proposition_secularexpansion}.
\begin{enumerate}
\item
The first nontrivial term in the expansion of $\L$ is $\frac{1}{L_2^2} \alpha_2^{12} \L_2^{12}$ where
\begin{align}
\L_2^{12}\left( \tilde{\gamma}_2, \tilde{\Gamma}_2 \right) ={}& \int_{- \infty}^{\infty} \left( H_2^{12} \left( Z^0 \left( t, \tilde{\gamma}_2, \tilde{\psi}_1, \tilde{\gamma}_3, \tilde{\Gamma}_2, \tilde{\Psi}_1, \tilde{\Gamma}_3 \right) \right) - H_2^{12} \left( Z^0_* \left( t, \tilde{\gamma}_2, \tilde{\psi}_1, \tilde{\gamma}_3, \tilde{\Gamma}_2, \tilde{\Psi}_1, \tilde{\Gamma}_3 \right) \right)\right) \, dt \\
={}& \tilde{\L}_2^{12} \left( \tilde{\Gamma}_2\right) \sin \tilde{\gamma}_2
\end{align}
and where $\tilde{\L}_2^{12}$ is an analytic function of $\tilde{\Gamma}_2$ that is nonvanishing for $\tilde \Gamma_2 \in [\zeta_1,\zeta_2]$ (see \eqref{eq_nhimdef}, \eqref{eq_nhimparametersdef}, and Appendix \ref{appendix_errata}). 
\item
The first term in the expansion of $\L$ that depends on $\tilde{\psi}_1$ is $\frac{L_2^9}{L_3^6} \alpha_2^{23} \L_2^{23}$ where
\begin{align}
\L_2^{23} \left(\tilde{\gamma}_2, \tilde{\psi}_1, \tilde{\Gamma}_2 \right) ={}& \int_{- \infty}^{\infty} \left( H_2^{23} \left( Z^0 \left( t, \tilde{\gamma}_2, \tilde{\psi}_1, \tilde{\gamma}_3, \tilde{\Gamma}_2, \tilde{\Psi}_1, \tilde{\Gamma}_3 \right) \right) - H_2^{23} \left( Z^0_* \left( t, \tilde{\gamma}_2, \tilde{\psi}_1, \tilde{\gamma}_3, \tilde{\Gamma}_2, \tilde{\Psi}_1, \tilde{\Gamma}_3 \right) \right)\right) \, dt \\
={}& \kappa \left( \frac{ \pi \, \tilde{\Gamma}_2}{A_2 \, L_1^2} \right) \left[ c_1 \, \cos \tilde{\gamma}_2 \, \cos \tilde{\psi}_1 + c_2 \, \sin \tilde{\gamma}_2 \, \sin \tilde{\psi}_1 \right]
\end{align}
where the function $\kappa$ is defined in \eqref{eq_melnikovcoefficientfn}, the constants $c_1, \, c_2$ are defined in \eqref{eq_c1c2def}, and $\chi, \, A_2$ are the constants defined in \eqref{eq_chia2def}.
\item
The first term in the expansion of $\L$ that depends on $\tilde{\gamma}_3$ is $\frac{L_2^{11}}{L_3^8} \alpha_5^{23} \L_5^{23}$ where
\begin{align}
\L_5^{23} \left(\tilde{\gamma}_2, \tilde{\psi}_1, \tilde{\gamma}_3, \tilde{\Gamma}_2 \right) ={}& \int_{- \infty}^{\infty} \left( H_5^{23} \left( Z^0 \left( t, \tilde{\gamma}_2, \tilde{\psi}_1, \tilde{\gamma}_3, \tilde{\Gamma}_2, \tilde{\Psi}_1, \tilde{\Gamma}_3 \right) \right) - H_5^{23} \left( Z^0_* \left( t, \tilde{\gamma}_2, \tilde{\psi}_1, \tilde{\gamma}_3, \tilde{\Gamma}_2, \tilde{\Psi}_1, \tilde{\Gamma}_3 \right) \right)\right) \, dt \\
={}& \kappa \left( \frac{ \pi \, \tilde{\Gamma}_2}{A_2 \, L_1^2} \right) \left( J_1 \left( \tilde{\psi}_1, \tilde{\gamma}_3 \right) \cos \tilde{\gamma}_2 + J_2 \left( \tilde{\psi}_1, \tilde{\gamma}_3 \right) \sin \tilde{\gamma}_2 \right)
\end{align}
where again $\kappa$ is the function defined by \eqref{eq_melnikovcoefficientfn}, and where
\[
\begin{split}
J_1 \left( \tilde{\psi}_1, \tilde{\gamma}_3 \right) &= 
30\,\delta_{3}^2\,\sin \tilde{\gamma}_{3}\,\cos \tilde{\psi}_{1}\,\sin 
 \tilde{\psi}_{1}-10\,\delta_{1}^2\,\sin \tilde{\gamma}_{3}\,\cos \tilde{\psi}_{1}\,
 \sin \tilde{\psi}_{1}-20\,\delta_{1}\,\delta_{3}\,\cos \tilde{\gamma}_{3}\,\cos ^2
 \tilde{\psi}_{1}+10\,\delta_{1}\,\delta_{3}\,\cos \tilde{\gamma}_{3}\\
% \end{equation}
% and
% \begin{align}
J_2 \left( \tilde{\psi}_1, \tilde{\gamma}_3 \right) &= 
-50\,\delta_{1}\,\delta_{3}\,\cos \tilde{\gamma}_{3}\,\cos \tilde{\psi}_{1}\,\sin 
 \tilde{\psi}_{1}+ \frac{70\,\delta_{3}}{\delta_1}\,\cos \tilde{\gamma}_{3}\,\cos \tilde{\psi}_{1}\,
 \sin \tilde{\psi}_{1}+\frac{105\,\delta_{3}^2}{\delta_1^2}\,\sin 
 \tilde{\gamma}_{3}\,\cos ^2\tilde{\psi}_{1} \\
 & -75\,\delta_{3}^2\,
 \sin \tilde{\gamma}_{3}\,\cos ^2\tilde{\psi}_{1}+25\,\delta_{1}^2\,\sin \tilde{\gamma}_{3}\,\cos 
 ^2\tilde{\psi}_{1}-35\,\sin \tilde{\gamma}_{3}\,\cos ^2\tilde{\psi}_{1}- \frac{105\,
 \delta_{3}^2}{\delta_1^2} \,\sin \tilde{\gamma}_{3} \\
 & +60\,\delta_{3}^2\,
 \sin \tilde{\gamma}_{3}-17\,\delta_{1}^2\,\sin \tilde{\gamma}_{3}+28\,\sin \tilde{\gamma}_{3}.
\end{split}
\]
\end{enumerate}
\end{proposition}

The following result guarantees the existence of two homoclinic channels, and thus two scattering maps; furthermore it provides first-order approximations of the scattering maps in `tilde' variables. 

\begin{lemma}\label{lemma_scatteringtildecoords}
The secular Hamiltonian has two homoclinic channels $\Gamma_{\pm}$ corresponding to the normally hyperbolic invariant manifold $\Lambda$, and there are two scattering maps defined globally on $\Lambda$ (see \eqref{eq_nhimdef} and \eqref{eq_nhimparametersdef}) in the variables \eqref{eq_changeofcoordstilde} by
\begin{equation}
\tilde{S}_{\pm} : \left( \tilde{\gamma}_2, \tilde{\psi}_1, \tilde{\gamma}_3, \tilde{\Gamma}_2, \tilde{\Psi}_1, \tilde{\Gamma}_3 \right) \longmapsto \left( \tilde{\gamma}_2^*, \tilde{\psi}_1^*, \tilde{\gamma}_3^*, \tilde{\Gamma}_2^*, \tilde{\Psi}_1^*, \tilde{\Gamma}_3^* \right)
\end{equation}
with
\begin{equation}
\tilde{\gamma}_2^* = \tilde{\gamma}_2 + \Delta_2^{\pm} \left( \tilde{\Gamma}_2; \cdots \right), \quad
\tilde{\psi}_1^* = \tilde{\psi}_1 + \frac{1}{L_2^2} \Delta_1^{\pm} \left( \tilde{\Gamma}_2 ; \cdots \right), \quad
\tilde{\gamma}_3^* = \tilde{\gamma}_3 + \frac{L_2^{10}}{L_3^7} \Delta_3^{\pm} \left( \tilde{\gamma}_2, \tilde{\psi}_1, \tilde{\gamma}_3, \tilde{\Gamma}_2, \tilde{\Psi}_1, \tilde{\Gamma}_3 \right) 
\end{equation}
and
\begin{equation}
\tilde{\Gamma}_2^* = \tilde{\Gamma}_2 + \frac{L_2^8}{L_3^6} \Theta_2^{\pm} \left( \tilde{\psi}_1, \tilde{\Gamma}_2; \cdots \right), \quad
\tilde{\Psi}_1^* = \tilde{\Psi}_1 + \frac{L_2^9}{L_3^6} \Theta_1^{\pm}\left( \tilde{\psi}_1, \tilde{\Gamma}_2; \cdots \right), \quad
\tilde{\Gamma}_3^* = \tilde{\Gamma}_3 + \frac{L_2^{11}}{L_3^8} \Theta_3^{\pm} \left( \tilde{\psi}_1, \tilde{\gamma}_3, \tilde{\Gamma}_2; \cdots \right)
\end{equation}
where the ellipsis after the semicolon denotes dependence on the remaining variables at higher order, and where
\begin{equation}
\Delta_2^{\pm} \left( \tilde{\Gamma}_2 ; \cdots \right) = 2 \, \arctan 
\chi^{-1} + \cdots, \quad \Delta_1^{\pm} \left( \tilde{\Gamma}_2 ; \cdots 
\right) = \frac{L_1}{6} \sqrt{\frac{3}{2}}  \, \beta_2 
\tilde{\Gamma}_2 \sqrt{1 - \frac{5}{3} \frac{\tilde{\Gamma}_2^2}{L_1^2}} + 
\cdots,\quad \Delta_3^{\pm}=O(1)
\end{equation}
and
\begin{align}
\Theta_2^{\pm} \left( \tilde{\psi}_1, \tilde{\Gamma}_2; \cdots \right) ={}& \pm  \alpha_1^{12} \, \alpha_2^{23} \, \kappa \left( \frac{\pi \, \tilde{\Gamma}_2}{A_2 \, L_1^2} \right) \, c_2 \, \left(3 \frac{\tilde{\Gamma}_2^2}{L_1^2} - 1 \right) \frac{L_1^2}{2 \, \tilde{\Gamma}_2} \, \cos \tilde{\psi}_1 + \cdots \\
\Theta_1^{\pm}\left( \tilde{\psi}_1, \tilde{\Gamma}_2; \cdots \right) ={}& \mp \alpha_2^{23} \, \kappa \left( \frac{\pi \, \tilde{\Gamma}_2}{A_2 \, L_1^2} \right) \, c_2 \, \cos \tilde{\psi}_1 + \cdots \\
\Theta_3^{\pm} \left( \tilde{\psi}_1, \tilde{\gamma}_3, \tilde{\Gamma}_2; \cdots \right) ={}& \mp  \alpha_5^{23} \, \kappa \left( \frac{\pi \, \tilde{\Gamma}_2}{A_2 \, L_1^2} \right) \, \frac{\partial J_2}{\partial \tilde{\gamma}_3} \left( \tilde{\psi}_1, \tilde{\gamma}_3 \right) + \cdots
\end{align}
where $\kappa$ is the function defined by \eqref{eq_melnikovcoefficientfn}. 
\end{lemma}

\begin{proof}
Denote by $\left( \omega_0, \omega_1, \omega_2 \right)$ the frequency vector of 
the angles $\left( \tilde{\gamma}_2, \tilde{\psi}_1, \tilde{\gamma}_3 \right)$ 
on a torus on $\Lambda$ corresponding to fixed values of the actions 
$\tilde{\Gamma}_2, \tilde{\Psi}_1, \tilde{\Gamma}_3$ for the Hamiltonian 
$\tilde{F}_{\mathrm{quad}}^{12}$. Its particular form, given in Lemma \ref{lemma_quad12expansion}, implies that 
\[
\omega_0=2\alpha_0^{12}\frac{\tilde\Gamma_2}{L_1^2}+O\left(\frac{1}{L_2}
\right)\qquad \omega_1=O\left(\frac{1}{L_2}
\right),\qquad \omega_2=0.
\]
Consider the function
\begin{equation}\label{eq_melnikovfrequencymap}
\tau \longmapsto \L \left( \tilde{\gamma}_2 - \omega_0 \, \tau, \tilde{\psi}_1 - \omega_1 \, \tau, \tilde{\gamma}_3 - \omega_2 \, \tau, \tilde{\Gamma}_2, \tilde{\Psi}_1, \tilde{\Gamma}_3 \right) 
\end{equation}
where $\L$ is the Poincar\'e-Melnikov potential defined by 
\eqref{eq_melnikovdef}. Results  of \cite{delshams2006biggaps} imply that 
nondegenerate critical points of \eqref{eq_melnikovfrequencymap} correspond to 
transverse homoclinic intersections of the stable and unstable manifolds of 
$\Lambda$. Equation \eqref{eq_melnikovdef}, Proposition 
\ref{proposition_secularexpansion}, and Proposition 
\ref{proposition_melnikovtildeexp} imply that 
\begin{equation}
\L = \frac{1}{L_2^2} \alpha_2^{12} \L_2^{12} + O \left( \frac{L_2^9}{L_3^6}\right). 
\end{equation}
The function $\tau \mapsto \L_2^{12} \left( \tilde{\gamma}_2 - \omega_0 \, \tau, \tilde{\Gamma}_2 \right)$ has nondegenerate critical points $\tau_{\pm}$ where $\omega_0 \tau_{\pm} = \tilde{\gamma}_2 \pm \frac{\pi}{2} $. It follows that there are functions 
\begin{equation}
\tau^*_{\pm} \left( \tilde{\gamma}_2, \tilde{\psi}_1, \tilde{\gamma}_3, \tilde{\Gamma}_2, \tilde{\Psi}_1, \tilde{\Gamma}_3 \right) = \frac{1}{\omega_0} \left( \tilde{\gamma}_2 \pm \frac{\pi}{2} \right) + O \left( \frac{L_2^{11}}{L_3^6} \right)
\end{equation}
such that
\begin{equation}
\left. \frac{d}{d \tau} \right|_{\tau = \tau^*_{\pm}} \L \left( \tilde{\gamma}_2 - \omega_0 \, \tau, \tilde{\psi}_1 - \omega_1 \, \tau, \tilde{\gamma}_3 - \omega_2 \, \tau, \tilde{\Gamma}_2, \tilde{\Psi}_1, \tilde{\Gamma}_3 \right)  = 0. 
\end{equation}
We now introduce the reduced Poincar\'e-Melnikov potentials
\begin{equation}
\L^*_{\pm} \left( \tilde{\gamma}_2, \tilde{\psi}_1, \tilde{\gamma}_3, \tilde{\Gamma}_2, \tilde{\Psi}_1, \tilde{\Gamma}_3 \right) = \L \left( \tilde{\gamma}_2 - \omega_0 \, \tau^*_{\pm}, \tilde{\psi}_1 - \omega_1 \, \tau^*_{\pm}, \tilde{\gamma}_3 - \omega_2 \, \tau^*_{\pm}, \tilde{\Gamma}_2, \tilde{\Psi}_1, \tilde{\Gamma}_3 \right). 
\end{equation}
Now, following again \cite{delshams2006biggaps}, the changes in the actions 
coming from the scattering maps $\tilde{S}_{\pm}$ are defined using the 
functions $\L^*_{\pm}$ via
\begin{equation}
\tilde{\Gamma}_2^* = \tilde{\Gamma}_2 + \frac{\partial \L^*_{\pm}}{\partial 
\tilde{\gamma}_2} + \cdots, \quad \tilde{\Psi}_1^* = \tilde{\Psi}_1 + 
\frac{\partial \L^*_{\pm}}{\partial \tilde{\psi}_1} + \cdots, \quad 
\tilde{\Gamma}_3^* = \tilde{\Gamma}_3 + \frac{\partial \L^*_{\pm}}{\partial 
\tilde{\gamma}_3} + \cdots
\end{equation}
Note that the cylinder frequencies in the model in \cite{delshams2006biggaps} 
all have the same time scale and moreover the first order of the perturbation 
depends on all the angles. On the contrary, in our model all have 
different speeds and the angles $\tilde\psi_1$ and $\tilde\gamma_3$ appear 
only at higher order terms (see Proposition 
\ref{proposition_secularexpansion}). Still, one can easily check that the 
statements in \cite{delshams2006biggaps} are still valid in the present setting. 
The only difference is that the first order of the scattering maps in the 
actions $\tilde\Psi_1$ and $\tilde\Gamma_3$ come from higher orders of the 
Melnikov potential.

Indeed, we have
\begin{align}
\frac{\partial \L^*_{\pm}}{\partial \tilde{\gamma}_2} ={}& \frac{L_2^9}{L_3^6} \, \alpha_2^{23} \, \frac{\partial \L_2^{23}}{\partial \tilde{\psi}_1} \, \frac{\partial}{\partial \tilde{\gamma}_2} \left( \tilde{\psi}_1 - \tau_{\pm}^* \, \omega_1 \right) + \cdots \\
={}& - \frac{L_2^9}{L_3^6} \, \alpha_2^{23} \, \kappa \left( \frac{\pi \, \tilde{\Gamma}_2}{A_2 \, L_1^2} \right) \, c_2 \, \sin \left( \mp \frac{\pi}{2} \right) \cos \left( \tilde{\psi}_1 \right) \omega_1 \, \frac{\partial \tau_{\pm}^*}{\partial \tilde{\gamma}_2} + \cdots \\
={}& \pm \frac{L_2^8}{L_3^6} \, \alpha_1^{12} \, \alpha_2^{23} \, \kappa \left( \frac{\pi \, \tilde{\Gamma}_2}{A_2 \, L_1^2} \right) \, c_2 \, \left(3 \frac{\tilde{\Gamma}_2^2}{L_1^2} - 1 \right) \frac{1}{\omega_0} \, \cos \tilde{\psi}_1 + \cdots \\
={}& \pm \frac{L_2^8}{L_3^6} \, \alpha_1^{12} \, \alpha_2^{23} \, \kappa \left( \frac{\pi \, \tilde{\Gamma}_2}{A_2 \, L_1^2} \right) \, c_2 \, \left(3 \frac{\tilde{\Gamma}_2^2}{L_1^2} - 1 \right) \frac{L_1^2}{2 \, \tilde{\Gamma}_2} \, \cos \tilde{\psi}_1 + \cdots,
\end{align}
\begin{equation}
\frac{\partial \L^*_{\pm}}{\partial \tilde{\psi}_1} = \frac{L_2^9}{L_3^6} \, \alpha_2^{23} \, \frac{\partial \L_2^{23}}{\partial \tilde{\psi}_1} + \cdots = \mp \frac{L_2^9}{L_3^6} \, \alpha_2^{23} \, \kappa \left( \frac{\pi \, \tilde{\Gamma}_2}{A_2 \, L_1^2} \right) \, c_2 \, \cos \tilde{\psi}_1 + \cdots,
\end{equation}
and
\begin{equation}
\frac{\partial \L^*_{\pm}}{\partial \tilde{\gamma}_3} = \frac{L_2^{11}}{L_3^8} \, \alpha_5^{23} \, \frac{\partial \L_5^{23}}{\partial \tilde{\gamma}_3} + \cdots = \mp \frac{L_2^{11}}{L_3^8} \, \alpha_5^{23} \, \kappa \left( \frac{\pi \, \tilde{\Gamma}_2}{A_2 \, L_1^2} \right) \, \frac{\partial J_2}{\partial \tilde{\gamma}_3} \left( \tilde{\psi}_1, \tilde{\gamma}_3 \right) + \cdots
\end{equation}

For the angles $\tilde{\gamma}_2, \tilde{\psi}_1, \tilde{\gamma}_3$, the first-order term under application of the scattering map is a so-called \emph{phase shift}. This is a change in the angle that comes from the integrable part of the Hamiltonian along the separatrix, and does not necessarily depend on the functions $\L^*_{\pm}$ at first order. The phase shift in $\tilde{\gamma}_2$ comes from \eqref{eq_gamma22separatrix} as follows:
\begin{equation}
\Delta_2 \left( \tilde{\Gamma}_2; \cdots \right) = \lim_{t \to + \infty} \left( \tilde{\gamma}_2^2 (t) - \tilde{\gamma}_2^2 (-t) \right) = 2 \, \arctan \chi^{-1}. 
\end{equation}
The phase shift in $\tilde{\psi}_1$ is computed in Appendix \ref{appendix_psi1phaseshift}. As for the phase shift in $\tilde{\gamma}_3$, we simply estimate that it cannot be larger than $O \left( \frac{L_2^{10}}{L_3^7}\right)$ for the following reason: the first order term in the expansion of the secular Hamiltonian containing $\tilde{\Gamma}_3$ is of order $\frac{L_2^3}{L_3^6}$ (see Proposition \ref{proposition_secularexpansion}). Since this term provides no phase shift in $\tilde{\gamma}_3$ at first order, the largest possible term that can produce a phase shift has another factor of $\frac{L_2}{L_3}$. Since we normalise the entire secular Hamiltonian by $L_2^6$, we see that the phase shift in $\tilde{\gamma}_3$ cannot be larger than terms of order $\frac{L_2^3}{L_3^6} \frac{L_2}{L_3} L_2^6 = \frac{L_2^{10}}{L_3^7}$. 
\end{proof}

\begin{remark}
Note that we do not give an expression for $\Delta_3$ in Lemma \ref{lemma_scatteringtildecoords}. Indeed we require only an estimate on its order; as we will see in the proof of Lemma \ref{lemma_scatteringhatcoords}, the phase shift in $\tilde{\gamma}_3$ is small enough that we can ignore it. 
\end{remark}

\sloppy Next, we combine Lemma \ref{lemma_scatteringtildecoords} and the coordinate 
transformation $\left( \tilde{\gamma}_2, \tilde{\psi}_1, \tilde{\gamma}_3, 
\tilde{\Gamma}_2, \tilde{\Psi}_1, \tilde{\Gamma}_3 \right) \mapsto \left( 
\hat{\gamma}_2, \hat{\psi}_1, \hat{\gamma}_3, \hat{\Gamma}_2, \hat{\Psi}_1, 
\hat{\Gamma}_3 \right)$ provided by Theorem \ref{theorem_innerdynamics} (see also Remark \ref{remark_shrinkingnhim}) to 
produce an expression for the scattering maps $S_{\pm}$ in `hat' variables, thus 
completing the proof of Theorem \ref{theorem_outerdynamics}. 

\begin{lemma}\label{lemma_scatteringhatcoords}
In the `hat' coordinates, the scattering maps $S_{\pm} : \Lambda \to \Lambda'$ introduced in Lemma \ref{lemma_scatteringtildecoords} are given by 
\begin{equation}
S_{\pm}: \left( \hat{\gamma}_2, \hat{\psi}_1, \hat{\gamma}_3, \hat{\Gamma}_2, \hat{\Psi}_1, \hat{\Gamma}_3 \right) \longmapsto \left( \hat{\gamma}_2^*, \hat{\psi}_1^*, \hat{\gamma}_3^*, \hat{\Gamma}_2^*, \hat{\Psi}_1^*, \hat{\Gamma}_3^* \right)
\end{equation}
with
\begin{equation}
\hat{\Psi}_1^* = \hat{\Psi}_1 + \frac{L_2^9}{L_3^6} \S_1^{\pm} \left( \hat{\psi}_1, \hat{\Gamma}_2 \right) + \cdots, \quad \hat{\Gamma}_3^* = \hat{\Gamma}_3 + \frac{L_2^{11}}{L_3^8} \S^{\pm}_3 \left( \hat{\psi}_1, \hat{\gamma}_3, \hat{\Gamma}_2 \right) + \cdots
\end{equation}
where $\S_1^{\pm}$ and $\S_3^{\pm}$ are given by \eqref{eq_scatteringpsi1} and \eqref{eq_scatteringgamma3} respectively. 
\end{lemma}

\begin{proof}

Denote by $\hat{z}$ a point on $\Lambda$ in `hat' coordinates, by $\tilde{z}$ the same point expressed in `tilde' coordinates, and by $\Phi : \tilde{z} \mapsto \hat{z}$ the coordinate transformation provided by Theorem \ref{theorem_innerdynamics}. The maps $\tilde{S}_{\pm}$, defined in Lemma \ref{lemma_scatteringtildecoords} send $\tilde{z} \mapsto \tilde{z}^*$. In `hat' coordinates, the scattering maps $S_{\pm}$ are therefore defined by
\begin{equation}
\hat{z}^* = S_{\pm} (\hat{z}) = \Phi \circ \tilde{S}_{\pm} \circ \Phi^{-1} (\hat{z}).
\end{equation}
We compute the effect of these maps in the $\hat{\Psi}_1$, $\hat{\Gamma}_3$ variables. Comparing Lemma \ref{lemma_straightsympform} and Lemma \ref{lemma_averaging} in the context of the coordinate transformation $\Phi$, notice that we can write
\begin{equation}
\hat{\Psi}_1 = \tilde{\Psi}_1 + \frac{L_2^{11}}{L_3^6} \Phi_1 \left( \tilde{\psi}_1, \tilde{\Gamma}_2; \cdots \right), \quad \hat{\Gamma}_3 = \tilde{\Gamma}_3 + \frac{L_2^{13}}{L_3^8} \Phi_3 \left( \tilde{\psi}_1, \tilde{\gamma}_3, \tilde{\Gamma}_2 ; \cdots \right),
\end{equation}
where (see equation \eqref{eq_innersecexpansion})
\begin{align}
\Phi_1 \left( \tilde{\psi}_1, \tilde{\Gamma}_2; \cdots \right) ={}& - \frac{\alpha_0^{23}}{\alpha_1^{12}} \, \frac{ \cos \left( 2 \tilde{\psi}_1 \right)}{ 2 \left(3 \frac{\tilde{\Gamma}_2 ^2}{L_1^2} - 1 \right) } + \cdots \\
\Phi_3 \left( \tilde{\psi}_1, \tilde{\gamma}_3, \tilde{\Gamma}_2 ; \cdots \right) ={}& \tilde{\alpha} \, \left( 3 \frac{\tilde{\Gamma}_2^2}{L_1^2} - 1 \right)^{-1} \bigg[ \frac{1}{3} \, \nu_0 \, \cos \left( \tilde{\gamma}_3 + 3 \tilde{\psi}_1 \right) + \nu_1 \, \cos \left( \tilde{\gamma}_3 + \tilde{\psi}_1 \right) \\
& \quad- \nu_2 \, \cos \left( \tilde{\gamma}_3 - \tilde{\psi}_1 \right) - \frac{1}{3} \, \nu_3 \, \cos \left( \tilde{\gamma}_3 - 3 \tilde{\psi}_1 \right) \bigg] + \cdots
\end{align}

where $\tilde \alpha$ is a nontrivial constant. Thus 
\begin{align}
\hat{\Psi}_1^* ={}& \tilde{\Psi}_1^* + \frac{L_2^{11}}{L_3^6} \Phi_1 \left( \tilde{\psi}_1^*, \tilde{\Gamma}_2^*; \cdots \right) \\
={}& \tilde{\Psi}_1 + \frac{L_2^9}{L_3^6} \Theta_1^{\pm} \left( \tilde{\psi}_1, \tilde{\Gamma}_2 ; \cdots \right) + \frac{L_2^{11}}{L_3^6} \Phi_1 \left( \tilde{\psi}_1 + \frac{1}{L_2^2} \Delta_1^{\pm} \left( \tilde{\Gamma}_2; \cdots \right), \tilde{\Gamma}_2 + \frac{L_2^8}{L_3^6} \Theta^{\pm}_2 \left( \tilde{\psi}_1, \tilde{\Gamma}_2 ; \cdots \right); \cdots \right) \\
={}& \hat{\Psi}_1 + \frac{L_2^9}{L_3^6} \left[ \Theta_1^{\pm} \left( \tilde{\psi}_1, \tilde{\Gamma}_2 ; \cdots \right) + \partial_{\tilde{\psi}_1} \Phi_1 \left( \tilde{\psi}_1, \tilde{\Gamma}_2 ; \cdots \right) \Delta_1^{\pm} \left( \tilde{\Gamma}_2 ; \cdots \right) \right] + \cdots \\
={}& \hat{\Psi}_1 + \frac{L_2^9}{L_3^6} \left[ \Theta_1^{\pm} \left( \hat{\psi}_1, \hat{\Gamma}_2 ; \cdots \right) + \partial_{\hat{\psi}_1} \Phi_1 \left( \hat{\psi}_1, \hat{\Gamma}_2 ; \cdots \right) \Delta_1^{\pm} \left( \hat{\Gamma}_2 ; \cdots \right) \right] + \cdots 
\end{align}
and
\begin{align}
\hat{\Gamma}_3^* ={}& \tilde{\Gamma}_3^* + \frac{L_2^{13}}{L_3^8} \Phi_3 \left( \tilde{\psi}_1^*, \tilde{\gamma}_3^*, \tilde{\Gamma}_2^* ; \cdots \right) \\
={}& \tilde{\Gamma}_3 + \frac{L_2^{11}}{L_3^8} \Theta_3^{\pm} \left( \tilde{\psi}_1, \tilde{\gamma}_3, \tilde{\Gamma}_2; \cdots \right) \\
& + \frac{L_2^{13}}{L_3^8} \Phi_3 \left( \tilde{\psi}_1 + \frac{1}{L_2^2} \Delta_1^{\pm} \left( \tilde{\Gamma}_2 ; \cdots \right), \tilde{\gamma}_3 + \frac{L_2^{10}}{L_3^7} \Delta_3^{\pm} \left( \cdots \right), \tilde{\Gamma}_2 + \frac{L_2^8}{L_3^6} \Theta_2^{\pm} \left( \tilde{\psi}_1, \tilde{\Gamma}_2 ; \cdots \right) ; \cdots \right) \\
={}& \hat{\Gamma}_3 + \frac{L_2^{11}}{L_3^8} \left[ \Theta_3^{\pm} \left( \tilde{\psi}_1, \tilde{\gamma}_3, \tilde{\Gamma}_2; \cdots \right) + \partial_{\tilde{\psi}_1} \Phi_3 \left( \tilde{\psi}_1, \tilde{\gamma}_3, \tilde{\Gamma}_2; \cdots \right) \Delta_1^{\pm} \left( \tilde{\Gamma}_2 ; \cdots \right) \right] + \cdots \\
={}& \hat{\Gamma}_3 + \frac{L_2^{11}}{L_3^8} \left[ \Theta_3^{\pm} \left( \hat{\psi}_1, \hat{\gamma}_3, \hat{\Gamma}_2; \cdots \right) + \partial_{\hat{\psi}_1} \Phi_3 \left( \hat{\psi}_1, \hat{\gamma}_3, \hat{\Gamma}_2; \cdots \right) \Delta_1^{\pm} \left( \hat{\Gamma}_2 ; \cdots \right) \right] + \cdots 
\end{align}
Therefore it remains only to compute
\begin{align}
\S_1^{\pm}\left( \hat{\psi}_1, \hat{\Gamma}_2 \right) ={}&\Theta_1^{\pm} \left( \hat{\psi}_1, \hat{\Gamma}_2 \right) + \partial_{\hat{\psi}_1} \Phi_1 \left( \hat{\psi}_1, \hat{\Gamma}_2 \right) \Delta_1^{\pm} \left( \hat{\Gamma}_2 ; \cdots \right) \\
={}& \mp \alpha_2^{23} \, \kappa \left( \frac{\pi \, \hat{\Gamma}_2}{A_2 \, L_1^2} \right) \, c_2 \, \cos \hat{\psi}_1 + \frac{\alpha_0^{23}}{\alpha_1^{12}} \, \beta_2 \frac{L_1}{6} \, \sqrt{ \frac{3}{2} } \hat{\Gamma}_2 \sqrt{1 - \frac{5}{3} \frac{\hat{\Gamma}_2^2}{L_1^2}} \frac{ \sin 2 \hat{\psi}_1}{3 \frac{\hat{\Gamma}_2^2}{L_1^2} - 1}
\end{align}
and
\begin{align}
\S^{\pm}_3 \left( \hat{\psi}_1, \hat{\gamma}_3, \hat{\Gamma}_2 \right) ={}& \Theta_3^{\pm} \left( \hat{\psi}_1, \hat{\gamma}_3, \hat{\Gamma}_2 \right) + \partial_{\hat{\psi}_1} \Phi_3 \left( \hat{\psi}_1, \hat{\gamma}_3, \hat{\Gamma}_2 \right) \Delta_1^{\pm} \left( \hat{\Gamma}_2 \right) \\
={}& \mp \alpha_5^{23} \, \kappa \left( \frac{\pi \, \hat{\Gamma}_2}{A_2 \, L_1^2} \right) \, \frac{\partial J_2}{\partial \hat{\gamma}_3} \left( \hat{\psi}_1, \hat{\gamma}_3 \right) + \frac{L_1}{6} \sqrt{ \frac{3}{2}} \, \beta_2 \, \hat{\Gamma}_2 \sqrt{1 - \frac{5}{3} \frac{\hat{\Gamma}_2^2}{L_1^2}} \tilde{\alpha} \, \left( 3 \frac{\hat{\Gamma}_2^2}{L_1^2} - 1 \right)^{-1} \\
& \times  \bigg[  \, \nu_0 \, \cos \left( \hat{\gamma}_3 + 3 \hat{\psi}_1 \right) + \nu_1 \, \cos \left( \hat{\gamma}_3 + \hat{\psi}_1 \right) + \nu_2 \, \cos \left( \hat{\gamma}_3 - \hat{\psi}_1 \right)  + \nu_3 \, \cos \left( \hat{\gamma}_3 - 3 \hat{\psi}_1 \right) \bigg] \\
={}& \mp \alpha_5^{23} \, \kappa \left( \frac{\pi \, \hat{\Gamma}_2}{A_2 \, L_1^2} \right) \, \bigg[ 50\,\delta_{1}\,\delta_{3}\,\sin \hat{\gamma}_{3}\,\cos \hat{\psi}_{1}\,\sin 
 \hat{\psi}_{1}- \frac{70\,\delta_{3}}{\delta_1}\,\sin \hat{\gamma}_{3}\,\cos \hat{\psi}_{1}\,
 \sin \hat{\psi}_{1} \\
 & +\frac{105\,\delta_{3}^2}{\delta_1^2}\,\cos \hat{\gamma}_{3}\,\cos ^2\hat{\psi}_{1} -75\,\delta_{3}^2\,\cos \hat{\gamma}_{3}\,\cos ^2\hat{\psi}_{1} +25\,\delta_{1}^2\,\cos \hat{\gamma}_{3}\,\cos^2\hat{\psi}_{1} -35\,\cos \hat{\gamma}_{3}\,\cos ^2\hat{\psi}_{1} \\
  & - \frac{105\, \delta_{3}^2}{\delta_1^2} \,\cos \hat{\gamma}_{3} +60\,\delta_{3}^2\, \cos \hat{\gamma}_{3}-17\,\delta_{1}^2\,\cos \hat{\gamma}_{3}+28\,\cos \hat{\gamma}_{3} \bigg] \\
  & +  \frac{L_1}{6} \sqrt{ \frac{3}{2}} \beta_2 \hat{\Gamma}_2 \sqrt{1 - \frac{5}{3} \frac{\hat{\Gamma}_2^2}{L_1^2}} \tilde{\alpha} \, \left( 3 \frac{\hat{\Gamma}_2^2}{L_1^2} - 1 \right)^{-1} \bigg[  \, \nu_0 \, \cos \left( \hat{\gamma}_3 + 3 \hat{\psi}_1 \right) \\
  &+ \nu_1 \, \cos \left( \hat{\gamma}_3 + \hat{\psi}_1 \right) + \nu_2 \, \cos \left( \hat{\gamma}_3 - \hat{\psi}_1 \right)  + \nu_3 \, \cos \left( \hat{\gamma}_3 - 3 \hat{\psi}_1 \right) \bigg]. 
\end{align}

\end{proof}

\subsection{Computation of the Poincar\'e-Melnikov potential}
\label{section_melnikovcomp}

In this section we address the proof of Proposition 
\ref{proposition_melnikovtildeexp}. First of all, note that part 1 of the 
proposition was proved in \cite{fejoz2016secular} (see also Appendix \ref{appendix_errata}), since $F_{\mathrm{sec}}^{12}$ 
coincides with the secular Hamiltonian from the three body problem. It remains 
to prove parts 2 and 3 of the proposition. By Proposition 
\ref{proposition_secularexpansion}, the first terms that could potentially split 
the separatrices in the $\tilde{\Psi}_1, \tilde{\Gamma}_3$ directions are 
$H_2^{23}, H_5^{23}$ respectively, since they are the first-order terms 
combining the separatrix variables $\gamma_1, \Gamma_1,\tilde{\gamma}_2$ with the angles 
$\tilde{\psi}_1, \tilde{\gamma}_3$ respectively. Since the variables 
$\tilde{\psi}_1, \tilde{\gamma}_3$ are constant with respect to $H_0^{12}$, we 
can easily write the periodic orbits $Z^0_{\mathrm{min, max}}$ and the 
heteroclinic orbit $Z^0$ as functions of $(t, \tilde{\gamma}_2, \tilde{\psi}_1, 
\tilde{\gamma}_3, \tilde{\Gamma}_2)$. We must therefore compute
\begin{align}
\L_j^{23} \left(\tilde{\gamma}_2, \tilde{\psi}_1, \tilde{\gamma}_3, \tilde{\Gamma}_2 \right) ={}& \int_0^{\infty} \left( H_j^{23} \circ Z^0 \left(t, \tilde{\gamma}_2, \tilde{\psi}_1, \tilde{\gamma}_3, \tilde{\Gamma}_2 \right) - H_j^{23} \circ Z_{\mathrm{min}}^0 \left(t, \tilde{\gamma}_2, \tilde{\psi}_1, \tilde{\gamma}_3, \tilde{\Gamma}_2 \right) \right) \, dt + \\
& \quad + \int^0_{- \infty} \left( H_j^{23} \circ Z^0 \left(t, \tilde{\gamma}_2, \tilde{\psi}_1, \tilde{\gamma}_3, \tilde{\Gamma}_2 \right) - H_j^{23} \circ Z_{\mathrm{max}}^0 \left(t, \tilde{\gamma}_2, \tilde{\psi}_1, \tilde{\gamma}_3, \tilde{\Gamma}_2 \right) \right) \, dt
\end{align}
for $j=2,5$ where $H_2^{23}$ is defined by \eqref{eq_H223def} and $H_5^{23}$ is defined by \eqref{eq_H523def}. Now, notice that, since $\tilde{\psi}_1, \tilde{\gamma}_3$ are constant with respect to $H_0^{12}$, we can write
\begin{equation}\label{eq_melnikovsplitintwo}
\begin{split}
\L_2^{23} \left(\tilde{\gamma}_2, \tilde{\psi}_1, \tilde{\Gamma}_2 \right)& = c_1 \, \cos \tilde{\psi}_1 \, \L_1 \left( \tilde{\gamma}_2, \tilde{\Gamma}_2 \right) + c_2 \, \sin \tilde{\psi}_1 \, \L_2 \left( \tilde{\gamma}_2, \tilde{\Gamma}_2 \right)\\
% \end{equation}
% and
% \begin{equation}\label{eq_melnikovsplitintwo2}
\L_5^{23} \left(\tilde{\gamma}_2, \tilde{\psi}_1, \tilde{\gamma}_3, \tilde{\Gamma}_2 \right) &= J_1 \left( \tilde{\psi}_1, \tilde{\gamma}_3 \right) \, \L_1 \left( \tilde{\gamma}_2, \tilde{\Gamma}_2 \right) + J_2 \left( \tilde{\psi}_1, \tilde{\gamma}_3 \right) \, \L_2 \left( \tilde{\gamma}_2, \tilde{\Gamma}_2 \right)
% \label{eq_melnikovsplitintwo2}
\end{split}
\end{equation}
(with the constants $c_1$, $c_2$ defined in Lemma \ref{lemma_quad23exp} and the expressions $J_1$, $J_2$ defined in Lemma \ref{lemma_oct23exp}) where we have discarded the higher-order terms for convenience, where
\begin{align}
\L_j \left(\tilde{\gamma}_2, \tilde{\Gamma}_2 \right) ={}& \int_0^{\infty} \left( \F_j \circ Z^0 \left(t, \tilde{\gamma}_2, 0, 0, \tilde{\Gamma}_2 \right) - \F_j \circ Z^0_{\mathrm{min}} \left(t, \tilde{\gamma}_2, 0, 0, \tilde{\Gamma}_2 \right) \right) \, dt + \\
& \quad + \int_{- \infty}^0 \left( \F_j \circ Z^0 \left(t, \tilde{\gamma}_2, 0, 0, \tilde{\Gamma}_2 \right) - \F_j \circ Z^0_{\mathrm{max}} \left(t, \tilde{\gamma}_2, 0, 0, \tilde{\Gamma}_2 \right) \right) \, dt,
\end{align}
and where the functions
\begin{equation}
\F_1 = \sqrt{\Gamma_1^2 - \tilde{\Gamma}_2^2} \, \cos \tilde{\gamma}_2, \quad \F_2 = \sqrt{\Gamma_1^2 - \tilde{\Gamma}_2^2} \, \sin \tilde{\gamma}_2
\end{equation}
do not depend on $\tilde{\psi}_1, \tilde{\gamma}_3$. 

For the rest of this section, to simplify notation, we write $\gamma, \Gamma$ instead of $\tilde{\gamma}_2, \tilde{\Gamma}_2$, respectively. Recall from Lemma 5 that $\gamma=\gamma^0+\gamma^1+\gamma^2$ along the separatrix, where $\gamma^2$ is defined by \eqref{eq_gamma22separatrix}. Since $\gamma^2$ does not tend to 0 as $t$ goes to $\pm \infty$, we must take into account the phase shifts
\begin{equation} \label{eq_phaseshifts}
\Delta_{\pm} = \lim_{t \to \pm \infty} \arctan \left( \chi^{-1} \tanh \left( A_2 \, t \right)\right) = \pm \arctan \chi^{-1} 
\end{equation}
along the periodic orbits $Z^0_{\mathrm{min},\mathrm{max}}$ in order to ensure that the integrals $\L_j$ are well-defined. This is done by taking $\gamma=\gamma^0+\gamma^1+\Delta_+$ on $Z^0_{\mathrm{min}}$ and $\gamma=\gamma^0+\gamma^1+\Delta_-$ on $Z^0_{\mathrm{max}}$.

Expanding the Fourier series of $\L_j$ in $\gamma^0$ for $j=1,2$, we have
\begin{equation}\label{eq_ljfourier}
\L_j \left( \gamma^0, \Gamma \right) = \L^*_j \, e^{i \gamma^0} + \overline{\L^*_j} \, e^{-i \gamma^0}
\end{equation}
where
\begin{equation}
\L_j^* = \frac{1}{2} \int_0^{\infty} \left(\xi^+_j(t) + i \, \eta^+_j(t) \right) \, e^{i \gamma^1(t)} \, dt + \frac{1}{2} \int_{- \infty}^0 \left(\xi^-_j(t) + i \, \eta^-_j(t) \right) \, e^{i \gamma^1(t)} \, dt 
\end{equation}
with
\begin{equation} \label{eq_xietadef1}
\begin{dcases}
\xi^{\pm}_1 (t) = -\eta^{\pm}_2 (t) =\sqrt{\Gamma_1(t)^2 - \Gamma^2} \, \cos \gamma^2(t) - \sqrt{L_1^2 - \Gamma^2} \, \cos \Delta_{\pm} \\
\eta^{\pm}_1 (t) =\xi^{\pm}_2 (t) = \sqrt{\Gamma_1(t)^2 - \Gamma^2} \, \sin \gamma^2(t) - \sqrt{L_1^2 - \Gamma^2} \, \sin \Delta_{\pm}
\end{dcases}
\end{equation}
% and
% \begin{equation} \label{eq_xietadef2}
% \begin{dcases}
% \xi^{\pm}_2 (t) = \sqrt{\Gamma_1(t)^2 - \Gamma^2} \, \sin \gamma^2(t) - \sqrt{L_1^2 - \Gamma^2} \, \sin \Delta_{\pm} \\
% \eta^{\pm}_2 (t) = - \sqrt{\Gamma_1(t)^2 - \Gamma^2} \, \cos \gamma^2(t) + \sqrt{L_1^2 - \Gamma^2} \, \cos \Delta_{\pm},
% \end{dcases}
% \end{equation}
and where $\Gamma_1(t)$, which describes the motion of $\Gamma_1$ along the separatrix, is defined by \eqref{eq_Gamma1separatrix}. 

\begin{lemma}
We can write
\begin{equation}\label{eq_l1starintegraltau}
\L_1^* = \frac{i}{2} \, \sqrt{\frac{2}{3}} \, \frac{\Gamma}{A_2 \, \chi} \left[ \int_0^{\infty} f^+ (\tau) \, d \tau + \int_{-\infty}^0 f^- (\tau) \, d \tau \right]
\end{equation}
where
\begin{equation}\label{eq_fpmdef}
f^{\pm} (\tau) = \left( \tanh \tau \mp 1 \right) \, e^{i \frac{2 \, \Gamma}{A_2 \, L_1^2} \tau}
\end{equation}
and where $\tau = A_2 \, t$. Furthermore
\begin{equation}\label{eq_l2starintermsofl1star}
\L_2^* = -i \, \L^*_1.
\end{equation}
\end{lemma}
\begin{proof}
As in \cite{fejoz2016secular} (see Lemma 5.1, 5.2, and 5.3 of that paper) we first write the integrands of $\L_1^*$ in terms of $\gamma_1$ and then in terms of $\tau=A_2 \, t$. It follows directly from \eqref{eq_cosgamma1separatrix}, \eqref{eq_Gamma1separatrix}, and \eqref{eq_gamma22separatrix} (see also Lemma 5.1 of \cite{fejoz2016secular}) that, on the separatrix, we have
\begin{equation}\label{eq_cossingamma2wrtgamma1}
\cos \gamma^2 = \sqrt{1 - \frac{5}{3} \, \cos^2 \gamma_1}, \quad  \sin \gamma^2 = \sqrt{\frac{5}{3}} \, \cos \gamma_1, \quad
% \end{equation}
% and
% \begin{equation} \label{eq_Gamma1intermsofgamma1}
\Gamma_1 = \sqrt{\frac{5}{3}} \, \Gamma \, \frac{\sin \gamma_1}{\sqrt{1 - \frac{5}{3} \, \cos^2 \gamma_1}}.
\end{equation}
It follows  that
\begin{equation}\label{eq_sqrtgamma1gammasqdiff}
\sqrt{\Gamma_1^2 - \Gamma^2} = \sqrt{\frac{2}{3}} \, \Gamma \, \frac{1}{\sqrt{1 - \frac{5}{3} \, \cos^2 \gamma_1}}.
\end{equation}
Moreover, from \eqref{eq_chia2def} and \eqref{eq_phaseshifts} we get
\begin{equation}\label{eq_phaseshiftcosine}
\cos \Delta_{\pm} = \frac{\chi}{\sqrt{1 + \chi^2}} = \sqrt{\frac{2}{3}} \, \Gamma \, \frac{1}{\sqrt{L_1^2 - \Gamma^2}},\qquad 
% \end{equation}
% and
% \begin{equation}\label{eq_phaseshiftsine}
\sin \Delta_{\pm} = \pm \frac{1}{\sqrt{1 +\chi^2}} = \pm \sqrt{\frac{2}{3}} \, \frac{\Gamma}{\chi} \, \frac{1}{\sqrt{L_1^2 - \Gamma^2}}.
\end{equation}
Combining \eqref{eq_xietadef1} with \eqref{eq_cossingamma2wrtgamma1}, \eqref{eq_sqrtgamma1gammasqdiff}, and \eqref{eq_phaseshiftcosine}
% and \eqref{eq_phaseshiftsine} 
yields
\begin{equation}\label{eq_xieta1pmintermsofgamma1}
\xi_1^{\pm} = 0, \quad \eta_1^{\pm} = \sqrt{\frac{2}{3}} \, \sqrt{\frac{5}{3}} \, \Gamma \, \frac{\cos \gamma_1}{\sqrt{1 - \frac{5}{3} \, \cos^2 \gamma_1}} \mp \sqrt{\frac{2}{3}} \, \frac{\Gamma}{\chi}.
\end{equation}
Now that $\eta_1^{\pm}$ is written in terms of $\gamma_1$, we proceed to write it in terms of $\tau$. Equation \eqref{eq_cosgamma1separatrix} implies that
\begin{equation}
\sqrt{1 - \frac{5}{3} \, \cos^2 \gamma_1} = \frac{\chi \, \cosh A_2 \, t}{\sqrt{\chi^2+(1+\chi^2)\, \sinh^2 (A_2 \, t)}}.
\end{equation}
This expression, together with \eqref{eq_gammafrequency}, \eqref{eq_cosgamma1separatrix}, and \eqref{eq_xieta1pmintermsofgamma1} proves \eqref{eq_l1starintegraltau} and \eqref{eq_fpmdef}. Finally, comparing \eqref{eq_xietadef1}  proves \eqref{eq_l2starintermsofl1star}.
\end{proof}

In order to determine the integral of $f^{\pm}$ over the positive and negative halves of the real line, we instead integrate $f^{\pm}$ over appropriately chosen contours $C^{\pm}$ in the complex plane using the residue theorem. Initially we assume the contour has ``width'' $R>0$. Then, taking the limit as $R \to \infty$ allows us to derive an expression for $\L_1^*$. 

\begin{figure}[h]
    \centering
    \begin{subfigure}[b]{0.36\textwidth}
        \includegraphics[width=\textwidth]{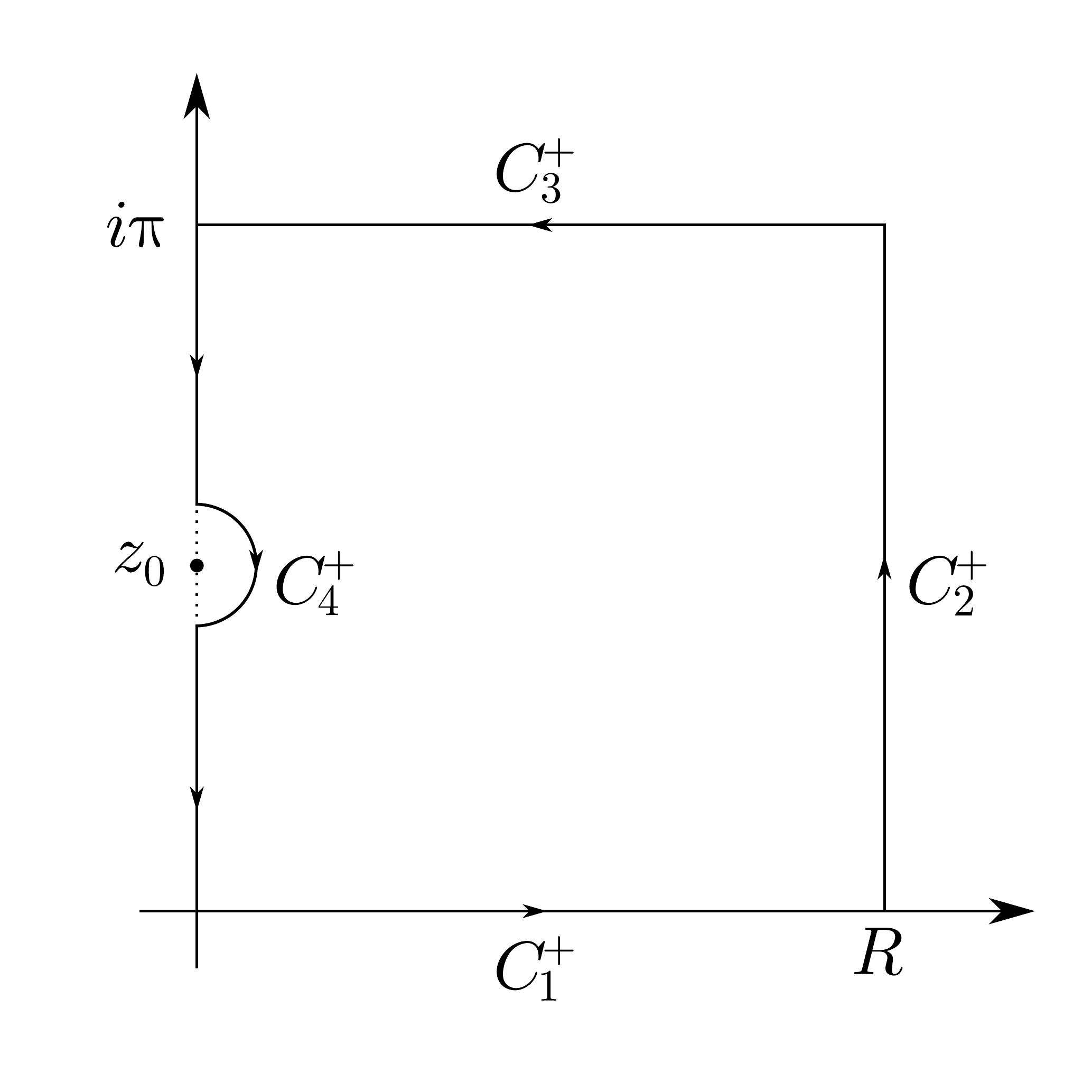}
        \caption{}
        \label{figure_contour_cplus}
    \end{subfigure} 
    \begin{subfigure}[b]{0.36\textwidth}
        \includegraphics[width=\textwidth]{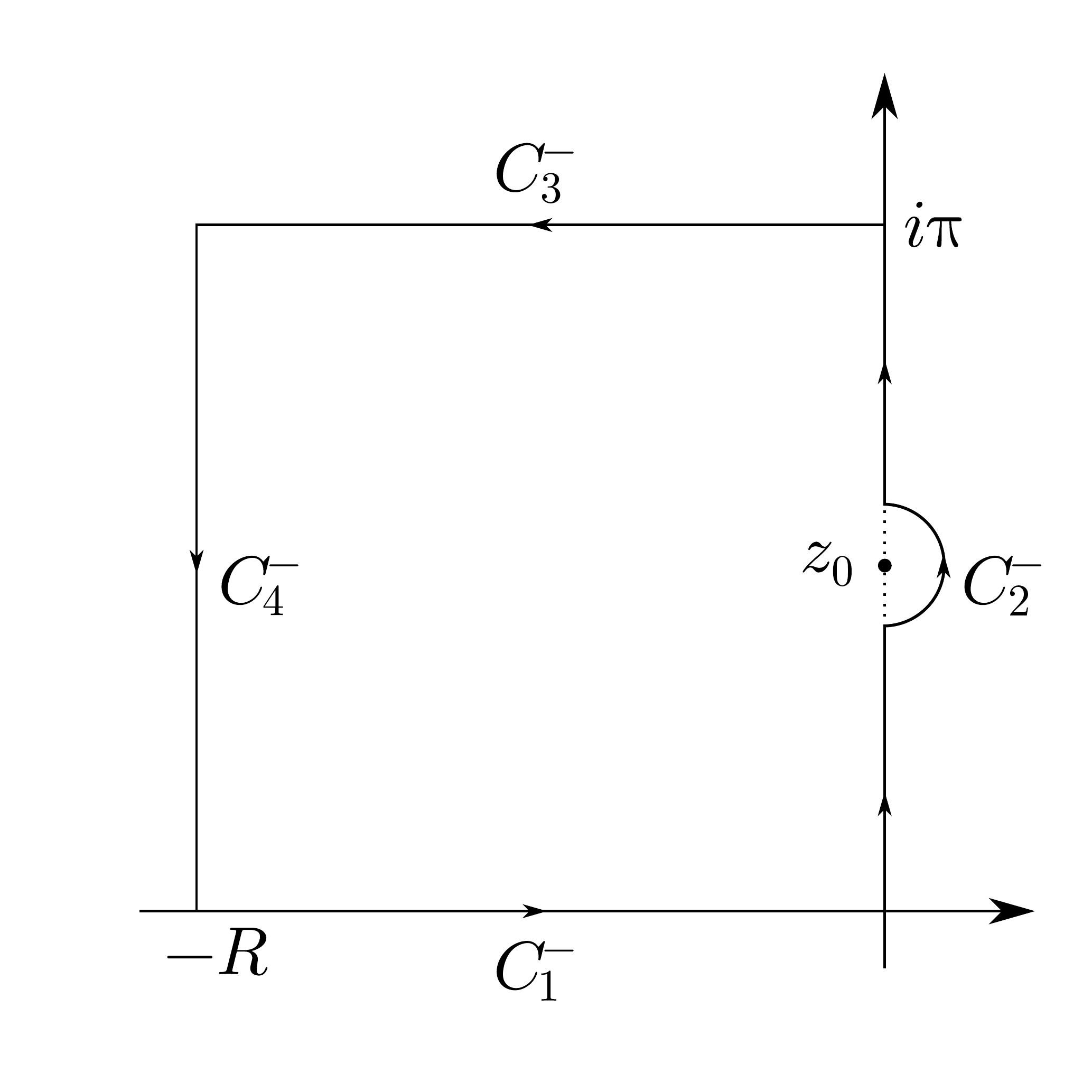}
        \caption{}
        \label{figure_contour_cminus}
    \end{subfigure}
   
    \caption{The contours $C^{\pm}$ over which we integrate $f^{\pm}$.}\label{figure_contours}
\end{figure}

From \eqref{eq_fpmdef}, it can be seen that
\begin{equation}\label{eq_fpmperiodicity}
f^{\pm} (\tau + i \, \pi) = e^{- \frac{2 \, \pi \, \Gamma}{A_2 \, L_1^2}} f^{\pm} (\tau).
\end{equation}
Moreover the only singularity of $\tanh \tau$, and thus of $f^{\pm}$, in the region $\{ 0 < \mathrm{Im} \, \tau < \pi \}$ is
% \begin{equation}
$z_0 = i \, \frac{\pi}{2}$,
% \end{equation}
and the residue of $f^{\pm}$ at $z_0$ is
\begin{equation}
\mathrm{Res}(f^{\pm},z_0) = e^{- \frac{\pi \, \Gamma}{A_2 \, L_1^2}}.
\end{equation}
We therefore define the contours $C^{\pm}$ as follows. Fix some $R>0$, and denote by $C^+_1$ the segment of the real line from $0$ to $R$; denote by $C_2^+$ the straight line from $R$ to $R + i \, \pi$; and denote by $C_3^+$ the straight line from $R+i \, \pi$ to $i \, \pi$. The last part $C^+_4$ of the contour $C^+$ is composed of two straight lines and a small semicircle centred at $z_0$, so as to exclude $z_0$ from the interior of the region bounded by $C^+$ (see Figure \ref{figure_contours}(a)). Similarly, $C_1^-$ is the segment of the real line from $-R$ to $0$; $C_2^-$ is $C_4^+$ traversed in the opposite direction; $C_3^-$ is the straight line from $i \, \pi$ to $-R + i \, \pi$; and $C_4^-$ is the straight line from $-R + i \, \pi$ to $-R$ (see Figure \ref{figure_contours}(b)). Thus we define
\begin{equation}
C^{\pm}=C^{\pm}_1 \cup C^{\pm}_2 \cup C^{\pm}_3 \cup C^{\pm}_4.
\end{equation}
The residue theorem implies that
\begin{equation}
\int_{C^+} f^+(\tau) \, d \tau =0, \quad \int_{C^-} f^-(\tau) \, d \tau = 2 \pi i \, e^{- \frac{\pi \, \Gamma}{A_2 \, L_1^2}}.
\end{equation}
In the limit as $R \to \infty$ we have
\begin{equation}
\lim_{R \to \infty} \int_{C_2^+} f^+ (\tau) \, d \tau = 0 = \lim_{R \to \infty} \int_{C_4^-} f^- (\tau) \, d \tau
\end{equation}
and
\begin{equation}
% \begin{dcases}
\lim_{R \to \infty} \int_{C_1^\pm \cup C_3^\pm} f^\pm (\tau) \, d \tau = \pm\left(1 - e^{- \frac{2 \, \pi \, \Gamma}{A_2 \, L_1^2}} \right) \int_0^{\pm\infty} f^\pm (\tau) \, d \tau \\
% \lim_{R \to \infty} \int_{C_1^- \cup C_3^-} f^- (\tau) \, d \tau = \left(1 - e^{- \frac{2 \, \pi \, \Gamma}{A_2 \, L_1^2}} \right) \int_{-\infty}^0 f^- (\tau) \, d \tau
% \end{dcases}
\end{equation}
% 
% \begin{equation}
% \begin{dcases}
% \lim_{R \to \infty} \int_{C_1^+ \cup C_3^+} f^+ (\tau) \, d \tau = \left(1 - e^{- \frac{2 \, \pi \, \Gamma}{A_2 \, L_1^2}} \right) \int_0^{\infty} f^+ (\tau) \, d \tau \\
% \lim_{R \to \infty} \int_{C_1^- \cup C_3^-} f^- (\tau) \, d \tau = \left(1 - e^{- \frac{2 \, \pi \, \Gamma}{A_2 \, L_1^2}} \right) \int_{-\infty}^0 f^- (\tau) \, d \tau
% \end{dcases}
% \end{equation}
due to \eqref{eq_fpmperiodicity}. The last part of the integral is
\begin{equation}
\int_{C_4^+} f^+ (\tau) \, d \tau + \int_{C_2^-} f^- (\tau) \, d \tau = \int_{C_2^-} \left(f^-(\tau) - f^+(\tau)\right) \, d \tau = 2 \int_{C_2^-} e^{i \frac{2 \, \Gamma}{A_2 \, L_1^2} \tau} \, d \tau.
\end{equation}
As the integrand of the last expression is holomorphic, this equals
\begin{equation}
2 \int_{i \, [0, \pi]} e^{i \frac{2 \, \Gamma}{A_2 \, L_1^2} \tau} \, d \tau = i \, \frac{A_2 \, L_1^2}{\Gamma} \left(1 - e^{- \frac{2 \, \pi \, \Gamma}{A_2 \, L_1^2}} \right).
\end{equation}
Therefore summing the integral of $f^+$ over $C^+$ and the integral of $f^-$ over $C^-$, and taking the limit as $R \to \infty$ yields
\begin{equation}
2 \pi i \, e^{- \frac{\pi \, \Gamma}{A_2 \, L_1^2}} = \left(1 - e^{- \frac{2 \, \pi \, \Gamma}{A_2 \, L_1^2}} \right) \left[ \int_0^{\infty} f^+ (\tau) \, d \tau + \int_{- \infty}^0 f^- (\tau) \, d \tau + i \, \frac{A_2 \, L_1^2}{\Gamma}  \right].
\end{equation}
Rearranging terms and using \eqref{eq_l1starintegraltau} gives
\begin{equation}
\L^*_1 = \frac{1}{2} \, \kappa \left( \frac{\pi \, \Gamma}{A_2 \, L_1^2} \right)
\end{equation}
where $\kappa$ is the function defined by \eqref{eq_melnikovcoefficientfn}. As this is real-valued, we find from \eqref{eq_ljfourier} and \eqref{eq_l2starintermsofl1star} that
\begin{equation}
\L_1 \left( \gamma^0 \right) = \kappa \left( \frac{\pi \, \Gamma}{A_2 \, L_1^2} \right) \cos \gamma^0, \quad  \L_2 \left( \gamma^0 \right) = \kappa \left( \frac{\pi \, \Gamma}{A_2 \, L_1^2} \right) \sin \gamma^0
\end{equation}
which, together with \eqref{eq_melnikovsplitintwo}, completes the proof of Proposition \ref{proposition_melnikovtildeexp}.

\section{Reduction to a Poincar\'e map and the shadowing argument}
\label{sec:mapreduction}

The purpose of this section is to show that there are orbits of the four body problem along which the quantities $\Psi_1$ and $\Gamma_3$ drift along any itinerary we choose. The idea is to apply the shadowing results proved by the authors in \cite{clarke2022topological} (see Appendix \ref{appendix_shadowing} of the present paper for a summary of those results). As the results of that paper are for maps (as is typically more convenient for shadowing results), we must first reduce the secular system to a Poincar\'e map; this is done in Section \ref{subsec_poincaremap}. In doing so, we show that the map satisfies the assumptions of Theorem \ref{theorem_main1}. We then show in the subsequent sections that the Poincar\'e map of the full four body problem thus satisfies the assumptions of Theorem \ref{theorem_main2}, and so, there are drifting orbits.

\subsection{The Poincar\'e map}\label{subsec_poincaremap}

In this section we show that we can extend the coordinates $\left(\hat{\gamma}_2, \hat{\Gamma}_2, \hat{\psi}_1, \hat{\Psi}_1, \hat{\gamma}_3, \hat{\Gamma}_3 \right)$on $\Lambda$ to (probably non-symplectic) coordinates on a subset of the secular phase space, and show that we can define a Poincar\'e map in the region in which these coordinates are defined, within an energy level of the secular Hamiltonian. We show that the inner map satisfies a twist condition, and deduce first-order expressions for the scattering maps in the action variables. 

Consider the region $\tilde \cD$ of the phase space in which the variables $\left(\xi, \eta, \tilde{\gamma}_2, \tilde{\Gamma}_2, \tilde{\psi}_1, \tilde{\Psi}_1, \tilde{\gamma}_3, \tilde{\Gamma}_3 \right)$ (introduced in Sections \ref{section_secularexpansion} and \ref{section_analysisofh0}) satisfy $\tilde \gamma_2, \tilde \psi_1, \tilde \gamma_3 \in \T$, $\tilde \Gamma_2 \in [\zeta_1 , \zeta_2 ]$, $\tilde \Psi_1, \tilde \Gamma_3 \in [-1, 1 ]$, and $(\xi, \eta)$ belongs to the open ball of radius $\sqrt{2 \, L_1}$ centred at the origin in $\mathbb{R}^2$ (see \eqref{eq_poincarevariables}). Recall the constants $\zeta_1, \zeta_2$ were defined in \eqref{eq_nhimparametersdef}. We now further restrict these constants by assuming that 
\begin{equation}\label{eq_isoenernondeg}
0 < \zeta_1 < \zeta_2 < \frac{L_1}{\sqrt{3}}.
\end{equation}
The reason for this further restriction is to guarantee that the Poincar\'e map which we will construct in this section satisfies a twist condition; see Lemma \ref{lemma_twistcondition}.

By slightly shrinking this region $\tilde \cD$ in the sense of Remark \ref{remark_shrinkingnhim}, we obtain a domain $\cD$ in which there is a (non-symplectic) coordinate transformation 
\begin{equation}\label{eq_globaltildehatcoordtransf}
:\left(\xi, \eta, \tilde{\gamma}_2, \tilde{\Gamma}_2, \tilde{\psi}_1, \tilde{\Psi}_1, \tilde{\gamma}_3, \tilde{\Gamma}_3 \right) \mapsto \left(\xi, \eta, \hat{\gamma}_2, \hat{\Gamma}_2, \hat{\psi}_1, \hat{\Psi}_1, \hat{\gamma}_3, \hat{\Gamma}_3 \right)
\end{equation}
where $\left(\hat{\gamma}_2, \hat{\Gamma}_2, \hat{\psi}_1, \hat{\Psi}_1, \hat{\gamma}_3, \hat{\Gamma}_3 \right)$ are the coordinates on $\Lambda$ constructed in Section \ref{sec:innerdyn}. Define the subset $\E \subset \mathbb{R}$ by $\E = \left\{ F_{\mathrm{sec}} (z) : z \in \cD \right\}$. The following theorem is the main result of this section. 

\begin{theorem}\label{theorem_reductiontomap}
Fix some energy level $\{ F_{\mathrm{sec}} = E_0 \}$ of the secular Hamiltonian where $E_0 \in \E$, and consider the Poincar\'e section $M = \cD \cap  \{ F_{\mathrm{sec}} = E_0 \} \cap \{ \hat{\gamma}_2 = 0 \}$. 
\begin{enumerate}
\item
There is a well-defined Poincar\'e map $F:M \to M$, and the set $\hat{\Lambda} = \Lambda \cap M$ is a normally hyperbolic invariant manifold for $F$.
\item
The variables $\left( \hat{\psi}_1, \hat{\gamma}_3, \hat{\Psi}_1, \hat{\Gamma}_3 \right)$ define coordinates on $\hat{\Lambda}$, and the inner map $f= \left. F \right|_{\hat{\Lambda}}$ has the form
\begin{equation}\label{eq_innerpoincaremap0}
f:
\begin{dcases}
\left( \hat{\psi}_1^*, \hat{\gamma}_3^* \right) =& \left( \hat{\psi}_1, \hat{\gamma}_3 \right) + g \left( \hat{\Psi}_1, \hat{\Gamma}_3 \right) + O \left( \epsilon^{k_1 - 6} \mu^{k_2}\right) \\
\left( \hat{\Psi}_1^*, \hat{\Gamma}_3^* \right) =& \left( \hat{\Psi}_1, \hat{\Gamma}_3 \right) + O \left( \epsilon^{k_1 - 6} \mu^{k_2} \right)
\end{dcases}
\end{equation}
where $\epsilon=\frac{1}{L_2}, \, \mu = \frac{L_2}{L_3}$, where $k_1, k_2 \in \mathbb{N}$ come from part 2 of Theorem \ref{theorem_innerdynamics}, and where
\begin{equation}
\det D g \left( \hat{\Psi}_1, \hat{\Gamma}_3 \right) \neq 0.
\end{equation}
Moreover the bottom eigenvalue of $D g \left( \hat{\Psi}_1, \hat{\Gamma}_3 \right)$ is of order $\frac{L_2^8}{L_3^6}$.
\item
There are two scattering maps 
\begin{equation}
\hat{S}_{\pm} : \hat{\Lambda} \longrightarrow \hat{\Lambda}'
\end{equation}
where $\hat{\Lambda}'$ is an open cylinder in $\mathbb{T}^2 \times \mathbb{R}^2$ containing $\hat \Lambda$. Moreover the actions $\left(\hat{\Psi}_1^*, \hat{\Gamma}_3^* \right)$ of the image of a point $\left( \hat{\psi}_1, \hat{\gamma}_3, \hat{\Psi}_1, \hat{\Gamma}_3 \right) \in \hat \Lambda$ under the scattering maps $\hat{S}_{\pm}$ is given by
\begin{equation}
\hat{\Psi}_1^* = \hat{\Psi}_1 + \frac{L_2^9}{L_3^6} \S_1^{\pm} \left( \hat{\psi}_1, \hat{\Gamma}_2 \right) + \cdots, \quad \hat{\Gamma}_3^* = \hat{\Gamma}_3 + \frac{L_2^{11}}{L_3^8} \S^{\pm}_3 \left( \hat{\psi}_1, \hat{\gamma}_3, \hat{\Gamma}_2 \right) + \cdots
\end{equation}
where $\S_1^{\pm}$ is defined by \eqref{eq_scatteringpsi1} and $\S_3^{\pm}$ is defined by \eqref{eq_scatteringgamma3}. 
\end{enumerate}
\end{theorem}

We divide the proof of Theorem \ref{theorem_reductiontomap} into three lemmas. The following lemma establishes the existence of coordinates in a neighbourhood of $\Lambda$ involving the `hat' variables, and the existence of a Poincar\'e map. 

\begin{lemma}
The variables $\left(\hat{\xi}, \hat{\eta}, \hat{\psi}_1, \hat{\Psi}_1, \hat{\gamma}_3, \hat{\Gamma}_3 \right)$ define coordinates on the section $M = \cD \cap \{ F_{\mathrm{sec}} = E_0 \} \cap \{ \hat{\gamma}_2 = 0 \}$ for $E_0 \in \E$, and there is a well-defined Poincar\'e map $F:M \to M$. Furthermore the set $\hat{\Lambda} = \Lambda \cap M$ is a normally hyperbolic invariant manifold for $F$.
\end{lemma}

\begin{proof}
Denote by $\hat \Phi$ the coordinate transformation \eqref{eq_globaltildehatcoordtransf}. The transformation $\hat \Phi$ may not be symplectic, but the flow possesses an integral $\hat{H} = L_2^6 \, F_{\mathrm{sec}} \circ \hat{\Phi}$ nonetheless. Fix some level set $\{ \hat{H} = E_0 \}$ of the integral. Since $\hat{\Phi}$ is $O \left( \frac{L_2^{11}}{L_3^6} \right)$ close to the identity, we have (differentiating \eqref{eq_H012def} with respect to $\tilde{\Gamma}_2$ and effecting the coordinate transformation)
\begin{equation}\label{eq_globalhatintegralderivative}
\frac{\partial \hat{H}}{\partial \hat{\Gamma}_2} = \hat{\Gamma}_2 \, \left[ \frac{2}{L_1^2} + \frac{10}{\Gamma_1^2} \left( 1 - \frac{\Gamma_1^2}{L_1^2} \right) \, \sin^2 \gamma_1 \right] + O \left( \frac{L_2^{11}}{L_3^6} \right). 
\end{equation}
Observe that this is nonzero as long as $\hat{\Gamma}_2  \neq 0$; indeed the term $1 - \frac{\Gamma_1^2}{L_1^2}$ is $e_1^2$ and thus belongs to $(0,1)$, and so the term inside the square brackets is strictly positive. Therefore, by the implicit function theorem, we may express $\hat{\Gamma}_2$ as a function of the other variables in the energy level $\{ \hat{H} = E_0 \}$. Moreover, notice that $\frac{d \hat{\gamma}_2}{dt}$ is equal, up to higher order terms, to \eqref{eq_globalhatintegralderivative}. Since the time derivative of $\hat{\Gamma}_2$ is of higher order, it follows that the return map $F$ to the section $M = \cD \cap \{ F_{\mathrm{sec}} = E_0 \} \cap \{ \hat{\gamma}_2 = 0 \}$ is well-defined as long as $\hat{\Gamma}_2 \neq 0$, which is true on $\cD$. 

Finally, the fact that $\hat{\Lambda} = \Lambda \cap M$ is a normally hyperbolic invariant manifold for the return map $F$ follows immediately from the fact that $\Lambda$ is a normally hyperbolic invariant manifold for the flow. 
\end{proof}

\begin{lemma}\label{lemma_twistcondition}
The inner map $f= \left. F \right|_{\hat{\Lambda}}$ has the form
\begin{equation}\label{eq_innerpoincaremap}
f:
\begin{dcases}
\left( \hat{\psi}_1^*, \hat{\gamma}_3^* \right) =& \left( \hat{\psi}_1, \hat{\gamma}_3 \right) + g \left( \hat{\Psi}_1, \hat{\Gamma}_3 \right) + O \left( \epsilon^{k_1 - 6} \mu^{k_2} \right) \\
\left( \hat{\Psi}_1^*, \hat{\Gamma}_3^* \right) =& \left( \hat{\Psi}_1, \hat{\Gamma}_3 \right) + O \left( \epsilon^{k_1 - 6} \mu^{k_2} \right)
\end{dcases}
\end{equation}
where $\epsilon=\frac{1}{L_2}, \, \mu = \frac{L_2}{L_3}$, where $k_1, k_2 \in \mathbb{N}$ come from part 2 of Theorem \ref{theorem_innerdynamics}, and where
\begin{equation}\label{eq_twistproperty}
\det D g \left( \hat{\Psi}_1, \hat{\Gamma}_3 \right) \neq 0.
\end{equation}
Moreover the bottom eigenvalue of $D g \left( \hat{\Psi}_1, \hat{\Gamma}_3 \right)$ is of order $\frac{\mu^6}{\epsilon^2}$.
\end{lemma}
\begin{proof}
Consider the normalised inner Hamiltonian $L_2^6 \hat{F}$ where $\hat{F}$ is defined by \eqref{eq_innerhamiltonianave}. By part 2 of Theorem \ref{theorem_innerdynamics}, the first order term of this Hamiltonian is $c_0 \, \hat{\Gamma}_2^2$, and so the rate of change of $\hat{\gamma}_2$ is $2 \, c_0 \, \hat{\Gamma}_2 + O \left( \epsilon \right)$. Since this is of order 1, the return time of the flow associated with the Hamiltonian function $L_2^6 \hat{F}$ to the Poincar\'e section $\{ \hat{\gamma}_2 = 0\}$ is itself of order 1. Due to equation \eqref{eq_innerhamiltonianave}, the angles $\hat{\psi}_1, \hat{\gamma}_3$ do not appear in the Hamiltonian $L_2^6 \hat{F}$ before the terms of order $\epsilon^{k_1 - 6} \mu^{k_2}$, and so the return map $f$ has the form \eqref{eq_innerpoincaremap}. It remains to prove the formula \eqref{eq_twistproperty}, and to determine the order of the bottom eigenvalue of $Dg$.

Write $P_0 = \hat{\Gamma}_2$, $P_1 = \hat{\Psi}_1$, $P_2 = \hat{\Gamma}_3$, and $P = (P_1,P_2)$. Write
\begin{equation}
\omega_i \left( P_0, P \right) = \frac{\partial \hat{F}_0}{\partial P_i}
\end{equation}
for $i=0,1,2$, where $\hat F_0$ is defined in \eqref{eq_innerhamiltonianave}. By the implicit function theorem, in the energy level $\{ \hat{F}_0 = E_0 \}$ where $E_0 \in \E$, we can write $P_0 = \alpha (P)$, and the derivatives of $\alpha$ are given by
\begin{equation}
\frac{\partial \alpha}{\partial P_j} (P) = - \frac{\hat{\omega}_j (P)}{\hat{\omega}_0 (P)}
\end{equation}
where we have defined $\hat{\omega}_i (P) = \omega_i (\alpha(P),P)$. Using this notation we have $g_i(P) = \hat{\omega}_0(P)^{-1} \hat{\omega}_i(P)$. In the following computation, for convenience, we suppress the dependence of functions and their derivatives on $P$ and $P_0 = \alpha (P)$. 
\begin{align}
\hat{\omega}_0^3 \left( D g \right)_{ij} ={}& \hat{\omega}_0^3 \frac{\partial g_i}{\partial P_j} = \hat{\omega}_0^2 \left( \frac{\partial \omega_i}{\partial P_0} \frac{\partial \alpha}{\partial P_j} + \frac{\partial \omega_i}{\partial P_j} \right) - \hat{\omega}_0 \, \hat{\omega}_i \left( \frac{\partial \omega_0}{\partial P_0} \frac{\partial \alpha}{\partial P_j} + \frac{\partial \omega_0}{\partial P_j} \right) \\
={}& - \hat{\omega}_0 \, \hat{\omega}_j \frac{\partial \omega_i}{\partial P_0} + \hat{\omega}_0^2 \frac{\partial \omega_i}{\partial P_j} + \hat{\omega}_i \, \hat{\omega}_j \frac{\partial \omega_0}{\partial P_0} - \hat{\omega}_0 \, \hat{\omega}_i \frac{\partial \omega_0}{\partial P_j} \\
={}& - \frac{\partial \hat{F}_0}{\partial P_0} \frac{\partial \hat{F}_0}{\partial P_j} \frac{\partial^2 \hat{F}_0}{\partial P_0 \partial P_i} + \left( \frac{\partial \hat{F}_0}{\partial P_0} \right)^2 \frac{\partial^2 \hat{F}_0}{\partial P_j \partial P_i} + \frac{\partial \hat{F}_0}{\partial P_i} \frac{\partial \hat{F}_0}{\partial P_j} \frac{\partial^2 \hat{F}_0}{\partial P_0^2} - \frac{\partial \hat{F}_0}{\partial P_0} \frac{\partial \hat{F}_0}{\partial P_i} \frac{\partial^2 \hat{F}_0}{\partial P_j \partial P_0}. 
\end{align}
Combining this formula with the derivatives given in Lemma \ref{lemma_innerhamderivatives}, we see that
\begin{align}
\hat{\omega}_0^3 \left( D g \right)_{11} ={}& \epsilon^{20} \, C_{12}^3\, {{54\, \left(L_{1}^2-3\,\hat{\Gamma}_{2}^2\right)\,\left(L_{1}^2+\hat{\Gamma}_{2} ^2\right)}\over{L_{1}^6\,\delta_{1}^{11}}}+ \cdots\\
\hat{\omega}_0^3 \left( D g \right)_{12} ={}& - \epsilon^{
 16}\,\mu^6 \, C_{12}^2\,C_{23} \, {{72 \,\left(3\,L_{1}^2\,\delta_{1}^2+9\, \hat{\Gamma}_{2}^2\,
 \delta_{1}^2-5\,L_{1}^2+5\, \hat{\Gamma}_{2}^2\right)\,\delta_{3}}\over{L_{1}^4\,\delta_{1}^9\,\delta_{2}^3}} + \cdots \\
\hat{\omega}_0^3 \left( D g \right)_{21} ={}& - \epsilon^{
 16}\,\mu^6 \, C_{12}^2\,C_{23} \, {{72 \,\left(3\,L_{1}^2\,\delta_{1}^2+9\, \hat{\Gamma}_{2}^2\,
 \delta_{1}^2-5\,L_{1}^2+5\, \hat{\Gamma}_{2}^2\right)\,\delta_{3}}\over{L_{1}^4\,\delta_{1}^9\,\delta_{2}^3}}  + \cdots \\
\hat{\omega}_0^3 \left( D g \right)_{22} ={}&  - \epsilon^{16}\,\mu^6 \, C_{12}^2\,C_{23}\, {{36\, \hat{\Gamma}_{2}^2\,\left(12\,\delta_{1}^2-20\right)}\over{L_{1}^4\,\delta_{1}^8\,\delta_{2}^3}} + \cdots
\end{align}
As a result of \eqref{eq_assumption1} we have $\epsilon^{20} \gg \epsilon^{16} \mu^6$. Since $O \left( \hat{\omega}_0^3 \right) = \epsilon^{18}$, the top eigenvalue of $Dg$ is of order $O \left( (Dg)_{11} \right) = \epsilon^2$. Moreover, the determinant is equal, up to higher order terms, to the product of the entries on the diagonal of $Dg$; therefore it is of order $\mu^6$, and it is nonzero as long as 
\begin{equation}
\hat{\Gamma}_{2}^2 \left(L_{1}^2-3\,\hat{\Gamma}_{2}^2\right)\,\left(L_{1}^2+\hat{\Gamma}_{2} ^2\right) \left(12\,\delta_{1}^2-20\right)  \neq 0. 
\end{equation} 
This is true whenever $\hat \Gamma_2 \in (\zeta_1, \zeta_2)$ where $\zeta_1, \zeta_2$ satisfy condition \eqref{eq_isoenernondeg} is satisfied since the second last factor is strictly positive, and since $\delta_1^2 \in (0,1)$ implies that the last factor never vanishes. The fact that the bottom eigenvalue is of order $\epsilon^{-2} \mu^6$ follows from the fact that the top eigenvalue is of order $\epsilon^2$ and the determinant is of order $\mu^6$. 
\end{proof}

\begin{remark}
The fact that condition \eqref{eq_isoenernondeg} implies the twist property \eqref{eq_twistproperty} can also be obtained via Arnold's \emph{isoenergetic nondegeneracy} condition; that is, the condition that the so-called bordered Hessian of $\hat{F}_0$ (where $\hat{F}_0$ is defined by \eqref{eq_innerhamiltonianave}) has nonzero determinant (see Appendix 8 of \cite{arnol1978mathematical} and the references therein). The reason we have computed explicitly the derivative of the frequency vector $g$ is that we need to know the order of the eigenvalues of $Dg$ in order to obtain estimates on the diffusion time; the bordered Hessian and its determinant do not readily provide this data. 
\end{remark}

\begin{lemma}
There are two scattering maps 
\begin{equation}
\hat{S}_{\pm} : \hat{\Lambda} \longrightarrow \hat{\Lambda}'
\end{equation}
where $\hat{\Lambda}'$ is an open cylinder in $\mathbb{T}^2 \times \mathbb{R}^2$ containing $\hat \Lambda$. Moreover the actions $\hat{\Psi}_1^*, \hat{\Gamma}_3^* $ of the image of a point $\left( \hat{\psi}_1, \hat{\gamma}_3, \hat{\Psi}_1, \hat{\Gamma}_3 \right) \in \Lambda$ under the scattering maps $\hat{S}_{\pm}$ are given by
\begin{equation}
\hat{\Psi}_1^* = \hat{\Psi}_1 + \frac{L_2^9}{L_3^6} \S_1^{\pm} \left( \hat{\psi}_1, \hat{\Gamma}_2 \right) + \cdots, \quad \hat{\Gamma}_3^* = \hat{\Gamma}_3 + \frac{L_2^{11}}{L_3^8} \S^{\pm}_3 \left( \hat{\psi}_1, \hat{\gamma}_3, \hat{\Gamma}_2 \right) + \cdots
\end{equation}
where $\hat{\Gamma}_2$ is a function of the coordinates $\left( \hat{\psi}_1, \hat{\gamma}_3, \hat{\Psi}_1, \hat{\Gamma}_3 \right)$ and the energy $E_0 \in \E$ on the cylinder $\hat{\Lambda}$, where $\S_1^{\pm}$ is defined by \eqref{eq_scatteringpsi1} and $\S_3^{\pm}$ is defined by \eqref{eq_scatteringgamma3}. 
\end{lemma}

\begin{proof}
By Lemma \ref{lemma_scatteringtildecoords}, there are two homoclinic channels $\Gamma_{\pm}$ corresponding to the normally hyperbolic invariant manifold $\Lambda$, which give rise to two globally defined scattering maps $S_{\pm} : \Lambda \to \Lambda'$ for some cylinder $\Lambda'$ such that $\Lambda \subset \Lambda' \subset \mathbb{T}^3 \times \mathbb{R}^3$. It follows that the sets $\hat{\Gamma}_{\pm} = \Gamma_{\pm} \cap \{ \hat{\gamma}_2 = 0 \} \cap \{ F_{\mathrm{sec}} = E_0 \}$ are homoclinic channels for the return map $F$ to the section $\{ \hat{\gamma}_2 = 0 \} \cap \{ F_{\mathrm{sec}} = E_0 \}$. Let $\hat{\Lambda}' = \Lambda' \cap \{ \hat{\gamma}_2 = 0 \} \cap \{ F_{\mathrm{sec}} = E_0 \}$, and denote by $\hat{S}_{\pm} : \hat{\Lambda} \to \hat{\Lambda}'$ the scattering maps corresponding to the homoclinic channels $\hat{\Gamma}_{\pm}$. Denote by $\hat{\phi}^t$ the time-$t$ map of the flow of the Hamiltonian function $L_2^6 \, \hat{F}$ where the inner Hamiltonian $\hat{F}$ is defined by \eqref{eq_innerhamiltonianave}. Proposition 3 of \cite{delshams2013transition} implies that there are smooth functions $\tau_{\pm} : \hat{\Lambda} \to \mathbb{R}$ such that 
\begin{equation}\label{eq_returnscatteringwrtnonreturn}
\hat{S}_{\pm} \left( \hat{z} \right) = \hat{\phi}^{\tau_{\pm} ( \hat{z} )} \circ S_{\pm}  \left( \hat{z} \right)
\end{equation}
and $\tau_{\pm} \left( \hat{z} \right) = O(1)$ for all $\hat{z} \in \hat{\Lambda}$. Indeed, the fact that $\tau_{\pm} ( \hat{z} ) = O(1)$ is due to the functions $\tau_{\pm}$ being constructed as return times with respect to the Hamiltonian $L_2^6 \, \hat{F}$ to the section $\{ \hat{\gamma}_2 = 0 \}$; since the time derivative of $\hat{\gamma}_2$ with respect to $L_2^6 \, \hat{F}$ is of order 1, so too is the return time.

Now, choose any point $\hat{z} = \left( \hat{\psi}_1, \hat{\gamma}_3, \hat{\Psi}_1, \hat{\Gamma}_3 \right) \in \hat{\Lambda}$. This gives rise to a point $z = \left( \hat{\gamma}_2, \hat{\psi}_1, \hat{\gamma}_3, \hat{\Gamma}_2, \hat{\Psi}_1, \hat{\Gamma}_3 \right)$ in $\Lambda$ where $\hat{\gamma}_2 = 0$ and $\hat{\Gamma}_2$ is a function of $\hat{\psi}_1, \hat{\gamma}_3, \hat{\Psi}_1, \hat{\Gamma}_3$ and the energy $E_0$. Write $\bar{z} = S_{\pm} \left( z \right)$. Then the $\bar{\Psi}_1, \, \bar{\Gamma}_3$ components of $\bar{z}$ are 
\begin{equation}\label{eq_returnerrorsscattering0}
\bar{\Psi}_1 = \hat{\Psi}_1 + \frac{L_2^9}{L_3^6} \, \S_1^{\pm} \left( \hat{\psi}_1, \hat{\Gamma}_2 \right) + \cdots, \quad \bar{\Gamma}_3 = \hat{\Gamma}_3 + \frac{L_2^9}{L_3^6} \, \S_3^{\pm} \left( \hat{\psi}_1, \hat{\gamma}_3, \hat{\Gamma}_2 \right) + \cdots
\end{equation}
Write $\hat z^* = \hat S_{\pm} \left( \hat z \right) = \hat \phi^{\tau_{\pm} \left( \hat z \right)} \left( \bar z \right)$. Since $\tau_{\pm} ( \hat{z} )$ is of order 1, part 2 of Theorem \ref{theorem_innerdynamics} implies that the $\hat{\Psi}_1^*, \hat{\Gamma}_3^*$ components of $\hat z^*$ satisfy
\begin{equation}\label{eq_returnerrorsscattering}
 \left( \hat{\Psi}_1^*, \hat{\Gamma}_3^* \right) = \left( \bar{\Psi}_1, \bar{\Gamma}_3 \right) + O \left( \epsilon^{k_1 - 6} \mu^{k_2} \right)
\end{equation}
where $k_1, k_2$ are assumed to be sufficiently large. Combining \eqref{eq_returnerrorsscattering0} and \eqref{eq_returnerrorsscattering} completes the proof of the lemma. 

\end{proof}

\subsection{Constructing transition chains of almost invariant tori}\label{sec_transitionchainconstruction}

Observe that the normally hyperbolic invariant manifold $\hat{\Lambda}$ constructed in Theorem \ref{theorem_reductiontomap} has a foliation by leaves
\begin{equation}\label{eq_foliationoflambdahat}
\L \left( \hat{\Psi}_1^*, \hat{\Gamma}_3^* \right) = \left\{ \left( \hat{\psi}_1, \hat{\gamma}_3, \hat{\Psi}_1, \hat{\Gamma}_3 \right) \in \hat{\Lambda} : \hat{\Psi}_1 = \hat{\Psi}_1^*, \, \hat{\Gamma}_3 = \hat{\Gamma}_3^* \right\},
\end{equation}
each of which is almost invariant by the inner map $f$, since it is of the form \eqref{eq_innerpoincaremap0}. The following result allows us to construct sequences of leaves $\{\L_j \}$ of the foliation connected by the scattering maps $\hat{S}_{\pm}$ (i.e. \emph{transition chains}) such that the image of each leaf $\L_j$ under one of the scattering maps $\hat{S}_{\beta_j}$ (where $\beta_j \in \{ +, - \}$) is mapped transversely across $\L_{j+1}$ in the following sense: there is a $z \in \L_j$ such that $\hat{S}_{\beta_j}(z)  \in \L_{j+1}$, and
\begin{equation}\label{eq_transversalityoffoliations}
T_{\hat{S}_{\beta_j}(z)} \Lambda = T_{\hat{S}_{\beta_j}(z)} \left( \hat{S}_{\beta_j}(\L_j) \right)  \oplus T_{\hat{S}_{\beta_j}(z)} \L_{j+1}. 
\end{equation}
This, combined with the results of Section \ref{subsec_poincaremap}, will allow us to apply shadowing results from \cite{clarke2022topological}. Recall that the first order terms $\S_1^{\pm}$, $\S_3^{\pm}$ in the expansion of the images $\bar{\Psi}_1$, $\bar{\Gamma}_3$ of a point $\left( \hat{\psi}_1, \hat{\gamma}_3, \hat{\Psi}_1, \hat{\Gamma}_3 \right)$ under the scattering maps $\hat{S}_{\pm} : \hat{\Lambda} \to \hat{\Lambda}'$ are given by \eqref{eq_scatteringpsi1}, \eqref{eq_scatteringgamma3} respectively due to Theorem \ref{theorem_reductiontomap}. 

\begin{lemma}\label{lemma_scatteringmapjumps}
Consider the set $\hat{\Lambda} = \Lambda \cap M$ defined in Theorem \ref{theorem_reductiontomap}. Then there are constants $\nu^{\pm}, \hat{\nu}^{\pm}, \xi^{\pm}, \hat{\xi}^{\pm}>0$ and $C^{\pm} >0$ such that for $L_2$ large enough (see \eqref{eq_semimajorassumptionsinlvariable}) and any leaf $\L^* = \L \left( \hat{\Psi}_1^*, \hat{\Gamma}_3^* \right)$ of the foliation of $\hat{\Lambda}$ there are open sets $U_j^{\pm} \subset \L^* \simeq \T^2$ for $j=1,2,3,4$ such that $\mu \left( U_j^{\pm} \right) > C^{\pm}$ where $\mu$ is the Lebesgue measure on $\T^2$, and:
\begin{enumerate}
\item
For all $z \in U_1^{\pm}$ we have $\S_1^{\pm} (z) > \nu^{\pm}$,  $\left| \partial_{\hat{\psi}_1} \S_1^{\pm} (z) \right| > \hat{\nu}^{\pm}$,  $\S_3^{\pm} (z) > \xi^{\pm}$,  $\left| \partial_{\hat{\gamma}_3} \S_3^{\pm} (z) \right| > \hat{\xi}^{\pm}.$
\item
For all $z \in U_2^{\pm}$ we have $\S_1^{\pm} (z) < - \nu^{\pm}$,  $\left| \partial_{\hat{\psi}_1} \S_1^{\pm} (z) \right| > \hat{\nu}^{\pm}$,  $\S_3^{\pm} (z) < - \xi^{\pm}$,  $\left| \partial_{\hat{\gamma}_3} \S_3^{\pm} (z) \right| > \hat{\xi}^{\pm}$.
\item
For all $z \in U_3^{\pm}$ we have $\S_1^{\pm} (z) > \nu^{\pm}$,  $\left| \partial_{\hat{\psi}_1} \S_1^{\pm} (z) \right| > \hat{\nu}^{\pm}$,  $\S_3^{\pm} (z) < - \xi^{\pm}$,  $\left| \partial_{\hat{\gamma}_3} \S_3^{\pm} (z) \right| > \hat{\xi}^{\pm}$.
\item
For all $z \in U_4^{\pm}$ we have $\S_1^{\pm} (z) < - \nu^{\pm}$,  $\left| \partial_{\hat{\psi}_1} \S_1^{\pm} (z) \right| > \hat{\nu}^{\pm}$,  $\S_3^{\pm} (z) > \xi^{\pm}$,  $\left| \partial_{\hat{\gamma}_3} \S_3^{\pm} (z) \right| > \hat{\xi}^{\pm}$.
\end{enumerate}
\end{lemma}

\begin{proof}
We prove part 1 of the lemma for the scattering map $\hat{S}_+$, as the cases for parts 2-4 of the lemma and for $\hat{S}_-$ are analogous. Since $\hat{\Gamma}_2 \in (\zeta_1, \zeta_2)$ (see \eqref{eq_isoenernondeg}), the coefficient of each term in \eqref{eq_scatteringpsi1} is bounded away from zero. It follows that $\S_1^+$ is a nonconstant trigonometric polynomial with zero average, and with two different harmonics; therefore we have $\partial_{\hat{\psi}_1} \S_1^+  =0$ only on a set of measure 0 in each fixed $\L^*$. Moreover, this implies that there is an open set $V_1^+ \subset \T$ of positive measure in $\T$ and constants $\nu^+,  \hat{\nu}^+>0$ such that, for each $\hat{\psi}_1 \in V_1^+$, we have $\S_1^+ \left( \hat{\psi}_1, \hat{\Gamma}_2 \right) > \nu^+$ and $\left| \partial_{\hat{\psi}_1} \S_1^+ \left( \hat{\psi}_1, \hat{\Gamma}_2 \right) \right| > \hat{\nu}^+$. 

Since $\hat{\Gamma}_2 \in (\zeta_1, \zeta_2)$ (see \eqref{eq_isoenernondeg}), the coefficients in front of the two expressions in square brackets in \eqref{eq_scatteringgamma3} are bounded away from zero. Recall the quantities $\nu_j$ for $j=0,1,2,3$ appearing in \eqref{eq_scatteringgamma3} are defined by \eqref{eq_constantsmuline1} and \eqref{eq_constantsmuline2}. Observe that $\nu_j$ cannot all be zero simultaneously. Indeed, $\nu_j$ depends only on $\delta_1, \delta_3$, where $\delta_1 \in (0,1)$ and $\delta_3 >0$. Therefore $\nu_0=0$ if and only if $\delta_1 = \delta_3$, in which case $\nu_1 \neq 0$. Since the harmonics appearing in the second set of square brackets in \eqref{eq_scatteringgamma3} are different from those appearing in the first set of square brackets, it follows that $\S_3^+$ is a nonconstant trigonometric polynomial with zero average for any fixed value of $\hat \Gamma_2$.

Now, suppose  $\hat{\psi}_1 \in V_1^+$, and choose $\hat{\gamma}_3 \in \T$ such that $\S_3^+ \left( \hat{\psi}_1, \hat{\gamma}_3, \hat{\Gamma}_2 \right) > 0$. By slightly adjusting $\hat{\psi}_1, \hat{\gamma}_3$ if necessary, we can obtain $\hat{\psi}_1', \hat{\gamma}_3', \hat{\Gamma}_2'$ where $\hat{\psi}_1' \in V_1^+$, such that $\S_3^+ \left( \hat{\psi}_1', \hat{\gamma}_3', \hat{\Gamma}_2' \right) > 0$ and $ \left| \partial_{\hat{\gamma}_3} \S_3^+ \left( \hat{\psi}_1', \hat{\gamma}_3', \hat{\Gamma}_2' \right) \right| > 0$. Therefore we can find constants $\xi^+, \hat{\xi}^+ >0$ and an open set $U_1^+ \subset \T^2$ such that part 1 of the lemma holds. 
\end{proof}

It is clear that Lemma \ref{lemma_scatteringmapjumps} provides us with large open sets $U_j^{\pm}$ in each leaf $\L$ of the foliation of $\hat{\Lambda}$ such that the scattering map can jump a fixed distance either up or down in the $\hat{\Psi}_1, \hat{\Gamma}_3$ directions. In the next lemma we show that the conditions satisfied by the scattering maps in these sets $U_j^{\pm}$ imply transversality of the foliations in the sense of \eqref{eq_transversalityoffoliations}. 

\begin{lemma}\label{lemma_transversalityoffoliations}
Suppose $\hat{z} \in U_j^{\pm} \cap \L_0$ and $\hat{z}^* = \hat{S}_{\pm} \left( \hat{z} \right) \in \L_1$, where $U_j^{\pm}$ are the open sets found in Lemma \ref{lemma_scatteringmapjumps} and $\L_j$ are leaves of the foliation of $\hat{\Lambda}$. Then $\hat{S}_{\pm}$ maps $\L_0$ transversely across $\L_1$ at the point $\hat{z}^* = \hat{S}_{\pm} \left( \hat{z} \right)$ in the sense that
\begin{equation}
T_{\hat{z}^*} \hat{\Lambda} = T_{\hat{z}^*} \left( \hat{S}_{\pm} \left( \L_0 \right) \right) + T_{\hat{z}^*} \L_1. 
\end{equation}
\end{lemma}

\begin{proof}
We prove the lemma for $\hat{S}_+$, assuming $\hat{z} \in U_1^+$ and $\hat{z}^* = \hat{S}_+ (\hat{z})$. The set $T_{\hat{z}^*} \hat{\Lambda}$ is a 4-dimensional vector space, and its elements can be written in the form $v = (Q,P)$ where $Q \in \mathbb{R}^2$ is a tangent vector in the $\hat{\psi}_1, \hat{\gamma}_3$ directions, and $P \in \mathbb{R}^2$ is a tangent vector in the $\hat{\Psi}_1, \hat{\Gamma}_3$ directions. Since $\hat{S}_+$ is smooth we have $T_{\hat{z}^*} \left( \hat{S}_+ \left( \L_0 \right) \right) = D_{\hat{z}} \hat{S}_+ \left( T_{\hat{z}} \L_0 \right)$. With $\hat{z} = \left( \hat{\psi}_1, \hat{\gamma}_3, \hat{\Psi}_1, \hat{\Gamma}_3 \right)$ and $\hat{z}^* = \left( \hat{\psi}_1^*, \hat{\gamma}_3^*, \hat{\Psi}_1^*, \hat{\Gamma}_3^* \right)$, we can write
\begin{equation}
D_{\hat{z}} \hat{S}_+ = \left(
\begin{matrix}
A_1 & A_2 \\
A_3 & A_4
\end{matrix}
\right)
\end{equation}
where
\begin{equation}
A_1 = \frac{\partial \left( \hat{\psi}_1^*, \hat{\gamma}_3^* \right)}{\partial \left( \hat{\psi}_1, \hat{\gamma}_3 \right)}, \quad A_2 = \frac{\partial \left( \hat{\psi}_1^*, \hat{\gamma}_3^* \right)}{\partial \left( \hat{\Psi}_1, \hat{\Gamma}_3 \right)}, \quad A_3 = \frac{\partial \left( \hat{\Psi}_1^*, \hat{\Gamma}_3^* \right)}{\partial \left( \hat{\psi}_1, \hat{\gamma}_3 \right)}, \quad A_4 = \frac{\partial \left( \hat{\Psi}_1^*, \hat{\Gamma}_3^* \right)}{\partial \left( \hat{\Psi}_1, \hat{\Gamma}_3 \right)}. 
\end{equation}
Therefore
\begin{equation}
A_1 = \left(
\begin{matrix}
1 & 0 \\
0 & 1 \\
\end{matrix}
\right) + \cdots, \quad A_3 = \left(
\begin{matrix}
\frac{L_2^9}{L_3^6} \, \partial_{\hat{\psi}_1} \S_1^+ \left( \hat{z} \right) & 0 \\[6pt]
\frac{L_2^{11}}{L_3^8} \, \partial_{\hat{\psi}_1} \S_3^+ \left( \hat{z} \right) & \frac{L_2^{11}}{L_3^8} \, \partial_{\hat{\gamma}_3} \S_3^+ \left( \hat{z} \right)
\end{matrix}
\right)  + \cdots
\end{equation}
since $\S_1^+$ does not depend on $\hat{\gamma}_3$. 

Now, let $v^0 \in T_{\hat{z}} \L_0$ and $v^1 \in T_{\hat{z}^*} \L_1$. Due to the topology of the leaves $\L_j$ we have $v^j = (Q^j,0)$ for $j=0,1$. Therefore we have
\begin{equation}
\left(
\begin{matrix}
\bar{Q} \\
\bar{P}
\end{matrix}
\right) = D_{\hat{z}} \hat{S}_+ (v^0) + v^1 = \left(
\begin{matrix}
A_1 Q^0 + Q^1 \\
A_3 Q^0
\end{matrix}
\right) + \cdots
\end{equation}
Since $\hat{z} \in U_1^+$, the entries on the diagonal of $A_3$ are nontrivial by part 1 of Lemma \ref{lemma_scatteringmapjumps}, and therefore it is invertible. This completes the proof of the lemma. 
\end{proof}

\subsection{Application of the shadowing theorems}\label{section_applicshadowingresults}

In this section we show that the shadowing results of \cite{clarke2022topological} (which are summarised in Appendix \ref{appendix_shadowing} for convenience) can be applied to \eqref{eq_secularhamiltoniandef} and, consequently, to \eqref{eq_4bphamafterave}. In fact, we have already proved that the secular Hamiltonian $F_{\mathrm{sec}}$ satisfies the assumptions [A1-3] of Theorem \ref{theorem_main1}. Indeed, in Section \ref{subsec_poincaremap}, we constructed a Poincar\'e map $F$ to the section $\{ F_{\mathrm{sec}} = E_0 \} \cap \{ \hat{\gamma}_2 = 0 \}$; we showed that $F$ has a normally hyperbolic invariant cylinder $\hat{\Lambda} \simeq \mathbb{T}^2 \times [0,1]^2$, that the restriction $f = F |_{\hat{\Lambda}}$ is a near-integrable twist map in the sense of Definition \ref{def_nearlyintegrabletwist} (satisfying a non-uniform twist condition of order $\frac{L_2^8}{L_3^6}$), that the stable and unstable manifolds meet transversely (with an order of splitting equal to $\frac{1}{L_2^2}$), and that there are two scattering maps $\hat{S}_{\pm}: \hat{\Lambda} \to \hat{\Lambda}'$ (see Theorem \ref{theorem_reductiontomap}). This implies conditions [A1] and [A2]. In Section \ref{sec_transitionchainconstruction} we constructed a foliation of $\hat{\Lambda}$ by leaves of the form \eqref{eq_foliationoflambdahat}, each of which is almost invariant under the map $f$. In Lemma \ref{lemma_scatteringmapjumps}, we showed that we can make jumps of a fixed distance of order $\frac{L_2^{11}}{L_3^8}$ (either up or down) in each of the $\hat{\Psi}_1, \hat{\Gamma}_3$ directions using the scattering maps; choose some such sequence of leaves $\{ \L_j \}_{j \in \mathbb{N}}$ such that for each $j \in \mathbb{N}$ there is $\beta_j \in \{ +, - \}$ such that $\L_{j+1} \cap \hat{S}_{\beta_j} ( \L_j ) \neq \emptyset$. By Lemma \ref{lemma_transversalityoffoliations}, the scattering map $\hat{S}_{\beta_j}$ maps $\L_j$ transversely across $\L_{j+1}$, and the angle of transversality is of order $\frac{L_2^{11}}{L_3^8}$. Thus assumption [A3] of Theorem \ref{theorem_main1} is also satisfied. 

Recall in Section \ref{section_secularexpansion} we made the change of variables \eqref{eq_changeofcoordstilde} to the `tilde' variables, which were the basis for all further analysis in the paper. This change of variables, however, is local, whereas the drift in eccentricity and inclination described in Theorem \ref{thm:main} is global. In order to define these coordinates  we introduced constants $\delta_j$ for $j=1,2,3$. Here $\delta_2$ is the coefficient of total angular momentum, and is therefore fixed for the secular system. The constants $\delta_1, \delta_3$ on the other hand are allowed to vary, and by varying them we simply obtain a different system of `tilde' coordinates. It is not hard to see that the subsequent analysis of this paper holds equally for any value of the constants\footnote{More properly, one has to consider $\delta_1$, $\delta_3$ in closed intervals within the open intervals to be uniformly away from the singularities of the Deprit coordinates.} $\delta_1 \in (0,1)$ and $\delta_3 \in (-1,1)$. Denote by $\tilde{C}_{\delta_1, \delta_3}$ the system of `tilde' coordinates corresponding to the values $\delta_1, \delta_3$. Then the results of Section \ref{section_analysisofh0} apply in $\tilde{C}_{\delta_1, \delta_3}$ coordinates for each value of $\delta_1, \delta_3$, so we have a normally hyperbolic invariant manifold $\Lambda_{\delta_1, \delta_3}$ in each such system of coordinates. Moreover, since the cylinder depends smoothly on the parameters $\delta_1, \delta_3$, this construction allows us to obtain one large normally hyperbolic invariant cylinder $\Lambda^*$. 

Observe that the contents of Sections \ref{subsec_poincaremap} and \ref{sec_transitionchainconstruction} apply equally in each system of coordinates $\tilde{C}_{\delta_1, \delta_3}$. Furthermore, since the $\hat \gamma_2$ variable does not depend on $\delta_1, \delta_3$, the Poincar\'e section is global, and so we obtain a large 4-dimensional cylinder $\hat \Lambda^*$ for the return map to the Poincar\'e section. We now fix a global transition chain on the cylinder $\hat \Lambda^*$ such that the actions $\hat \Psi_1, \hat \Gamma_3$ drift by an amount of order $L_2$ along the chain, and choose some sequence $\{ \tilde{C}_{\delta_1^k, \delta_3^k} \}_{k =1, \ldots, K}$ so that we have an appropriate system of coordinates to apply the analysis of the earlier sections near each torus in the chain. The analysis of Sections \ref{subsec_poincaremap} and \ref{sec_transitionchainconstruction} applies in each coordinate system $\tilde{C}_{\delta_1^k, \delta_3^k}$. Note that the shadowing results of \cite{clarke2022topological} apply equally well using the many different coordinate systems, as the coordinates used in that paper are purely local. Thus the assumptions of Theorems \ref{theorem_main1} and \ref{theorem_main2} apply to the global transition chain on the cylinder $\hat \Lambda^*$.

Denote by $\{ \L_i \}$ the global transition chain. Then, by Theorem \ref{theorem_main1}, for any $\eta > 0$, all sufficiently large $L_2$, and all $L_3$ satisfying \eqref{eq_semimajorassumptionsinlvariable}, there exists a sequence $\{ z_j \}_{j \in \mathbb{N}}$ in the secular phase space and times $t_j > 0$ such that
\begin{equation}
z_{i+1} = \phi^{t_j}_{\mathrm{sec}} (z_i), \quad d(z_i, \L_i) < \eta
\end{equation}
for each $j \in \mathbb{N}$ where $\phi^t_{\mathrm{sec}}$ is the flow associated with the secular Hamiltonian (see \eqref{eq_secularhamexpansions}). Observe that this completes the proof of Proposition \ref{prop:SecularChains}. Moreover, the time to move a distance of order $L_2$ in the $\hat{\Psi}_1, \hat{\Gamma}_3$ directions (which is the time to move a distance of order 1 in the eccentricity $e_2$ and the inclination $\theta_{23}$) is of order 
\begin{equation}\label{eq_seculartimeestimate}
L_2 \, L_2^6 \, \frac{L_3^{16}}{L_2^{22}} \, \frac{L_3^6}{L_2^8} \, \frac{L_3^8}{L_2^{11}} = \frac{L_3^{30}}{L_2^{34}}.
\end{equation}
Indeed, this follows from the following facts: the distance we drift in the $\hat{\Psi}_1, \hat{\Gamma}_3$ variables is of order $L_2$; each $t_j$ obtained via Theorem \ref{theorem_main1} is of order $L_2^6$ as this is the order of the return time to the Poincar\'e section; the order of splitting of separatrices of the cylinder is $\frac{1}{L_2^2}$ in the $\hat{\Gamma}_3$ direction by Proposition \ref{proposition_melnikovtildeexp}; the twist property is of order $\frac{L_2^8}{L_3^6}$ by Theorem \ref{theorem_reductiontomap}; and the order of transversality of foliations is $\frac{L_2^{11}}{L_3^8}$ by the proof of Lemma \ref{lemma_transversalityoffoliations}. Combining these values with the formula \eqref{eq_timeestimate1} gives the time estimate \eqref{eq_seculartimeestimate}. 

Consider now the Hamiltonian $F_{\mathrm{4bp}}$ (see \eqref{eq_4bphamafterave}) of the four body problem after averaging the mean anomalies $\ell = (\ell_1, \ell_2, \ell_3) \in \mathbb{T}^3$. For $j=1,2,3$ fix some $L_j^{\pm} \in \mathbb{R}$ such that $0 < L_j^- < L_j^+$ and if $L = (L_1, L_2, L_3) \in [L_1^-, L_1^+] \times [L_2^-, L_2^+] \times [L_3^-, L_3^+]$ then \eqref{eq_semimajorassumptionsinlvariable} is satisfied. Write $\Sigma = \mathbb{T}^3 \times [L_1^-, L_1^+] \times [L_2^-, L_2^+] \times [L_3^-, L_3^+]$. Recall from the beginning of Section \ref{subsec_poincaremap} the definition of the domain $\cD$ in the secular phase space. Denote by $\E_{\mathrm{4bp}} = \left\{ F_{\mathrm{4bp}} (z, \ell, L) : z \in \cD, \, (\ell, L) \in \Sigma \right\}$ the set of values of the energy of the full four-body problem that we consider. Fix $E_1 \in \E_{\mathrm{4bp}}$, and denote by $\Psi$ the Poincar\'e map of the flow of $F_{\mathrm{4bp}}$ to the section $\cD \times \Sigma \cap \{ \hat{\gamma}_2 = 0 \} \cap \{F_{\mathrm{4bp}} = E_1 \}$. With $z$ denoting a point in $\cD$, we write $(\bar{z}, \bar{\ell}, \bar{L}) = \Psi (z, \ell, L)$ where $\bar{z} = G(z, \ell, L)$ and $(\bar{\ell}, \bar{L}) = \phi (z, \ell, L)$. Since the Hamiltonian $F_{\mathrm{4bp}}$ is obtained by averaging the mean anomalies, there are $\hat{k}_1, \hat{k}_2 \in \mathbb{N}$ such that the variables $\ell_j$ do not appear in $F_{\mathrm{4bp}}$ until terms of order $\epsilon^{\hat{k}_1} \, \mu^{\hat{k}_2}$ where $\epsilon = \frac{1}{L_2}$ and $\mu = \frac{L_2}{L_3}$, and where we can choose $\hat{k}_j$ to be as large as we like. It follows that the map $G$ can be written in the form $G(z, \ell, L) = \widetilde{G}(z ; L) + O \left(\epsilon^{\hat{k}_1 - 6} \, \mu^{\hat{k}_2} \right)$ where the higher-order terms are uniformly bounded in the $C^r$ topology (for any $r \in \mathbb{N}$), and where for any fixed values of $L$, the map $z \mapsto \widetilde{G}(z ; L)$ is a Poincar\'e map $F$ of the type constructed above in Theorem \ref{theorem_reductiontomap}, and therefore satisfies assumptions [A1-3] of Theorem \ref{theorem_main1}. This in turn implies that the map $\Psi$ satisfies assumption [B1] of Theorem \ref{theorem_main2}. Moreover, writing $\phi (z, \ell, L) = (\phi_1 (z, \ell, L), \phi_2 (z, \ell, L))$ such that $\bar{\ell} = \phi_1 (z, \ell, L)$ and $\bar{L} = \phi_2 (z, \ell, L)$, we have $\phi_2 (z, \ell, L) = L + O \left(\epsilon^{\hat{k}_1 - 6} \, \mu^{\hat{k}_2} \right)$ where the higher-order terms are uniformly bounded in the $C^r$ topology. Therefore assumption [B2] is also satisfied. As explained in Appendix \ref{appendix_shadowing}, as a consequence of results of \cite{delshams2008geometric}, the map $\Psi$ has a normally hyperbolic (locally) invariant manifold $\widetilde{\Lambda}$ that is close to $\hat{\Lambda} \times \Sigma$. We can thus use the coordinates $w, \ell, L$ on $\widetilde{\Lambda}$ where $w$ are the coordinates on $\hat{\Lambda}$, and define a foliation of $\widetilde{\Lambda}$ by leaves
\begin{equation}\label{eq_foliationoflambdatilde}
\tilde{\L} \left( \hat{\Psi}_1^*, \hat{\Gamma}_3^*, L^* \right) = \left\{ \left( w, \ell, L \right) \in \widetilde{\Lambda} : w \in \L \left( \hat{\Psi}_1^*, \hat{\Gamma}_3^* \right), \, L = L^* \right\}
\end{equation}
where $\L \left( \hat{\Psi}_1^*, \hat{\Gamma}_3^* \right)$ is the leaf of the foliation of $\hat{\Lambda}$ defined by \eqref{eq_foliationoflambdahat}. Fix $\eta > 0$, $K_1, K_2 \in \mathbb{N}$, and choose some initial values $L^1_* = (L_1^1, L_2^1, L_3^1)$ of the $L_j$ variables such that $(\ell, L_*^1) \in \mathrm{Int} \, \Sigma$ for any $\ell \in \mathbb{T}^3$. Choose $N \leq \epsilon^{-K_1} \, \mu^{-K_2}$, and values $P^1_*, \ldots, P^N_*$ of the actions $\hat{\Psi}_1, \hat{\Gamma}_3$ such that the leaves $\L_j = \L \left( P^j_* \right)$ of the foliation of $\hat{\Lambda}$ are connected by the scattering maps (corresponding to the secular Hamiltonian with $L = L^1_*$) in the sense of Lemma \ref{lemma_scatteringmapjumps}. Then by Theorem \ref{theorem_main2} there are $L^2_*, \ldots, L^N_* \in [L_1^-, L_1^+] \times [L_2^-, L_2^+] \times [L_3^-, L_3^+]$ such that, with $\tilde{\L}_j = \tilde{\L} (P^j_*, L^j_*)$, there are points $(z^1, \ell^1, L^1), \ldots, (z^N, \ell^N, L^N)$ in the phase space of the full four body problem and times $t_j^* >0$ such that 
\begin{equation}
(z^{i+1}, \ell^{i+1}, L^{i+1}) = \phi^{t_j^*}_{\mathrm{4bp}} (z^i, \ell^i, L^i), \quad d \left( (z^i,\ell^i,L^i), \tilde{\L}_i \right) < \eta
\end{equation}
where $\phi^{t_j}_{\mathrm{4bp}}$ is the flow of the Hamiltonian $F_{\mathrm{4bp}}$ of the full four body problem. Moreover, the time estimate \eqref{eq_seculartimeestimate} holds also in this case as a lower bound on the time to move a distance of order 1 in the inclination $\theta_{23}$ and the eccentricity $e_2$.

\section{Continuation in the planetary regime}
\label{sec:planetary}
We would like to discuss the robustness of our instability mechanism with respect to modifications of the four masses and the three semimajor axes. (There are other parameters, e.g. the eccentricities of the planets, but these do not play a similar role.) Due to the invariance of physical laws by change of mass and length units, without loss of generality we may set $m_0= a_1 = 1$ and are reduced to a 5 dimensional parameter space.

Up to now we have investigated the moderate hierarchical regime,
\begin{equation} \label{eq_assumption1bis}
O(1) = a_1 \ll a_2  \to \infty \qquad \text{and}\qquad
a_2^{\frac{11}{6}}  \ll a_3 \ll a_2^2
\end{equation} 
('moderate' meaning that while $a_3$ is much smaller than $a_2$, it is not allowed to be arbitrarily small). Weakening the hypothesis \eqref{eq_assumption1bis} and considering the general hierarchical regime $a_1 \ll a_2 \ll a_3$ would require another proof since the fast dynamics of planet 3 could be slower than the secular dynamics of planets 1 and 2, which would completely destroy the frequency hierarchy we have extensively used, for averaging and so on.

We have also assumed that $m_0 \neq m_1$ and $m_0+m_1\neq m_2$, because $m_1-m_0$ and $m_0+m_1- m_2$ are in factor of the two quadrupolar Hamiltonians respectively. Rather than fixing the masses, we may as well let  $(m_0,m_1,m_2,m_3)$ vary arbitrarily in some fixed compact subset of $K\subset M_\eta$ (see \eqref{eq:M})
and the conclusions our main theorems will hold, all our construction being uniform with respect to $K$.

\medskip Another important regime in the parameter space is the so-called planetary regime, where
\begin{equation}
  \label{eq:planetary}
  m_1, m_2,m_3 \to 0
\end{equation}
while, in turn, semimajor axes are fixed (or vary in some compact subset)~\cite{Fejoz:2015:nbp}.

\begin{proposition}
  The instability mechanism which we have shown to exist in the hierarchical regime in Sections~\ref{section_secularexpansion}-\ref{sec:mapreduction} continues in the planetary regime and, as $\rho$ tends to $0$, the instability time is of the order of $\rho^2$.
\end{proposition}

\begin{proof}
  Write
  \[m_j = \rho \bar m_j, \quad j=1,2,3,\]
  so that, when $\mu\ll 1$, 
  \[M_j \sim 1, \quad \sigma_{0,j} \sim 1, \quad \sigma_{ij} \sim \rho \quad \mbox{and} \quad \mu_j \sim \rho \quad (i,j=1,2,3)\]
  (notations defined in Section~\ref{sec:DepritResults}). Our proof in the hierarchical regime assumed $\rho$ constant, and we will now show why the construction holds when $\rho$ is small. 

  The key remark is the following. Consider for instance the part of the perturbing function $F_{\mathrm{per}}^{12}$ regarding planets $1$ and $2$ (first introduced in equation~\eqref{eq_perfn12}): 
  \[F_{\mathrm{per}}^{12} ={} \frac{\mu_2 M_2}{\| q_2 \|} - \frac{m_0 \, m_2}{\| q_2 + \sigma_{11} \, q_1 \|} - \frac{m_1 \, m_2}{\| q_2 - \sigma_{01} \, q_1 \|}.\]
  The first two terms are $O(\rho)$, while the third one (describing the interaction of planets $1$ and $2$) is $O(\rho^2)$; this could impair the construction of the hierarchical regime. But the third term is qualitatively so similar to the second one, that the relative smallness of $m_1m_2$ will not qualitatively change the dynamics. This can be seen explicitely in the expansion in Legendre polynomials: 
  \[F_{\mathrm{per}}^{12} = - \frac{\mu_1 m_2}{\| q_2 \|} \sum_{n=2}^{\infty} \tilde{\sigma}_{1,n} P_n (\cos \zeta_{1}) \left( \frac{\| q_1 \|}{\| q_2 \|} \right)^n\]
  where
  \[\tilde{\sigma}_{1,n} = \sigma_{01}^{n-1} + (-1)^n \sigma_{11}^{n-1} \sim 1,\]
  so that the hierarchy of terms inside the infinite sum is unchanged when we move away from the hierarchical regime to the planetary regime. Similar is the case of $F_{\mathrm{per}}^{23}$. 

  Because the secular frequencies are $\rho^2$-small compared to the Keplerian frequencies, the splitting is $O(\rho^2)$ in the secular directions (while it is exponentially small with respect to $\rho$ in the Keplerian directions).
\end{proof}

%%% Local Variables: 
%%% mode: latex
%%% TeX-master: "4bp_secular_diffusion_v6.tex"
%%% End: 

\section{From Deprit coordinates to elliptic elements: proof of Theorem \ref{thm:main}}
\label{sec:FromDepritToElements}
Theorem \ref{thm:main} and the behaviour described in \eqref{def:drifte2theta23}, \eqref{def:drifte1theta12}, \eqref{def:inclincation12relation}, \eqref{def:drifte3aj} are a direct consequence of Theorems \ref{thm:MainHierarch:Deprit} and \ref{thm:PlanetHierarch:Deprit}. 
 We first explain how to obtain the evolution of the orbital elements and then deduce Theorem \ref{thm:main}.

To obtain the evolution of $e_2$ in  \eqref{def:drifte2theta23}, one just has to take into account its definition in \eqref{def:eccentricity}. For  $\theta_{23}$, note that \eqref{eq_semimajorassumptionsinlvariable} implies $\theta_{23}=i_{23}+O(L_2^{-1})$ where $i_{23}$ is defined by \eqref{def:inclinationi23}. Then, the evolution of $\theta_{23}$ can be deduced from the evolution of the Deprit variables and \eqref{def:inclinationi23} (and taking $L_2$ large enough). The evolution of the node, even if not stated in Theorem \ref{thm:MainHierarch:Deprit} can be easily deduced from the shadowing argument explained in Section \ref{section_applicshadowingresults}, by iterating the Poincar\'e map along leaves of the almost invariant foliation of the normally hyperbolic cylinder.

The behaviour in \eqref{def:drifte1theta12}, \eqref{def:inclincation12relation} can be obtained as follows. First, note that the times  $\{t_k\}$ obtained in Theorem \ref{thm:MainHierarch:Deprit} are chosen so that the corresponding points are very close to the normally hyperbolic invariant cylinder. In particular, $\Gamma_1$ is very close to $L_1$. This implies that $e_1$ is very small. 

Moreover, note that the energy must be preserved. The Keplerian terms in the Hamiltonian are constant up to an arbitrarily high order. This implies that the next term in the Hamiltonian expansion (see Proposition \ref{proposition_secularexpansion}), that is the quadrupolar Hamiltonian must be preserved up to small errors. If one evaluates it (see \eqref{eq_quad12unscaled} in Lemma \ref{lemma_quadoct12comp}) at the normally hyperbolic invariant manifold, one obtains the expression
\[
 \frac{(1-e_1^2)\cos^2\theta_{12}}{(1-e_2^2)^{3/2}}+\ldots=\text{constant}.
\]
(note that for the first two planets $\theta_{12}$ coincides with $i_{12}$, see \eqref{def:inclinationi12}).
% is constant along the orbits that we consider since, roughly speaking, it corresponds to the (first order of the) energy of the quadrupolar Hamiltonian.
Since $e_1\sim 0$ on the cylinder,  the drift in $e_2$ implies a drift in the inclination $\theta_{12}$. This is precisely what is described in \eqref{def:drifte1theta12}, \eqref{def:inclincation12relation}.

The evolution of $e_3$ and $a_j$ in  \eqref{def:drifte3aj} is a direct consequence of the evolution of $\Gamma_3$ and $L_j$ in Theorems \ref{thm:MainHierarch:Deprit} and \ref{thm:PlanetHierarch:Deprit}.

Finally, note that \eqref{def:drifte2theta23}, \eqref{def:drifte1theta12}, \eqref{def:inclincation12relation}, \eqref{def:drifte3aj} imply Theorem \ref{thm:main} since the inclination, the node, and the eccentricity determine the normalised angular momentum vector. Note that in Theorem \ref{thm:main} we allow the orbit to shadow angular momenta corresponding to degeneracies of the Deprit variables. In order to do this, it is enough to choose $\delta$ and the ratio of the semimajor axes adequately. Clearly we can choose these ratios polynomially in $\delta$. Then, the time estimates in Theorems \ref{thm:MainHierarch:Deprit} and \ref{thm:PlanetHierarch:Deprit} imply the time estimates \eqref{def:timefinal} and \eqref{def:timefinal2}.

% In fact, assume that we start with $\Psi_1-\Gamma_2$ close to its highest value, that is 
% \[
%  \Psi_1-\Gamma_2\sim L_1\sqrt{\frac{3}{5}}
% \]
% Note that this can be rewritten as 
% \[
% ( \Psi_1-\Gamma_2)^2=L_1^2 (1-e_1^2)^2\cos^2\theta_{12}
% \]
% Then, at the begining, close to the saddle
% \[
%  \cos^2\theta_{12}\sim \frac{3}{5}
% \]
% whereas at the end it is given by the conservation of the expression above. In particular, if we choose the final $e_2$ very close to 1, we obtain
% \[
%  \cos^2\theta_{12}\sim 0
% \]
% that is $\theta_{12}\sim \frac{\pi}{2}$.
% 
% Note that in the separatrix excursion this osculating elements are oscillating. In particular the eccentricity behaves as follows. Note that we have
% \[
%  \frac{(1-e_1^2(t))\cos^2\theta_{12}(t)}{(1-(e_2(t))^2)^{3/2}}\sim  \frac{\cos^2\theta_{12}^0}{(1-(e_2^0)^2)^{3/2}}
% \]
% Then,
% \[
%  e_1\sim 0\qquad \text{to}\qquad e_1\sim \sqrt{1-\frac{5}{3}\frac{\cos^2\theta_{12}^0}{(1-(e_2^0)^2)^{3/2}}(1-e_2^2(t_k))^{3/2}}
% \]
% In particular, if we have chosen initially
% \[
%  \cos^2\theta_{12}\sim \frac{3}{5}\qquad \text{and}\qquad e_2^0\sim0
% \]
% we have
% \[
%  e_1\sim 0\qquad \text{to}\qquad e_1\sim \sqrt{1-(1-e_2^2(t_k))^{3}}
% \]
% and therefore $e_1$ has also a drastic change from close to $0$ to close to $1$

%%% Local Variables: 
%%% mode: latex
%%% TeX-master: "4bp_secular_diffusion_v6.tex"
%%% End: 

\appendix
\section{Deprit's coordinates}
\label{sec:app:Deprit}
This appendix is a reminder on the Deprit coordinates. These coordinates are well suited for the reduction by $4$ dimensions due to the symmetry of rotations. Here we show a direct way to check that they form a symplectic coordinate system, alternative to some other presentations~\cite{chierchia2011deprit, deprit1983}. We restrict to $3$ planets, as needed in this article, but the argument extends to any number of planets (it is not an induction on the number of bodies). 

Denote by 
\begin{equation}
C_j = q_j \times p_j
\end{equation} 
the angular momentum of the $j^{\mathrm{th}}$ fictitious Keplerian body (Keplerian refering to $F_{\mathrm{Kep}}$), and let $k_j$ be the $j^{\mathrm{th}}$ element of the standard orthonormal basis of $\mathbb{R}^3$. Define the nodes $\nu_j$ by
\begin{equation}
\nu_1=\nu_2=C_1 \times C_2, \quad \nu_3 = (C_1 + C_2) \times C_3, \quad \nu_4=k_3 \times C
\end{equation}
where
\begin{equation}
C = C_1 + C_2 + C_3
\end{equation}
is the total angular momentum vector. For a non-zero vector $z \in \mathbb{R}^3$ and two non-zero vectors $u,v$ lying in the plane orthogonal to $z$, denote by $\alpha_z (u,v)$ the oriented angle between $u,v$, with orientation defined by the right hand rule with respect to $z$. Denote by $\Pi_j$ the pericenter of $q_j$ on its Keplerian ellipse.
% \footnote{Despite the notation, there is no risk of confusion with Legendre polynomials.} 
The Deprit variables $(\ell_j,L_j,\gamma_j,\Gamma_j,\psi_j,\Psi_j)_{j=1,2,3}$ are defined as follows:
\begin{itemize}
\item $\ell_j$ is the mean anomaly of $q_j$ on its Keplerian ellipse;
\item
$L_j = \mu_j \sqrt{M_j a_j}$;
\item
$\gamma_j= \alpha_{C_j}(\nu_j,\Pi_j)$;
\item
$\Gamma_j = \left\| C_j \right\|$;
\item
$\psi_1 = \alpha_{(C_1+C_2)}(\nu_3,\nu_2)$, $\psi_2=\alpha_C(\nu_4,\nu_3)$, $\psi_3 = \alpha_{k_3} (k_1, \nu_4)$;
\item
$\Psi_1 = \| C_1 + C_2 \|$, $\Psi_2 = \| C_1 + C_2 +C_3 \| = \| C \|$, $\Psi_3 = C \cdot k_3$.
\end{itemize}
The Deprit variables are analytic over the open subset $\cD$ over which the $3$ terms of $F_{\mathrm{Kep}}$ are negative, the eccentricities of the Keplerian ellipses lie strictly between 0 and 1, and the nodes $\nu_j$ are nonzero.

\begin{lemma}
  The Deprit variables form a symplectic analytic coordinate system over $\cD$.
\end{lemma}

% Let $(\ell_i,L_i,g_i,G_i,\theta_i,\Theta_i)$ denote the Delaunay
% coordinates. For fixed $(\theta_i,\Theta_i)$, the map
% $(\ell_i,L_i,g_i,G_i) \mapsto (\ell_i,L_i,\gamma_i,G_i)$ is a local
% diffeomorphism. So, there remains only to prove that, for fixed
% $(\ell_i,L_i,g_i,G_i)_{i=1,2,3}$, the map
% $(\theta_i,\Theta_i)_{i=1,2,3} \mapsto (\psi_i,\Psi_i)_{i=1,2,3}$ is a
% local diffeomorphism too.

That Deprit's variables are independent, follows from their symplectic character. So, we only need to prove symplecticity.

The flow of $L_i$ is a reparameterization of the Kepler flow of planet $i$, such that $\{L_i,\ell_i\}=1$ (as is known from the Delaunay coordinates) and the bracket of $L_i$ with any other Deprit variable vanishes.

\begin{figure}[h]
  \centering
  \begin{center}
    \begin{tikzpicture}[scale=0.5]
      \draw[thick,->] (0,0)
      -- (0,1.5) node[left]{$G_i$}
      -- (0,3) node[above]{$C_i$};
      \draw[thick,->] (0,0) -- (1.5,1.5) node[right]{$\Pi_i$};
      \draw[thick,->] (0,0) -- (3,-1) node[right]{$\nu_i$};
      \draw [domain=-18:45,->] plot ({cos(\x)}, {sin(\x)});
      \draw [domain=-18:10] plot ({cos(\x)}, {sin(\x)})
      node[right]{$\gamma_i$}; 
    \end{tikzpicture}
    \hspace{1cm}
    \begin{tikzpicture}[scale=0.5]
      \draw[thick,->] (0,0)
      -- (0,1.5) node[left]{$\Psi_1$}
      -- (0,3) node[above]{$C_1+C_2$};
      \draw[thick,->] (0,0) -- (2,2) node[right]{$\nu_2 = C_1 \times
        C_2$}; 
      \draw[thick,->] (0,0) -- (3,-1) node[right]{$\nu_3 = C_3 \times
        (C_1+C_2)$}; 
      \draw [domain=-18:45,->] plot ({cos(\x)}, {sin(\x)});
      \draw [domain=-18:10] plot ({cos(\x)}, {sin(\x)})
      node[right]{$\psi_1$}; 
    \end{tikzpicture}
  \end{center}

  \begin{center}
    \begin{tikzpicture}[scale=0.5]
      \draw[thick,->] (0,0)
      -- (0,1.5) node[left]{$\Psi_2$}
      -- (0,3) node[above]{$C$};
      \draw[thick,->] (0,0) -- (3,-1) node[right]{$\nu_4 = k_3 \times C$}; 
      \draw[thick,->] (0,0) -- (2,2) node[right]{$\nu_3$};
      \draw [domain=-18:45,->] plot ({cos(\x)}, {sin(\x)});
      \draw [domain=-18:10] plot ({cos(\x)}, {sin(\x)})
      node[right]{$\psi_2$}; 
    \end{tikzpicture}
    \hspace{1cm}
    \begin{tikzpicture}[scale=0.5]
      \draw[thick,->] (0,0)
      -- (0,1.5) node[left]{$\Psi_3$}
      -- (0,3) node[above]{$C \cdot k_3$};
      \draw[thick,->] (0,0) -- (3,-1) node[right]{$k_1$};
      \draw[thick,->] (0,0) -- (2,2) node[right]{$\nu_4 = k_3 \times C$};
      \draw [domain=-18:45,->] plot ({cos(\x)}, {sin(\x)});
      \draw [domain=-18:10] plot ({cos(\x)}, {sin(\x)})
      node[right]{$\psi_3$}; 
    \end{tikzpicture}
  \end{center}

  \caption{The flows of the $G_i$'s and the $\Psi_i$'s are rotations}
  \label{fig:flows}
\end{figure}
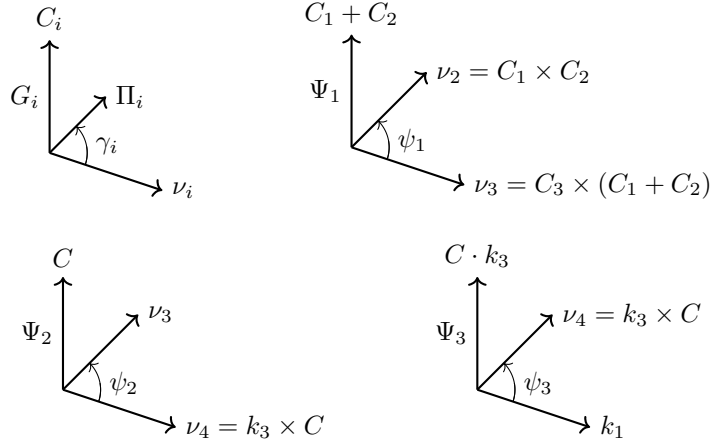

The following facts follow from the property that the angular momentum generates rotations, from rotational equivariance of the angular momentum and from the bare definition of the action variables (see Figure~\ref{fig:flows}): 

The time-$t$ map of the flow of $G_i$ is a rotation of angle $t$ of
the ellipse $i$ in its plane, around the center of attraction. So
$\{G_i,\gamma_i\}=1$ and the bracket of $G_i$ with any other
variable vanishes.

$\Psi_1=|C_1+C_2|$ generates rotations of ellipses $1$ and $2$ around
$C_1+C_2$. So $\{\Psi_1,\psi_1\}=1$ and the bracket of $\Psi_1$ with any
other variable vanishes.

$\Psi_2=|C|$ generates rotations of all three ellipses around $C$.  So
$\{\Psi_2,\psi_2\}=1$ and the bracket of $\Psi_2$ with any other
variable vanishes.

$\Psi_3=C \cdot k_3$ generates rotations of all three ellipses
around the third axis.  So $\{\Psi_3,\psi_3\}=1$ and the bracket of
$\Psi_3$ with any other variable vanishes.

It remains only to check that the brackets of pairs of angles vanish. Due to Jacobi's identity, it is sufficient to check these brackets when angles $(\ell_j,\gamma_j,\psi_j)_{j=1,2,3}$ all vanish (cf.~\cite{fejoz2013Poincare}), i.e. when planets are at their pericenters and all three pericenters $\Pi_i$ lie on the half line generated by $k_1$. From now on we restrict to this (alledgedly Lagrangian) submanifold. 

Recall that $q_a = (q_{a,b})_{b=1,2,3}$ and $p_a = (p_{a,b})_{b=1,2,3}$ are the position and impulsion of the $a$th fictitious Kepler body. Fix bodies $j$ and $k$.

First,
\[
  \{\gamma_j,\gamma_k\} =\sum_{a}
                         \left(
                         \begin{array}[h]{l}                           
                           \sum_{b = 1} \left( \cancel{\frac{\partial \gamma_j}{\partial q_{a,b}}} \displaystyle\frac{\partial \gamma_k}{\partial p_{a,b}}  -  \displaystyle\frac{\partial \gamma_j}{\partial p_{a,b}}   \cancel{\frac{\partial \gamma_k}{\partial q_{a,b}}} \right) + \\
  \sum_{b \neq 1} \left( \displaystyle\frac{\partial \gamma_j}{\partial q_{a,b}}   \cancel{\frac{\partial \gamma_k}{\partial p_{a,b}}} - \cancel{\frac{\partial \gamma_j}{\partial p_{a,b}}} \displaystyle\frac{\partial \gamma_k}{\partial q_{a,b}} \right)
                         \end{array} \right)
  =0;
\]
indeed,
\begin{itemize}
\item when $b=1$, $\frac{\partial \gamma_j}{\partial q_{a,b}}=0$ because a variation of $q_{a,b}$ leaves $p_j$ orthogonal to $q_j$ so the pericenter $\Pi_j$ remains on the $k_1$-axis,
\item when $b \neq1$, $\frac{\partial \gamma_j}{\partial p_{a,b}}=0$ because after a variation of $p_{a,b}$,  $p_j$ remains orthogonal to $q_j$,
\item $\gamma_j$ and $\gamma_k$ play symmetric roles. 
\end{itemize}

Second,
\[
  \{\gamma_j,\psi_k\} =\sum_{a}
  \left(
    \begin{array}[h]{l}                           
      \sum_{b = 1} \left( \cancel{\frac{\partial \gamma_j}{\partial q_{a,b}}}   \displaystyle\frac{\partial \psi_k}{\partial p_{a,b}}  - \displaystyle\frac{\partial \gamma_j}{\partial p_{a,b}}   \cancel{\frac{\partial \psi_k}{\partial q_{a,b}}} \right) + \\
      \sum_{b \neq 1} \left(
      \displaystyle\frac{\partial \gamma_j}{\partial q_{a,b}}   \cancel{\frac{\partial \psi_k}{\partial p_{a,b}}} - \cancel{\frac{\partial \gamma_j}{\partial p_{a,b}}} \displaystyle\frac{\partial \psi_k}{\partial q_{a,b}} \right)
    \end{array} \right) = 0\]
and 
\[
  \{\psi_j,\psi_k\} =\sum_{a}
  \left(
    \begin{array}[h]{l}                           
      \sum_{b = 1} \left( \cancel{\frac{\partial \psi_j}{\partial q_{a,b}}}   \displaystyle\frac{\partial \psi_k}{\partial p_{a,b}}  - \displaystyle\frac{\partial \psi_j}{\partial p_{a,b}}   \cancel{\frac{\partial \psi_k}{\partial q_{a,b}}} \right) + \\
      \sum_{b \neq 1} \left(
      \displaystyle\frac{\partial \psi_j}{\partial q_{a,b}}   \cancel{\frac{\partial \psi_k}{\partial p_{a,b}}} - \cancel{\frac{\partial \psi_j}{\partial p_{a,b}}} \displaystyle\frac{\partial \psi_k}{\partial q_{a,b}} \right)
    \end{array} \right) = 0\]
for similar reasons. 

%%% Local Variables: 
%%% mode: latex
%%% TeX-master: "4bp_secular_diffusion_v6.tex"
%%% End: 

\section{The scattering map of a normally hyperbolic invariant manifold}
\label{sec:heuristics}

In this section we denote by $M$ a $C^r$ smooth manifold, and by $\phi^t : M \to M$ a smooth flow with $\left. \frac{d}{dt} \right|_{t=0} \phi^t = X$ where $X \in C^r (M, TM)$. Let $\Lambda \subset M$ be a compact $\phi^t$-invariant submanifold, possibly with boundary. By $\phi^t$-invariant we mean that $X$ is tangent to $\Lambda$, but that orbits can escape through the boundary (a concept sometimes referred to as \emph{local} invariance). 

\begin{definition}
We call $\Lambda$ a \emph{normally hyperbolic invariant manifold} for $\phi^t$ if there is $0 < \lambda < \mu^{-1}$, a positive constant $C$ and an invariant splitting of the tangent bundle
\begin{equation}
T_{\Lambda} M = T \Lambda \oplus E^s \oplus E^u
\end{equation}
such that:
\begin{equation} \label{eq_normalhyperbolicity}
\def\arraystretch{1.5}
\begin{array}{c}
\| D \phi^t |_{E^s} \| \leq C \lambda^t \mbox{ for all } t \geq 0, \\
\| D \phi^t |_{E^{u}} \| \leq C \lambda^{-t} \mbox{ for all } t \leq 0, \\
\| D \phi^t |_{T \Lambda} \| \leq C \mu^{| t |} \mbox{ for all } t \in \mathbb{R}.
\end{array}
\end{equation}
Moreover, $\Lambda$ is called an $r$-\emph{normally hyperbolic invariant manifold} if it is $C^r$ smooth, and
\begin{equation}\label{eq_largespectralgap}
0 < \lambda < \mu^{-r} < 1
\end{equation}
 for $r \geq 1$. This is called a \emph{large spectral gap} condition.
\end{definition}

This definition guarantees the existence of stable and unstable invariant manifolds $W^{s,u} (\Lambda)\subset M$ defined as follows. The local stable manifold $W^{s}_{\mathrm{loc}}(\Lambda)$ is the set of points in a small neighbourhood of $\Lambda$ whose forward orbits never leave the neighbourhood, and tend exponentially to $\Lambda$. The local unstable manifold $W^{u}_{\mathrm{loc}}(\Lambda)$ is the set of points in the neighbourhood whose backward orbtis stay in the neighbourhood and tend exponentially to $\Lambda$.  We then define
\begin{equation}
W^s(\Lambda) = \bigcup_{t \geq 0}^{\infty} \phi^{-t} \left( W^{s}_{\mathrm{loc}}(\Lambda) \right), \quad W^u(\Lambda) = \bigcup_{t \geq 0}^{\infty} \phi^{t} \left( W^{u}_{\mathrm{loc}}(\Lambda) \right).
\end{equation}
On the stable and unstable manifolds we have the strong stable and strong unstable foliations, the leaves of which we denote by $W^{s,u}(x)$ for $x \in \Lambda$. For each $x \in \Lambda$, the leaf $W^s(x)$ of the strong stable foliation is tangent at $x$ to $E^s_x$, and the leaf $W^u(x)$ of the strong unstable foliation is tangent at $x$ to $E^u_x$. Moreover the foliations are invariant in the sense that $\phi^t \left( W^s (x) \right) = W^s \left( \phi^t (x) \right)$ and $\phi^t \left( W^u (x) \right) = W^u \left( \phi^t (x) \right)$ for each $x \in \Lambda$ and $t \in \mathbb{R}$. We thus define the \emph{holonomy maps} $\pi^{s,u} : W^{s,u} (\Lambda) \to \Lambda$ to be projections along leaves of the strong stable and strong unstable foliations. That is to say, if $x \in W^s (\Lambda)$ then there is a unique $x_+ \in \Lambda$ such that $x \in W^s(x_+)$, and so $\pi^s(x)=x_+$. Similarly, if $x \in W^u (\Lambda)$ then there is a unique $x_- \in \Lambda$ such that $x \in W^u(x_-)$, in which case $\pi^u(x)=x_-$.

Now, suppose that $x \in \left(W^s(\Lambda) \pitchfork W^u (\Lambda)\right) \setminus \Lambda$ is a transverse homoclinic point such that $x \in W^s(x_+) \cap W^u(x_-)$. We say that the homoclinic intersection at $x$ is \emph{strongly transverse} if
\begin{equation}
\begin{split} \label{eq_strongtransversality}
T_x W^s (x_+) \oplus T_x \left( W^s(\Lambda) \cap W^u(\Lambda) \right) = T_x W^s (\Lambda), \\
T_x W^u (x_-) \oplus T_x \left( W^s(\Lambda) \cap W^u(\Lambda) \right) = T_x W^u (\Lambda).
\end{split}
\end{equation}
In this case we can take a sufficiently small neighbourhood $\Gamma$ of $x$ in $W^s(\Lambda) \cap W^u(\Lambda)$ so that \eqref{eq_strongtransversality} holds at each point of $\Gamma$, and the restrictions to $\Gamma$ of the holonomy maps are bijections onto their images. We call $\Gamma$ a \emph{homoclinic channel} (see Figure \ref{figure_homoclinicchannel}). We can then define the scattering map as follows \cite{delshams2008geometric}.

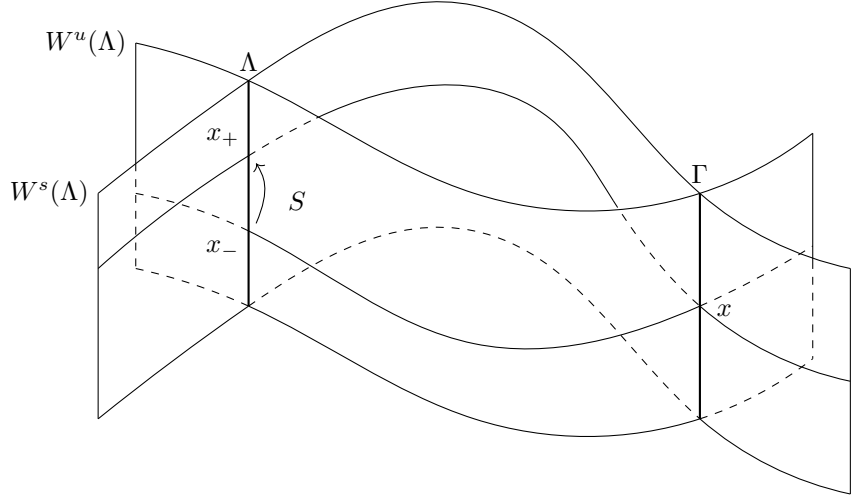
\begin{figure}
\centering
  \begin{tikzpicture}[use Hobby shortcut]
%    \draw (0,-4) grid +(10,7);

    % Strong stable manifolds
    \draw (0,0) .. (2,1.5) .. (5,2.5) .. (8,0) .. (10,-1);
    \draw (0,0) node[anchor=east] {$W^s(\Lambda)$};
    \draw (0,-3) .. (2,-1.5) .. ([blank=soft]5,-0.5) .. ([blank=soft]8,-3) .. (10,-4);
    \draw[dashed,use previous Hobby path={invert soft blanks,disjoint}]; 
    % Strong unstable manifolds
    \draw (0.5,2) .. (2,1.5) .. (5,0) .. (8,0) .. (9.5,0.8);
    \draw (0.5,2) node[anchor=east] {$W^u(\Lambda)$};
    \draw (0.5,-1) .. ([blank=soft]2,-1.5) .. (5,-3) .. (8,-3)
    .. ([blank=soft]9.5,-2.2); 
    \draw[dashed,use previous Hobby path={invert soft blanks,disjoint}]; 
    % (Un)stable manifold boundaries
    \draw (0,0) -- (0,-3);
    \draw (10,-1) -- (10,-4);
    \draw (0.5,2) .. (0.5,0.4) .. ([blank=soft]0.5,-1);
    \draw[dashed,use previous Hobby path={invert soft blanks,disjoint}]; 
    \draw (9.5,0.8) .. (9.5,-0.9) .. ([blank=soft]9.5,-2.2);
    \draw[dashed,use previous Hobby path={invert soft blanks,disjoint}]; 
    % Lambda and Gamma
    \draw[thick] (2,-1.5) -- (2,1.5) node[anchor=south] {$\Lambda$};
    \draw[thick] (8,-3) -- (8,0) node[anchor=south] {$\Gamma$};
    % Orbits
    \draw (0,-1) .. (2,0.5) .. ([blank=soft]2.9,1) .. (6,1)
    .. (6.85,-0.1) .. ([blank=soft]8,-1.5) .. (10,-2.5);
    \draw[dashed,use previous Hobby path={invert soft blanks,disjoint}]; 
    \draw (0.5,0) .. ([blank=soft]2,-0.5) .. (5,-2) .. (8,-1.5) .. ([blank=soft]9.5,-0.7);
    \draw[dashed,use previous Hobby path={invert soft blanks,disjoint}]; 
    % Scattering map
    \draw (2,-0.5) node[anchor=north east] {$x_-$};
    \draw (2,0.5) node[anchor=south east] {$x_+$};
    \draw[->] (2.1,-0.4) .. (2.4,-0.1) node [anchor=west] {$S$} .. (2.1,0.4);
    \draw (8.1,-1.55) node[anchor=west] {$x$};
  \end{tikzpicture}
  \caption{The scattering map $S$ takes a point $x_- \in \Lambda$, follows the unique leaf of the strong unstable foliation passing through $x_-$ to the point $x$ in the homoclinic channel $\Gamma$, and from there follows the unique leaf of the strong stable foliation passing through $x$ to the point $x_+$ on $\Lambda$.} \label{figure_homoclinicchannel}
\end{figure}

\begin{definition}\label{def:scattering}
Let $y_-\in \pi^u \left( \Gamma \right)$, let $y = \left(\left. \pi^u \right|_{\Gamma} \right)^{-1} (y_-)$, and let $y_+ = \pi^s(y)$. The \emph{scattering map} $S : \pi^u (\Gamma) \to \pi^s (\Gamma)$ is defined by
\begin{equation}
S = \pi^s \circ \left( \pi^u \right)^{-1} : y_- \longmapsto y_+.
\end{equation}
\end{definition}

Suppose now that the smoothness $r$ of $M$ and $X$ is at least $2$, suppose the normally hyperbolic invariant manifold $\Lambda$ is a $C^r$ submanifold of $M$, and suppose the large spectral gap condition \eqref{eq_largespectralgap} holds. This implies $C^{r-1}$ regularity of the strong stable and strong unstable foliations \cite{hirsch1970invariant}, which in turn implies that the scattering map $S$ is $C^{r-1}$ \cite{delshams2008geometric}.

In general, the scattering map may be defined only locally, as the transverse homoclinic intersection of stable and unstable manifolds can be very complicated; however in the setting of the present paper, the scattering maps we study turn out to be globally defined.

\section{A general shadowing argument}\label{appendix_shadowing} 
We follow the notation and exposition of \cite{clarke2022topological}. Let $M$ be a $C^r$ manifold of dimension $d=2(m+n)$ where $r \geq 4$. Let $F \in \Diff^4(M)$, and assume $F$ depends smoothly on a small parameter $\epsilon$, with uniformly bounded derivatives. Suppose $F$ has a normally hyperbolic invariant (or locally invariant) manifold $\Lambda \subset M$ of dimension $2n$ satisfying the large spectral gap condition \eqref{eq_largespectralgap}; suppose moreover that $\Lambda$ is diffeomorphic to $\mathbb{T}^n \times [0,1]^n$. Furthermore, we assume that $\dim W^s (\Lambda) = \dim W^u (\Lambda) = m + 2n$. In order to state the remaining assumptions and the shadowing theorems, we must consider some definitions. 

Suppose the scattering map $S$ is defined relative to a homoclinic channel $\Gamma$ for all sufficiently small $\epsilon >0$. We allow for the possibility that the angle between $W^{s,u}(\Lambda)$ along the homoclinic channel $\Gamma$ goes to 0 as $\epsilon \to 0$. Denote by $\alpha (v_1, v_2)$ the angle between two vectors $v_1, v_2$ in the direction that yields the smallest result (i.e. $\alpha (v_1,v_2) \in [0, \pi]$). For $x \in \Gamma$, let 
\begin{equation}
\alpha_{\Gamma} (x) = \inf \alpha (v_+, v_-)
\end{equation}
where the infimum is over all $v_+ \in T_x W^s (\Lambda)^{\perp}$ and $v_- \in T_x W^u (\Lambda)^{\perp}$ such that $\| v_{\pm} \| = 1$. 
\begin{definition}
For $\sigma \geq 0$, we say that \emph{the angle of the splitting along $\Gamma$ is of order $\epsilon^{\sigma}$} if there is a positive constant $C$ (independent of $\epsilon$) such that
\begin{equation}
\alpha_{\Gamma} (x) \geq C \epsilon^{\sigma}
\end{equation}
for all $x \in \Gamma$.
\end{definition}

Recall we have assumed that the normally hyperbolic invariant manifold $\Lambda$ is diffeomorphic to $\mathbb{T}^n \times [0,1]^n$, and denote by $(q,p) \in \mathbb{T}^n \times [0,1]^n$ smooth coordinates on $\Lambda$. Define $f \coloneqq F|_{\Lambda}$, which also depends on the small parameter $\epsilon$.
\begin{definition}\label{def_nearlyintegrabletwist}
We say that $f: \Lambda \to \Lambda$ is a \emph{near-integrable twist map} if there is some $k \in \mathbb{N}$ such that
\begin{equation}\label{eq_inttwistmap}
f:
\begin{cases}
\bar{q} = q + g(p) + O(\epsilon^k) \\
\bar{p} = p + O(\epsilon^k)
\end{cases}
\end{equation}
where
\begin{equation}
\det D g (p) \neq 0
\end{equation}
for all $p \in [0,1]^n$, and where the higher order terms are uniformly bounded in the $C^1$ topology. If the higher order terms are 0 then $f$ is an \emph{integrable twist map}.
\end{definition}

It follows from the definition that if $f : \Lambda \to \Lambda$ is a near-integrable twist map, then there exist twist parameters $T_+ > \widetilde{T}_- >0$ such that 
\begin{equation}\label{eq_twistcondition}
\widetilde{T}_- \| v \| \leq \left\| D g(p) v \right\| \leq T_+ \| v \|
\end{equation}
for all $p \in [0,1]^n$ and all $v \in \mathbb{R}^n$. We can always choose $T_+$ to be independent of $\epsilon$. Our formulation of the problem allows the parameter $\widetilde{T}_-$ to depend on $\epsilon$: there is $\tau \in \mathbb{N}_0$ and a strictly positive constant $T_-$ (independent of $\epsilon$) such that $\widetilde{T}_- = \epsilon^{\tau} T_-$. 
\begin{definition}
Suppose $f : \Lambda \to \Lambda$ is a near-integrable twist map. Denote by $T_+ > \widetilde{T}_- = \epsilon^{\tau} T_- >0$ the twist parameters. We say that $f$ satisfies:
\begin{itemize}
\item
A \emph{uniform twist condition} if $\tau=0$; 
\item
A \emph{non-uniform twist condition (of order $\epsilon^{\tau}$)} if $\tau>0$, and the order $\epsilon^k$ of the error terms in the definition of the near-integrable twist map $f$ is such that $k > \tau$.
\end{itemize}
\end{definition}

In the coordinates $(q,p)$, we may define a foliation of $\Lambda$, the leaves of which are given by
\begin{equation}\label{eq_foliationleaves}
\L (p^*) = \left\{ (q,p) \in \Lambda : p=p^* \right\}.
\end{equation}
If $f : \Lambda \to \Lambda$ is a near-integrable twist map in the sense of Definition \ref{def_nearlyintegrabletwist}, then each leaf of the foliation is almost invariant under $f$, up to terms of order $\epsilon^k$. Denote by $U \subset \Lambda$ the domain of definition of the scattering map $S$.

\begin{definition}
We say that the scattering map $S$ is \emph{transverse to leaves along leaves, and that the angle of transversality is of order $\epsilon^{\upsilon}$ (with respect to the leaves \eqref{eq_foliationleaves} of the foliation of $\Lambda$)} if there are $c, \, C > 0$ such that for all $p_0^* \in [0,1]^n$ and all $p^* \in [0,1]^n$ satisfying $\| p^* - p_0^* \| < c \, \epsilon^{\upsilon}$ we have 
\[
S \left( \L (p_0^*) \cap U \right) \pitchfork \L (p^*) \neq \emptyset
\]
and there is $x \in S \left( \L (p_0^*) \cap U \right) \pitchfork \L (p^*)$ such that
\[
\inf \alpha (v_0,v) \geq C \epsilon^{\upsilon}
\]
where the infimum is taken over all $v_0 \in T_x S \left( \L (p_0^*) \cap U \right)$ and $v \in T_x \L (p^*)$ such that $\| v_0 \| = \| v \| = 1$. 
\end{definition}

Using these definitions, we may now state the main assumptions of the first shadowing theorem, which will be applied to the secular Hamiltonian \eqref{eq_secularhamiltoniandef} to prove the existence of drifting orbits for the secular Hamiltonian defined by \eqref{eq_secularhamiltoniandef}. 
\begin{enumerate}[{[}{A}1{]}]
\item
The stable and unstable manifolds $W^{s,u} (\Lambda)$ have a strongly transverse homoclinic intersection along a homoclinic channel $\Gamma$, and so we have an open set $U \subseteq \Lambda$ and a scattering map $S : U \to \Lambda$. The angle of the splitting along $\Gamma$ is of order $\epsilon^{\sigma}$. 
\item
The inner map $f = F|_{\Lambda}$ is a near-integrable twist map with error terms of order $\epsilon^k$ satisfying a non-uniform (or uniform) twist condition of order $\epsilon^{\tau}$. 
\item
The scattering map $S$ is transverse to leaves along leaves (with respect to the leaves \eqref{eq_foliationleaves} of the foliation of $\Lambda$), and the angle of transversality is of order $\epsilon^{\upsilon}$. 
\end{enumerate}

\begin{theorem} \label{theorem_main1} 
Fix $\eta >0$, let $\epsilon > 0$ be sufficiently small, and suppose
% \begin{equation} \label{eq_kavgcondition}
$k \geq 2 \left( \rho + \tau \right) + 1$
% \end{equation}
where
% \begin{equation} \label{eq_rhodef}
$\rho = \max \{2 \sigma,  2\upsilon, \tau \}$.
% \end{equation}
Choose $\{ p_j^* \}_{j=1}^{\infty} \subset [0,1]^n$ such that 
\begin{equation}
S \left( \L_j \cap U \right) \cap \L_{j+1} \neq \emptyset,
\end{equation}
and $S \left( \L_j \cap U \right)$ is transverse to $\L_{j+1}$, where $\L_j = \L (p_j^*)$. Suppose the distance between $\L_j$ and $\L_{j+1}$ is of order $\epsilon^{\upsilon}$ for each $j$. Then there are $\{z_i\}_{i=1}^{\infty} \subset M$ and $n_i \in \mathbb{N}$ such that 
% \begin{equation}
$z_{i+1} = F^{n_i} (z_i)$
% \end{equation}
and 
\begin{equation}
d ( z_i, \L_i) < \eta.
\end{equation}
Moreover, the time to move a distance of order 1 in the $p$-direction is bounded from above by a term of order
\begin{equation} \label{eq_timeestimate1}
\epsilon^{- \rho - \tau - \upsilon}.
\end{equation}
\end{theorem}

Observe that Theorem \ref{theorem_main1} cannot be applied to \eqref{eq_4bphamafterave}. Indeed, a crucial assumption in Theorem \ref{theorem_main1} is that the scattering map $S$ is transverse to leaves along leaves. For \eqref{eq_4bphamafterave}, we have no information about the behaviour of the scattering map in the $L_i$ directions, and so we cannot check assumption [A3] for the Hamiltonian \eqref{eq_4bphamafterave}. Theorem \ref{theorem_main2} below generalises Theorem \ref{theorem_main1} to settings where transversality is only known in some directions, and thus allows us to complete the proof of Theorem \ref{thm:MainHierarch:Deprit}. 

To state Theorem \ref{theorem_main2} we consider, as before, a $C^r$ manifold $M$ of dimension $2(m+n)$ where $r \geq 4$ and $m,n \in \mathbb{N}$. Let $\Sigma = \mathbb{T}^{\ell_1} \times [0,1]^{\ell_2}$ for some $\ell_1, \ell_2 \in \mathbb{N}_0$, and denote by $(\theta, \xi) \in \mathbb{T}^{\ell_1} \times [0,1]^{\ell_2}$ coordinates on $\Sigma$. Write $\widetilde{M} = M \times \Sigma$. Suppose $\Psi \in \Diff^4 \left(\widetilde{M} \right)$ such that 
\begin{equation}
\Psi (z, \theta, \xi) = \left( G(z, \theta, \xi), \phi (z, \theta, \xi) \right)
\end{equation}
where $z \in M$, $G \in C^4 \left( \widetilde{M}, M \right)$, and $\phi \in C^4 \left( \widetilde{M}, \Sigma \right)$. Suppose $\Psi$ depends on a small parameter $\epsilon$. We make the following assumptions on $\Psi$. 
\begin{enumerate}[{[}{B}1{]}]
\item
There is some $L \in \mathbb{N}$ such that
\begin{equation}
G (z, \theta, \xi) = \widetilde{G}(z; \xi ) + O \left(\epsilon^L \right)
\end{equation}
where the higher order terms are uniformly bounded in the $C^4$ topology, and for each $\xi \in [0,1]^{\ell_2}$ the map
\begin{equation}
\widetilde{G} ( \cdot ; \xi ): z \in M \longmapsto \widetilde{G} ( z ; \xi ) \in M
\end{equation}
satisfies the assumptions [A1-3] of Theorem \ref{theorem_main1}.
\item
Moreover, the map $\phi$ has the form
\begin{equation}
\phi:
\begin{cases}
\bar{\theta} = \phi_1 (z, \theta, \xi) \\
\bar{\xi} = \phi_2(z, \theta, \xi) = \xi + O \left( \epsilon^L \right)
\end{cases}
\end{equation}
where the higher order terms are uniformly bounded in the $C^4$ topology. 
\end{enumerate}

Results from \cite{delshams2008geometric} imply that $\Psi$ has a normally hyperbolic invariant manifold $\widetilde{\Lambda}$ that is $O \left(\epsilon^L \right)$ close in the $C^4$ topology to $\Lambda \times \Sigma$ where $\Lambda \subset M$ is the normally hyperbolic invariant manifold of $\widetilde{G} ( \cdot ; \xi )$. Moreover there is an open set $\widetilde{U} \subset \widetilde{\Lambda}$ and a scattering map $\widetilde{S} : \widetilde{U} \to \widetilde{\Lambda}$ such that the $z$-component of $\widetilde{S}(z, \theta, \xi)$ is $O \left( \epsilon^L \right)$ close in the $C^3$ topology to $S \left( z ; \xi \right)$ where $S \left( \cdot ; \xi \right) : U \to \Lambda$ is the scattering map corresponding to $\widetilde{G} ( \cdot ; \xi )$. 

We use the coordinates $(q,p, \theta, \xi)$ on $\widetilde{\Lambda}$ where $(q,p)$ are the coordinates on $\Lambda$ and $(\theta, \xi)$ are the coordinates on $\Sigma$. Notice that the sets
\begin{equation}
\widetilde{\L} \left(p^*, \xi^* \right) = \left\{ (q,p, \theta, \xi) \in \widetilde{\Lambda}: p = p^*, \xi=\xi^* \right\} = \L \left(p^* \right) \times \mathbb{T}^{\ell_1} \times \left\{ \xi^* \right\}
\end{equation}
for $p^* \in [0,1]^n$ and $\xi^* \in [0,1]^{\ell_2}$ define the leaves of a foliation of $\widetilde{\Lambda}$, where $\L(p^*)$ are the leaves of the foliation of $\Lambda$ defined by \eqref{eq_foliationleaves}. 

\begin{theorem}\label{theorem_main2}
Fix $\eta>0$ and $K\in\mathbb{N}$ and let $\epsilon >0$ be sufficiently small. Choose $N \in \mathbb{N}$ satisfying
\[
 N\leq \frac{1}{\epsilon^K}
\]
$\xi^*_1 \in \mathrm{Int} \left([0,1]^{\ell_2} \right)$ so that $\widetilde{G} ( \cdot ; \xi^*_1 )$ satisfies assumptions [A1-3],
and $p_1^*, \ldots, p_N^* \in [0,1]^n$ as in Theorem \ref{theorem_main1} such that
\begin{equation}
S \left( \L_j \cap U ; \xi^*_1 \right) \cap \L_{j+1} \neq \emptyset
\end{equation}
and $S \left( \L_j \cap U; \xi^*_1 \right)$ is transverse to $\L_{j+1}$, where $\L_j = \L (p_j^*)$. Suppose the distance between $\L_j$ and $\L_{j+1}$ is of order $\epsilon^{\upsilon}$ for each $j$, and $L >0$ is sufficiently large, depending on $K$. Then there are $\xi^*_2, \ldots, \xi^*_N \in [0,1]^{\ell_2}$ such that, with $\widetilde{\L}_j = \widetilde{\L} \left(p^*_j, \xi^*_j \right)$, there are $w_1, \ldots, w_N \in \widetilde{M}$ and $n_i \in \mathbb{N}$ such that the $\xi$ component of $w_1$ is $\xi^*_1$,
\begin{equation}
w_{i+1} = \Psi^{n_i} (w_i),
\end{equation}
and
\begin{equation}
d \left(w_i, \widetilde{\L}_i \right) < \eta
\end{equation}
where $\rho, \sigma, \tau$ are as in the statement of Theorem \ref{theorem_main1}. Moreover, the time to move a distance of order 1 in the $p$-direction is of order
% \begin{equation} \label{eq_timeestimate2}
$\epsilon^{- \rho - \tau - \upsilon}$.
% \end{equation}
\end{theorem}

Note that the transition chain obtained in Theorem \ref{theorem_main2} is only of finite length, while the one obtained in Theorem \ref{theorem_main1} may be infinite.

\section{Computation of the phase shift in $\tilde{\psi}_1$}\label{appendix_psi1phaseshift}

Recall the secular Hamiltonian (see \eqref{eq_secularhamexpansions}) can be written
\begin{equation}
F_{\mathrm{sec}} = C_{12} F_{\mathrm{quad}}^{12} + O \left( \frac{1}{L_2^8} \right)\qquad \text{where}\qquad C_{12} = -\frac{\mu_1 m_2}{(2 \pi )^2}.
\end{equation}
% where
% \begin{equation}
% C_{12} = -\frac{\mu_1 m_2}{(2 \pi )^2}.
% \end{equation}
Since $F_{\mathrm{sec}} \sim L_2^{-6}$ (Proposition \ref{proposition_secularexpansion}), we scale time to have an order one Hamiltonian, $\tilde F_{\mathrm{sec}}=L_2^6F_{\mathrm{sec}}$. From now on in this section, we use this tilde notation to denote the scaled (secular, quadrupolar, octupolar etc.) Hamiltonians.

Denote by $\Phi^t$ the flow of $\tilde F_{\mathrm{sec}}$ and let us consider points $z^\pm$ in $\Lambda$ such $z^+=S(z^-)$. Then, there exists a point $z^*$ in the homoclinic channel such that 
\[
\left| \Phi^t(z^*)- \Phi^t(z^\pm)\right|\lesssim e^{-\nu |t|}\qquad \text{as}\qquad  t\to\pm\infty
\]
for some $\nu>0$ independent of $L_2$. The change in the $\tilde{\psi}_1$ component from $z_-$ to $z_+$ is 
\begin{equation}\label{def:psi1integral}
\begin{split}
\frac{1}{L_2^2} \Delta \left( \tilde{\psi}_1 \right) = \tilde{\psi}_1^+-\tilde{\psi}_1^-=&\int^{+\infty}_0\left(\partial_{\tilde{\Psi}_1}\tilde F_{\mathrm{sec}}(\Phi^t(z^*))- \partial_{\tilde{\Psi}_1}\tilde F_{\mathrm{sec}}(\Phi^t(z^+))\right)dt\\
&+\int_{-\infty}^0\left(\partial_{\tilde{\Psi}_1}\tilde F_{\mathrm{sec}}(\Phi^t(z^*))- \partial_{\tilde{\Psi}_1}\tilde F_{\mathrm{sec}}(\Phi^t(z^-))\right)dt
\end{split}
\end{equation}
The phase shift is the first order of this integral. Since $\partial_{\tilde{\Psi}_1}\tilde F_{\mathrm{sec}}$ is of order $L_2^{-1}$, it may be expected that the phase shift will also be of order $L_2^{-1}$; however, we will see below that the first order term does not contribute to the integral \eqref{def:psi1integral}, and so we have to go the second order.

Now, notice that, since the first order of $\tilde F_{\mathrm{oct}}^{12}$ does not depend on $\tilde{\Psi}_1$ (Lemma \ref{lemma_oct12expansion}),
\[
\left|\partial_{\tilde{\Psi}_1}\tilde F_{\mathrm{sec}}- C_{12} \partial_{\tilde{\Psi}_1} \tilde F_{\mathrm{quad}}^{12}\right|\lesssim \frac{1}{L_2^{3}}
\]
and therefore, to compute $\Delta \left(\tilde{\psi}_1 \right)$, one can replace $F_{\mathrm{sec}}^{12}$ by $C_{12}\tilde F_{\mathrm{quad}}^{12}$ in the integrals \eqref{def:psi1integral}. Analogusly, one can replace the flow $\Phi^t$ by the flow $\Phi_{\mathrm{quad}}^t$ of $C_{12}\tilde F_{\mathrm{quad}}^{12}$. That is, one can conclude that 
\[
\begin{split}
\frac{1}{L_2^2} \Delta \left( \tilde{\psi}_1 \right)=C_{12} &\int^{+\infty}_0\left(\partial_{\tilde{\Psi}_1}\tilde F_{\mathrm{quad}}^{12}(\Phi^t(z^*))- \partial_{\tilde{\Psi}_1}\tilde F_{\mathrm{quad}}^{12}(\Phi_{\mathrm{quad}}^t(z^+))\right)dt\\
&+C_{12} \int_{-\infty}^0\left(\partial_{\tilde{\Psi}_1}\tilde F_{\mathrm{quad}}^{12}(\Phi_{\mathrm{quad}}^t(z^*))- \partial_{\tilde{\Psi}_1}\tilde F_{\mathrm{quad}}^{12}(\Phi_{\mathrm{quad}}^t(z^-))\right)dt+O(L_2^{-3}). 
\end{split}
\]
Now, recall that $\tilde F_{\mathrm{quad}}^{12}$ is an integrable Hamiltonian which has a saddle with a homoclinic connection (in Poincar\'e variables). This homoclinic connection is $L_2^{-1}$ close to that of $H_0^{12}$ given in Lemma \ref{lemma_separatrixformulas}. We denote by 
$z_{\mathrm{quad}}^*$ and $z_0^*$ the saddles of  $H_0^{12}$ and $F_{\mathrm{quad}}^{12}$  respectively and we denote by $z^h_{\mathrm{quad}}(t)$ and $z_0^h(t)$ their homoclinic connections to the saddle. Therefore we can rewrite the phase shift as
\begin{equation}\label{def:phaseshiftintegralquad}
\frac{1}{L_2^2} \Delta \left(\tilde{\psi}_1 \right) = C_{12} \int_{-\infty}^{+\infty}\left(\partial_{\tilde{\Psi}_1}\tilde F_{\mathrm{quad}}^{12}(z^h_{\mathrm{quad}}(t))- \partial_{\tilde{\Psi}_1}\tilde F_{\mathrm{quad}}^{12}(z^*_{\mathrm{quad}})\right)dt +O\left(\frac{1}{L_2^{3}}\right).
\end{equation}
Finally, it remains only to take advantage of the particular form of $\tilde F_{\mathrm{quad}}^{12}$ (see Lemma \ref{lemma_quad12expansion}) to compute a first order for this integral. Write $\hat{F}_{\mathrm{quad}}^{12}=\frac{L_2^6}{\alpha_0^{12}} \left( -\tilde{c}_0^{12} + \tilde F_{\mathrm{quad}}^{12} \right)$. Then we know from Lemma \ref{lemma_quad12expansion} that
\[
\hat{F}_{\mathrm{quad}}^{12}(\gamma_1,\Gamma_1,\tilde{\Gamma}_2,\tilde{\Psi}_1)=H_0^{12}(\gamma_1,\Gamma_1,\tilde{\Gamma}_2)+L_2^{-1} \beta_1 \,  H_1^{12} (\gamma_1,\Gamma_1,\tilde{\Gamma}_2,\tilde{\Psi}_1) + L_2^{-2} \beta_2 \, \tilde{H}_2^{12} (\gamma_1,\Gamma_1,\tilde{\Gamma}_2,\tilde{\Psi}_1) +O(L_2^{-3})
\]
where $\beta_j=\alpha_j^{12}/\alpha_0^{12}$. 

Write
\begin{align}
A \left( \gamma_1, \Gamma_1, \tilde{\Gamma}_2 \right) ={}& - 4 \tilde{\Gamma}_2 H_0^{12} \left( \gamma_1, \Gamma_1, \tilde{\Gamma}_2 \right) + 3 \tilde{\Gamma}_2 - \frac{\Gamma_1^2 \tilde{\Gamma}_2}{L_1^2} \\
B \left( \gamma_1, \Gamma_1, \tilde{\Gamma}_2 \right) ={}& \left( 6 - 8 H_0^{12} \left( \gamma_1, \Gamma_1, \tilde{\Gamma}_2 \right) - 2 \frac{\Gamma_1^2}{L_1^2} \right) \tilde{\Gamma}_2. 
\end{align}
It follows that
\[
\begin{split}
\partial_{\tilde{\Psi}_1} \hat{F}_{\mathrm{quad}}^{12}(\gamma_1,\Gamma_1,\tilde{\Gamma}_2,\tilde{\Psi}_1)=&L_2^{-1} \beta_1 \, (3H_0^{12}(\gamma_1,\Gamma_1,\tilde{\Gamma}_2)-1)\\
&+L_2^{-2} \beta_2 \, \left[2 \tilde{\Psi}_1(3H_0^{12}(\gamma_1,\Gamma_1,\tilde{\Gamma}_2)-1)+B(\gamma_1,\Gamma_1,\tilde{\Gamma}_2)\right]+O(L_2^{-3}).
\end{split}
\]
Notice that $\frac{\beta_1^2}{\beta_2} = \frac{1}{2}$. Using these formulas one can write $\partial_{\tilde{\Psi}_1} \hat{F}_{\mathrm{quad}}^{12}$ as
\begin{align}
 \begin{split}
\partial_{\tilde{\Psi}_1} \hat{F}_{\mathrm{quad}}^{12}(\gamma_1,\Gamma_1,\tilde{\Gamma}_2,\tilde{\Psi}_1)=&L_2^{-1} \beta_1 (3 \hat{F}_{\mathrm{quad}}^{12}(\gamma_1,\Gamma_1,\Gamma_2)-1)  \\
&+L_2^{-2}\beta_2 \left[ 2 \left( 3 \hat{F}_{\mathrm{quad}}^{12} - 1 \right) \tilde{\Psi}_1 + B(\gamma_1,\Gamma_1,\tilde{\Gamma}_2)-3 \frac{\beta_1^2}{\beta_2 \, }H_1^{12} \right]+O(L_2^{-3}) \\
=& L_2^{-1} \beta_1 (3 \hat{F}_{\mathrm{quad}}^{12}(\gamma_1,\Gamma_1,\Gamma_2)-1) \\
&+L_2^{-2} \, \beta_2 \left[ \frac{1}{2} \left( 3 \hat{F}_{\mathrm{quad}}^{12} - 1 \right) \tilde{\Psi}_1 + B(\gamma_1,\Gamma_1,\tilde{\Gamma}_2)- \frac{3}{2} A (\gamma_1,\Gamma_1,\tilde{\Gamma}_2) \right]+O(L_2^{-3}). 
\end{split}
\end{align}
Now, since $\hat{F}_{\mathrm{quad}}^{12}$ is constant along its solutions, 
\[
 \hat{F}_{\mathrm{quad}}^{12}(z^h_{\mathrm{quad}}(t))=\hat{F}_{\mathrm{quad}}^{12}(z^*_{\mathrm{quad}})\qquad \forall \quad t\in\mathbb{R}
\]
and therefore the integral \eqref{def:phaseshiftintegralquad} can be rewritten as
\[
 \Delta \left( \tilde{\psi}_1 \right) =C_{12} \, \beta_2 \left[ \int_{-\infty}^{+\infty}\left[B(z^h_{\mathrm{quad}}(t))- B(z^*_{\mathrm{quad}})\right]dt 
- \frac{3}{2}\int_{-\infty}^{+\infty}\left[A(z^h_{\mathrm{quad}}(t))- A(z^*_{\mathrm{quad}})\right]dt\right]
+O(L_2^{-1}). 
\]
% \[
% \begin{split}
%  \Delta \left( \tilde{\psi}_1 \right) =&C_{12} \, \beta_2 \, \int_{-\infty}^{+\infty}\left[B(z^h_{\mathrm{quad}}(t))- B(z^*_{\mathrm{quad}})\right]dt  \\
% &- \frac{3}{2} C_{12} \, \beta_2 \int_{-\infty}^{+\infty}\left[A(z^h_{\mathrm{quad}}(t))- A(z^*_{\mathrm{quad}})\right]dt
% +O(L_2^{-1}). 
% \end{split}
% \]
Finally, using that the separatrix and saddle of $\hat{F}_{\mathrm{quad}}^{12}$ and $H_0^{12}$ are $L_2^{-1}$-close one can conclude that 
\[
 \Delta \left( \tilde{\psi}_1 \right) =C_{12}  \beta_2 \left[\int_{-\infty}^{+\infty}\left[B(z^h_0(t))- B(z^*_0)\right]dt - \frac{3}{2}  \int_{-\infty}^{+\infty}\left[A(z^h_0(t))- A(z^*_0)\right]dt\right]
+O(L_2^{-1}). 
\]
where $z^h_{0}$ is the separatrix analysed in Lemma \ref{lemma_separatrixformulas} and $z_0^*$ the saddle to which it is asymptotic\footnote{In Deprit coordinates it is asymptotic to two different saddles in the past and future due to the blow up of circular motions. However the values of $A$ and $B$ on the two saddles are the same and therefore we abuse notation and we identify them.}. Using the fact that $H_0^{12} \left(z_0^h(t) \right) = H_0^{12} \left(z_0^*(t) \right)$, and that $\tilde{\Gamma}_2$ is constant with respect to $H_0^{12}$, we see that
\begin{equation}
\Delta \left( \tilde{\psi}_1 \right) = -\frac{1}{2} \frac{\tilde{\Gamma}_2}{L_1^2} C_{12} \,  \beta_2 \, \int_{- \infty}^{+ \infty} \left( \Gamma_1 (t)^2 - L_1^2 \right) dt
\end{equation}
where the value $\Gamma_1 (t)$ of $\Gamma_1$ on the separatrix is given by \eqref{eq_Gamma1separatrix}. It follows that
\begin{equation}
\Delta \left( \tilde{\psi}_1 \right) = -\frac{1}{2 A_2} \frac{\tilde{\Gamma}_2}{L_1^2} C_{12} \, \beta_2 \, \int_{- \infty}^{+ \infty} \frac{ \frac{5}{3} \tilde{\Gamma}_2 - L_1^2}{\cosh^2 \tau} d \tau = \frac{L_1}{6} \sqrt{\frac{3}{2}} C_{12} \, \beta_2 \, \tilde{\Gamma}_2 \sqrt{1 - \frac{5}{3} \frac{\tilde{\Gamma}_2^2}{L_1^2}}
\end{equation}
where we have used the formula \eqref{eq_chia2def} for $A_2$.

\section{Expansion of the inner Hamiltonian}\label{sec_appendixexpformulas}

In the proof of Lemma \ref{lemma_innerhamderivatives}, we performed an expansion of the restriction $K_2$ of $F_{\mathrm{quad}}^{23}$ to $\Lambda$, after averaging over all of the angles. The terms defined in \eqref{eq_fquadaveragedexp} are as follows. 
\begin{align}
\bar{K}_{0,0} ={}&  -{{\left(18\,\delta_{1}^2-30\right)\,\delta_{3}^2-6\,\delta_{1}^4+
 10\,\delta_{1}^2}\over{3\,\delta_{1}^2\,\delta_{2}^3}} \\
\bar{K}_{0,1} ={}& {{
 \left(12\,\tilde{\Gamma}_{2}\,\delta_{1}^2-20\,\tilde{\Psi}_{1}\right)\,\delta_{3}^2+\left(-
 12\,\tilde{\Gamma}_{3}\,\delta_{1}^3+20\,\tilde{\Gamma}_{3}\,\delta_{1}\right)\,\delta_{3}+
 \left(4\,\tilde{\Psi}_{1}-4\,\tilde{\Gamma}_{2}\right)\,\delta_{1}^4
 }\over{\delta_{1}^3\,\delta_{2}^3}} \\
 \bar{K}_{0,2} ={}& - \frac{1}{\delta_{1}^4\,\delta_{2}^3}  \Big[ \left(\left(12\,\tilde{\Gamma}_{2}\,
 \tilde{\Psi}_{1}-9\,L_{1}^2+15\,\tilde{\Gamma}_{2}^2\right)\,\delta_{1}^2-30\,\tilde{\Psi}_{1}^2+15\,
 L_{1}^2-15\,\tilde{\Gamma}_{2}^2\right)\,\delta_{3}^2 
+\left(-24\,\tilde{\Gamma}_{2}\,\tilde{\Gamma}_{3}\,\delta_{1}^3+40\,\tilde{\Gamma}_{3}\,\tilde{\Psi}_{1}\,\delta_{1}\right)\,\delta_{3} \\
\quad&+\left(-2\,\tilde{\Psi}_{1}^2+4\,\tilde{\Gamma}_{2}\,\tilde{\Psi}_{1}+3\,L_{1}^2+6\,\tilde{\Gamma}_{3}^2-5\,\tilde{\Gamma}_{2}^2\right)\,
 \delta_{1}^4 
+\left(-5\,L_{1}^2-10\,\tilde{\Gamma}_{3}^2+5\,\tilde{\Gamma}_{2}^2\right)\,
 \delta_{1}^2\Big] \\
 \bar{K}_{1,0} ={}& -{{\left(24\,\delta_{1}^2
 -40\right)\,\delta_{3}^3+\left(-12\,\delta_{1}^4+20\,\delta_{1}^2
 \right)\,\delta_{3}}\over{\delta_{1}^2\,\delta_{2}^4}} \\
 \bar{K}_{1,1} ={}& {{
 \left(48\,\tilde{\Gamma}_{2}\,\delta_{1}^2-80\,\tilde{\Psi}_{1}\right)\,\delta_{3}^3+\left(-
 72\,\tilde{\Gamma}_{3}\,\delta_{1}^3+120\,\tilde{\Gamma}_{3}\,\delta_{1}\right)\,\delta_{3}^2+
 \left(24\,\tilde{\Psi}_{1}-24\,\tilde{\Gamma}_{2}\right)\,\delta_{1}^4\,\delta_{3}+12\,\tilde{\Gamma}_{3}
 \,\delta_{1}^5-20\,\tilde{\Gamma}_{3}\,\delta_{1}^3}\over{
 \delta_{1}^3\,\delta_{2}^4}} \\
 \bar{K}_{1,2} ={}& - \frac{1}{\delta_{1}^4\,\delta_{2}^4} \Big[ \left(\left(48\,\tilde{\Gamma}_{2}\,\tilde{\Psi}_{1}-36
 \,L_{1}^2+60\,\tilde{\Gamma}_{2}^2\right)\,\delta_{1}^2-120\,\tilde{\Psi}_{1}^2+60\,L_{1}^2-
 60\,\tilde{\Gamma}_{2}^2\right)\,\delta_{3}^3+\left(-144\,\tilde{\Gamma}_{2}\,\tilde{\Gamma}_{3}\,
 \delta_{1}^3+240\,\tilde{\Gamma}_{3}\,\tilde{\Psi}_{1}\,\delta_{1}\right)\,\delta_{3}^2 \\
\quad & +\left(\left(-12\,\tilde{\Psi}_{1}^2+24\,\tilde{\Gamma}_{2}\,\tilde{\Psi}_{1}+18\,L_{1}^2+72\,\tilde{\Gamma}_{3}^2-30
 \,\tilde{\Gamma}_{2}^2\right)\,\delta_{1}^4+\left(-30\,L_{1}^2-120\,\tilde{\Gamma}_{3}^2+30\,
 \tilde{\Gamma}_{2}^2\right)\,\delta_{1}^2\right)\,\delta_{3} \\
\quad & +\left(-24\,\tilde{\Gamma}_{3}\,
 \tilde{\Psi}_{1}+24\,\tilde{\Gamma}_{2}\,\tilde{\Gamma}_{3}\right)\,\delta_{1}^5 \Big] \\
 \bar{K}_{2,0} ={}& -{{
 \left(123\,\delta_{1}^2-205\right)\,\delta_{3}^4+\left(-78\,
 \delta_{1}^4+130\,\delta_{1}^2\right)\,\delta_{3}^2+3\,\delta_{1}^6-
 5\,\delta_{1}^4}\over{2\,\delta_{1}^2\,\delta_{2}^5}} \\
 \bar{K}_{2,1} ={}& \frac{1}{\delta_{1}^3\,\delta_{2}^5} \Big[\left(
 123\,\tilde{\Gamma}_{2}\,\delta_{1}^2-205\,\tilde{\Psi}_{1}\right)\,\delta_{3}^4+\left(-246
 \,\tilde{\Gamma}_{3}\,\delta_{1}^3+410\,\tilde{\Gamma}_{3}\,\delta_{1}\right)\,\delta_{3}^3+
 \left(78\,\tilde{\Psi}_{1}-78\,\tilde{\Gamma}_{2}\right)\,\delta_{1}^4\,\delta_{3}^2 \\
 & +\left(
 78\,\tilde{\Gamma}_{3}\,\delta_{1}^5-130\,\tilde{\Gamma}_{3}\,\delta_{1}^3\right)\,\delta_{3}+
 \left(-6\,\tilde{\Psi}_{1}+3\,\tilde{\Gamma}_{2}\right)\,\delta_{1}^6+5\,\tilde{\Psi}_{1}\,\delta_{1}^4
 \Big] \\
 \bar{K}_{2,2} ={}& - 1\frac{1}{4\,\delta_{1}^4\,\delta_{2}^5} \Big[
 \left(\left(492\,\tilde{\Gamma}_{2}\,\tilde{\Psi}_{1}-369\,L_{1}^2+615\,\tilde{\Gamma}_{2}^2\right)\,
 \delta_{1}^2-1230\,\tilde{\Psi}_{1}^2+615\,L_{1}^2-615\,\tilde{\Gamma}_{2}^2\right)\,
 \delta_{3}^4 \\
 &+\left(-1968\,\tilde{\Gamma}_{2}\,\tilde{\Gamma}_{3}\,\delta_{1}^3+3280\,\tilde{\Gamma}_{3}\,
 \tilde{\Psi}_{1}\,\delta_{1}\right)\,\delta_{3}^3+\Big( \left(-156\,\tilde{\Psi}_{1}^2+312
 \,\tilde{\Gamma}_{2}\,\tilde{\Psi}_{1}+234\,L_{1}^2+1476\,\tilde{\Gamma}_{3}^2-390\,\tilde{\Gamma}_{2}^2\right)\,
 \delta_{1}^4 \\
 & +\left(-390\,L_{1}^2-2460\,\tilde{\Gamma}_{3}^2+390\,\tilde{\Gamma}_{2}^2\right)\,
 \delta_{1}^2\Big)\,\delta_{3}^2+\left(-624\,\tilde{\Gamma}_{3}\,\tilde{\Psi}_{1}+624\,
 \tilde{\Gamma}_{2}\,\tilde{\Gamma}_{3}\right)\,\delta_{1}^5\,\delta_{3} \\
 & +\left(36\,\tilde{\Psi}_{1}^2-36\,
 \tilde{\Gamma}_{2}\,\tilde{\Psi}_{1}-9\,L_{1}^2-156\,\tilde{\Gamma}_{3}^2+15\,\tilde{\Gamma}_{2}^2\right)\,\delta_{1}^
 6+\left(-10\,\tilde{\Psi}_{1}^2+15\,L_{1}^2+260\,\tilde{\Gamma}_{3}^2-15\,\tilde{\Gamma}_{2}^2\right)\,
 \delta_{1}^4\Big]
\end{align}

\section{Corrigendum of \cite{fejoz2016secular}}\label{appendix_errata}

As this paper uses several ideas and formulas from \cite{fejoz2016secular}, we include here corrections to some errata in that paper. 
\begin{enumerate}
\item
In Lemma 2.1, equation (12), the sign of the last term $\frac{\Gamma^2}{L_1^2}$ should be $+$, and so the Hamiltonian $H_0$ should look the same as the Hamiltonian $H_0^{12}$ defined in equation \eqref{eq_H012def} of the present paper. 
\item
In Lemma 3.1, equation (26), the sign should be $+$, and so that lemma is equivalent to Lemma \ref{lemma_separatrixformulas} of the present paper. 
\item
In Lemma 5.1, the sign of $\sin \gamma^2$ should be $+$. 
\item
Lemma 5.2 should be as follows:
\begin{lemma*}
The function $\mathcal{F}^+$ can be written, on the separatrix, as a function of $g_1$ as $\F^+ = \frac{1}{2} \left( \mathcal{F}_1 + i \, \mathcal{F}_2 \right)$ with
\[
\begin{dcases}
\mathcal{F}_1 =& C_1 \, \frac{\sqrt{1 - \frac{5}{3} \, \left( 1 + \chi^2 \right) \, \cos^2 g_1}}{1 - \frac{5}{3} \, \cos^2 g_1} \, \cos g_1 \\
\mathcal{F}_2 =& C_2 \, \frac{\sqrt{1 - \frac{5}{3} \, \left( 1 + \chi^2 \right) \, \cos^2 g_1}}{1 - \frac{5}{3} \, \cos^2 g_1} \, \left( - \frac{21}{5} + \frac{1}{3} \, \frac{\Gamma^2}{L_1^2} \, \frac{15 - 13 \, \cos^2 g_1}{1 - \frac{5}{3} \, \cos^2 g_1} \right)
\end{dcases}
\]
where
\[
C_1 = \frac{20}{3} \, \sqrt{\frac{2}{3}} \, \frac{\Gamma^3}{L_1^3 \, \chi} \, A_{\mathrm{oct}}, \quad
C_2 = \frac{\sqrt{10}}{3} \, \frac{\Gamma}{L_1 \, \chi} \, A_{\mathrm{oct}}, \quad A_{\mathrm{oct}} = - \frac{15}{64} \, \frac{a_1^3}{a_2^4} \, \frac{e_2}{ \left( 1 - e_2^2 \right)^{\frac{5}{2}}}. 
\]
\end{lemma*}
\item
As a consequence, there are many cancellations in the Poincar\'e-Melnikov computation in Section 5 of \cite{fejoz2016secular}, the complex integrand has only one singularity, and the integral $\L \left( \gamma^0 \right)$ defined in equation (34) takes the simple form
\[
\L \left( \gamma^0 \right) = \sqrt{\frac{3}{2}} \,  \frac{ \pi \, A_{\mathrm{oct}} \, e^{\frac{\pi \, \Gamma}{A_2 \, L_1^2}}}{6 \, \sqrt{15} \, L_1 \, \left( 1 + e^{\frac{2 \, \pi \, \Gamma}{A_2 \, L_1^2}} \right) } \, \left( 24 \, L_1^2 - 37 \, \Gamma^2 \right) \, \sin \gamma^0. 
\]
Note that the notation from \cite{fejoz2016secular} is $\Gamma=\tilde{\Gamma}_2$, and moreover the function in front of $\sin\gamma^0$, called $\tilde{\L}_2^{12}$ in Proposition \ref{proposition_melnikovtildeexp} of the present paper, does note vanish under condition \eqref{eq_equilibriumcondition2}.
\end{enumerate}

{\small
\bibliographystyle{abbrv}
\bibliography{4bp_secular_diffusion_refs} 

\begin{thebibliography}{10}

\bibitem{Albouy:2013}
A.~Albouy.
\newblock Histoire des \'{e}quations de la m\'{e}canique analytique: rep\`eres
  chronologiques et difficult\'{e}s.
\newblock In {\em Sim\'{e}on-{D}enis {P}oisson}, Hist. Math. Sci. Phys., pages
  229--280. Ed. \'{E}c. Polytech., Palaiseau, 2013.

\bibitem{Alekseev68}
V.~M. Alekseev.
\newblock Quasirandom dynamical systems. {I, II, III}.
\newblock {\em Math. USSR}, 5,6,7, 1968--1969.

\bibitem{Arnold:1963}
V.~I. Arnold.
\newblock Small denominators and problems of stability of motion in classical
  and celestial mechanics.
\newblock {\em Uspehi Mat. Nauk}, 18(6 (114)):91--192, 1963.

\bibitem{arnold1964instability}
V.~I. Arnold.
\newblock Instability of dynamical systems with many degrees of freedom.
\newblock {\em Dokl. Akad. Nauk SSSR}, 156:9--12, 1964.

\bibitem{arnol1978mathematical}
V.~I. Arnold.
\newblock {\em Mathematical methods of classical mechanics}, volume~60.
\newblock Springer, 1978.

\bibitem{Arnold:1990}
V.~I. Arnold.
\newblock {\em Huygens and {B}arrow, {N}ewton and {H}ooke}.
\newblock Birkh{\"a}user Verlag, Basel, 1990.
\newblock Pioneers in mathematical analysis and catastrophe theory from
  evolvents to quasicrystals, Translated from the Russian by Eric J. F.
  Primrose.

\bibitem{batygin2015chaotic}
K.~Batygin, A.~Morbidelli, and M.~J. Holman.
\newblock Chaotic disintegration of the inner solar system.
\newblock {\em The Astrophysical Journal}, 799(2):120, 2015.

\bibitem{Bernard08}
P.~Bernard.
\newblock The dynamics of pseudographs in convex {H}amiltonian systems.
\newblock {\em J. Amer. Math. Soc.}, 21(3):615--669, 2008.

\bibitem{Kaloshin:2016}
P.~Bernard, V.~Kaloshin, and K.~Zhang.
\newblock Arnold diffusion in arbitrary degrees of freedom and normally
  hyperbolic invariant cylinders.
\newblock {\em Acta Math.}, 217(1):1--79, 2016.

\bibitem{bolotin2006symbolic}
S.~Bolotin.
\newblock Symbolic dynamics of almost collision orbits and skew products of
  symplectic maps.
\newblock {\em Nonlinearity}, 19(9):2041--2063, 2006.

\bibitem{Bolotin:1999}
S.~Bolotin and D.~Treschev.
\newblock Unbounded growth of energy in nonautonomous {H}amiltonian systems.
\newblock {\em Nonlinearity}, 12(2):365--388, 1999.

\bibitem{Boue:2010:collisionless}
G.~Bou{\'e} and J.~Laskar.
\newblock A collisionless scenario for {U}ranus tilting.
\newblock {\em Astrophysics J. Letters}, 712(1):L44, 2010.

\bibitem{Boue:2012:simple}
G.~Bou{\'e}, J.~Laskar, and F.~Farago.
\newblock A simple model of the chaotic eccentricity of mercury.
\newblock {\em Astron. \& Astrophysics}, 2012.

\bibitem{Capinski12}
M.~J. Capi\'{n}ski.
\newblock Computer assisted existence proofs of {L}yapunov orbits at {$L_2$}
  and transversal intersections of invariant manifolds in the {J}upiter-{S}un
  {PCR}3{BP}.
\newblock {\em SIAM J. Appl. Dyn. Syst.}, 11(4):1723--1753, 2012.

\bibitem{CapinskiGidea}
M.~J. Capi\'{n}ski and M.~Gidea.
\newblock Arnold diffusion, quantitative estimates, and stochastic behavior in
  the three-body problem.
\newblock {\em Communications on Pure and Applied Mathematics}, 2021.

\bibitem{capinski2017diffusion}
M.~J. Capi\'{n}ski, M.~Gidea, and R.~de~la Llave.
\newblock Arnold diffusion in the planar elliptic restricted three-body
  problem: mechanism and numerical verification.
\newblock {\em Nonlinearity}, 30(1):329--360, 2017.

\bibitem{Cheng:2017}
C.~Cheng.
\newblock Dynamics around the double resonance.
\newblock {\em Camb. J. Math.}, 5(2):153--228, 2017.

\bibitem{ChengY04}
C.~Cheng and J.~Yan.
\newblock Existence of diffusion orbits in a priori unstable {H}amiltonian
  systems.
\newblock {\em J. Differential Geom.}, 67(3):457--517, 2004.

\bibitem{CherchiaG94}
L.~Chierchia and G.~Gallavotti.
\newblock Drift and diffusion in phase space.
\newblock {\em Ann. Inst. H. Poincar\'{e} Phys. Th\'{e}or.}, 60(1):144, 1994.

\bibitem{chierchia2011deprit}
L.~Chierchia and G.~Pinzari.
\newblock Deprit’s reduction of the nodes revisited.
\newblock {\em Celestial Mechanics and Dynamical Astronomy}, 109(3):285--301,
  2011.

\bibitem{Chierchia:2011}
L.~Chierchia and G.~Pinzari.
\newblock The planetary {$N$}-body problem: symplectic foliation, reductions
  and invariant tori.
\newblock {\em Invent. Math.}, 186(1):1--77, 2011.

\bibitem{Chirikov:1959}
B.~V. Chirikov.
\newblock The passage of a nonlinear oscillating system through resonance.
\newblock {\em Soviet Physics. Dokl.}, 4:390--394, 1959.

\bibitem{clarke2022topological}
A.~Clarke, J.~Fejoz, and M.~Guardia.
\newblock Topological shadowing methods in arnold diffusion: Weak torsion and
  multiple time scales.
\newblock {\em arXiv preprint arXiv:2204.14135}, 2022.

\bibitem{clarke2022arnold}
A.~Clarke and D.~Turaev.
\newblock Arnold diffusion in multi-dimensional convex billiards.
\newblock {\em Duke Mathematical Journal}, to appear, 2022.

\bibitem{DelshamsLS00}
A.~Delshams, R.~de~la Llave, and T.~Seara.
\newblock A geometric approach to the existence of orbits with unbounded energy
  in generic periodic perturbations by a potential of generic geodesic flows of
  $\mathbb{T}\sp 2$.
\newblock {\em Comm. Math. Phys.}, 209(2):353--392, 2000.

\bibitem{DelshamsLS06b}
A.~Delshams, R.~de~la Llave, and T.~Seara.
\newblock Orbits of unbounded energy in quasi-periodic perturbations of
  geodesic flows.
\newblock {\em Adv. Math.}, 202(1):64--188, 2006.

\bibitem{delshams2006biggaps}
A.~Delshams, R.~de~la Llave, and T.~M. Seara.
\newblock A geometric mechanism for diffusion in {H}amiltonian systems
  overcoming the large gap problem: heuristics and rigorous verification on a
  model.
\newblock {\em Mem. Amer. Math. Soc.}, 179(844):viii+141, 2006.

\bibitem{delshams2008geometric}
A.~Delshams, R.~De~La~Llave, and T.~M. Seara.
\newblock Geometric properties of the scattering map of a normally hyperbolic
  invariant manifold.
\newblock {\em Advances in Mathematics}, 217(3):1096--1153, 2008.

\bibitem{MR3479576}
A.~Delshams, R.~de~la Llave, and T.~M. Seara.
\newblock Instability of high dimensional {H}amiltonian systems: multiple
  resonances do not impede diffusion.
\newblock {\em Adv. Math.}, 294:689--755, 2016.

\bibitem{delshams2013transition}
A.~Delshams, M.~Gidea, and P.~Roldan.
\newblock Transition map and shadowing lemma for normally hyperbolic invariant
  manifolds.
\newblock {\em Discrete and Continuous Dynamical Systems}, 33(3):1089--1112,
  2013.

\bibitem{DelshamsH05}
A.~Delshams and G.~Huguet.
\newblock Geography of resonances and {A}rnold diffusion in a priori unstable
  {H}amiltonian systems.
\newblock {\em Nonlinearity}, 22(8):1997--2077, 2009.

\bibitem{delshams2019instability}
A.~Delshams, V.~Kaloshin, A.~de~la Rosa, and T.~M. Seara.
\newblock Global instability in the restricted planar elliptic three body
  problem.
\newblock {\em Comm. Math. Phys.}, 366(3):1173--1228, 2019.

\bibitem{deprit1983}
A.~Deprit.
\newblock Elimination of the nodes in problems of {$n$} bodies.
\newblock {\em Celestial Mech.}, 30(2):181--195, 1983.

\bibitem{fejoz2002quasiperiodic}
J.~Fejoz.
\newblock Quasiperiodic motions in the planar three-body problem.
\newblock {\em Journal of Differential Equations}, 183(2):303--341, 2002.

\bibitem{fejoz2004arnold}
J.~F{\'e}joz.
\newblock D{\'e}monstration du `th{\'e}or{\`e}me d'{A}rnold' sur la
  stabilit{\'e} du syst{\`e}me plan{\'e}taire (d'apr{\`e}s {H}erman).
\newblock {\em Ergodic Theory Dynam. Systems}, 24(5):1521--1582, 2004.

\bibitem{fejoz2013Poincare}
J.~Fejoz.
\newblock On action-angle coordinates and the {P}oincar\'{e} coordinates.
\newblock {\em Regul. Chaotic Dyn.}, 18(6):703--718, 2013.

\bibitem{Fejoz:2015:nbp}
J.~F{\'e}joz.
\newblock {\em Celestial Mechanics}, chapter The $N$-body problem.
\newblock Encyclopedia of life support systems. Unesco-EOLSS, 2015.

\bibitem{fejoz2016secular}
J.~Fejoz and M.~Guardia.
\newblock Secular instability in the three-body problem.
\newblock {\em Archive for rational mechanics and analysis}, 221(1):335--362,
  2016.

\bibitem{Fejoz:2018:overlap}
J.~Fejoz and M.~Guardia.
\newblock An example of resonance overlap.
\newblock unpublished, 2018.

\bibitem{fejoz2016kirkwood}
J.~F\'{e}joz, M.~Gu\`ardia, V.~Kaloshin, and P.~Rold\'{a}n.
\newblock Kirkwood gaps and diffusion along mean motion resonances in the
  restricted planar three-body problem.
\newblock {\em J. Eur. Math. Soc. (JEMS)}, 18(10):2315--2403, 2016.

\bibitem{fenichel1971persistence}
N.~Fenichel.
\newblock Persistence and smoothness of invariant manifolds for flows.
\newblock {\em Indiana University Mathematics Journal}, 21(3):193--226, 1971.

\bibitem{fenichel1974asymptotic}
N.~Fenichel.
\newblock Asymptotic stability with rate conditions.
\newblock {\em Indiana University Mathematics Journal}, 23(12):1109--1137,
  1974.

\bibitem{fenichel1977asymptotic}
N.~Fenichel.
\newblock Asymptotic stability with rate conditions, ii.
\newblock {\em Indiana University Mathematics Journal}, 26(1):81--93, 1977.

\bibitem{MR1784083}
E.~Fontich and P.~Mart\'{\i}n.
\newblock Arnold diffusion in perturbations of analytic exact symplectic maps.
\newblock {\em Nonlinear Anal.}, 42(8):1397--1412, 2000.

\bibitem{MR1806373}
E.~Fontich and P.~Mart\'{\i}n.
\newblock Arnold diffusion in perturbations of analytic integrable
  {H}amiltonian systems.
\newblock {\em Discrete Contin. Dynam. Systems}, 7(1):61--84, 2001.

\bibitem{Froeschle:1989}
C.~{Froeschle} and H.~{Scholl}.
\newblock {The three principal secular resonances nu(5), nu(6), and nu(16) in
  the asteroidal belt}.
\newblock {\em Celestial Mechanics and Dynamical Astronomy}, 46:231--251, Sept.
  1989.

\bibitem{Gelfreich:2008}
V.~Gelfreich and D.~Turaev.
\newblock Unbounded energy growth in {H}amiltonian systems with a slowly
  varying parameter.
\newblock {\em Comm. Math. Phys.}, 283(3):769--794, 2008.

\bibitem{GelfreichTuraev2017}
V.~Gelfreich and D.~Turaev.
\newblock Arnold diffusion in a priori chaotic symplectic maps.
\newblock {\em Comm. Math. Phys.}, 353(2):507--547, 2017.

\bibitem{gidea2006topological}
M.~Gidea and R.~de~la Llave.
\newblock Topological methods in the instability problem of hamiltonian
  systems.
\newblock {\em Discrete \& Continuous Dynamical Systems}, 14(2):295, 2006.

\bibitem{MR4033892}
M.~Gidea, R.~de~la Llave, and T.~M-Seara.
\newblock A general mechanism of diffusion in {H}amiltonian systems:
  qualitative results.
\newblock {\em Comm. Pure Appl. Math.}, 73(1):150--209, 2020.

\bibitem{Guardia:2013:oscillatory}
M.~Guardia, P.~Mart{\'{\i}}n, and T.~M. Seara.
\newblock Homoclinic solutions to infinity and oscillatory motions in the
  restricted planar circular three body problem.
\newblock In {\em Progress and challenges in dynamical systems}, volume~54 of
  {\em Springer Proc. Math. Stat.}, pages 265--280. Springer, Heidelberg, 2013.

\bibitem{guardia2016oscillatory}
M.~Guardia, P.~Mart\'{\i}n, and T.~M. Seara.
\newblock Oscillatory motions for the restricted planar circular three body
  problem.
\newblock {\em Invent. Math.}, 203(2):417--492, 2016.

\bibitem{GuardiaPSV20}
M.~Guardia, J.~Paradela, T.~M. Seara, and C.~Vidal.
\newblock Symbolic dynamics in the restricted elliptic isosceles three body
  problem.
\newblock {\em J. Differential Equations}, 294:143--177, 2021.

\bibitem{Harrington:1968}
R.~S. Harrington.
\newblock Dynamical evolution of triple stars.
\newblock {\em Astronom. J.}, pages 190--194, 1968.

\bibitem{herman1998icm}
M.~Herman.
\newblock Some open problems in dynamical systems.
\newblock In {\em Proceedings of the {I}nternational {C}ongress of
  {M}athematicians ({B}erlin, 1998)}, volume Extra Vol. II, pages 797--808
  (electronic), 1998.

\bibitem{hirsch1970invariant}
M.~W. Hirsch, C.~C. Pugh, and M.~Shub.
\newblock Invariant manifolds.
\newblock {\em Bulletin of the American Mathematical Society}, 76(5), 1970.

\bibitem{jefferys1966}
W.~H. Jefferys and J.~Moser.
\newblock Quasi-periodic solutions for the three-body problem.
\newblock {\em Astronom. J.}, 71:568--578, 1966.

\bibitem{Kaloshin:2020}
V.~Kaloshin and K.~Zhang.
\newblock {\em Arnold diffusion for smooth systems of two and a half degrees of
  freedom}, volume 208 of {\em Annals of Mathematics Studies}.
\newblock Princeton University Press, Princeton, NJ, 2020.

\bibitem{Kozai:1962}
Y.~Kozai.
\newblock Secular perturbations of asteroids with high inclination and
  eccentricity.
\newblock {\em The Astronomical Journal}, 1962.

\bibitem{Laplace:1785}
P.-S. Laplace.
\newblock Th{\'e}orie de {J}upiter et de {S}aturne.
\newblock {\em M{\'e}m. Acad. royale des sciences de Paris}, {\OE}uvres
  compl{\`e}tes, Tome~XI, 1785--1788.

\bibitem{laskar1989chaos}
J.~Laskar.
\newblock A numerical experiment on the chaotic behavior of the {S}olar
  {S}ystem.
\newblock {\em Nature}, 338:237--238, 1989.

\bibitem{laskar2006Lagrange}
J.~Laskar.
\newblock {\em Sfogliando La M{\'e}chanique analitique, Giornata di studio su
  Louis Lagrange}, chapter Lagrange et la stabilit{\'e} du syst{\`e}me solaire.
\newblock Edizioni Universitarie di Lettere Economia Diritto, Milano, 2006.

\bibitem{Laskar:2008:chaotic}
J.~Laskar.
\newblock Chaotic diffusion in the {S}olar {S}ystem.
\newblock {\em Icarus}, 196(1):1--15, 2008.

\bibitem{laskar2010}
J.~Laskar.
\newblock Le syst{\`e}me solaire est-il stable~?
\newblock In {\em Le Chaos}, number XIV in S{\'e}minaire Poincar{\'e}, pages
  221--246. Birkh{\"a}user, 2010.

\bibitem{Laskar:1993:chaotic}
J.~Laskar and P.~Robutel.
\newblock The chaotic obliquity of the planets.
\newblock {\em Nature}, 361(6413):608--612, 1993.

\bibitem{Lazzarini19}
L.~Lazzarini, J.-P. Marco, and D.~Sauzin.
\newblock Measure and capacity of wandering domains in {G}evrey near-integrable
  exact symplectic systems.
\newblock {\em Mem. Amer. Math. Soc.}, 257(1235):vi+110, 2019.

\bibitem{lecar2001chaos}
M.~Lecar, F.~A. Franklin, M.~J. Holman, and N.~J. Murray.
\newblock Chaos in the solar system.
\newblock {\em Annual Review of Astronomy and Astrophysics}, 39:581--631, 2001.

\bibitem{Lidov:1962}
M.~L. Lidov.
\newblock The evolution of orbits of artificial satellites of planets under the
  action of gravitational perturbations of external bodies.
\newblock {\em Planetary and Space Science}, 9(10):719--759, 1962.

\bibitem{SimoL80}
J.~Llibre and C.~Sim{\'o}.
\newblock Oscillatory solutions in the planar restricted three-body problem.
\newblock {\em Math. Ann.}, 248(2):153--184, 1980.

\bibitem{LlibreS80}
J.~Llibre and C.~Sim{\'o}.
\newblock Some homoclinic phenomena in the three-body problem.
\newblock {\em J. Differential Equations}, 37(3):444--465, 1980.

\bibitem{moeckel1989chaotic}
R.~Moeckel.
\newblock Chaotic dynamics near triple collision.
\newblock {\em Arch. Rational Mech. Anal.}, 107(1):37--69, 1989.

\bibitem{moeckel2002drift}
R.~Moeckel.
\newblock Generic drift on {C}antor sets of annuli.
\newblock In {\em Celestial mechanics ({E}vanston, {IL}, 1999)}, volume 292 of
  {\em Contemp. Math.}, pages 163--171. Amer. Math. Soc., Providence, RI, 2002.

\bibitem{moeckel2007symbolicdynamics}
R.~Moeckel.
\newblock Symbolic dynamics in the planar three-body problem.
\newblock {\em Regul. Chaotic Dyn.}, 12(5):449--475, 2007.

\bibitem{Mogavero:2022}
F.~Mogavero and J.~Laskar.
\newblock The origin of chaos in the solar system through computer algebra.
\newblock {\em A\&A}, 662:L3, 2022.

\bibitem{Montgomery:2007}
R.~Montgomery.
\newblock The zero angular momentum, three-body problem: all but one solution
  has syzygies.
\newblock {\em Ergodic Theory Dynam. Systems}, 27(6):1933--1946, 2007.

\bibitem{Montgomery:2019}
R.~Montgomery.
\newblock Oscillating about coplanarity in the 4 body problem.
\newblock {\em Invent. Math.}, 218(1):113--144, 2019.

\bibitem{morbidelli2002}
A.~Morbidelli.
\newblock {\em Modern celestial mechanics. Aspects of solar system dynamics}.
\newblock CRC Press, 2002.

\bibitem{moser1973stableunstable}
J.~Moser.
\newblock {\em Stable and random motions in dynamical systems}.
\newblock Princeton Landmarks in Mathematics. Princeton University Press,
  Princeton, NJ, 2001.
\newblock With special emphasis on celestial mechanics, Reprint of the 1973
  original, With a foreword by Philip J. Holmes.

\bibitem{Naoz:2016}
S.~Naoz.
\newblock The eccentric kozai-lidov effect and its applications.
\newblock {\em Annual Review of Astronomy and Astrophysics}, 54(1):441--489,
  2016.

\bibitem{Newton:1687}
I.~Newton.
\newblock {\em Philosophiae naturalis principia mathematica}.
\newblock Maclehose, 1871.

\bibitem{MR1419468}
L.~Niederman.
\newblock Stability over exponentially long times in the planetary problem.
\newblock {\em Nonlinearity}, 9(6):1703--1751, 1996.

\bibitem{pinzari2009kolmogorov}
G.~Pinzari.
\newblock {\em On the {Kolmogorov} set for many-body problems}.
\newblock PhD thesis, Universit{}\`a degli Studi di Roma Tre, 2009.

\bibitem{poincare1892methodes}
H.~Poincar{\'e}.
\newblock {\em Les m{\'e}thodes nouvelles de la m{\'e}canique c{\'e}leste}.
\newblock Gauthier-Villars, 1892.

\bibitem{Robutel:1995}
P.~Robutel.
\newblock Stability of the planetary three-body problem. {II}. {KAM} theory and
  existence of quasiperiodic motions.
\newblock {\em Celestial Mech. Dynam. Astronom.}, 62(3):219--261, 1995.

\bibitem{Sitnikov60}
K.~Sitnikov.
\newblock The existence of oscillatory motions in the three-body problems.
\newblock {\em Soviet Physics. Dokl.}, 5:647--650, 1960.

\bibitem{Treschev04}
D.~Treschev.
\newblock Evolution of slow variables in a priori unstable hamiltonian systems.
\newblock {\em Nonlinearity}, 17(5):1803--1841, 2004.

\bibitem{Treschev:2012}
D.~Treschev.
\newblock Arnold diffusion far from strong resonances in multidimensional {\it
  a priori} unstable {H}amiltonian systems.
\newblock {\em Nonlinearity}, 25(9):2717--2757, 2012.

\bibitem{Xue:2014:4bp}
J.~Xue.
\newblock Arnold diffusion in a restricted planar four-body problem.
\newblock {\em Nonlinearity}, 27(12):2887--2908, 2014.

\bibitem{zhao2014quasi}
L.~Zhao.
\newblock Quasi-periodic solutions of the spatial lunar three-body problem.
\newblock {\em Celestial Mechanics and Dynamical Astronomy}, 119(1):91--118,
  2014.

\bibitem{Ziglin:1975}
S.~L. Ziglin.
\newblock Secular evolution of the orbit of a planet in a binary-star system.
\newblock 1975.

\end{thebibliography}
}
\Addresses

\end{document}